\documentclass[a4paper,11pt]{amsart}
\usepackage[utf8]{inputenc}
\usepackage{latexsym}
\usepackage{amsfonts}
\usepackage{amssymb}

\usepackage{default}
\usepackage[T1]{fontenc}
\usepackage{amsmath}
\usepackage{amssymb}
\usepackage{amsfonts}
\usepackage{amsthm}
\usepackage{indentfirst}
\usepackage{psfrag}
\usepackage{amscd,stmaryrd,latexsym}
\usepackage[margin=1.3in]{geometry}
\usepackage[pdftex]{graphicx}

\newtheorem{thm}{Theorem}[section]
\newtheorem{lem}{Lemma}[section]
\newtheorem{hyp}{Hypotheses}[section]
\newtheorem{cor}{Corollary}[section]

\newtheorem{Prop}{Proposition}[section]
\newtheorem*{St*}{Statement}

\newtheorem*{theorem*}{Theorem}
\theoremstyle{remark}

\theoremstyle{definition}
\newtheorem{df}{Definition}[section]

\theoremstyle{remark}
\newtheorem{oss}{Remark}[section]

\newcommand{\be}{\begin{equation}}
\newcommand{\ee}{\end{equation}}
\newcommand{\R}{\mathbb{R}}
\newcommand{\N}{\mathbb{N}}

\newcommand{\G}{\Gamma}

\newcommand{\Greg}{\text{gen-reg}\,}
\newcommand{\Reg}{\text{reg}\,}
\newcommand{\spt}[1]{\text{spt}\,\|#1\|}
\newcommand{\B}[1]{({\mathcal B{\emph #1}})}

\newcommand\res{\mathop{\hbox{\vrule height 7pt width .5pt depth 0pt
\vrule height .5pt width 6pt depth 0pt}}\nolimits}

\def\eps{\mathop{\varepsilon}}
\def\e{\mathop{\varepsilon}}
\def\k{\mathop{\kappa}}
\def\g{\mathop{\gamma}}

\def\Uc{\mathop{\mathcal{U}_{}}}
\def\Oc{\mathop{\mathcal{O}}}

\def\Hc{\mathop{\mathcal{H}}}

\def\th{\theta}

\def\b{\beta}
\def\t{\tau}
\def\d{\delta}

\def\e{\varepsilon}

\def\t{\tau}
\def\z{\zeta}

\def\s{\sigma}
\def\r{\rho}
\def\Om{\Omega}
\def\om{\omega}

\def\p{\partial}

\def\eps{\mathop{\varepsilon}}

\def\Om{\Omega}
\def\om{\omega}
\def\p{\partial}

\DeclareMathAlphabet{\mathscr}{OT1}{pzc}{m}{it}

\begin{document} 

\title[Stable CMC integral varifolds of codimension $1$]{Stable CMC integral varifolds of codimension $1$: regularity and compactness}
\author[C. Bellettini \& N. Wickramasekera ]{Costante Bellettini \& Neshan Wickramasekera}


\begin{abstract}
We give two structural conditions on a codimension $1$ integral $n$-varifold with first variation locally summable to an exponent $p>n$ that imply the following: whenever each orientable portion of the $C^{1}$-embedded part of the varifold (which is non-empty by the Allard regularity theory) is stationarity and the $C^{2}$-immersed part of it is stable with respect to the area functional for volume preserving deformations, its support, except possibly on a closed set of codimension $7$, is an immersed constant-mean-curvature (cmc) hypersurface of class $C^{2}$ that can fail to be embedded only at points where locally the support is the union of two $C^{2}$ embedded cmc disks with only tangential intersection. Both  structural conditions are necessary for the conclusions and involve only those parts of the varifold that are made up of embedded $C^{1, \alpha}$-regular pieces coming together in a regular fashion, making them easy to check in principle. We show also that any family of codimension 1 integral varifolds satisfying these structural and variational hypotheses as well as locally uniform mass and mean curvature bounds is compact in the varifold topology.  Our results generalize both the regularity theory \cite{WicAnnals} (for stable minimal hypersurfaces) and the regularity theory of Schoen--Simon (\cite{SS}, for hypersurfaces satisfying a priori a smallness hypothesis on the singular set in addition to the variational hypotheses). Corollaries of the main varifold regularity theorem are obtained for sets of locally finite perimeter, which generalize the regularity theory of  Gonzalez--Massari--Tamanini (\cite{GonzMassTaman}) for boundaries that locally minimize perimeter subject to the fixed enclosed volume constraint.

\end{abstract}

\maketitle
\tableofcontents
\section{Introduction} 

Our purpose here is to develop a local regularity and compactness theory  for a class of hypersurfaces (codimension 1 integral varifolds) that are stationary and stable on their regular parts with respect to the area functional for deformations that preserve the ``enclosed volume'' functional. 
There is a rich literature concerning local and global geometric consequences  of these variational hypotheses for a hypersurface of a smooth Riemannian manifold whenever the hypersurface is a priori assumed to be of class $C^{2}$.  For instance, it is well known that if a (piece of a) hypersurface is of class $C^{2}$,  then it is area-stationary for volume preserving deformations  if and only if  it is a constant-mean-curvature (CMC) hypersurface; and whenever the ambient manifold is simply connected and has constant sectional curvature, a compact $C^{2}$ hypersurface  is stationary and stable with respect to the area functional for volume preserving deformations if and only if it is a geodesic sphere (\cite{BarbDoCEsch}).  The main results of the present article give sharp local structural conditions implying  $C^{2}$ (and hence also higher) regularity, away from a small singular set, of a codimension 1 integral varifold having first variation locally summable to a power greater than its dimension and satisfying stationarity and stability hypotheses in the above sense.

The work described here is to be viewed as both a generalization,
to CMC hypersurfaces, of the regularity theory of \cite{WicAnnals}
(that if an $n$-dimensional minimal hypersurface has stable regular set and has no ``classical
singularities''---see definition ($\dag$) below---then it is regular except on a set of
Hausdorff dimension $n-7$, together with the associated compactness theory and
various applications) and a generalization of the work of Schoen--Simon (\cite{SS}) (that gives regularity and compactness conclusions as in the theorems here but subject to an a priori smallness assumption on the singular set in addition to the variational hypotheses). The proofs of the main results here, while making indispensable use of the estimates of \cite{SS} and the techniques of \cite{WicAnnals}, require also accounting for some new, considerable analytic difficulties which arise from the combination of the failure of the two-sided strong maximum principle (in contrast to the minimal hypersurface case) and the absence of any a priori size hypothesis on the singular set; moreover, there are some
subtleties in formulating the optimal set of hypotheses of the theorems. We shall elaborate on these aspects in the discussion that follows. In a forthcoming sequel \cite{BW2}, we shall generalize the main results obtained here even further, to include the setting where the scalar mean curvature of the hypersurface is prescribed by a given, appropriately regular non-negative function on the ambient manifold. This condition, much like the CMC case treated here, has a variational formulation. Also, in the present article we shall confine ourselves to ambient spaces that are open subsets of ${\mathbb R}^{n+1}$, deferring  to the sequel the discussion of the (routine) technical modifications necessary to extend the results to the case of general Riemannian ambient spaces. 
 In the case of CMC hypersurfaces  of a Euclidean ambient space, as considered here, the proofs of the main results not only allow for a more transparent exposition but also contain many of the main necessary geometric and analytic ingredients.

The theory developed here  gives two sharp  structural conditions, that are in principle easy to check, on a codimension 1 integral varifold $V$ of dimension $n \geq 2$ having first variation summable to an exponent $> n$ and satisfying (appropriate forms of) stationarity and stability hypotheses with respect to the area functional for volume preserving deformations so that: (i) these hypotheses imply that  $V$, possibly away from a much lower dimensional closed set $\Sigma$ of singularities, corresponds to a CMC hypersurface of class $C^{2}$ in the sense that away from $\Sigma$, the support of $V$ is locally either a single $C^{2}$ embedded CMC disk or precisely two $C^{2}$ embedded CMC disks with only tangential intersection along a set contained in an $(n-1)$-dimensional embedded submanifold, and (ii) any subcollection of  such hypersurfaces satisfying additionally uniform volume and mean curvature bounds is compact in the topology of measure-theoretic (i.e.\ varifold) convergence. The structural hypotheses and the precise form of the variational hypotheses, as well as the main conclusions we establish, are contained in Theorem~\ref{thm:mainregularity} (regularity theorem) and  Theorem~\ref{thm:compactness} (compactness theorem) below. Sections~\ref{intro-caccioppoli} and \ref{intro-varifolds} below contain a less technical discussion of these hypotheses and conclusions. (In particular, hypothesis (a) of the theorem in Section~\ref{intro-caccioppoli} gives the first structural hypothesis, namely, the absence of classical singularities, and hypothesis (a$^{\prime}$) of the theorem in Section~\ref{intro-varifolds} gives the second).

The results established here should be regarded as giving conditions implying ``embeddedness'' of a  stable codimension 1 CMC varifold away from a small set of genuine singularities, although our conclusion allows for two $C^{2}$ pieces of the varifold to intersect tangentially. The most general manner in which tangential intersection of $C^{2}$ pieces of an $n$-dimensional CMC hypersurface is possible is as permitted in our conclusion (i) above, i.e.\ along a set of dimension at most $(n-1)$. Such tangential intersection is in fact natural for CMC hypersurfaces when the mean curvature is non-zero; consider for instance two touching unit cylinders with parallel $(n-1)$-dimensional axes in ${\mathbb R}^{n+1}$, which is an example at one extreme where the touching set is $(n-1)$-dimensional, or two touching unit spheres, an example at the other extreme with just a single touching point. Of course without the full freedom of such intersection,   
no compactness assertion as in (ii) above can be true (consider e.g.\ two disjoint half-cylinders with equal radii coming together), and the usefulness of the theory would in principle be limited.   
Note  that since one of our structural hypotheses (the absence of classical singularities; see Section~\ref{intro-caccioppoli} below) rules out transverse self intersections of the hypersurface, it follows from  the maximum principle that three distinct $C^{2}$ pieces of a CMC hypersurface as in our theorems cannot have a common point.

 There is a rich variational theory of minimal hypersurfaces in Riemannian manifolds that has been developed over the past seven or so decades. In that theory, understanding regularity and compactness properties of stable minimal hypersurfaces has been indispensable for establishing existence of optimally regular minimal hypersurfaces. In particular, a recent approach to this existence theory (established through the combined works of Guaraco (\cite{Gua}), Hutchinson--Tonegawa (\cite{HutchTon}) and Tonegawa--Wickramasekera (\cite{TonWic})) shows that having at one's disposal a sharp regularity theory for stable minimal hypersurfaces (as in \cite{WicAnnals}) makes it possible to reduce the construction part of the theorem to a standard PDE mountain pass lemma, replacing the varifold min-max construction in the original Almgren--Pitts--Schoen--Simon approach. See Section~\ref{minimal} below for a brief discussion on this. The success of this PDE approach in that setting naturally leads to the question whether a similar theory for hypersurfaces of more generally prescribed mean curvature could be developed. We shall address this question in forthcoming work (\cite{BW3}). An essential step in such a theory is to develop sufficiently strong regularity and compactness theorems for the corresponding \emph{stable} solutions.  Our work here and in the sequel \cite{BW2} provide these. We remark that the results established in the present work and in \cite{BW2} however require no assumption that is specific to any existence construction; the work in fact produces considerably general local results that might conceivably be applied in a variety of different situations.

 \subsection{Caccioppoli sets} \label{intro-caccioppoli}
In the most general version of our main regularity and compactness theorems  (Theorem~\ref{thm:mainregularity} and Theorem~\ref{thm:compactness}), a hypersurface means a codimension 1 integral $n$-varifold whose first variation is absolutely continuous with respect to its weight measure and whose generalized mean curvature is locally in 
 $L^{p}$ for some $p > n.$ In that generality however the meaning of the notions of enclosed volume and volume preserving deformations is not immediately clear, and these notions need to be defined appropriately. Moreover, the theorem in that generality allows for multiplicity $> 1$. Before discussing these general theorems, it is perhaps instructive to mention a special case (the theorem below, which is a mildly imprecise re-statement of Corollary~\ref{cor:CaccioppoliwithoutLp}) which is simpler to state and yet involves a natural setting for CMC hypersurfaces---namely, that of boundaries of sets of locally finite perimeter, known also as Caccioppoli sets---in which it is clear what enclosed volume means. Moreover, as it turns out, in this setting only one of the two structural hypotheses is necessary. 
 
 Let $n \geq 2$. Recall that by definition, a subset $E$ of $\R^{n+1}$ is a Caccioppoli set if $E$ is ${\mathcal H}^{n+1}$ measurable and its characteristic function $\chi_{E} \in BV_{\rm loc}(\R^{n+1}).$ Thus if $E$ is a Caccioppoli set in $\R^{n+1}$, it follows from the Riesz representation theorem that  there is a Radon measure on $\R^{n+1}$, denoted $|D\chi_{E}|$, and a $|D\chi_{E}|$-measurable vector field $\nu_{E}$ on $\R^{n+1}$ with $|\nu_{E}| = 1$ $|D\chi_{E}|$-a.e.\ on $\R^{n+1}$ 
 satisfying $\int_{E} {\rm div} \, g \, d{\mathcal H}^{n+1} = -\int_{\R^{n+1}} g \cdot \nu_{E} \, |D\chi_{E}|$ for every smooth compactly supported vector field $g$ on $\R^{n+1}$.  (Thus in case $E$ is a $C^{1}$ domain, by the divergence theorem $|D\chi_{E}|$ is just ${\mathcal H}^{n} \res \partial \, E$ and $\nu_{E}$ is the unit normal to $\partial \, E$ pointing into $E$; in general, ${\rm spt} \, |D\chi_{E}|$ is to be thought of as the generalized boundary of $E$, and $\nu_{E}$ as the generalized unit normal to the generalized boundary pointing into $E$.) 
 
 For $\lambda \in {\mathbb R}$ and ${\mathcal O}$ an open subset of $\R^{n+1}$ with compact closure, let
$$J_{\mathcal O}(E) = |D\chi_{E}|({\mathcal O}) + \lambda {\mathcal H}^{n+1}(E \cap {\mathcal O})$$
for Caccioppoli sets $E$ in $\R^{n+1}$. Note that stationarity of $E$ with respect to $J_{\mathcal O}$ for some $\lambda$ and arbitrary ambient deformations fixing $E$ outside ${\mathcal O}$ is equivalent to stationarity of $E$ with respect to  the perimeter functional ($= |D\chi_{E}|({\mathcal O})$) for deformations  that fix $E$ outside ${\mathcal O}$ and preserve the enclosed volume ${\mathcal H}^{n+1}(E \cap {\mathcal O}).$ 

The first of the two structural hypotheses (and the only one needed in the setting of Caccioppoli sets), namely hypothesis (a) of the theorem below, requires the following general definition:
 
\begin{enumerate}
\item[($\dag$) ]A \emph{classical singularity} of a set $W \subset \R^{n+1}$ is a point $p \in W$ about which there is a neighborhood $U$ such that $W \cap \, U$ is, for some $\alpha \in (0, 1)$,  the union of 
 three or more embedded $C^{1, \alpha}$ hypersurfaces-with-boundary sharing a single common boundary containing $p,$ meeting pairwise only along the common boundary and with at least one pair meeting transversely.
\end{enumerate}

In order to formulate the stability assumption we need the following notion of volume-preserving deformation for immersions. Let $\iota:S \to \R^{n+1}$ denote a smooth immersion of an $n$-dimensional orientable manifold $S$ into $\R^{n+1}$. For $\Oc \subset \subset \R^{n+1}$, let $\psi(t,x)=\psi_t(x)$ (for $t\in(-\eps,\eps)$ for some $\eps > 0$ and $x\in S$) denote a one-parameter family of immersions (smooth on $(-\eps,\eps) \times S$) such that $\psi_0=\iota$ and $\psi_t(x)=\iota(x)$ for all $t\in(-\eps, \eps)$ and all $x \in S \setminus \iota^{-1}(\Oc)$. Let 
\begin{equation*}\label{enclosedvolume}
{\rm Vol} \, (t) = \int_{[0,t] \times S} \psi^*d\text{vol}
\end{equation*}
where $d\text{vol}$ denotes the usual volume form on $\R^{n+1}$.
We say that $\psi_t$ is a \emph{volume-preserving deformation of $\iota(S)$ in $\Oc$ as an immersion} if

\begin{equation}
 \label{eq:volumepreservingimmersions}
{\rm Vol}^{\prime}(t) =0 \text{ for all } t\in(-\eps,\eps).
\end{equation}

\medskip

 \noindent
 \emph{{\bf Theorem} (Corollary~\ref{cor:CaccioppoliwithoutLp}). For $n \geq 2$, let $E$ be a Caccioppoli set in $\R^{n+1}$ and $\Uc$ be an open set, with ${\mathcal H}^{n+1}(E\cap \Uc) >0.$ Let $\lambda \in {\mathbb R}$ be a constant. Suppose  that: 
 \begin{enumerate}
 \item[(a)] no point $p \in {\rm spt} \, |D\chi_{E}| \cap \Uc$ is a classical singularity of ${\rm spt} \, |D\chi_{E}|$;
 \item[(b)] For each open set ${\mathcal O}$ with compact closure in $\Uc$, $E$ is stationary with respect to the functional  $J_{\mathcal O}(\cdot)$ for ambient deformations that fix $E$ outside ${\mathcal O}$ and 
 \item[(c)] For each open set ${\mathcal O}$ with compact closure in $\Uc$, the smoothly immersed part $M$ of ${\rm spt} \, |D\chi_{E}|$ (which by (b) is CMC) is stable with respect to the area functional for volume-preserving deformations of $M$ in $\Oc$ as an immersion.
\end{enumerate} 
 Then there is a closed set 
 $\Sigma \subset {\rm spt} \, |D\chi_{E}|$ with $\Sigma = \emptyset$ if $n \leq 6$, $\Sigma$ discrete if $n=7$ and ${\rm dim}_{\mathcal H} \, (\Sigma) \leq n-7$ if $n \geq 8$ such that: 
 \begin{enumerate}
 \item [(i)] locally near each point  $p \in ({\rm spt} \, |D\chi_{E}| \cap \Uc) \setminus \Sigma$, either ${\rm spt} \, |D\chi_{E}|$ is a single smoothly embedded disk or ${\rm spt} \, |D\chi_{E}|$ is precisely two smoothly embedded disks with only tangential intersection along a subset contained in a smooth $(n-1)$-dimensional submanifold, and 
\item[(ii)] the mean curvature $H$ of $({\rm spt} \, |D\chi_{E}| \cap \Uc) \setminus \Sigma$ is given by 
 $H = - \lambda  \nu_{E}$ where $\nu_{E}$ is the unit normal to ${\rm spt} \, |D\chi_{E}|$ pointing into $E$.
 \end{enumerate}}

By virtue of the assumption that $E$ is a Caccioppoli  set, it follows from the well known structure theorem of De~Giorgi (\cite{DG1}, \cite{DG2}; see also \cite{Giusti}, \cite{Maggi}) that $|D\chi_{E}|$ is $n$-rectifiable, i.e.\ $|D\chi_{E}| = {\mathcal H}^{n} \res \partial^{\star} \, E$ where $\partial^{\star} \, E$ (the reduced boundary of $E$) is an $n$-rectifiable set having a multiplicity 1 tangent hyperplane at every point. 
Since $E$ is a stationary point of $J_{\mathcal O}$ for every open set ${\mathcal O} \subset\subset \R^{n+1}$, it follows (see the discussion in Remark \ref{oss:LpmeancurvCaccioppoli} below) that the first variation of the multiplicity 1 varifold $V$ associated with $|D\chi_{E}|$ is absolutely continuous with respect to its weight measure $\|V\|$ (= $|D\chi_{E}|$) and that the generalized mean curvature of $V$ is equal to $-\lambda \nu_{E}$. Thus Allard's regularity theorem (\cite{Allard})  implies that the singular set of ${\rm spt} \, |D\chi_{E}|$ (i.e.\ the set of points of ${\rm spt} \, |D\chi_{E}|$ where ${\rm spt} \, |D\chi_{E}|$  is not smoothly embedded) has zero $n$-dimensional Hausdorff measure. What is new in the above theorem is that, if additionally the $C^{2}$ immersed part  of ${\rm spt} \, |D\chi_{E}|$ 
is stable with respect to area for volume preserving variations, and if ${\rm spt} \, |D\chi_{E}|$  has no classical singularities, then the singular set of ${\rm spt} \, |D\chi_{E}|$ decomposes as  the disjoint union of the set of points near which ${\rm spt} \, |D\chi_{E}|$ consists locally of two smoothly embedded CMC disks intersecting tangentially and a closed set of codimension $\geq 7.$ 

Although the multiplicity of the hypersurface in the above theorem is 1 a.e., its proof is not much simpler than the proof of our more general varifold regularity result, Theorem~\ref{thm:mainregularity} (reproduced in Section~\ref{intro-varifolds} below). This is because even in the above special case, a multiplicity 2 tangent hyperplane can arise (e.g.\ as in the case of pieces of two touching unit spheres or unit cylinders in Euclidean space), and the occurrence of tangent hyperplanes with multiplicity $\geq 3$ cannot a priori be ruled out. It is part of the conclusion that there are no tangent hyperplanes with multiplicity $\geq 3$, and that a multiplicity 2 tangent hyperplane can only occur at a point where two $C^{2}$ pieces of the hypersurface meet tangentially (in particular, tangent cones along $\Sigma$ are non-planar).
 
\medskip

In view of the topology induced on the space of Caccioppoli sets by its embedding into $L^{1}_{\rm loc}$, there is a different and yet very natural notion of stationarity for Caccioppoli sets with respect to the functional 
$J_{\mathcal O}$. Although this stationarity condition is in fact stronger than the one assumed in the preceding theorem, its advantage is that it automatically rules out classical singularities, and moreover allows us to assume stability only on the smoothly \emph{embedded} part (and hence stability needs to be checked only for \textit{ambient} volume-preserving deformations). Let us now describe the deformations this stationarity condition entails and give the precise statement of the result it implies.

Let $E$ be a Caccioppoli set in $\R^{n+1}$. For each open set ${\mathcal O} \subset {\mathbb R}^{n+1}$ with compact closure, consider a one parameter family of sets $\{E_t\}_{t\in[0, \eps)}$ such that $E_t$ is a Caccioppoli set for each $t \in [0, \eps)$ with the properties:

\noindent $\bullet$ $E_0=E$; $E_t  \cap (\R^{n+1} \setminus {\mathcal O})  = E \cap (\R^{n+1} \setminus {\mathcal O})$ for all $t \in [0, \eps)$; and ${\mathcal H}^{n+1}(E_t\cap \mathcal{O})={\mathcal H}^{n+1}(E \cap \mathcal{O})$ for all $t \in [0, \eps)$;

\noindent $\bullet$ the map $t \in [0, \eps)\to \chi_{E_t}$ is continuous, where the topology on the characteristic functions $\chi_{E_t}$ is the one induced by the embeddeding in $L^1_{\text{loc}}$; moreover the associated map $t \in [0, \eps)\to J_{\mathcal O}(E_t)$ is differentiable from the right at $t=0$.

We refer to such a family $\{E_t\}_{t\in[0, \eps)}$ as a \emph{one-sided one-parameter volume-preserving family of deformations in $\mathcal{O}$ with respect to the $L^1_{\text{loc}}$-topology}.

\medskip

\noindent
 \emph{{\bf Theorem} (Corollary~\ref{cor:Caccioppolianalytic}). For $n \geq 2$, let $E$ be a Caccioppoli set in $\R^{n+1}$ and $\Uc$ be an open set, with ${\mathcal H}^{n+1}(E\cap \Uc) >0.$ Let $\lambda \in {\mathbb R}$ be a constant.  Suppose  that: 
 \begin{enumerate}
 \item[($b^{\prime}$)] For each open set ${\mathcal O}$ with compact closure in $\Uc$ and for each one-sided one-parameter volume-preserving family of deformations in $\mathcal{O}$ with respect to the $L^1_{\text{loc}}$-topology, we have that $\left.\frac{d}{dt}\right|_{t=0^+}J_{\mathcal O}(E_t)\geq 0$;
 \item[$(c^{\prime})$] the smoothly embedded part of ${\rm spt} \, |D\chi_{E}|$ is stable with respect to $J_{\mathcal O} (\cdot)$ for ambient deformations that fix $E$ outside ${\mathcal O}$ and preserve ${\mathcal H}^{n+1}(E \cap {\mathcal O}).$ 
\end{enumerate} 
 Then there is a closed set 
 $\Sigma \subset {\rm spt} \, |D\chi_{E}|$ with $\Sigma = \emptyset$ if $n \leq 6$, $\Sigma$ discrete if $n=7$ and ${\rm dim}_{\mathcal H} \, (\Sigma) \leq n-7$ if $n \geq 8$ such that: 
 \begin{enumerate}
\item [(i)] locally near each point  $p \in ({\rm spt} \, |D\chi_{E}| \cap \Uc) \setminus \Sigma$, either ${\rm spt} \, |D\chi_{E}|$ is a single smoothly embedded disk or ${\rm spt} \, |D\chi_{E}|$ is precisely two smoothly embedded disks with only tangential intersection along a subset contained in a smooth $(n-1)$-dimensional submanifold, and 
\item[(ii)] the mean curvature $H$ of $({\rm spt} \, |D\chi_{E}| \cap \Uc) \setminus \Sigma$ is given by 
 $H = - \lambda  \nu_{E}$ where $\nu_{E}$ is the unit normal to ${\rm spt} \, |D\chi_{E}|$ pointing into $E$.
 \end{enumerate}}
 
\medskip

We point out the following: Let $X\in C^1_c(\Oc, \R^{n+1})$ and $\psi_t:(-\eps,\eps)\times \Oc \to \Oc$ be a one-parameter family of ambient diffeomorphisms with $\psi_0=Id$ and $\left.\frac{d}{dt}\right|_{t=0} \psi_t=X$ and consider $t \in (-\eps, \eps)\to \chi_{E_t}$, where $E_t=\psi_t(E)$; then both $t \in [0, \eps)\to \chi_{E_t}$ and $t \in [0, \eps)\to \chi_{E_{-t}}$ are one-sided one-parameter volume-preserving families of deformations in $\mathcal{O}$ with respect to the $L^1_{\text{loc}}$-topology. Since in this case $t \in (-\eps, \eps)\to J_{\Oc}(E_t)$ is differentiable at $t=0$, it is immediate that assumption $(b^{\prime})$ implies, in particular, the stationarity of $E$ with respect to $J_{\Oc}$ for all ambient deformations. The importance of $1$-sided deformations allowed in $(b^\prime)$ lies in the fact that in the presence for example of a classical singularity the deformations (in $L^1_{\text{loc}}$) that preserve the class of Caccioppoli sets are naturally $1$-sided and, in such cases, the stationarity condition should entail ``not decreasing area to first order'' (hence the inequality).

The preceding theorem generalizes the result by Gonzales--Massari--Tamanini \cite{GonzMassTaman} that established regularity of boundaries that \emph{minimize} area subject to the fixed enclosed volume constraint. 
The natural generalizations of the two theorems above to the case of ambient Riemannian manifolds will be discussed in \cite{BW2}.

\subsection{More general varifolds} \label{intro-varifolds}

The most general setting in which  the concepts of first variation of $n$-dimensional area and area-stationarity can be understood is that of $n$-varifolds. Amongst $n$-varifolds, the space ${\mathcal V}_{n, p}$ of integral $n$-varifolds $V$ whose first variation is absolutely continuous with respect to the weight measure $\|V\|$ and generalized mean curvature $H_{V}$ is in $L^{p}_{\rm loc} (\|V\|)$ for some $p > n$ is amply general for the study of many geometric variational problems.  As established by the fundamental regularity theory of Allard (\cite{Allard}), there is an embryonic control of singularities of varifolds in ${\mathcal V}_{n, p}$ that allows one to directly extend classical geometric constraints  (that may, for instance, arise from variational conditions) to $V \in {\mathcal V}_{n, p}$, albeit on an a priori small part of $V$, namely, the regular part of ${\rm spt} \, \|V\|$;  indeed, by Allard's regularity theorem, if $V \in {\mathcal V}_{n, p}$ then the (open) subset ${\rm reg}_{1} \, V$ of points of ${\rm spt} \, \|V\|$ near which ${\rm spt} \, \|V\|$ is a an embedded $C^{1}$ submanifold is dense in ${\rm spt} \, \|V\|$, and ${\rm reg}_{1} \, V$ is in fact of class $C^{1, 1-\frac{n}{p}}$ (where $n < p < \infty$). The hypothesis $H_{V} \in L^{p}_{\rm loc} \, (\|V\|)$ however is not strong enough to give any control of the size (Hausdorff measure) of the singular set ${\rm spt} \, \|V\|  \setminus {\rm reg}_{1} \, V$; there is in fact a well known example due to Brakke (\cite{Brak}) of an integral 2-varifold in ${\mathbb R}^{3}$ with (variable) generalized mean curvature in $L_{\rm loc}^{\infty}$ and a singular set of positive 2-dimensional Hausdorff measure. 
 
 Our general regularity theorem (Theorem~\ref{thm:mainregularity}) is formulated and proved for varifolds $V \in {\mathcal V}_{n, p}$.  Although as mentioned above its proof does not require much more effort than the proof of the theorem above for Caccioppoli sets, there is some subtlety involved in the formulation of its hypotheses so that they, while being not too restrictive, still guarantee the conclusion that the hypersurface is ``classical,'' i.e.\ is of class $C^{2}$ (in the same sense as in the theorems above, allowing two $C^{2}$ pieces to touch) away from an $(n-7)$-dimensional closed set of genuine singularities. In particular, a second structural hypothesis (($\dag\dag$) below) is necessary. There are in fact two important aspects with regard to  the hypotheses of the varifold version of the theorem that are not apparent in the setting of Caccioppoli sets. Rather than reproducing a full, precise statement of Theorem~\ref{thm:mainregularity} here, let us just highlight these two main points and give a slightly informal statement of the theorem:

 \begin{enumerate}
\item[(i)]  First, in light of the assumption that the first variation of $V$ is locally summable to an exponent $p >n$, the stationarity requirement (i.e.\ the analogue of hypothesis (b) in the first theorem in Section~\ref{intro-caccioppoli})  needs to be  imposed only on ${\rm reg}_{1} \, V$. This stationarity requirement precisely is the following:  On every orientable portion of ${\rm reg}_{1} \, V$, there exists a choice of orientation such that  that portion is stationary with respect to area for deformations preserving the enclosed volume (in the sense that (\ref{eq:volumepreservingimmersions}) holds, or equivalently, with the enclosed volume taken to be defined by (\ref{enclosed-vol}) of Section~\ref{main-thms} with respect to the chosen orientation). See hypothesis 1 of Theorem~\ref{thm:mainregularity} (or  hypothesis (b) in the theorem below). A posteriori this orientation is determined, up to sign (independent of the connected component of the hypersurface), by the mean curvature vector. Existence a priori of such an orientation is necessary in order to make $C^{2}$ regularity conclusions, as we do, away from a subset of codimension $\geq 2$, as shown by Figure \ref{fig:C11graph} in Section~\ref{examples}. Note  that whenever $V$ is the varifold defined by $|D\chi_{E}|$ $(= {\mathcal H}^{n} \res \partial^{\star} \, E)$ for some Caccioppoli set $E$ as in Section \ref{intro-caccioppoli}, this stationarity hypothesis is implied by hypothesis (b) of the theorem in Section \ref{intro-caccioppoli}) since in that case we have that ${\rm reg}_{1} \, V = \partial^{\star} \, E$, $\nu_{E}$ is an orienting unit normal to ${\rm reg}_{1} \, V$ and by the divergence theorem, ${\mathcal H}^{n+1}(E) =  \frac{1}{n+1}\int_{\partial^{\star} \, E} x \cdot \nu_{E} \, d{\mathcal H}^{n}$.

\item[(ii)] In light of example in Figure \ref{fig:C11_Regularity} in Section~\ref{examples}, it is necessary to make an additional hypothesis in the varifold setting (in addition to the no-classical-singularities assumption, stationarity and stability) in order for the $C^{2}$ regularity conclusions to hold, as asserted, away from a singular set of codimension $\geq 2$.  Of course the additional hypothesis must automatically be satisfied in the setting of Caccioppoli sets, but note that the example referred to above shows that even when ${\mathcal H}^{n}({\rm spt} \, \|V\| \setminus {\rm reg}_{1} \, V) = 0$ (which is automatic for $J_{\mathcal O}$-stationary Caccioppoli sets, as pointed out above), an additional hypothesis is necessary. As it turns out, this additional hypothesis takes the form of a second structural condition; just as with the no-classical-singularities hypothesis, it requires verification of a property only in regions where the entire structure of the varifold is given by (two) $C^{1, \alpha}$ hypersurfaces. Specifically, this hypothesis says the following: 

\begin{enumerate}
\item[($\dag\dag$)] Whenever a point 
 $p \in {\rm spt} \, \|V\| \setminus {\rm reg}_{1} \, V$ that is not a classical singularity of ${\rm spt} \, \|V\|$ has a neighborhood $U$ in which ${\rm spt} \, \|V\|$ is, for some $\alpha \in (0, 1)$, the union of two embedded $C^{1, \alpha}$ hypersurfaces of $U$ (i.e.\ whenever $p$ is a ``touching singularity;'' see Definition~\ref{df:touchingsingularity}), $p$ has a possibly smaller neighborhood $\widetilde{U}$ such that ${\mathcal H}^{n} \, (\{q \in \widetilde{U} \, : \, \Theta \,(\|V\|, q) = \Theta \, (\|V\|, p)\}) = 0.$ 
 \end{enumerate}
Here $\Theta \, (\|V\|, p)$ denotes the density of $V$ at $p$. Since a point $p$ as in ($\dag\dag$) satisfies $\Theta \, (\|V\|, p) \geq 2$, this hypothesis is redundant when $V$ corresponds to $|D\chi_{E}|$ for some Caccioppoli set $E$ 
because in that case $\Theta \,(\|V\|, x) = 1$ for every $x \in {\rm reg}_{1} \, V$ and ${\mathcal H}^{n} \, ({\rm spt} \, \|V\| \setminus {\rm reg}_{1} \, V) = 0.$ 
\end{enumerate}
 
 Our general regularity and compactness theorem can now be stated, albeit a little imprecisely, as follows (see Theorems~\ref{thm:mainregularity} and ~\ref{thm:compactness} for the precise statements):\\    

\noindent
\emph{{\bf Theorem.} Let $V$ be an integral $n$-varifold ($n \geq 2$) in an open set $\Uc \subset \R^{n+1}$, whose first variation is locally bounded and absolutely continous with respect to $\|V\|$ and whose generalized mean curvature is in $L^{p}_{\rm loc} \, (\|V\|)$ for some $p >n$, and 
let $\lambda \in {\mathbb R}$ be a constant.  Let ${\rm reg}_{1} \, V$ denote the $C^{1}$ embedded part of ${\rm spt} \, \|V\|$ (which by Allard's regularity theorem is a dense open subset of ${\rm spt} \, \|V\|$). Suppose that: 
\begin{enumerate}
\item[(a)] no point of ${\rm spt} \, \|V\|$ is a  classical singularity of ${\rm spt} \, \|V\|$; 
\item[(a$^{\prime}$)] $V$ satisfies {\rm($\dag\dag$)}; 
\item[(b)] for each open set ${\Oc} \subset\subset \Uc \setminus ({\rm spt} \, \|V\| \setminus {\rm reg}_{1} \, V)$ such that ${\rm reg}_{1} \, V \cap {\Oc}$ is orientable,  $V \res {\Oc}$ is stationary with respect to the functional $J_{\Oc} \, : \, IV_{n}({\Oc}) \to {\mathbb R}$ given by  
$$J_{\Oc}(W)=\|W\|(\Oc)+\lambda\text{vol}_{\Oc}(W)$$
for any ambient deformation that only moves  $V \res {\Oc}$, where $\text{vol}_{\Oc} \, (\cdot)$ is the volume enclosed by $V \res {\Oc}$ relative to a choice of orientation on ${\rm reg}_{1} \, V \cap {\Oc}$; 
\item[(c)] the $C^{2}$ immersed part $\text{gen-reg} \, V$ of ${\rm spt} \, \|V\|$ (which is a classical CMC immersion by (b)) is stable (as an immersion) with respect to the functional $J_{\Oc}(\cdot)$ (on multiplicity 1 immersions) for any volume preserving deformation that only moves a compact region of $\text{gen-reg} \, V.$  
\end{enumerate}
Then there is a closed set 
 $\Sigma \subset {\rm spt} \, \|V\|$ with $\Sigma = \emptyset$ if $n \leq 6$, $\Sigma$ discrete if $n=7$ and ${\rm dim}_{\mathcal H} \, (\Sigma) \leq n-7$ if $n \geq 8$ such that: 
 \begin{enumerate}
\item [(i)] locally near each point  $p \in {\rm spt} \, \|V\| \setminus \Sigma$, either ${\rm spt} \, \|V\|$ is a single smoothly embedded disk or ${\rm spt} \, \|V\|$ is precisely two smoothly embedded disks with only tangential intersection along a subset contained in a smooth $(n-1)$-dimensional submanifold, and 
\item[(ii)] ${\rm spt} \, \|V\| \setminus \Sigma$ is an orientable immersion and there is a continuous choice of unit normal $\nu$ on ${\rm spt} \, \|V\| \setminus \Sigma$ such that the mean curvature $H_{V}$ of ${\rm spt} \, \|V\| \setminus \Sigma$ is given by $H_{V} =  \lambda  \nu$.
 \end{enumerate}
 Moreover, if $(V_j)$ is a sequence of integral $n$-varifolds in $\Uc$ and $(\lambda_{j})$ is a sequence of real numbers satisfying the above hypotheses with 
 $V_{j}$ in place of $V$ and $\lambda_{j}$ in place of $\lambda$, and if $\limsup_j (|\lambda_{j}| + \|V_j\| (K))<\infty$ for each compact set $K \subset \Uc$, then there exist an integral $n$-varifold $V$ in $\Uc$ and a number $\lambda \in {\mathbb R}$ satisfying the above hypotheses, and a subsequence $({j'})$ such that  $\lambda_{j^{\prime}} \to \lambda$ and $V_{j^{\prime}} \to V$ as varifolds in $\Uc$.}\\ 

\noindent
\emph{{\bf Remark}}. In fact a weaker stability assumption than (c) will suffice: we will need only a special type of volume-preserving deformations as immersions (see Theorem \ref{thm:mainregularity}). An equivalent formulation of this condition can be given by the so-called weak stability inequality, see (\ref{eq:weakstabilityinequality}).

\medskip

\noindent
\emph{{\bf Remark}}. In \cite{BCW} we expoit the above theorem and prove, for the class of weakly stable CMC hypersurfaces with bounds on the area and on the mean curvature, a priori curvature estimates for $n \leq 6$ and, under an additional necessary flatness assumption, sheeting theorems for arbitrary $n$.

\medskip
 
We emphasize that apart from the requirement that the generalized mean curvature $H_{V} \in L^{p}_{\rm loc}$ for some $p >n$,  each of the hypotheses of the above theorem is a condition on a part of the varifold $V$ where its regularity, at least of class $C^{1, \alpha}$ in some form, is known. Specifically, each of the two structural hypotheses rules out or controls a type of singularity that is formed when $C^{1, \alpha}$ embedded pieces of the varifold come together in a regular fashion; and the two variational hypotheses are, likewise, required only on the regular parts of the varifold---stationarity only on the $C^{1, 1 - \frac{n}{p}}$ embedded part ${\rm reg}_{1} \, V$, and stability only on the $C^{2}$ immersed part.  This is a very useful feature of the theorem because it makes these hypotheses easy to check in principle.  Beyond these requirements no hypothesis is necessary concerning the singular set, and the theorem guarantees lower dimensionality of the singular set. For an arbitrary varifold $V$ satisfying $H_{V} \in L^{p}_{\rm loc}$ for some $p > n,$ clearly no such conclusion is possible in view of Brakke's example referred to above. What is surprising is that it suffices to impose a set of hypotheses just on the ``regular parts'' of such a varifold $V$ in order to infer optimal size control of its singular set.

\subsection{Minimal hypersurface theory: the Allen--Cahn construction}\label{minimal}
The present work generalises the work \cite{WicAnnals} which established an analogous regularity and compactness theory for stable minimal hypersurfaces---more precisely, an analogous theory for codimension 1 integral $n$-varifolds $V$ that have no classical singularities; that are stationary with respect to the area functional for (unconstrained) ambient deformations that fix the region outside a compact subset; and that have stable regular parts ${\rm reg}_{1} \, V$ with respect to area for unconstrained ambient deformations that only move compact regions of ${\rm reg}_{1} \, V$. The work \cite{WicAnnals} showed that whenever these hypotheses are satisfied, ${\rm spt} \, \|V\|$ is smoothly embedded away from a closed singular set of Hausdorff dimension $\leq n-7$ which is empty if $n \leq 6$ and discrete if $n=7.$  (Note in particular that the present work in fact shows that the stationarity assumption in \cite{WicAnnals}, which is equivalent to the requirement that the generalized mean curvature $H_{V} = 0$ everywhere, can be weakened to the combined requirement that $H_{V} \in L^{p}_{\rm loc}$ for some  $p > n$ and the $C^{1}$ embedded part ${\rm reg}_{1} \, V$ be stationary; moreover, the stability condition can be relaxed to weak stability, i.e.\ stability (of ${\rm reg}_{1} \, V (= {\rm reg}\,  V$)) for volume preserving deformations.) 

The regularity theory of \cite{WicAnnals}  has subsequently been used to give a new proof of the celebrated existence theorem for embedded minimal hypersurfaces in compact Riemannian manifolds. This theorem asserts that in any given $(n+1)$-dimensional compact Riemannian manifold $N$ with $n \geq 2$, there is a closed embedded minimal hypersurface $M$ with a possible singular set whose Hausdorff dimension is $\leq n-7$. This result was first established by the combined work of Almgren (\cite{Alm1}), Pitts (\cite{Pitts}) and Schoen--Simon (\cite{SS}) in the early 1980's. The original proof was based on a geometric min-max construction due to Pitts (\cite{Pitts}) that refined earlier work of Almgren (\cite{Alm1}), giving a stationary integral $n$-varifold $V$ with a special ``almost minimizing'' property with respect to the area functional. This almost minimizing property makes it possible for the regularity of $V$ to be inferred from the compactness theory of Schoen and Simon (\cite{SS})  for stable minimal hypersurfaces with small singular sets. This approach of using a varifold min-max construction was subsequently adapted by Simon--Smith \cite{SimonSmith} to construct minimal 2-spheres in the 3-sphere with an arbitrary Riemannian metric. Both the general Almgren--Pitts argument and the Simon--Smith argument have been streamlined in the more recent works of De~Lellis--Tasnady \cite{DLTasnady}  and of Colding--De Lellis \cite{ColdingDL} respectively. The method has also been adapted in a very recent paper of Zhou--Zhu \cite{ZZ} which asserts the existence of a CMC hypersurface of prescribed mean curvature in dimensions $n$ with $2\leq n \leq 6$. 

In contrast to the direct varifold min-max arguments as in these works, the new proof of existence of minimal hypersurfaces is more PDE theoretic and is based on the basic idea of obtaining the minimal hypersurface as a weak limit of level sets of solutions to the (elliptic) Allen--Cahn equation on $N$. The first step of the argument is the work of Tonegawa and the second author (\cite{TonWic}), in which the regularity theory of \cite{WicAnnals} in its full strength is used to establish regularity of the minimal hypersurfaces $V_{\rm ac}$---the Allen--Cahn minimal hypersurfaces---arising as weak limits of level sets of stable solutions to  Allen--Cahn equations with perturbation parameters $\epsilon_{j} \to 0^{+}$; earlier work of Hutchinson--Tonegawa (\cite{HutchTon}) and  Tonegawa (\cite{Ton}) had established the existence of varifold limits of the level sets. The second step is the recent work of Guaraco (\cite{Gua}) that produces, by an elegant, simple PDE argument, an approriate solution $u_{\epsilon}$ of the $\epsilon$-Allen--Cahn equation on $N$ for every small $\epsilon>0;$ this construction is based on a standard PDE mountain pass lemma  and it produces a smooth solution $u_{\epsilon}$ such that the Morse index of $u_{\epsilon}$ (with respect to the Allen--Cahn energy functional) is bounded by 1, and the Allen--Cahn energy of $u_{\epsilon}$ is bounded above and away from zero independently of $\epsilon$, guaranteeing in particular the non-triviality of the limit varifold $V_{\rm ac}$ corresponding to a sequence $u_{\epsilon_{j}}$ with $\epsilon_{j} \to 0^{+}$. The desired regularity of the minimal hypersurface $M = {\rm spt} \, \|V_{\rm ac}\|$ follows in a straightforward manner by applying the result of \cite{TonWic} in a small arbitrary ball $B_{\rho}(x)$ or in $N \setminus \overline{B_{\rho}(x)}$ (in one of which regions a subsequence of $u_{\epsilon_{j}}$ must be stable, since the Morse index of $u_{\epsilon_{j}}$ is at most 1), where $x \in {\rm spt} \, \|V_{\rm ac}\|$ is arbitrary (see \cite{Gua}).

There are two important aspects of this new proof. First, it avoids the intricate Almgren--Pitts min-max construction, used in the original proof, that was carried out directly for the area functional on the space of codimension 1 integral cycles on $N$; in its place, the new proof uses a much simpler PDE min-max construction implemented in a Hilbert space, namely in $W^{1, 2}(N),$ giving $u_{\epsilon}$ as above. In particular, the uniform Morse index bound on $u_{\epsilon}$ removes the necessity of anything like an almost minimizing property to reduce the regularity question to that of stable hypersurfaces. The end result is a striking gain in simplicity on the part of the construction of a stationary varifold, whose justification---and this is the second key aspect---requires a heavier  investment in  regularity theory. This  new proof and the role in it played by the sharp regularity theory of \cite{WicAnnals} partly provide motivation for the present work. 

Finally, we remark that for $n=1, 2,$ different PDE approaches have been developed by various authors; for immersed closed geodesics or branched minimal surfaces see the works of Colding--Minicozzi (\cite{ColMin1}, \cite{ColMin2}), Rivi\`{e}re (\cite{Riv}), Michelat--Rivi\`{e}re (\cite{MR}) and Pigati--Rivi\`{e}re (\cite{PR}), and for prescribed mean curvature CMC surfaces the work of Struwe (\cite{Struwe}).

\subsection{Additional difficulties in the present work}
We end this introduction by briefly pointing out the main new challenges overcome in the proofs in the present work that were not present in the work \cite{WicAnnals}. The reader unfamiliar with \cite{WicAnnals} will benefit from reading Section~\ref{mainsteps} (below) before proceeding with the rest of this discussion. 

The main regularity result, Theorem \ref{thm:mainregularity}, is first reduced to Theorem~\ref{thm:mainregularityrestated} where ``strong stability'' (i.e.\ stability with respect to $J$ for uncontrained deformations) of the $C^{2}$ immersed part of the varifold $V$ can be assumed. Subsequently, the proof of Theorem~\ref{thm:mainregularityrestated} is divided into three main steps, the Sheeting Theorem (Theorem~\ref{thm:sheeting}), the Minimum Distance Theorem (Theorem~\ref{thm:min-dist}) and the Higher Regularity Theorem (Theorem~\ref{thm:higher-reg}), all proved simultaneously by induction. Much of the additional effort needed in the present work goes into the proof of the Higher Regularity Theorem. In the case of zero mean curvature as in \cite{WicAnnals}, unlike here, this step is an immediate consequence of the Hopf boundary point lemma and the standard elliptic regularity theory.

 The Sheeting Theorem roughly speaking says that if a varifold $V$ as in Theorem~\ref{thm:mainregularityrestated} and with a fixed bound on the scalar mean curvature of its regular part is close to a multiplicity $q$ hyperplane $P$ (in mass, $L^{2}$ distance and $L^{p}$ mean-curvature) then it decomposes locally as the sum of $q$ multiplicity 1 $C^{1, \alpha}$ graphs over $P$ with small $C^{1, \alpha}$ norm. The Higher Regularity Theorem says that if a varifold $V$ is the sum of $q$ multiplicity 1 $C^{1, \alpha}$ graphs over a hyperplane, and if $V$ satisfies the hypotheses of Theorem~\ref{thm:mainregularityrestated} (or Theorem~\ref{thm:mainregularity}) except for the stability hypothesis, then its support is the union of $\widetilde{q}$ ($\leq q$) $C^{2}$ graphs, each separately CMC, and hence also smooth. (We emphasize that when the mean curvature $H \neq 0$, this is only true for the support of the varifold; the original graphs giving the varifold with multiplicity are no more than $C^{1, 1}$ regular in general. See the example in Remark~\ref{oss:jumpsattouchingsing} and Figure \ref{fig:jumpsattouching}.) The Minimum Distance Theorem says that given a non-negative constant $H$ and a stationary integral cone ${\mathbf C}$ made up of three or more $n$-dimensional half-hyperplanes meeting along a common $(n-1)$-dimensional subspace, there is a fixed positive lower bound  (depending on ${\mathbf C}$ and $H$) on the Hausdorff distance at unit scale between ${\mathbf C}$ and any varifold $V$ with its regular part having  scalar mean curvature bounded by $H$ and satisfying the hypotheses of Theorem~\ref{thm:mainregularityrestated} as well as an appropriate mass bound. The induction parameter for the Sheeting Theorem and the Higher Regularity Theorem is the positive integer $q$, and that for the Minimum Distance Theorem is the density $\Theta \, (\|{\mathbf C}\|, 0)$ of ${\mathbf C}$ which takes values in $\{q, q+ 1/2\}$ for some integer $q \geq 1.$ 

The proofs of the inductive steps of the Sheeting Theorem and the Minimum Distance Theorem follow closely the corresponding argument in \cite{WicAnnals}, but with two key new aspects. One is that they make essential inductive use of the Higher Regularity Theorem. The other is that the conclusion of the Sheeting Theorem yields, initially, a weaker H\"older exponent for the gradient (of the functions defining the sheets) than in \cite{WicAnnals}. This exponent needs to be improved (as we do in the inductive step for the Higher Regularity Theorem) by independent arguments. The reason for this initially weaker conclusion is that the key excess-decay result needed for the Sheeting Theorem in the present context is obtained for an excess $\hat{E}$ that has, as is usual when the mean curvature $H$ is non-zero, e.g.\ as in \cite{Allard}, an extra lower order additive term (in addition to the $L^{2}$ height term)  involving $H$. In contrast to the multiplicity 1 setting of \cite{Allard} however, establishing excess-decay in the present higher-multiplicity setting requires a priori estimates for the varifold that make crucial use of the monotonicity formula (see Section~\ref{HardtSimon}). Consequently, the best possible choice for the lower order term in ${\hat E}$ is of the order $\sqrt{\|H\|_{L^{p}(\|V\|)}}$; see the definition of ${\hat E}$ in Theorem~\ref{thm:mainregularity}. This limitation arises precisely from the ``error term'' in the monotonicity formula when $H \neq 0$. Hence the excess-decay result we establish will initially only prove the Sheeting Theorem with $C^{1,\alpha}$ sheets for a value of $\alpha < \frac{1}{2}(1- \frac{n}{p}).$ Although we can improve this H\"older exponent by a second run of  the argument with the additional knowledge that $H$ is constant in the graph region, the best value of $\alpha$ we can get at this stage is still $<\frac{1}{2}$. 

In \cite{WicAnnals}, since $H = 0$, the value of $\alpha$ is irrelevant and higher regularity of the sheets is immediate. This is because by the Hopf boundary point lemma, the distinct sheets making up the support of the varifold are disjoint, and hence the functions defining the individual sheets  satisfy separately the minimal surface equation weakly. In the present case, the sheets do not separate in this manner, and our hypotheses in fact allow an a priori optimally large set $T$ of points where the sheets may touch each other; indeed, the only a priori control we have on $T$ is that  ${\mathcal H}^{n} \, (T) = 0$ (which follows from the structural hypothesis ($\dag\dag$) above, a sharp condition). Thus starting from just knowing $C^{1, \alpha}$ regularity, for some $\alpha < 1/2,$ of the distinct sheets of the support of the varifold which are allowed to touch on a set $T$ of measure zero, we need to prove their $C^{2}$ regularity. This requires considerable effort.

This is carried out, by means of PDE arguments, in Section \ref{higherreg} where the induction step for the Higher Regularity Theorem is completed. First we need to improve the H\"older exponent obtained in the Sheeting Theorem to some $\alpha\geq \frac{1}{2}$ (Section \ref{higherholderexponent}). Then, exploiting the improved exponent, we show that the regularity can be improved to $C^2$ (Sections \ref{semidifferenceisW22} and \ref{extensionPDEv}). We remark that the stronger hypothesis ${\mathcal H}^{n-1}( {\rm spt} \, \|V\| \setminus {\rm reg}_{1} \, V) = 0$ would lead to a substantially simpler proof of the Higher Regularity Theorem. This is because then ${\mathcal H}^{n-1} \,(T) = 0$ and hence by a straightforward cutoff function argument $T$ can be shown  to be removable for the PDE (the CMC equation) satisfied, in the complement of $T$, 
by the functions defining the sheets. This stronger hypothesis however is undesirable from the point of view of applications; for instance, it is not implied by the general structure theory of Caccioppoli sets, nor does it permit a full compactness theorem for the hypersurfaces as the one established here (as shown by the example of a sequence $V_{i} \to V$, where $V_{i}$ is made up of two disjoint half-cylinders of unit radius and parallel axes in ${\mathbb R}^{3}$ that come together in the limit $V$ made up of two half-cylinders  touching along a line). In the general case, we still of course show removability of $T$ (for $C^{1, \alpha}$ functions solving the the CMC equation away from $T$) but the proof is considerably more involved.

\section{Main theorems}
\subsection{Definitions and the statements of the main results}\label{main-thms}

 The hypotheses of our main theorems are motivated by the geometric variational problem of studying the hypersurfaces that are \textit{stable critical points with respect to the hypersurface-area functional for deformations that keep the volume enclosed by the hypersurface fixed}. There is a vast literature on this subject in the classical setting where the hypersurfaces are assumed to be smooth. 
However, in the geometric measure theory setting that we take up here, where smooth hypersurfaces are replaced by codimension 1 integral varifolds, it is not immediately clear how to define either the criticality or the stability for volume preserving deformations; indeed, the classical notion of volume-preserving variations, and stationarity with respect to such variations, require an \emph{oriented} immersion with $C^1$ regularity, while the notion of stability (of a stationary immersion) requires the immersion to be of class $C^2$. In the varifold setting, in addition to the hypersurfaces having possibly large singular sets a priori preventing their orientability,  they present also the extra difficulty caused by the presence of multiplicity $> 1$. Nevertheless, as will be clear soon, we will make as mild a set of hypotheses as possible on the varifolds in our theorems; roughly speaking, we will impose stationarity and stability only on regions of the varifold where these conditions make sense classically (i.e.\ away from singularities), and make the following assumption which is the only variational hypothesis that concerns the varifold in its entirety: the (unconstrained) first variation of the varifold is locally bounded, is absolutely continuous with respect to its weight measure and its generalized mean curvature (i.e.~the Radon-Nikodym derivative of the first variation with respect to the weight measure) is in $L^p_{\text{loc}}$ for some $p>n$. These conditions on the first variation of the varifold are natural from the point of view that the class of integral varifolds satisfying them enjoys good compactness properties under a uniform bound on the area and the $L^p$-norm of the mean curvature (\cite{Allard}, \cite{SimonNotes}). In addition to these variational hypotheses, we will also need two structural conditions on certain specific types of singularities of the varifold, to which we refer to as ``classical singularities'' and ``touching singularities'' (see the definitions below). 

Let $V$ be an integral varifold of dimension $n$ on and open set $\mathcal{U}\subset \R^{n+1}$ and let $\|V\|$ denote the weight measure associated with $V$. 

\begin{df}[\textit{Regular set $\Reg \, {V}$ and singular set $\text{sing} \, V$}]
A point $X \in {\mathcal U}$ is a regular point of $V$ if $X \in \spt{V}$ and if there exists $\sigma > 0$ such that $\spt{V} \cap B_{\sigma}^{n+1}(X)$ is an embedded smooth hypersurface of $B_{\sigma}^{n+1}(X)$. The regular set of $V$, denoted ${\rm reg} \, V,$ is the set of all regular points of 
$V.$ The (interior) singular set of $V$, denoted ${\rm sing} \, V$, is $(\spt{V}\setminus \Reg{V}) \cap {\mathcal U}$. By definition, ${\rm reg} \, V$ is relatively open in ${\rm spt} \, \|V\|$ and ${\rm sing} \, V$ is relatively closed in ${\mathcal U}$.  
\end{df}

\begin{df}[\textit{$C^1$-regular set $\text{reg}_1 V$}]
\label{df:C1embedded}
We define $\text{reg}_1 V$ to be the set of points $X \in \spt{V}$ with the property that there is $\sigma > 0$ such that $\spt{V} \cap B^{n+1}_{\sigma}(X)$ is an embedded hypersurface of $B^{n+1}_{\sigma}(X)$ of class $C^{1}$.
\end{df}
 
\begin{df}[\textit{Set of classical singularities $\text{sing}_C \, V$}]
\label{df:classicalsingularity}
A point $X \in \text{sing} \, V$ is a classical singularity of $V$ if there exists $\sigma >0$ such that, for some $\alpha \in (0, 1]$,  $\spt{V} \cap B^{n+1}_{\sigma}(X)$ is the union of three or more embedded $C^{1,\alpha}$ hypersurfaces-with-boundary meeting  pairwise only along their common $C^{1,\alpha}$ boundary $\gamma$ containing $X$ and such that at least one pair of the hypersurfaces-with-boundary meet transversely everywhere along $\gamma$.  

The set of all classical singularities of $V$ will be denoted by $\text{sing}_C \, V$.
\end{df}

\begin{df}[\textit{Set of touching singularities $\text{sing}_T V$}]
\label{df:touchingsingularity}
A point $X \in \text{sing}\,  V \setminus \text{reg}_1 V$ is a touching singularity of $V$ if $X \notin \text{sing}_C \, V$ and if there exists $\sigma>0$ such that 
$\spt{V} \cap B^{n+1}_\sigma(X)$ is the union of two embedded $C^{1,\alpha}$-hypersurfaces  of $B^{n+1}_{\sigma}(X)$. 
The set of all touching singularities of $V$ will be denoted by $\text{sing}_T \, V$. 
\end{df}

\begin{oss}[\textit{Graph structure around a point $X\in \text{sing}_T \, V$}]
\label{oss:touchingsinggraphs}
If $X \in \text{sing}_T V$ then each of the two $C^{1,\alpha}$-hypersurfaces corresponding to $X$ (as in Definition~\ref{df:touchingsingularity}) contains $X$ and they are tangential to each other at $X;$ the former is implied by the fact that $X \in \text{sing} \, V \setminus \text{reg}_1 \, V$ and the latter by the fact that $X \notin \text{sing}_C \, V$. Let $L$ be the common tangent plane to the two hypersurfaces at $X$. Upon possibly choosing a smaller $\sigma$ we see that 
$$\spt{V} \cap B^{n+1}_\sigma(X) = (\text{graph}\, u_1 \cup \text{graph}\, u_2) \cap B^{n+1}_\sigma(X)$$ for two functions $$u_{1}, u_{2} \, : \, \left( B_\sigma^{n+1}(X) \cap L \right) \to L^\bot$$ of class $C^{1,\alpha}$ such that $u_1(X)=u_2(X)$ and $Du_1(X)=Du_2(X)=0$. Note that $u_{1} \neq u_{2}$ since $X \in \text{sing} V\setminus \text{reg}_1 V$. 
\end{oss}

\begin{oss}[\textit{terminology}]
 \label{twofold}
For a general integral varifold $V$ and integer $\ell \geq 2$, one may speak of an \textit{$\ell$-fold touching singularity}: a point $X \in \text{sing} \, V$ is an $\ell$-fold touching singularity of $V$ if there exists $\sigma>0$ such that $$\spt{V} \cap B^{n+1}_\sigma(X) = \cup_{i=1}^\ell M_i$$ where $M_i$ are distinct $C^{1,\alpha}$ embedded submanifolds of $B^{n+1}_\sigma(X)$ with $X \in M_i$ for every $i\in\{1, ..., \ell\}$ and $T_X M_i = T_X M_j$ for any 
$i,j\in \{1, ..., \ell\}$. Denote by $\text{sing}^\ell_T \, V$ the set of all $\ell$-fold touching singularities of $V.$ For the varifolds $V$ in each of our theorems in this paper, the only type of touching singularities on which we need to make any assumption are those in $\text{sing}_T^{2} \, V,$ and we will in fact a posteriori rule out the occurrence of $\ell$-fold touching singularities in $V$ for $\ell \geq 3$. For this reason, we will just refer to a $2$-fold touching singularity simply as a ``touching singularity'' and write $\text{sing}_T \, V$ for $\text{sing}_T^2 \, V$. 
\end{oss}

We now precisely state the hypotheses (items labeled {\textbf{(I)}}-{\textbf{(V)}} below) on  $V$ together with some comments related to them: 
\begin{description}
 \item[{\textbf{(I)}}] The first variation of $V$ is locally bounded in ${\mathcal U}$ and is absolutely continuous with respect to $\|V\|,$ and the generalized mean curvature of $V$ is in $L^p_{\text{loc}}(\|V\|)$ for some $p>n$. 
\end{description}

\noindent 
Under the conditions (\textbf{I}) the monotonicity formula \cite[17.6]{SimonNotes} holds and implies that the density $\Theta(\|V\|,X):=\lim_{\rho\to 0}\frac{\|V\|(B^{n+1}_{\rho}(X))}{\om_n \rho^n}$ exists for every $X \in \Uc$, is upper-semi-continuous and that $\Theta(\|V\|,X)\geq 1$ for every $X \in \spt{V}$. Moreover, Allard's regularity theorem \cite{Allard} implies the existence of a dense open subset of $\spt{V}$ in which $\spt{V}$ agrees with an embedded $C^{1}$  hypersurface (which in fact is of class $C^{1,\alpha}$ where $\alpha = 1 - \frac{n}{p}$ if $p \in (n, \infty)$ and $\alpha$ is any number  $\in (0, 1)$ if $p = \infty$). This $C^{1,\alpha}$-embedded part of $\spt{V}$ coincides with the set $\text{reg}_1 \, V$ as in Definition~\ref{df:C1embedded}. As explained in Lemma \ref{lem:constancyfirst}, the density $\Theta \, (\|V\|, X)$ is a locally constant integer for $X \in \text{reg}_1\, V$. 
 
\medskip

\begin{description}
 \item[{\textbf{(II)}}] $V$ has no classical singularities, i.e. $\text{sing}_C \, V = \emptyset$.

 \item[{\textbf{(III)}}] For each $X \in \text{sing}_T \, V$ there exists $\rho>0$ such that $${\Hc}^{n}\left(\{Y \in \spt{V} \cap B_\rho^{n+1}(X) : \Theta(\|V\|,Y) = \Theta(\|V\|,X)\}\right)=0.$$

\end{description}

\begin{oss}
\label{oss:zeromeasuretouchingsing}
Note that hypthesis {\textbf{(III)}}, in view of Lemma \ref{lem:constancysecond}, implies that $$\mathcal{H}^n(\text{sing}_T \, V)=0.$$ Indeed, for every $X \in \text{sing}_T V$, Lemma \ref{lem:constancysecond} gives that $\Theta(\|V\|,X)=q$ for some $q\in \N$ and that there exists a ball $B_\rho^{n+1}(X)$ such that $$\text{sing}_T \, V \cap B_\rho^{n+1}(X)\subset \{Y \in \spt{V}\, : \, \Theta(\|V\|,Y\|)=q\};$$ then hypothesis ${\textbf{(III)}}$ implies that, possibly choosing a smaller $\rho$, $\text{sing}_T \, V \cap B_\rho^{n+1}(X)$ is $\mathcal{H}^n$-null. Taking the union on $q \in \N$ we conclude that $\mathcal{H}^n(\text{sing}_T \, V)=0$. (The same conclusion could be reached, without the use of Lemma \ref{lem:constancysecond}, by means of a Besicovitch covering argument.)  
\end{oss}

Let us now discuss the first and second variation hypotheses.

\textit{\underline{Stationarity}}. It is well-known that for an embedded $C^2$ hypersurface $M$ the stationarity of area with respect to volume-preserving deformations is equivalent to fact that $M$ has constant mean curvature---such critical points are indeed referred to as \textit{constant mean curvature (CMC) hypersurfaces}. By the divergence theorem, in the case that $M$ is a boundary, the enclosed volume is equivalently given by $\displaystyle \frac{1}{n+1} \int_{M} \vec{x} \cdot  \hat{\nu}\, d\mathcal{H}^n$, where $\vec{x}= (x_1, ...,  x_{n+1})$ and $ \hat{\nu}$ is the outward unit normal on $M$. The advantage of this formula lies in the fact that it makes sense in wider generality: $M$ need not be a boundary but merely orientable.
Given $V \in IV_n(\mathcal{U}$), let $\mathcal{O}\subset \Uc \setminus (\spt{V}\setminus \text{reg}_1 V)$, so that $\spt{V \res \mathcal{O}} \subset \text{reg}_1 V$: then, if $\text{reg}_1 V \cap \mathcal{O}$ is orientable, we define the enclosed volume of $V \res \mathcal{O}$ as
\begin{equation}\label{enclosed-vol}
\text{vol}_{\mathcal{O}}(V) = \frac{1}{n+1} \int_{\text{reg}_1 V \cap \mathcal{O}}\, \vec{x} \cdot  \hat{\nu} \,d\|V\|,
\end{equation}
where $\vec{x}= (x_1, ...,  x_{n+1})$ and $ \hat{\nu}$ is a continuous choice of unit normal on $\text{reg}_1 V \cap \mathcal{O}$. This formula generalises the notion to the case when we have an orientable embedded $C^1$-hypersurface not necessarily closed and endowed with an integer multiplicity. Note that this is a signed volume: a change of sign in the choice of the normal induces a change in the sign of the enclosed volume. Geometrically, for a $C^1$ embedded hypersurface $D$ of small size, $|\text{vol}(D)|$ is the volume of the cone on $D$ and vertex at the origin. With the previous discussion in mind we can introduce the stationarity assumption that we will make in our setting. 

Given a vector field $X \in C^1_c(\mathcal{O})$ (where, as above, $\spt{V \res \mathcal{O}} \subset \text{reg}_1 V$ and $\text{reg}_1 V \cap \mathcal{O}$ is orientable) we take an associated $1$-parameter family of deformations $\Psi_t$, i.e. a one-parameter family of diffeomorphisms $\Psi_t:\mathcal{O}\to \mathcal{O}$ such that $\left.\frac{d}{dt}\right|_{t=0}\Psi_t=X$ with $t\in(-\eps, \eps)$ for some $\eps>0$ small enough to ensure that $\Psi_t$ is the identity on $\p \mathcal{O}$ for $t\in(-\eps, \eps)$. In view of this, $\Psi_t$ can also be viewed as a diffeomorphism $\Uc \to \Uc$ that is the indentity on $\Uc \setminus \mathcal{O}$ for $t\in(-\eps, \eps)$. The variation $\Psi_t$ is called volume-preserving if $\text{vol}_{\Oc}\left((\Psi_t)_\sharp V\right)$ is constant for $t\in(-\eps, \eps)$. The stationarity condition on $V$ is the requirement that $V$ is critical for the hypersurface measure under volume-preserving variations, i.e. $\left.\frac{d}{dt}\right|_{t=0} \|(\Psi_t)_\sharp V\|=0 \text{ for any }\Psi_t \text{ that is volume-preserving}$.
A natural condition on $X$ that guarantees the existence of an associated volume-preserving variation is $\int_{\text{reg}_1 V \cap \Oc} X \cdot  \hat{\nu} \, d\|V\| =0$ (see \cite[Lemma 2.4]{BarbDoCarmo} the proof of which, notice, only requires $C^1$-regularity of the hypersurface; note also that multiplicity  is constant, by Lemma \ref{lem:constancyfirst}, on each connected component of $\text{reg}_1 V$). As explained in \cite{BarbDoCarmo} the first variation depends only on $X$ and not otherwise on the family of deformation $\Psi_t$. 

Equivalently, we can encode the fixed-enclosed-volume constraint by introducing a Lagrange multiplier \cite{BarbDoCarmo}: for $W \in IV_n(\Uc)$ and $\lambda \in \R$, we consider the functional
$$J_{\Oc}(W)=A_{\Oc}(W)+\lambda\text{vol}_{\Oc}(W),$$
where $A_{\Oc}(W)=\|W\|(\mathcal{O})$, and require that $V \res \mathcal{O}$ is stationary for $J_{\Oc}$ with respect to arbitrary deformations, i.e.~$\left.\frac{d}{dt}\right|_{t=0} J_{\Oc}((\Psi_t)_\sharp V)=0$ for every $X \in C^1_c(\mathcal{O})$. Again the first variation depends only on $X$. Thus our stationarity assumption is the following:

\begin{description}
   \item[{\textbf{(IV)}}] Whenever $\mathcal{O} \subset \left(\Uc \setminus (\spt{V} \setminus \text{reg}_1 V)\right)$ is such that $\text{reg}_1 V \cap \mathcal{O}$ is orientable, there exists an orientation $ \hat{\nu}$ on $\text{reg}_1 V \cap \mathcal{O}$ such that

$$\left.\frac{d}{dt}\right|_{t=0} \|(\Psi_t)_\sharp V\|=0$$
$  \text{ for any } X \in C^1_c(\Oc) \text{ with } \int_{\text{reg}_1 V \cap \Oc} X \cdot  \hat{\nu} \, d\|V\| =0 $ and any deformation $\Psi_t$ with $\left.\frac{d}{dt}\right|_{t=0}\Psi_t=X$,
or equivalently, there exists $\lambda\in \R$ such that 

$$\left.\frac{d}{dt}\right|_{t=0} J_{\Oc}((\Psi_t)_\sharp V)=0$$
$ \text{ for every } X \in C^1_c(\mathcal{O})$ and for any deformation $\Psi_t$ with $\left.\frac{d}{dt}\right|_{t=0}\Psi_t=X$.

\end{description}
 
\noindent \textit{Discussion}. Let us now analyse the local and global consequences of hypothesis \textit{\textbf{(IV)}}.
Since multiplicity on each connected component of $\text{reg}_1\,  V$ is constant by Lemma \ref{lem:constancyfirst}, every connected component of ${\rm reg}_{1} \, V$ can locally be expressed as a graph of a $C^{1,\alpha}$ function $u$ (over a tangent plane) which, taken with multiplicity 1, is stationary for $J_{\Oc}$; this yields that $u$ satisfies, in a weak sense, the $CMC$ equation 

$$\text{div}\left( \frac{Du}{\sqrt{1+|Du|^2}}\right) =\lambda $$
for a constant $\lambda$, where $u \in C^{1,1-\frac{n}{p}}(B_R^n(0))$. Standard elliptic theory yields that $u$ is of class $C^\infty,$ and therefore that $\text{reg}_1 V$ is a smooth hypersurface and thus $\text{reg}_1 V = \Reg{V}$. Moreover the equation is equivalent to the condition that $\vec{H}= \lambda \hat{\nu}$, where $\vec{H}$ is the mean curvature of $\Reg{V}$, i.~e. $\text{graph}\,u$ is a smooth CMC hypersurface with scalar mean curvature $h_0:= \lambda$. Note that at this stage the value $h_0$ of the mean curvature, while constant on each connected component, might still depend on the chosen connected component of $\text{reg}_1 V = \Reg{V}$. Note that, unless the mean curvature is zero\footnote{Note that the orientability of each connected component is obtained here as a consequence of the non-vanishing of the mean curvature. However our work covers the case $H=0$ as well, see Remark \ref{oss:samelson}.}, the fact that the mean curvature vector $\vec{H}$ is parallel implies that \textit{each connected component of ${\rm reg}_{1} \, V$ is orientable}. (We wish to emphasise that the preceding derivation only requires the local orientability of $\text{reg}_1 V$, which is always true, and that either of the two possible choices of orientation leads to the same conclusion.)

Let us next discuss the presence of distinct connected components. Note that the volume preserving condition, without a preferred orientation for the varifold, is ambiguous when we are dealing with two distinct connected components of $\text{reg}_1 V = \text{reg}\, V$. However, as we have seen, local considerations imply the existence of a (non-zero) parallel mean curvature vector on $\text{reg}\,  V$ and hence a canonical global orientation. This allows us to choose $\Oc$ to cover multiple connected components of $\text{reg} V$.

We will next show that assumption {\textbf{(IV)}} implies that $ \hat{\nu}=+\frac{\vec{H}}{|\vec{H}|}$ and $ \hat{\nu}=-\frac{\vec{H}}{|\vec{H}|}$ are the only choices of orientation for which that assumption can possibly hold; moreover, there exists a constant $h$ such that $\vec{H}=h  \hat{\nu}$. 

By the previous discussion $\text{reg}_1 V=\Reg V$ and moreover, for the chosen $ \hat{\nu}$, on each connected component $\mathcal{R}$ of $\text{reg}_1 V$, $\vec{H}=h_{\mathcal{R}} \hat{\nu}$ for a constant $h_{\mathcal{R}} \in \R$.
Now consider a volume-preserving variation (with respect to the chosen orientation $ \hat{\nu}$) that is supported on the union of two distinct connected components ${\mathcal{R}_1}$ and ${\mathcal{R}_2}$ of $\text{reg}_1 V=\text{reg} V$ and not separately volume-preserving on each of them. As we recalled earlier, such a volume-preserving variation can be induced by any $\zeta \in C^1_c({\mathcal{R}_1}\cup {\mathcal{R}_2})$ such that $\int_{\mathcal{R}_1} \zeta \theta_1 d\mathcal{H}^n+\int_{\mathcal{R}_2} \zeta \theta_2d\mathcal{H}^n= 0$, where $\theta_1, \theta_2 \in \N$ denote respectively the (constant) density on $\mathcal{R}_1$ and on $\mathcal{R}_2$, and such that $\int_{\mathcal{R}_1} \zeta \theta_1d\mathcal{H}^n\neq 0$, $\int_{\mathcal{R}_2} \zeta \theta_2d\mathcal{H}^n\neq 0$. Then by \cite[\S16]{SimonNotes} the first variation of area $\delta V$ evaluated on the vector field $\zeta  \hat{\nu}$ is given by
\begin{equation}
 \label{eq:computationsameH}
\delta V (\zeta  \hat{\nu})= \int \vec{H} \cdot \zeta  \hat{\nu} \, d\|V\| = 
\end{equation}
$$=\int_{\mathcal{R}_1} h_{\mathcal{R}_1} \zeta \theta_1 d\mathcal{H}^n+ \int_{\mathcal{R}_2} h_{\mathcal{R}_2} \zeta \theta_2d\mathcal{H}^n =h_{\mathcal{R}_1}\int_{\mathcal{R}_1}  \zeta \theta_1 d\mathcal{H}^n + h_{\mathcal{R}_2} \int_{\mathcal{R}_2}  \zeta \theta_2d\mathcal{H}^n .$$
This implies that $h_{\mathcal{R}_1}=h_{\mathcal{R}_2}$ and hence there exists a constant $h\in \R$ such that $\vec{H}=h  \hat{\nu}$. Thus the assertion holds.

\medskip

\textit{\underline{Stability and stability inequalities}}. Let us now discuss the stability hypothesis, i.e. non-negativity for the second variation of $V$ with respect to the area functional for volume-preserving deformations. In our theroems, the stability assumption will be made only on the smoothly immersed part of $V$, which we shall call the ``generalized regular set'' of $V$:

\begin{df}[\textit{Generalized regular set}]
\label{df:regularpoints}
Let $V\in IV_n(\mathcal{U})$. A point $X \in \text{spt}\,\|V\|$ is a generalized regular point if either (i) $X \in \Reg V$ or (ii) $X \in \text{sing}_T V$ and the two functions $u_1$ and $u_2$ corresponding to $X$ (as in Definition \ref{df:touchingsingularity}) are \textit{smooth}. The set of generalised regular points will be denoted by $\Greg{V}$. 
\end{df}

\begin{oss}
Under assumption \textit{\textbf{(II)}} $\Greg{V}$ is open in $\text{spt}\,\|V\|$. By definition $\Greg{V}$ can be realised as a smooth immersion in $\mathcal{U}$ of an abstract $n$-dimensional manifold (possibly with many connected components).
\end{oss}

\begin{oss}
\label{oss:separatesheetsingreg}

It is important to note the following. Assume \textit{\textbf{(I)}} \textit{\textbf{(II)}} and \textit{\textbf{(IV)}}. Locally near any point $X \in \Greg{V}$ we have that $\spt{V}$ is a smooth embedded hypersurface or the union of exactly two smooth embedded hypersurfaces. By Allard's regularity theorem, $\Reg{V}$ is dense in ${\rm spt} \, \|V\|$  and in particular any $y\in \text{sing} \, V$ is a limit point of $\Reg{V}$. Therefore the mean curvature is necessarily constant on each smooth embedded hypersurface describing $\Greg{V}$; in other words $\Greg{V}$ is a $C^2$ CMC immersion.
This condition is equivalent \cite[Proposition 2.7]{BarbDoCarmo} to the fact that the immersion is stationary for the (multiplicity 1) area measure under volume preserving variations. The definition of enclosed volume can be given for any oriented immersion \cite[(2.2)]{BarbDoCarmo} and we will discuss it in detail after Remark \ref{oss:maxprincsmooth}.
\end{oss}

\begin{oss}[Maximum principle and the measure of $\Greg{V} \setminus \Reg{V}$]
\label{oss:maxprincsmooth}
At any $y\in \Greg{V} \setminus \Reg{V}$ by definition $\spt{V}$ is locally given by the union of two smooth CMC hypersurfaces that intersect tangentially at $y$. Set coordinates such that $y=(0,0)$ and $u_j:B^n_\sigma(0) \to \R$ are smooth and satisfy the CMC equation and describe $\spt{V}$ around $y$, with $u_1 \leq u_2$, $u_1(0)=u_2(0)$ and $Du_1(0)=Du_2(0)=0$; here each function $u_1$ and $u_2$ solves one of the following PDEs 
$$\text{div}\left( \frac{Du}{\sqrt{1+|Du|^2}}\right) =+  |h| \,\,\, \text{ or }\,\,\,\text{div}\left( \frac{Du}{\sqrt{1+|Du|^2}}\right) =-   |h|.$$
If the mean curvature is $0$ then $\Greg{V}\setminus \Reg{V}=\emptyset$ by the maximum principle. When $h\neq 0$, analysing case by case and considering the possibilities for the signs on the right-hand side of the equation and writing the PDE for the difference $v=u_1-u_2$ we conclude, again by the maximum principle, that $u_1$ must necessarily solve the PDE with $- |h|$ on the right-hand side and $u_2$ must solve it with $+ |h|$ on the right-hand side. This means, in other words, that the mean curvature vector $\vec{H}$ of $\text{graph}\,u_1$ is such that $\vec{H}\cdot \hat{e}_{n+1}<0$ and the mean curvature vector $\vec{H}$ of $\text{graph}\,u_2$ is such that $\vec{H} \cdot \hat{e}_{n+1}>0$. Moreover, observing the Hessian of $v$, there must exist an index $\ell \in \{1, ..., n\}$ such that $D^2_{\ell \ell}v(0) \neq 0$ (because of the non-vanishing of the mean curvature). The implicit function theorem then gives that the set $\{D_\ell v =0\}$ is, locally around $0$, an $(n-1)$-dimensional submanifold. The set of points $\{u_1 =u_2\}$ is contained in the set $\{Du_1 = Du_2\}$ by assumption {\textbf{(II)}} and therefore $\{u_1 =u_2\} \subset \{Dv =0\} \subset \{D_\ell v =0\}$. This implies in particular the the set $\Greg{V} \setminus \Reg{V}$ has locally finite $\mathcal{H}^{n-1}$-dimensional measure.
\end{oss}

Since $\Greg{V}$ is a $C^2$ CMC immersion (possibly with many connected components), it is orientable and is stationary (as an immersion) for volume-preserving variations, where the enclosed volume for an oriented immersion $\mathscr{i}:M^n \to \mathcal{U}$ is given by the formula \cite[(2.2)]{BarbDoCarmo}
$$\mathscr{vol}(\mathscr{i})=\frac{1}{n+1}\int_{M^n} \vec{\mathscr{i}} \cdot  \hat{\nu} \,dM,$$
where $dM$ is the metric induced on $M^n$ by the immersion into $\mathcal{U}$, $ \hat{\nu}$ is the unit normal chosen by the orientation and $\vec{\mathscr{i}}$ is the vector $(\mathscr{i}_1, ..., \mathscr{i}_{n+1})$. (For the varifold $V$ under study this quantity is equal to 
$$\text{vol}_{\Uc\setminus(\text{sing}V \setminus \Greg{V})}(|\Reg{V}|)=\frac{1}{n+1} \int_{\Reg{V}} \vec{\mathscr{i}} \cdot  \hat{\nu} \,d\mathcal{H}^n ;$$
note that ${\mathcal{H}}^n\left(\Greg{V} \setminus \Reg{V}\right)=0$.) 
We stress that, when we consider volume-preserving variations of $\Greg{V}$ as an immersion, we allow a one-parameter family $\mathscr{i}_t:M^n \to \Uc$ of immersions, $t\in(-\eps, \eps)$, with $\mathscr{i}_0=\mathscr{i}$, $\mathscr{i}_t(x)=\mathscr{i}_0(x)$ for $x$ outside a fixed compact subset of $M$ and $\mathscr{vol}(\mathscr{i}_t)=\mathscr{vol}(\mathscr{i}_0)$. Such a deformation is not necessarily induced by an ambient vector field in $\mathcal{U}$: in a neighbourhood of a point in $\Greg{V} \setminus \Reg{V}$ the two touching sheets will generally be moved independently of each other by such variations (while preserving $\mathscr{vol}(\mathscr{i}_t)$). We will not require the stability for all possible volume-preserving variations as an immersion, but only for those induced by an ambient test function; more precisely, we only need to test the stability for variations with initial normal speed given by $\phi \nu$, where $\nu$ is the chosen unit normal and $\phi$ is an arbitrary ambient smooth function compactly supported in $\Uc\setminus(\text{sing}V \setminus \Greg{V})$ such that 
$\int_{\Greg{V}} \phi \, d\mathcal{H}^n =0$ (as we will discuss below, the last integral condition is necessary and sufficient for the existence of a volume-preserving variation with initial normal speed $\phi \nu$).

Our stability assumption precisely is as follows:

\medskip

  \begin{description}
   \item [{\textbf{(V)}}] For $V$ as above and for every  $\phi \in C^1_c(\Uc\setminus(\text{sing}V \setminus \Greg{V}))$ that satisfies
 $$\int_{\Greg{V}} \phi \,d\mathcal{H}^n =0,$$
 let $\mathscr{i}_t:M^n \to \Uc$, with $t\in(-\eps, \eps)$ for some $\eps>0$, be a smooth one-parameter family of immersions such that $\left.\frac{\p}{\p t}\right|_{t=0} \mathscr{i}_t = \phi \nu$, $\mathscr{i}_0(M^n)=\Greg{V}$, $\mathscr{i}_t=\mathscr{i}_0$ outside a fixed compact set for every $t\in (-\eps, \eps)$ and $\mathscr{vol}(\mathscr{i}_t)=\mathscr{vol}(\mathscr{i}_0)$ for $t\in(-\eps, \eps)$. Then 
$$\left.\frac{d^2}{dt^2}\right|_{t=0} \mathscr{a}(\mathscr{i}_t)\geq0,$$
where $\mathscr{a}(\mathscr{i}_t) = \int_{M^n} dM_t$ and $dM_t$ is the metric induced on $M^n$ by the immersion $i_t$.
 
  \end{description}

\noindent \textit{Discussion}. Let us now discuss hypothesis {\textbf{(V)}} and its consequences. First of all we recall some facts from \cite{BarbDoCarmo}. Given a $C^2$ CMC immersion $\mathscr{i}:M=M^n\to \Uc$ and $\zeta \in C^1_c(M)$ such that $\int_M \zeta dg=0$, where $g$ is the metric induced on $M$ by $\Uc$ via the immersion, there exists a volume-preserving (normal) variation $\mathscr{i}_t$ whose variation vector $\left.\frac{d}{dt}\right|_{t=0} \mathscr{i}_t= \zeta  \hat{\nu}$, where $ \hat{\nu}$ is a ($C^1$) choice of the unit normal vector on $M$ \cite[Lemma(2.4)]{BarbDoCarmo}. Using this fact it is straightforward to show that (see \cite[Proposition 2.10]{BarbDoCarmo}) stability with respect to volume-preserving variations implies the inequality 

\begin{equation}
 \label{eq:weakstabilityinequality}
\int_M  |A|^2\zeta^2 dg \leq \int_{M} |\nabla \zeta|^2 dg \,\,\,  \text{ for any }\zeta \in C^1_c(M) \text{ such that } \int_M \zeta dg =0,
\end{equation}
where $A$ denotes the second fundamental form on $M$ induced by the immersion and $\nabla$ is the gradient on $M$. This is called the \textit{weak stability inequality}. The terminology is used to distinguish it from the \textit{strong stability inequality}, i.e.~the same inequality for arbitrary $\zeta \in C^1_c(M)$ (that are not required to satisfy the condition $\int_M \zeta dg =0$). In fact \cite[Proposition(2.10)]{BarbDoCarmo} shows that, given an immersed $C^2$ CMC hypersurface, stability with respect to volume-preserving variations and the validity of the weak stability inequality are actually \textit{equivalent}. Let us outline this argument below.

Given an oriented immersion $\mathscr{i}:M=M^n \to \Uc$ with constant mean curvature $h_0  \hat{\nu}$, consider the functional 

\begin{equation}
 \label{eq:Jfunctional}
J(\mathscr{i}) = \mathscr{a}(\mathscr{i}) + h_0 \mathscr{vol}(\mathscr{i}),
\end{equation}
where $\mathscr{vol}(\mathscr{i})$ is the enclosed volume and $\mathscr{a}(\mathscr{i})$ is the area defined above. For any one-parameter variation $\mathscr{i}_t:M \to \Uc$ with $\mathscr{i}_0(M^n)=\Greg{V}$, $\mathscr{i}_t=\mathscr{i}_0$ for all $t\in (-\eps, \eps)$ outside a fixed compact set, and $\mathscr{vol}(\mathscr{i}_t)=\mathscr{vol}(\mathscr{i}_0)$ for $t\in(-\eps, \eps)$, we let $f =  \hat{\nu} \cdot \left(\left.\frac{d}{dt}\right|_{t=0} \mathscr{i}_t\right)$. Then writing $J(t)=J(\mathscr{i}_t)$, $\mathscr{a}(t)=\mathscr{a}(\mathscr{i}_t)$ and $\mathscr{vol}(t)=\mathscr{vol}(\mathscr{i}_t)$ we have, by the constancy of the mean curvature and \cite[Proposition(2.7)]{BarbDoCarmo}, that $J'(0)=0$. Moreover by \cite[Lemma(2.8)]{BarbDoCarmo} \textit{$J''(0)$ depends only on $f$} and 

\begin{equation}
 \label{eq:Jsecondvariation}
J''(0)=\int_M (-|A|^2 f^2 +  |\nabla f|^2 )dg.
\end{equation}
(See \cite[Appendix]{BarbDoCarmo} for the computation; the difficulty in the preceding statement is that the same $f$ can be associated to many distinct variations, not necessarily normal variations.) Once this is established, \cite[Proposition(2.7)]{BarbDoCarmo} completes the proof of the implication ``weak stability inequality $\Rightarrow$ stability for $\mathscr{a}$ under volume-preserving variations'' as follows: given any volume-preserving variation $\mathscr{i}_t$, it easily follows that its normal component $f  \hat{\nu}$ is such that $\int_M f dM_0=0$ and so by the weak stability inequality (taken with $\zeta=f$) we have $J''(0)(f) \geq 0$. On the other hand $J''(0) = \mathscr{a}''(0) + h_0 \mathscr{vol}''(0) = \mathscr{a}''(0)$ because $\mathscr{i}_t$ preserves $\mathscr{vol}$, and hence $\mathscr{a}''(0)\geq 0$. 

In light of this discussion, assumption {\textbf{(V)}} can be equivalently phrased by requiring that the weak stability inequality 
\begin{equation}
 \label{eq:ambientweakstabilityinequality}
 \int_{\Greg{V}}  |A|^2 \phi^2 \,d\mathcal{H}^n \leq \int_{\Greg{V}} |\nabla \phi|^2 \,d\mathcal{H}^n 
\end{equation}
holds for every $\phi \in C^\infty_c(\Uc\setminus(\text{sing}V \setminus \Greg{V}))$ such that $\int_{\Greg{V}} \phi \,d\mathcal{H}^n =0$, where $\nabla$ stands for the gradient on $\Greg{V}$.

Let us now come to an additional result (Remark \ref{oss:weakimpliesstrongforJ} below) that will be important for our later purposes. First we need the following:

\begin{oss}[\textit{weak stability inequality $\Rightarrow$ strong stability inequality at smaller scales}]
\label{oss:weakimpliesstrong}
Let $\mathcal{M}=\mathscr{i}(M)$ be a $C^2$ immersed $CMC$ hypersurface in the open set $\mathcal{U}$ and $g$ is the metric induced on $M$ by the immersion. assume that it satisfies 

$$\int_M |A|^2\zeta^2 dg \leq \int_M |\nabla \zeta|^2 d g$$ 
for all $\zeta \in C^1_c(M)$ with $\int_{M}\zeta  dg=0$. Then whenever $\mathcal{V}_1$ and $\mathcal{V}_2$ are disjoint non-empty open subsets of $M$ it must be true that it (at least) one of them the inequality 

$$\int_{\mathcal{V}_j} |A|^2\zeta^2 dg \leq \int_{\mathcal{V}_j} |\nabla \zeta|^2 dg$$ 
holds for all $\zeta \in C^1_c(\mathcal{V}_j)$ without the zero-average restriction. Indeed, assume that this fails in both $\mathcal{V}_1$ and $\mathcal{V}_2$, then we can find $\zeta_1$ and $\zeta_2$ compactly supported respectively in $\mathcal{V}_1$ and $\mathcal{V}_2$ such that 

$$\int_{\mathcal{V}_j} |A|^2\zeta_j^2  > \int_{\mathcal{V}_j} |\nabla \zeta_j|^2$$ 
for $j\in \{1,2\}$. Then we can find $c_1, c_2 \in \R$ such that $c_1 \zeta_1 + c_2 \zeta_2$ satisfies the zero average condition on $M$ , i.e. $c_1\int_{\mathcal{V}_1} \zeta_1 = -c_2 \int_{\mathcal{V}_2} \zeta_2 $, and since the supports of $\zeta_1$ and $\zeta_2$ are disjoint we then have $(c_1\zeta_1+c_2 \zeta_2)^2 = c_1^2\zeta_1^2+c_2^2 \zeta_2^2$ from which it follows that

$$\int_{M} |A|^2(c_1\zeta_1+c_2 \zeta_2)^2  > \int_{M} |\nabla(c_1\zeta_1+c_2 \zeta_2)|^2,$$
contradicting the assumption. Thus the weak stability inequality implies the stong one in at least one of two arbitrary disjoint subsets.

By using this fact we can see now that the weak stability inequality assumed on $\mathscr{i}(M)\subset \mathcal{U}$ with $p \in \mathscr{i}(M)$ actually implies that there exists an ambient open ball $B$ around $p$ in which the strong stability holds (i.e. without the restriction of the zero-average on the test function $\zeta \in C^1_c(B\setminus(\text{sing}V \setminus \Greg{V}))$). To see this, consider, for $R>0$ fixed such that $B^{n+1}_R(p) \subset \mathcal{U}$ and $0<r<R$, the ball $B_r^{n+1}(p)$ and the annulus $B_R^{n+1}(p) \setminus \overline{B}_r^{n+1}(p)$. By the previous discussion, the strong stability inequality must hold in at least one of the disjoint open sets $\mathscr{i}^{-1}(B_R^{n+1}(p) \setminus \overline{B}_r^{n+1}(p))$ and $\mathscr{i}^{-1}({B}_r^{n+1}(p))$. We have either (i) for some $r$ the strong stability inequality holds for all $\zeta\in C^1_c\left( {B}_r^{n+1}(p) \setminus(\text{sing}V \setminus \Greg{V})\right)$ or (ii) the strong stability holds with any $\zeta\in C^1_c\left( B_R^{n+1}(p)\setminus(\text{sing}V \setminus \Greg{V}) \setminus \{p\}\right)$. In the latter case the inequality can be shown to hold for an arbitrary $\zeta$ supported in $B_R^{n+1}(p)\setminus(\text{sing}V \setminus \Greg{V})$ by a standard capacity argument, since $n \geq 2$. In either case we reach the same conclusion: there is a ball $B$ around $p$ such that the strong stability inequality holds for all $\zeta \in C^1_c(B\setminus(\text{sing}V \setminus \Greg{V}))$.
\end{oss}

\begin{oss}[\textit{assumption \rm\textbf{(V)} $\Rightarrow$ local strong stability for $J$}]
\label{oss:weakimpliesstrongforJ}
Remark \ref{oss:weakimpliesstrong} says that the requirement that a immersed CMC hypersurface is variationally stable (as an immersion) in an open set with respect to volume-preserving variations induced by ambient test functions (which is equivalent, as mentioned earlier, to the validity of the weak stability inequality in the same open set) implies the validity of the strong stability inequality in a neighbourhood of every point and therefore it gives the non-negativity of $J''(0)$ for any variation supported in that neighbourhood (non necessarily volume preserving) that is induced by an ambient test function. Therefore the (geometrically natural) variational stability of a CMC hypersurface for volume-preserving variations implies that the hypersurface is \textit{locally} a stable critical point for the functional $J$, for the variations that we allowed. The importance of this observation lies in the fact that $J$ is an admissible functional for the validity of the results in \cite{SS}. 
\end{oss}

\medskip

This concludes the discussion on the ``CMC stable'' assumptions and we are now ready to state the main regularity result.

\begin{thm}[\textbf{regularity for (weakly) stable CMC integral varifolds}]
\label{thm:mainregularity}
Let $n \geq 2$ and let $V$ be an integral $n$-varifold on an open set $\mathcal{U} \subset \R^{n+1}$ such that the hypotheses \textbf{(I)}-\textbf{(V)} above hold; specifically: 
\begin{enumerate}
 \item the first variation of $V$ with respect to the area functional is locally bounded in ${\mathcal U}$ and is absolutely continuous with respect to $\|V\|,$ and the generalized mean curvature $\vec{H}$ of $V$ is in $L^p_{\text{loc}}(\|V\|)$ for some $p>n$;
 \item $\text{sing}_C V = \emptyset$;
 \item whenever $X \in \text{sing}_T V$ there exists $\rho>0$ such that\footnote{In particular from this assumption it follows that $\mathcal{H}^n\left(\text{sing}_T V\right)=0$, see Remark \ref{oss:zeromeasuretouchingsing}.} $${\Hc}^{n}\left(\{Y \in \spt{V} \cap B_\rho^{n+1}(X) : \Theta(\|V\|,Y) = \Theta(\|V\|,X)\}\right)=0;$$
 \item if an open set $\mathcal{O} \subset \left(\Uc \setminus (\spt{V} \setminus \text{reg}_1 V)\right)$ is such that $\text{reg}_1 V \cap \mathcal{O}$ is orientable, then, relative to one (of the two possible choices of)  orientation on $\text{reg}_1 \, V \cap \mathcal{O}$, $V \res \mathcal{O}$ is stationary with respect to the area functional under volume-preserving variations;\footnote{Of course locally on $\text{reg}_{1} \, V$ there is always an orientation; by the discussion following  \textbf{(IV)} all of $\text{reg}_1 \, V$ is orientable (and smooth) whenever \textbf{(IV)} holds.}
 \item for every  $\phi \in C^\infty_c(\Uc\setminus(\text{sing} \, V \setminus \Greg{V}))$ that satisfies $\int_{\Greg{V}} \phi \,d\mathcal{H}^n =0$, $\Greg{V}$ is stable with respect to the area functional under (volume-preserving) variations with initial normal speed $\phi \nu.$ \footnote{The fact that $\Greg{V}$ is a CMC $C^2$-immersion (possibly with several connected components) is not an assumption here, it is an immediate consequence of assumption 4, in view of Remark \ref{oss:separatesheetsingreg}. Only the stability is an assumption.}
\end{enumerate}
Then $\text{sing} \, V \setminus \text{sing}_T \, V $ is empty if $n \leq 6$, discrete if $n=7$ and is a closed set of Hausdorff dimension at most $n-7$ if $n \geq 8$. Moreover $\text{sing}_T \, V \subset \Greg{V}$ (see Definition \ref{df:regularpoints} above) and $\text{sing}_T \, V$ is locally contained in a smooth submanifold of dimension $(n-1)$, and $\Greg{V}$ is a classical CMC immersion.
\end{thm}

\begin{oss}
It follows directly from the definition of classical singularity that the no-classical-singularities assumption (hypothesis 2) is equivalent  to the following: there exists a set $Z \subset {\rm spt} \, \|V\|$ with ${\mathcal H}^{n-1}(Z) = 0$ (not assumed closed) such that ${\rm sing}_{C} \, V \cap ({\rm spt} \, \|V\| \setminus Z) = \emptyset.$ 
\end{oss}

\begin{oss}
\label{oss:simplifyassumptions}
The preceding regularity result is of local nature, so it suffices to prove it locally around any point of $\spt{V}$, i.e.\ taking ${\mathcal U}$ to be a small open ball around any given point in $\spt{V}$. On the other hand, in view of Remark \ref{oss:weakimpliesstrongforJ}, around any $X \in \spt{V}$ we can find a ball such that $\Greg{V}$ is, in that ball, a $C^2$ CMC immersion that is \emph{strongly} stable with respect to the functional $J = \mathscr{a} + h_0 \mathscr{vol}$ (i.e.\, stable with respect to $J$ for variations induced by arbitrary ambient test functions not necessarily having zero average), where $h_0$ is the constant value for the scalar mean curvature (implied by assumption 4, see the discussion after \textit{\textbf{(IV)}} above). 
\end{oss}

In view of Remark~\ref{oss:simplifyassumptions} we see that Theorem \ref{thm:mainregularity} will be implied by the following theorem in which strong stability is assumed. It turns out that, for the proof, we only need to require the stability with respect to variations with initial speed $f \nu$ where $f$ is a \textit{non-negative} ambient test function.

\begin{thm}[\textbf{regularity for strongly stable CMC integral varifolds}]
\label{thm:mainregularityrestated}
Let $n \geq 2$ and let $V$ be an integral $n$-varifold on an open set $\mathcal{U} \subset \R^{n+1}$ that satisfies the following assumptions:
\begin{enumerate}
 \item the first variation of $V$ with respect to the area functional is locally bounded and is absolutely continuous with respect to $\|V\|,$ and the generalized mean curvature $\vec{H}$ of $V$ is in $L^p(\|V\|)$ for some $p>n$;
 \item $\text{sing}_C V = \emptyset$;
 \item whenever $X \in \text{sing}_T V$ there exists $\rho>0$ such that\footnote{In particular from this assumption it follows that $\mathcal{H}^n\left(\text{sing}_T \, V\right)=0$, see Remark \ref{oss:zeromeasuretouchingsing}.} $${\Hc}^{n}\left(\{Y \in \spt{V} \cap B_\rho^{n+1}(X) : \Theta(\|V\|,Y) = \Theta(\|V\|,X)\}\right)=0;$$
 \item $\text{reg}_1 V = \Reg{V}$ and there exists a continuous choice of unit normal $ \hat{\nu}$ on $\Reg{V}$ and a constant $h \in \R$ such that $\vec{H}=h  \hat{\nu}$ everywhere on $\Reg{V}$;
 \item for each $f\in C^1_c(\mathcal{U}\setminus(\text{sing}V \setminus \Greg{V}))$ with $f\geq 0$, $\Greg{V}$ is stable with respect to the functional $J$ defined in (\ref{eq:Jfunctional}) under variations (as an immersion) with initial normal speed $f \nu$; equivalently, $$\int_{\Greg{V}} |A|^2 f^2 \, d\mathcal{H}^n\leq \int_{\Greg{V}} |\nabla f|^2 \, d\mathcal{H}^n \text{ for all such } f,$$ with notation as in (\ref{eq:ambientweakstabilityinequality}).
\end{enumerate}
Then $\text{sing} \, V \setminus \text{sing}_T \, V $ is empty for $n \leq 6$, discrete for $n=7$ and for $n \geq 8$ it is a closed set of Hausdorff dimension at most $n-7$. Moreover $\text{sing}_T V \subset \Greg{V}$, in the sense of Definition \ref{df:regularpoints} and $\text{sing}_T V$ is locally contained in a smooth submanifold of dimension $(n-1)$, and $\Greg{V}$ is a classical CMC immersion.
\end{thm}

\begin{oss}[\textit{the case $H=0$}]
\label{oss:samelson}
For the minimal case ($H=0$) Theorem \ref{thm:mainregularity} provides the same result as \cite{WicAnnals} but with weaker assumptions, namely the fact that $H$ is identically $0$ is replaced by the requirements that (assumption 1) $H \in L^p(\|V\|)$ and (assumption 3) $H=0$ on $\text{reg}_1 V$, the $C^{1,\alpha}$ embedded set. The global vanishing of $H$, which is an assumption in \cite{WicAnnals}, is for us a conclusion. Assumption 2 becomes redundant in the minimal case, as it follows from the remaining assumptions and from the maximum principle that $\text{sing}_T V=\emptyset$. Moreover the variational stability (assumption 5) only needs to be assumed for volume preserving variations rather than for arbitrary ones; note, to this end, that a complete minimal hypersurface in an open ball is always orientable by \cite{Samelson}, and the argument extends to the case of a singular set having codimension at least $7$ (it will be clear from the proof that this is all that is needed).
\end{oss}

The class of varifolds in Theorem \ref{thm:mainregularity} is moreover compact under mass and mean curvature bounds:

\begin{thm}[\textbf{compactness for stable CMC integral varifolds}]
\label{thm:compactness}
Let $n \geq 2$ and consider an open set $\mathcal{U} \subset \R^{n+1}$. The class of integral $n$-varifolds $V$ that satisfy assumptions (1)-(5) of Theorem \ref{thm:mainregularity} and have uniformly bounded masses 
$\|V\| ({\mathcal U}) \leq K_0$ and uniformly bounded mean curvatures $|H|\leq H_0$ for $K_0, H_0 \in \R$ (where $H$ is the generalized mean curvature of $V$ as in assumption 3) is compact in the varifold topology.
\end{thm}

\begin{oss}
In the presence of touching singularities one could consider a stronger stability assumption, namely one that allows variations that move the two $C^{1,\alpha}$ hypersurfaces independently at the touching set. Such an assumption would make it possible to employ techniques similar to those used in \cite{Caff} for the so-called obstacle problem, and in particular it would permit the regularity improvement from $C^{1,\alpha}$ to $C^{1,1}$. Note that we are not allowing these variations; we impose only the weaker, classical assumption that stability holds when we already know that the hypersurface is $C^2$. 
\end{oss}

\subsection{Optimality of the theorems: some examples}\label{examples} 

\begin{oss}
 As observed in Section \ref{intro-varifolds}, the stationarity assumption must be fulfilled on \textit{any} orientable portion of $\text{reg}_1\,V$ for $C^2$ regularity of $\text{reg}_{1} \, V$ to follow, see Figure \ref{fig:C11graph}. 
\end{oss}

\begin{figure}[h]
\centering
 \includegraphics[width=4cm]{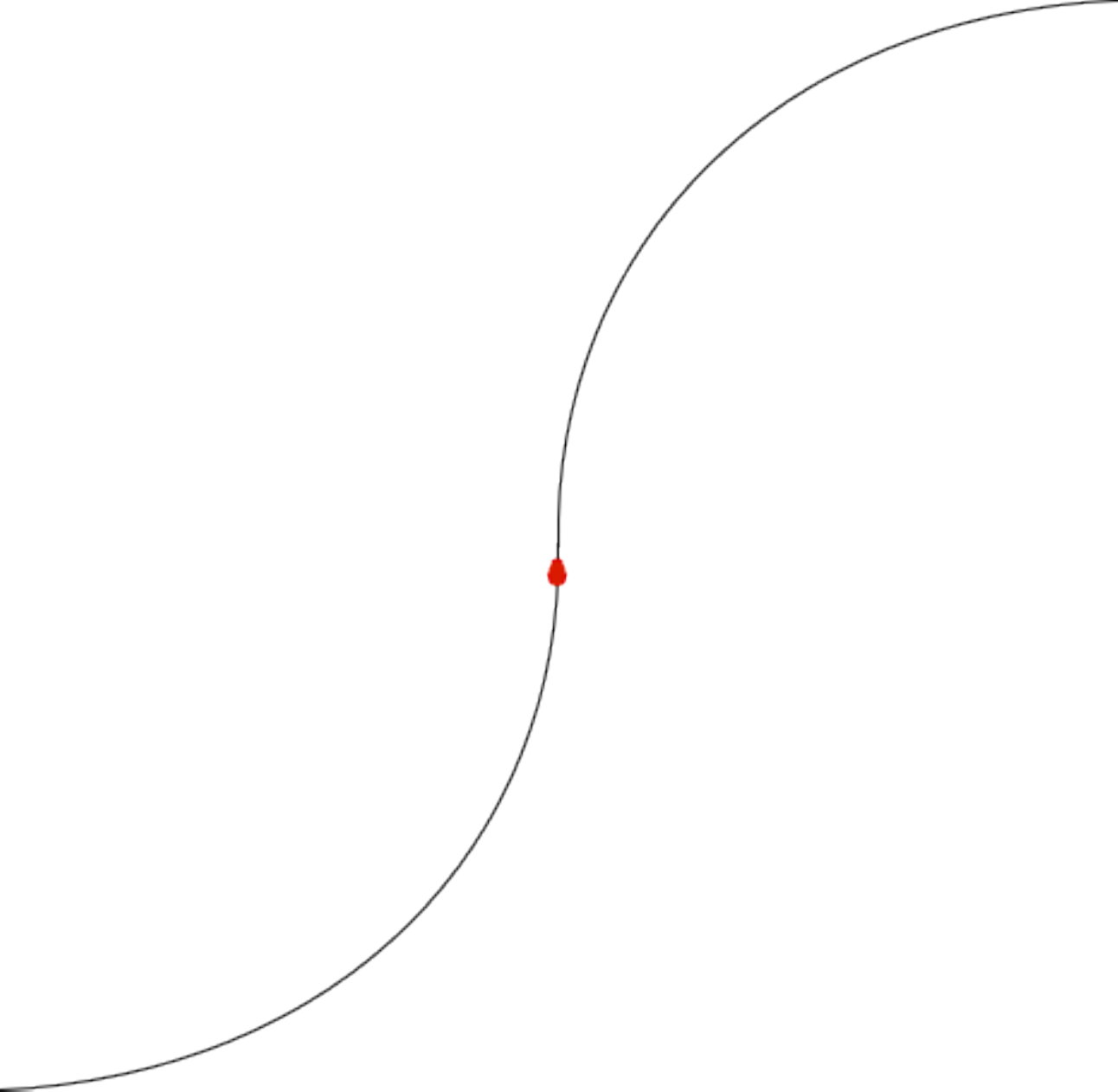} 
\caption{The 1-dimensional varifold $V$ depicted here consists of two quarters-of-circle with equal radii joined together in a $C^{1,1}$ fashion, taken with multiplicity 1. Every point of the varifold is in $\text{reg}_1\,V$ but $V$ is not of class $C^2$. Although all of $V$ is orientable, $V$ is not stationarity (for volume preserving deformations) with respect to either choice of orientation; $V$ is stationary only away from the point where the two circular arcs meet.}
 \label{fig:C11graph}
\end{figure}

\begin{oss}
\label{oss:withoutzeromeasureC11only}
In the absence of hypothesis 3, the $C^{2}$ regularity conclusion of Theorem~\ref{thm:mainregularity} away from a codimension 7 set cannot hold. This is easily seen by the following 1-dimensional example $V$ in 
${\mathbb R}^{2}$ which satisfies hypotheses 1, 2, 4, 5 but not hypothesis 3 of Theorem \ref{thm:mainregularity}, and has one point where it is not $C^{2}$ (but is $C^{1,1}$) immersed. (Of course, an $n$-dimensional example is obtained, with an $(n-1)$-dimensional set where the varifold is not $C^{2}$ immersed, by taking the cartesian product of $V$ with ${\mathbb R}^{n-1}$). In this example, $V$ is supported on the set $S\subset \R^2$ defined, with $(x,y)\in\R^2$, by
$$S=\{y\leq 1, x^2 + (y-1)^2 =1\} \cup \{y \geq -1, x\leq 0, x^2 + (y+1)^2 =1\}$$
and has multiplicity $2$ on the portion $\{(x,y)\in\R^2: y\leq 1, x\geq 0, x^2 + (y-1)^2 =1\}$ and multiplicity $1$ on the rest. See Figure \ref{fig:C11_Regularity}. 
Observe that the origin is a touching singularity and the mean curvature is constant on $\text{reg}_1 V=S\setminus\{(0,0)\}$. The stability is also true on $S\setminus\{(0,0)\}$ since we have graphical portions of a CMC curve. However writing the support $S$ at this touching singularity as the union of two graphs on the line $\{y=0\}$ we are forced to use, for one of the graphs, the function $u_1$ on $[-1,1]$ that takes the value $\sqrt{1-x^2} +1$ for $x\geq 0$ and the value $\sqrt{1-x^2} -1$ for $x\leq 0$, which is $C^{1,1}$ and enjoys no better regularity (the other graph is the one of the function $u_2=\sqrt{1-x^2} +1$, that is $C^2$). 
\end{oss}

\begin{figure}[h]
\centering
 \includegraphics[width=4cm]{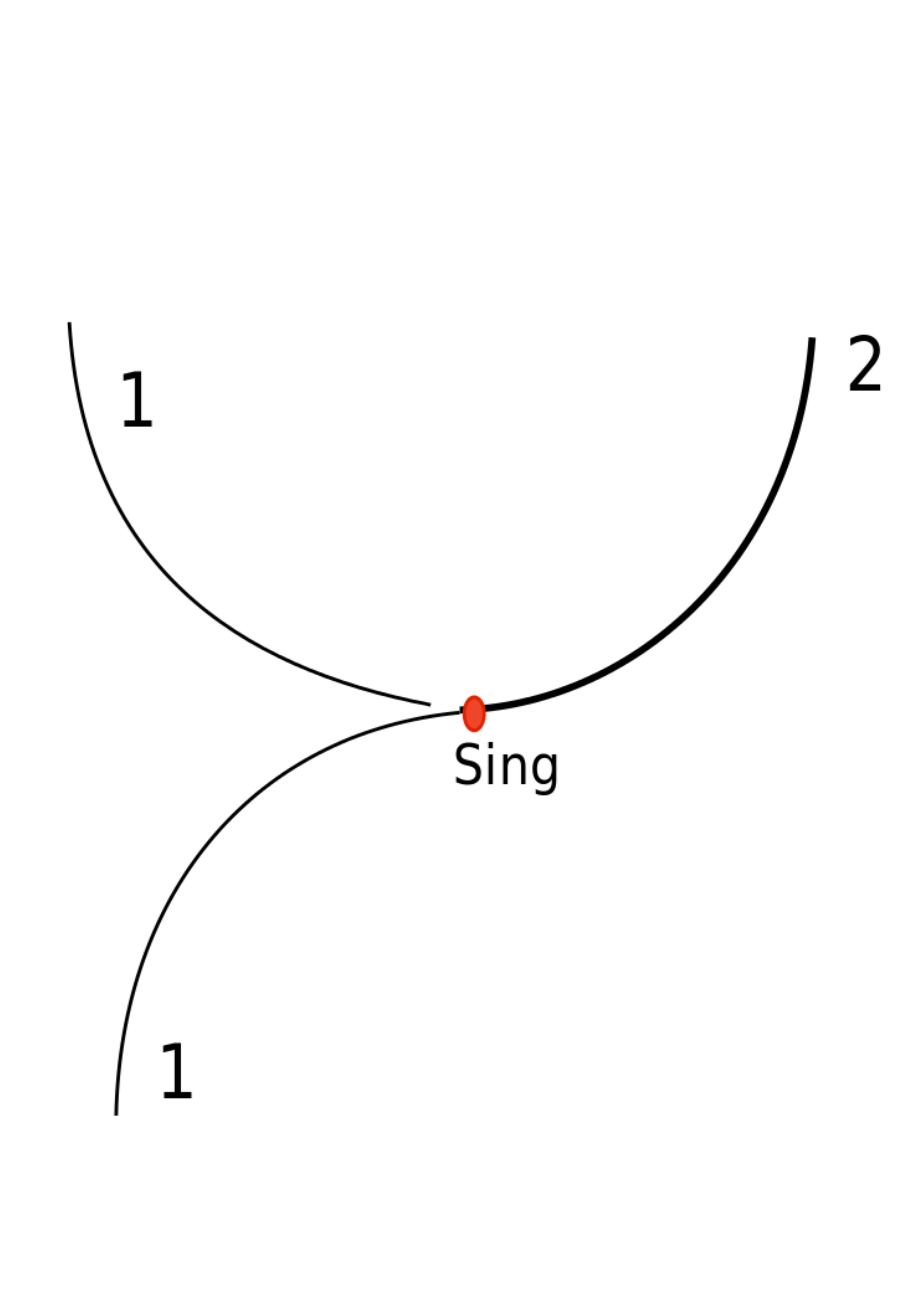} 
\caption{Lack of $C^2$ regularity in the absence of assumption 3, see Remark \ref{oss:withoutzeromeasureC11only}.}
 \label{fig:C11_Regularity}
\end{figure}

\begin{oss} 
\label{oss:withoutclassicalsingularityassumption}
If we drop assumption 2 (absence of classical singularities) we have the examples of two spheres of equal radii crossing along an equator or two transversely intersecting graphical pieces of spheres of equal radii. Both these examples have stable regular parts, and in fact satisfy assumptions 1, 3, 4, 5, but clearly do not satisfy the regularity conclusion. We conjecture that in the absence of assumption 2 the optimal regularity should be that $\text{sing} V $ is at most $(n-1)$-dimensional. 
\end{oss}

\begin{oss}[\textit{jumps in the multiplicities at the touching points}]
\label{oss:jumpsattouchingsing}
We wish to stress that the stability condition is given only on $\spt{V}$, i.e. we neglect multiplicities: indeed, with the notation from Definition \ref{df:regularpoints} and implicitly restricting to a neighbourhood of $p\in \Greg{V}\cap \text{sing}_T V$, we do not generally have that $V=q_1 |\text{graph}\,  u_1| + q_2 |\text{graph}\,  u_2|$ for some constants $q_1, q_2 \in \N$, as the following examples show.
Consider the $1$-dimensional integral varifold $V$ (higher dimensional examples follow by a trivial product with a linear subspace) whose support is given by (see Figure \ref{fig:jumpsattouching}) the set $D\subset \R^2$ defined by (here $(x,y)\in\R^2$)
$$D=\{y\geq -1, x^2 + (y+1)^2 =1\} \cup \{y \leq 1,  x^2 + (y-1)^2 =1\}$$
with multiplicity $2$ on the portions $\{(x,y)\in\R^2: -1\leq y\leq 0, x\leq 0, x^2 + (y+1)^2 =1\}$ and $\{(x,y)\in\R^2: 0\leq y \leq 1, x\geq 0, x^2 + (y-1)^2 =1\}$, and multiplicity $1$ on the rest. The support of $V$ agrees with $\Greg{V}$, the origin is a touching singularity and all assumptions of Theorem \ref{thm:mainregularityrestated} are satisfied. 

Note that each of the two sheets $\{y=\sqrt{1-x^2}-1\}$ and $\{y=\sqrt{1-x^2}+1\}$, taken separately with the assigned multiplicity, is not stationary for the variational problem, as it does not even have generalized mean curvature in $L^p$, due to the multiplicity jump (the origin belongs to the so-called varifold boundary). For this reason the stability assumption \textit{\textbf{(V)}} is stated for $\spt{V}$.

We can turn the given example, which is of local nature, into a global one in the standard sphere $S^2$, where the support of the varifold is given by the union of four tangential circles of radius $\sqrt{2}/2$, as follows. Let $S^2=\{(x,y,z):x^2+y^2+z^2=1\}$ and 
$$C_1=\left\{x^2+y^2=\frac{1}{2}, z=\frac{\sqrt{2}}{2}\right\},\, C_2=\left\{x^2+z^2=\frac{1}{2}, y=\frac{\sqrt{2}}{2}\right\},  $$
$$C_3=\left\{x^2+y^2=\frac{1}{2}, z=-\frac{\sqrt{2}}{2}\right\},\, C_4=\left\{x^2+z^2=\frac{1}{2}, y=-\frac{\sqrt{2}}{2}\right\}.  $$
We set multiplicities as follows: on the half-circles $C_1 \cap \{x>0\}$, $C_2 \cap \{x<0\}$, $C_3 \cap \{x>0\}$ and $C_4 \cap \{x<0\}$ we set the multiplicity equal to $2$ and on the remaining four half-circles we set it equal to $1$. 
\end{oss}

\begin{figure}[h]
\centering
 \includegraphics[width=4cm]{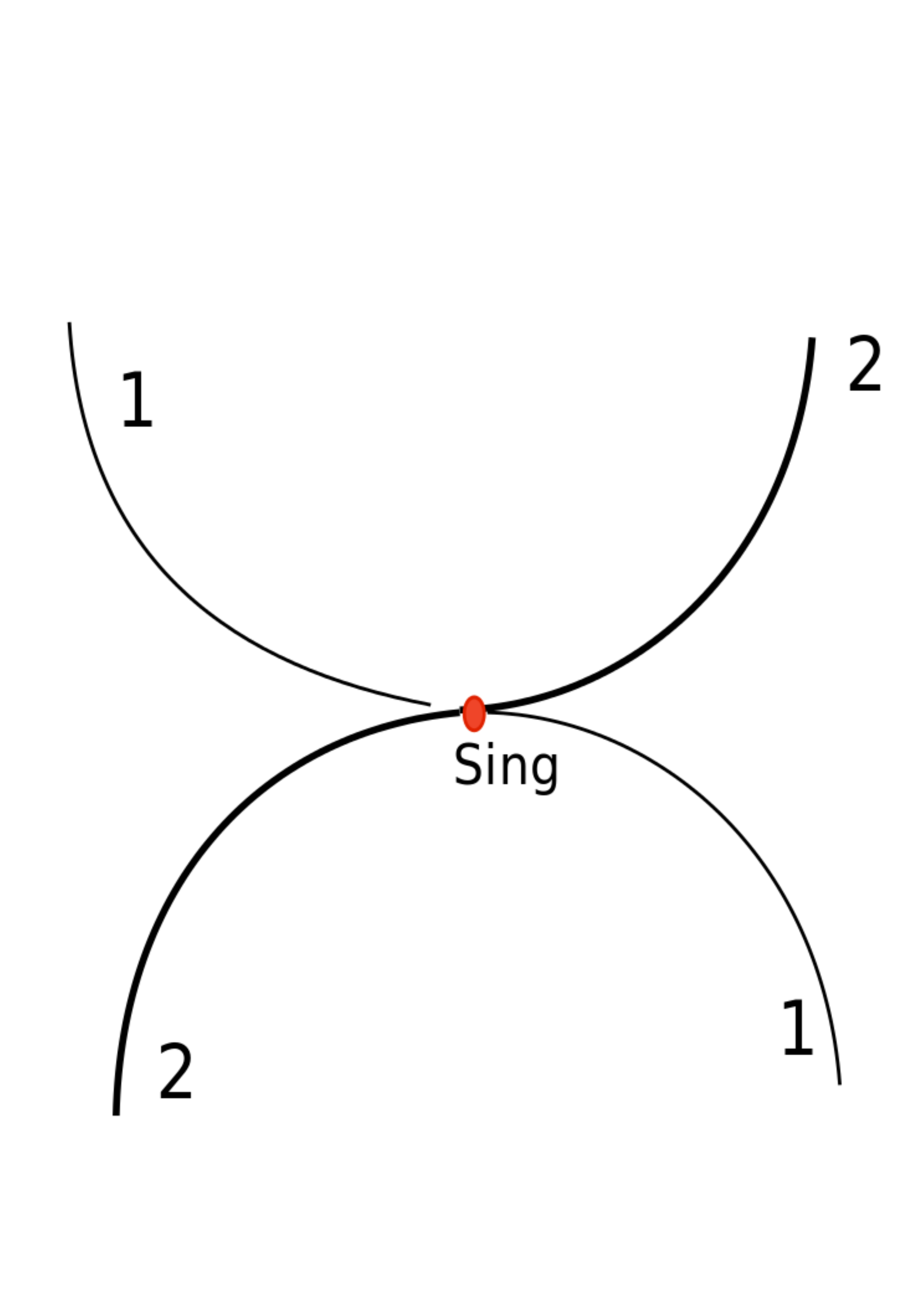} 
\caption{A jumps in the multiplicity of each of two $C^2$ CMC graphs along the touching set, see Remark \ref{oss:jumpsattouchingsing}.}
 \label{fig:jumpsattouching}
\end{figure}

\begin{oss}
In view of Remark \ref{oss:jumpsattouchingsing} it is natural to consider the restricted class of varifolds that satisfy the assumptions of Theorem \ref{thm:mainregularity} with the extra constraint that for $p\in \text{sing}_T V \cap \Greg{V}$ the two embedded hypersurfaces going through $p$ have separately constant multiplicity. It will follow from the proof of Theorem \ref{thm:compactness} that this class also enjoys the same compactness result (it is immediate that the regularity theorem holds for this restricted class as well). 
\end{oss}

\begin{oss}
\label{oss:Irving}
The possibility, allowed in the conclusion of Theorems \ref{thm:mainregularity} and \ref{thm:compactness}, that a codimension-$7$ singular set $\Sigma$ may be present for $n\geq 7$ is not surprising, in view of the analogous statements for stable minimal hypersurfaces, shown to be optimal by the example of Simons' cone.
The recent work \cite{Irving} constructs, in a spirit similar to \cite{CHS}, examples of CMC hypersurfaces with an isolated singularity that are asymptotic to a singular minimal cone. These hypersurfaces are stable when the minimal cone is strictly stable (e.g.~Simon's cone), showing the optimality of our conclusion.
\end{oss}

\subsection{Consequences for Caccioppoli sets}
\label{Caccioppoli}

In this subsection we focus on a special class of integral varifolds, namely multiplicity $1$ varifolds associated to the reduced boundary of Caccioppoli sets. The latter is a natural class for the variational problem of minimizing boundary area for a fixed enclosed volume, indeed the literature on the subject in the minimizing case is rich and rather complete, see e.g.~\cite{GonzMassTaman} for the Euclidean case and \cite{Morgan} for the extension to Riemannian manifolds. For ``stationary Caccioppoli sets'' and for ``stationary-stable Caccioppoli sets'' there is not even a partial local theory available for the variational problem under consideration and the notion of stationarity/stability itself is not immediately clear. In the following we point out how a very natural stationarity condition (on a Caccioppoli set) for ambient deformations fits very well with hypotheses 1 and 3 in Theorem \ref{thm:mainregularity} and thus makes the class of varifolds used in Theorem \ref{thm:mainregularity} suited to the context of Caccioppoli sets.

\begin{oss}[\textit{stationarity for ambient deformations $\Rightarrow$ hypothesis 1}]
\label{oss:LpmeancurvCaccioppoli}
Any Caccioppoli set $E$ admits a natural notion of enclosed volume, namely $\int \chi_E$, where $\chi_E$ denotes the characteristic function of $E$. In order to make sense of this notion when the enclosed volume is not necessarily finite, one restricts to an arbitrary open set with compact closure. For $\Oc \subset \subset \R^{n+1}$ we consider the functional (for a certain $\lambda \in \R$)
\begin{equation}
 \label{eq:JforCaccioppoli}
J_{\Oc}(E)=\|\p^* E\|(\Oc) + \lambda \int_{\Oc} \chi_E,
\end{equation}
where $\|\p^* E\|(\Oc)$ denotes the total mass of the boundary measure $\p^* E = D\chi_E$ in $\Oc$, and impose the stationarity condition for the varifold $|\p^* E|$ as follows. For any one-parameter family of deformations $\psi_t$ for $t\in(-\eps, \eps)$ with initial velocity $X \in C^1_c(\Oc;\R^{n+1})$ we obtain a one-parameter family of Caccioppoli sets $\{E_t=\psi_t(E)\}_{t\in(-\eps, \eps)}$ such that $E_0=E$ and $E_t\setminus \Oc=E\setminus \Oc$ for $t\in(-\eps,\eps)$; we require 

\begin{equation}
\label{eq:firstvarCaccioppoli}                                                                                                                                                                                                                                                                                                                                                                                                  
\left.\frac{d}{dt}\right|_{t=0} J_{\Oc}(E_t) =0 .                                                                                                                                                                                                                                                                                                                                                                               \end{equation}
This is a stronger hypothesis compared to assumption \textit{\textbf{(IV)}} above (we are allowing variations not necessarily supported on $\text{reg}_1$). In this case the stationarity implies automatically that the generalised mean curvature of the multiplicity $1$ $n$-varifold $|\p^* E|$ associated to the reduced boundary is a constant multiple of the unit normal with no singular part, as we will show now. The first variation of  $\int_{\Oc} \chi_{E_t}$ is equal to $\int_{\p^* E \cap {\Oc}} \nu  \cdot X d\mathcal{H}^n \res \p^* E$ by the divergence theorem on $E$, where $\nu$ denotes the outer normal on $\p^* E$. The first variation of $\|\p^* E_t\|(\Oc)=|\mathcal{H}^n \res( \p^* E_t \cap \Oc)|$ is, on the other hand, by the first variation formula, given by $\delta_{\p^* E} (X)=\int_{ \p^* E_t \cap \Oc} \text{div}_{\p^* E} X \, d\mathcal{H}^n \res \p^* E$ and is by definition a continuous linear functional on $C^1_c(\Oc;\R^{n+1})$. The stationarity assumption implies that 

$$\delta_{\p^* E} (X) -  \lambda \int \nu \cdot X (d\mathcal{H}^n \res \p^* E) =0 $$
for every $X\in C^1_c(\Oc;\R^{n+1})$, i.e. $\delta_{\p^* E}$ and the vector measure $\lambda (d\mathcal{H}^n \res \p^* E)\nu$ are equal as elements of the dual of $C^1_c(\Oc;\R^{n+1})$. The fact that $(d\mathcal{H}^n \res \p^* E)\nu$ is a measure implies therefore that $\delta_{\p^* E}(X)\leq C_E \lambda |X|_{C^0}$ for every $X \in C^1_c(\Oc;\R^{n+1})$, for some constant $C_E$, in other words the varifold $|\p^* E|$ has locally bounded first variation in the sense of \cite[\S 39]{SimonNotes} and $\delta_{\p^* E}$ extends to a continuous linear functional on $X \in C^0_c(\Oc;\R^{n+1})$: this extension necessarily agrees with $\lambda (d\mathcal{H}^n \res \p^* E)\nu$. The latter is absolutely continuous with respect to the varifold measure $\mathcal{H}^n \res \p^* E$ and we conclude that the generalized mean curvature of $|\p^* E|$ in $\Oc$ is $\vec{H}=\lambda \nu$ for the constant $\lambda$, in particular it is $L^\infty$. We point out that the stationarity for $J_{\Oc}$ for arbitrary ambient deformations is equivalent to the stationarity of the perimeter measure under volume-preserving ambient deformations, see \cite{Maggi}.
\end{oss}

\begin{oss}[\textit{hypothesis 1 $\Rightarrow$ hypothesis 3}]
\label{oss:touchingCaccioppoli}
When $V$ is the multiplicity $1$ varifold naturally associated to the reduced boundary $\p^*E$ of Caccioppoli set $E \subset \R^{n+1}$, hypothesis 3 in Theorems \ref{thm:mainregularity} and \ref{thm:mainregularityrestated} is automatically satisfied in the presence of assumption 1. Indeed, let $\|V\|=\mathcal{H}^n \res \p^* E$; the assumption that the generalized mean curvature is in $L^p(\|V\|)$ for $p>n$ implies, by the monotonicity formula \cite[17.6]{SimonNotes}, that the density $\Theta(\|V\|,x)$ exists everywhere and is $\geq 1$ on $\spt{V}$. Moreover by \cite[Theorem 3.15]{SimonNotes} we have that $\Theta(\|V\|,x)=0$ for ${\Hc}^n$-a.e. $x \in \spt{V}\setminus \p^*E$, so we must have ${\Hc}^n\left(\spt{V}\setminus \p^*E\right)=0$. De Giorgi's rectifiability theorem further gives that, for $x \in \p^*E$, $\Theta(\|V\|,x)=1$. Since, by the definition of $\text{sing}_T V$, for any $p \in \text{sing}_T V$ we have $\Theta(\|V\|,p)\geq 2$, it follows that hypothesis 3 of Theorem \ref{thm:mainregularity} holds. 
\end{oss}

Remarks \ref{oss:touchingCaccioppoli} and \ref{oss:LpmeancurvCaccioppoli} imply immediately the validity of the following corollary of Theorem \ref{thm:mainregularity}. It is worthwhile pointing out that, to our knowledge, proving this corollary alone is not easier than proving the more general Theorem \ref{thm:mainregularity}; the only slight simplification lies in the fact that the jumps in multiplicities on $\Greg{V}\setminus \Reg{V}$ described in Remark \ref{oss:jumpsattouchingsing} would be prevented, as the multiplicity is necessarily $1$ on $\Reg{V}$, but this would not contribute significantly to shortening the arguments.

\begin{cor}[\textbf{stable CMC Caccioppoli sets}]  
\label{cor:CaccioppoliwithoutLp}
Let $n \geq 2$ and let $|\p^* E|$ be the multiplicity $1$ integral $n$-varifold associated to the reduced boundary $\p^*E$ of a Caccioppoli set $E \subset \R^{n+1}$. Let $\Oc \subset \subset \R^{n+1}$ and assume that:

\noindent (i) $\text{sing}_C |\p^* E| \cap \Oc =\emptyset$; 

\noindent (ii)  the set $E$ is stationary with respect to the functional $J_{\Oc}$ as in (\ref{eq:JforCaccioppoli}), i.e.\ the condition~(\ref{eq:firstvarCaccioppoli}) holds (for ambient deformations $\psi_{t}$ as specified in (\ref{eq:firstvarCaccioppoli}));

\noindent (iii) $\Greg{|\p^* E|}\cap \Oc$ is stable as an immersion with respect to the functional $J$ in (\ref{eq:Jfunctional}), written with $h_0=\lambda$ and $\Oc$ instead of $\Uc$, for all volume-preserving variations with initial speed $f \nu$, where $f\in C^1_c(\Oc\setminus\left(\text{sing}\,|\p^* E| \setminus \Greg{|\p^* E|}\right))$ and $\int_{\Greg{|\p^* E|}} f d\mathcal{H}^n=0$.

Then $\Oc \cap (\text{sing}\, |\p^* E| \setminus \text{sing}_T \, |\p^* E|)$ is empty for $n \leq 6$, closed and discrete for $n=7$ and for $n \geq 8$ it is a closed set of Hausdorff dimension at most $n-7$. Moreover $\Oc \cap \text{sing}_T \, |\p^* E| \subset \Greg{|\p^* E|}$ and $\text{sing}_T |\p^* E| \cap \Oc$ is locally contained in a smooth submanifold of dimension $(n-1)$. 
\end{cor}

\begin{oss}
The regularity conclusion in the preceding corollary is sharp, as shown by the examples constructed in \cite{Irving}. 
Very recent remarkable work by Delgadino--Maggi \cite{MaggiDelg} classifies Caccioppoli sets in $\R^{n+1}$ with \textit{finite volume} that are stationary with respect to the perimeter for volume-preserving ambient deformations, showing that they are unions of balls. Even in the Euclidean case, a local analogue of this regularity result does not hold under stationarity only, in view of \cite{Irving}.   
\end{oss}

\begin{oss}
\label{oss:Caccioppolistabilityimersions}
At first sight the stability assumption \textit{(iii)} of Corollary \ref{cor:CaccioppoliwithoutLp} might seem unsuited to the context of Caccioppoli sets, since we are requiring variations as an immersion of $\Greg{|\p^* E|}$ and, in doing so, we may exit the class of Caccioppoli sets. We wish to point out however that
assumption \textit{(iii)} can be rephrased as requiring non-negativity at $t=0$ of the second variation of the perimeter measure computed along a deformation within the class of Caccioppoli sets that enclose the same volume and that are close to the initial one with respect to the $L^1_{\text{loc}}$-topology. In particular is satisfied under the area-minimizing assumption in Gonzales--Massari--Tamanini \cite{GonzMassTaman}. Figure \ref{fig:CaccioppoliDeformationImmersion} shows the idea behind this claim, which can be made precise.
\end{oss}

\begin{figure}[h]
\centering
 \includegraphics[width=5cm]{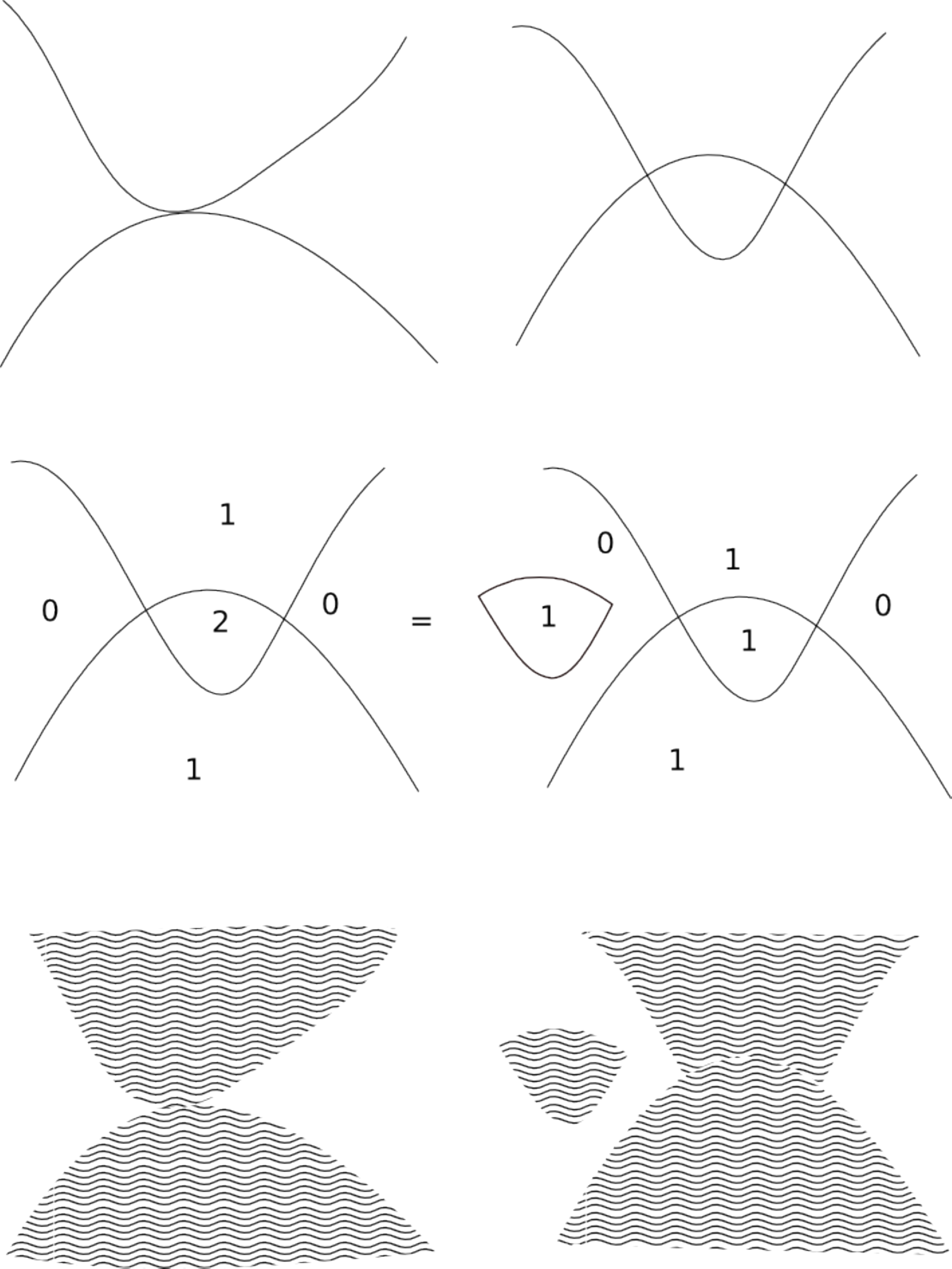} 
\caption{We consider the Caccioppoli set $E$ depicted in the bottom left picture. In the first row, a volume-preserving deformation, as an immersion, of the smoothly immersed boundary of $E$ (close to a touching singularity) is depicted. In order to see the preservation of enclosed volume as an immersion, we need to count with multiplicity two the overlapping region (second row, left picture). The same enclosed volume and the same perimeter can be realized by ``creating a copy'' of the overlapping region (second row, right picture). The bottom row depicts the corresponding volume-preserving deformation within the class of Caccioppoli sets: note that the perimeter and enclosed volume of the Caccioppoli set on the bottom right picture are respectively the hypersurface area and enclosed volume of the immersion in the top right picture.}
 \label{fig:CaccioppoliDeformationImmersion}
\end{figure}
 
\begin{oss}
The notion of stability with respect to the $L^1_{\text{loc}}$-topology on Caccioppoli sets, as discussed in Remark \ref{oss:Caccioppolistabilityimersions}, leads to the natural question of what can be said in the case when both stationarity and stability hold with respect to the $L^1_{\text{loc}}$-topology (rather than assuming stationarity for ambient deformations, as in Corollary \ref{cor:CaccioppoliwithoutLp}). We will discuss this in the next subsection, where we will prove that under such variational assumptions a stronger result can be obtained. In fact, subject to this stationarity assumption, a weaker notion of stability suffices.
\end{oss}

\subsubsection{Stationarity among $L^1_{\text{loc}}$-close Caccioppoli sets}

In the previous Corollary \ref{cor:CaccioppoliwithoutLp} we required the stationarity condition for deformations induced by ambient vector fields and the stability of $\Greg{|\p^* E|}$ as an immersion. Depending on the application of the regularity theory, there might be more or less suited stationarity and stability conditions. Much effort has been devoted to the case in which the Caccioppoli set is minimizing for the isoperimetric problem (\cite{GonzMassTaman}, \cite{Morgan}): this assumption can be viewed as sitting at one end of the spectrum, where we are allowed to compare with any other Caccioppoli set and we have a minimization property. A slightly weaker notion, of similar flavour, is that of locally minimizing, where the Caccioppoli set minimizes the perimeter measure among all Caccioppoli sets that are close to it in the sense of the $L^1_{\text{loc}}$-topology. At the other end of the spectrum, we might require that stationarity and stability hold for volume-preserving deformations induced by ambient vector fields.

We give, in this subsection, a corollary of our main result that is in between these two ends. In this corollary, stationarity is required with respect to the $L^1_{\text{loc}}$-topology; somewhat surprisingly, under such a stationarity assumption (clearly stronger than the one in Corollary \ref{cor:CaccioppoliwithoutLp}), we only need a very minimal stability requirement, namely only stability of the smoothly embedded part of $\p^* E$ for volume-preserving deformations (which are therefore induced by ambient vector fields). Beyond these variational hypotheses, no further condition, structural or otherwise, is needed and we in fact obtain a stronger conclusion than in Corollary \ref{cor:CaccioppoliwithoutLp}. We here prove this in the Euclidean case; the routine extension to the case of an analytic ambient metric will be included in \cite{BW2}. We conjecture that the same result should hold in a smooth Riemannian manifold.

\begin{df}
\label{df:oneparameterfamilyCaccioppoli}
Let $E$ be a Caccioppoli set and $\Oc$ be a bounded open set. A one-sided one-parameter family of deformations of $E$ in $\Oc$ is a collection $\{E_t\}_{t\in[0,\eps)}$ of Caccioppoli sets, for some $\eps>0$, such that the curve $$t \in [0,\eps) \to \chi_{E_t}$$ is continuous in the $L^1_{\text{loc}}$-topology and such that $E_t=E$ in $\R^{n+1}\setminus \Oc$ for every $t\in[0,\eps)$ and $E_0=E$.   

A one-sided one-parameter volume-preserving family of deformations of $E$ in $\Oc$ is a one-sided one-parameter family of deformations of $E$ in $\Oc$ with the additional constraint that $|E\cap \Oc|=|E_t \cap \Oc|$ for every $t\in[0,\eps)$. 
\end{df}

\begin{df}
\label{df:Caccioppolistationary}
Let $E$ be a Caccioppoli set in $\R^{n+1}$ and $\Oc$ be a bounded open set. We say that $E$ is \textit{stationary in $\Oc$ for the perimeter measure among Caccioppoli sets that enclose the same volume} when the following condition holds:

 $$\left.\frac{d}{dt}\right|_{t=0^+}\|D \chi_{E_t}  \res \mathcal{O}\| \geq 0$$ for any choice of one-sided one-parameter volume-preserving family of deformations of $E$ in $\Oc$ as in Definition \ref{df:oneparameterfamilyCaccioppoli} such that the map $t\to \|D \chi_{E_t}  \res \mathcal{O}\|$ is differentiable from the right at $t=0$.
\end{df}

Definition \ref{df:Caccioppolistationary} imposes a condition that is stronger than stationarity under ambient volume-preserving deformations (i.e.~those deformations that are induced by $C^1_c$ ambient vector fields). More generally, for any one-parameter volume-preserving deformation $t\in (-\eps,\eps)\to E_t$, continuous with respect to the $L^1_{\text{loc}}$-topology on $E_t$ and such that $t\to \|D \chi_{E_t}  \res \mathcal{O}\|$ is differentiable $t=0$, then the requirement in Definition \ref{df:Caccioppolistationary} immediately implies the stationarity condition $\left.\frac{d}{dt}\right|_{t=0}\|D \chi_{E_t}  \res \mathcal{O}\|=0$. The purpose of the requirement in Definition \ref{df:Caccioppolistationary} is to impose a notion of stationarity in cases where the structure of the Caccioppoli set naturally gives rise only to one-sided deformations\footnote{Roughly speaking, and as will be clear from the proof in Section \ref{Caccioppolicorollariesproofs}, when the reduced boundary has a touching singularity or a classical singularity, then the $L^1_{\text{loc}}$-topology allows to ``break the singularity apart'' in one direction only. From this perspective, a Caccioppoli set with a touching singularity or a classical singularity should be thought of as sitting ``at the boundary'' of the space of Caccioppoli sets: its deformations are therefore naturally one-sided, and the stationarity condition must be formulated as an inequality.}.

\begin{cor}
\label{cor:Caccioppolianalytic}
Let $E$ be a Caccioppoli set in $\R^{n+1}$ such that $E$ is stationary in a bounded open set $\Oc$ for the perimeter measure among Caccioppoli sets with the same enclosed volume, in the sense of Definition \ref{df:Caccioppolistationary}. Moreover assume that the smoothly embedded part $\Reg{|\p^* E|}$ is weakly stable, i.e.~stable for the perimeter measure under volume-preserving ambient deformations (in the sense of hypothesis \textbf{(V)}, but with $\phi \in C^1_c(\mathcal{O}\setminus \text{sing}\,|\p^*E|)$).

Then there exists a closed set $\Sigma$ such that $\text{dim}_{\mathcal{H}} \Sigma \leq n-7$ and $\p^* E \cap (\Oc \setminus \Sigma)$ is a smoothly immersed CMC hypersurface (possibly with several connected components) and with the property that at every point $p\in \overline{\p^* E} \cap (\Oc \setminus \Sigma)$ at which $\overline{\p^* E}$ is not locally embedded there exists a neighbourhood $B^{n+1}_\rho(p)$ in which $\overline{\p^* E}$ is the union of exactly two smooth complete CMC hypersurfaces in $B^{n+1}_\rho(p)$ that intersect tangentially. Moreover $\text{sing}_T |\p^*E|$ is a finite union of submanifolds of dimensions between $0$ and $n-2$.
\end{cor}

Corollary \ref{cor:Caccioppolianalytic} generalizes the work of Gonzales--Massari--Tamanini (\cite{GonzMassTaman}) that established regularity of boundaries that \emph{minimize} area subject to the fixed enclosed volume constraint\footnote{In \cite{BW2} we will similarly generalize the result of \cite{Morgan}  to Caccioppoli sets in anlytic Riemannian manifolds that are stationary and stable in the sense of Corollary \ref{cor:Caccioppolianalytic}.}.

In Section \ref{Caccioppolicorollariesproofs} we will prove this result by reducing it to Corollary \ref{cor:CaccioppoliwithoutLp}. In particular, we will show that the stationarity condition in Corollary \ref{cor:Caccioppolianalytic} rules out classical singularities and moreover forces $\Greg{V}$ to have a more restrictive structure where tangential CMC sheets do not touch along a submanifold of dimension $n-1$ (whence by analyticity the touching set has locally finite $n-2$ dimensional Hausdorff measure). This is the reason why the stability requirement can be weakened to only involve $\Reg{V}$. We point out that there is actually only one type of one-sided volume-preserving deformation that we need to employ that is not induced by an ambient deformation.

\section{Main steps of the proof: Sheeting, Minimum Distance and Higher Regularity Theorems}
\label{mainsteps}

Let $H$ be a non-negative constant. We will denote by $\mathcal{S}_H$ the class of integral $n$-varifolds satisfying assumptions 1, 2, 3, 4, 5 of Theorem \ref{thm:mainregularityrestated} with ${\mathcal U} = B_{2}^{n+1}(0)$ and $|h| \leq H$ and with fixed $p > n$.

Although many intermediate steps will be required, we can summarise the strategy of the proof of Theorem \ref{thm:mainregularityrestated} with the following three main theorems, which will all be proved by simultaneous induction.

\begin{thm}[Sheeting Theorem]
\label{thm:sheeting}
Let $q$ be a positive integer. There exists $\eps = \eps(n, $p$, q, H) \in (0, 1)$ such that if $V \in {\mathcal S}_H$, 
$(\omega_{n}2^{n})^{-1}\|V\|(B_{2}^{n+1}(0)) < q + 1/2,$ $q - 1/2 \leq \omega_{n}^{-1}\|V\|\left(B_{1} \times {\mathbb R} ^{n}(0) \right) < q + 1/2$ and 

$$\int_{B_{1}^{n}(0) \times {\R} } |x^{n+1}|^{2} \, d\|V\|(X) + \frac{1}{\eps} \left(\int_{B_{1}^{n}(0) \times {\R} } |h|^{p} \, d\|V\|(X) \right)^{\frac{1}{p}} < \eps \quad \text{then}$$ 

$$V \res \left(B_{1/2}^{n}(0) \times {\R} \right) = \sum_{j=1}^{q} |{\rm graph} \, u_{j}|$$
where $u_{j} \in C^{1, \alpha} \, (B_{1/2}^{n}(0); {\mathbb R})$ and $u_{1} \leq u_{2} \leq \ldots \leq u_{q}$, with
$$\|u_{j}\|^2_{C^{1, \alpha}(B_{1/2}^{n}(0))} \leq C \left(\int_{B_{1}^{n}(0) \times {\mathbb R}} |x^{n+1}|^{2} \, d\|V\|(X) + \frac{1}{\eps} \left(\int_{B_{1}^{n}(0) \times {\R} } |h|^{p} \, d\|V\|(X) \right)^{\frac{1}{p}}\right)$$ 
for some fixed constants $\alpha  = \alpha(n,p, q, H) \in (0, 1/2),$ $C = C(n, p, q, H) \in (0, \infty)$ and each $j=1, 2, \ldots, q$.  
\end{thm}

\begin{thm}[Minimum Distance Theorem]
\label{thm:min-dist}
Let $\delta \in (0, 1/2)$ and let ${\mathbf C} \in {\rm IV}_{n}({\mathbb R}^{n+1})$ be a stationary cone in ${\mathbb R}^{n+1}$  such that ${\rm spt} \, \|{\mathbf C}\|$ consists of three or more $n$-dimensional half-hyperplanes meeting along a common $(n-1)$-dimensional subspace.  
There exists $\eps = \eps({\mathbf C}, \delta, n, p, H) \in (0, 1)$ such that if $V \in {\mathcal S}_H$ and
$(\omega_{n}2^{n})^{-1}\|V\|(B_{2}^{n+1}(0)) \leq \Theta \, (\|{\mathbf C}\|, 0) +\delta$ then 
$${\rm dist}_{\mathcal H} \, ({\rm spt} \, \|V\| \cap B_{1}^{n+1}(0), {\rm spt} \, \|{\mathbf C}\| \cap B_{1}^{n+1}(0)) > \eps.$$
\end{thm}

\begin{thm}[Higher Regularity Theorem]
\label{thm:higher-reg}
Let $q$ be a positive integer and let $V \in {\mathcal S}_H$ be such that 
$$V \res \left(B_{1/2}^{n}(0) \times {\R} \right) = \sum_{j=1}^{q} |{\rm graph} \, u_{j}|$$
where $u_{j} \in C^{1, \alpha} \, (B_{1/2}^{n}(0); {\mathbb R})$ for some $\alpha \in (0, 1/2),$ and $u_{1} \leq u_{2} \leq \ldots \leq u_{q}$. Then 
$${\rm spt} \, \|V\| \cap \left(B_{1/2}^{n}(0) \times {\R} \right) = \cup_{j=1}^{\tilde{q}} {\rm graph} \, \tilde{u}_{j}$$
for some $\tilde{q}\leq q$ and distinct functions $\tilde{u}_{j}  \, : \, B_{1/2}(0) \to {\mathbb R}$ with $\tilde{u}_{1} \leq \tilde{u}_{2}  \leq \ldots \leq \tilde{u}_{\tilde{q}}$ where: 
\begin{itemize}
\item[(i)] $\tilde{u}_j \in C^2(B_{1/2}^{n}(0); {\mathbb R})$ and solves the CMC equation on $B_{1/2}^{n}(0)$  (and hence by elliptic regularity $\tilde{u}_{j} \in C^\infty(B_{1/2}^{n}(0); {\mathbb R})$) for each $j \in \{1, 2, \ldots, \tilde{q}\}$.
\item[(ii)] if $\tilde{q} \geq 2$, the graphs of $\tilde{u}_{j}$ touch at most in pairs, i.e. if there exist $x  \in B_{1/2}^{n}(0)$ and $i \in \{1, 2, ... \tilde{q}-1\}$ such that $\tilde{u}_i(x) = \tilde{u}_{i+1}(x)$ then $D\tilde{u}_i(x) = D\tilde{u}_{i+1}(x)$ and $\tilde{u}_j (x) \neq \tilde{u}_i(x)$ for all $j \in \{1, 2, ... \tilde{q}\} \setminus \{i, i+1\}$.
\end{itemize}
\end{thm}

\noindent Thus the Higher Regularity Theorem says that the touching singularities of $V$ are always two-fold and they are in $\Greg{V}$.

\medskip

The induction scheme for the proofs of the above theorems is as follows. Let $q \geq 2$ be an integer, and assume the following:

\noindent
{\sc induction hypotheses:}\\
\noindent
(H1) Theorem~\ref{thm:sheeting} holds with any $q^{\prime} \in \{1, \ldots, q-1\}$ in place of $q$.\\
\noindent
(H2) Theorem~\ref{thm:min-dist} holds whenever $\Theta \, (\|{\mathbf C}\|, 0) \in \{3/2, \ldots, q-1/2, q\}.$  \\
\noindent
(H3) Theorem~\ref{thm:higher-reg} holds with any $q^{\prime} \in \{1, \ldots, q-1\}$ in place of $q$.\\

Completion of induction is achieved by carrying out, assuming (H1), (H2), (H3), the following four steps in the order they are listed:
\begin{itemize}
\item[(i)] Prove Theorem~\ref{thm:sheeting} (Sections \ref{coarseandlinear} and \ref{sketchsheeting});
\item[(ii)] prove Theorem~\ref{thm:min-dist} in case $\Theta \, (\|{\mathbf C}\|, 0) = q+1/2$ (Section \ref{min-dist-sec});
\item[(iii)]  prove Theorem~\ref{thm:min-dist} in case $\Theta \, (\|{\mathbf C}\|, 0)  = q+1$  (Section \ref{min-dist-sec});
\item[(iv)]  prove Theorem~\ref{thm:higher-reg} (Section \ref{higherreg}).
 \end{itemize}
  
The base case $q=1$ of Theorem~\ref{thm:sheeting} is a direct consequence of Allard's regularity theorem \cite{Allard}, and  the case $\Theta \, (\|{\mathbf C}\|, 0) = 3/2$ of Theorem~\ref{thm:min-dist} follows from a theorem of Simon \cite[Theorem 4]{Simon2}\footnote{Incidentally, neither of these requires the stability hypothesis, and they both hold for stationary integral varifolds of arbitrary co-dimension.}. We wish to point out that, within step (i), there is a large ``substep'' (Section \ref{coarseandlinear}) that is still part of the inductive scheme and is needed to develop the necessary ``linear theory'' for the Sheeting Theorem.

The case $\Theta \, (\|{\mathbf C}\|, 0) = 2$ of Theorem~\ref{thm:min-dist} follows by taking $q=1$ in the argument for step (iii) above, and using the case $q=1$ of Theorem~\ref{thm:sheeting} and the case $\Theta \, (\|{\mathbf C}\|, 0) = 3/2$ of Theorem~\ref{thm:min-dist} in place of the induction hypotheses (H1), (H2) respectively. In the case $q=1$ Theorem~\ref{thm:higher-reg} is void and just needs to be replaced with the consequence of Allard's regularity theorem and the CMC assumption to obtain $C^2$ regularity and the validity of the CMC equation (i.e.~it becomes the standard higher regularity for the base case $q=1$ of Theorem~\ref{thm:sheeting}).

\medskip

The three theorems above will be combined with the following  elementary proposition, used at a number of places in the induction argument (precisely at a number of places in the proof of the sheeting theorem and in the proof of the minimum distance theorem), in order to complete the proof of Theorem \ref{thm:mainregularity}.

\begin{Prop}
\label{Prop:elementaryconsequence}
Let $V \in {\mathcal S}_H(\Omega)$, where $\Omega$ is an open subset of $B_{2}^{n+1}(0)$, and for $2\leq q \in \N$ let $S_{q} = \{Z \, : \, \Theta \, (\|V\|, Z) \geq q\}$. Assume that $(H1)$, $(H2)$, $(H3)$ are satisfied and assume further that $S_{q} \cap \Omega = \emptyset$: then 

(i) $(\text{sing} V \setminus \text{sing}_T V) \cap \Omega = \emptyset$ if $n \leq 6,$ $(\text{sing} V \setminus \text{sing}_T V) \cap \Omega$ is discrete if $n=7$ and $\text{dim}_{\mathcal H} \,\left( \left(\text{sing} V \setminus \text{sing}_T V\right)  \cap \Omega\right)\leq n-7$ for $n \geq 8$. 

(ii) $\text{sing}_T V \cap \Omega$ is locally contained in a smooth submanifold of dimension $n-1$.

\end{Prop}

\begin{proof}[Proof of Proposition \ref{Prop:elementaryconsequence}]
\textit{Statement (i)}. This follows using a standard tangent cone analysis, by means of Federer stratification theorem and Simons' stability result, see \cite[Section 6 Remarks 2 and 3]{WicAnnals}.

\noindent \textit{Statement (ii)}.Given two $C^2$ functions $u_1 \leq u_2$ both satisfying the CMC equation (with the same modulus for the mean curvature) and such that the set $T=\{u_1=u_2\}$ has vanishing $\mathcal{H}^n$-measure and $u_1$ and $u_2$ are tangential at all points in $T$, then the mean curvature vector of $\text{graph}u_1$ points downwards and the mean curvature vector of $\text{graph}u_2$ points upwards. In particular the CMC equations read 
$$\text{div}\left(\frac{Du_1}{\sqrt{1+|Du_1|^2}}\right)=-h \,,\,\, \text{div}\left(\frac{Du_2}{\sqrt{1+|Du_2|^2}}\right)=h$$
for $h>0$. This follows from the maximum principle, more precisely from Hopf boundary point lemma applied to the difference\footnote{In Lemma \ref{lem:Hopf} we will be concerned with an instance of this fact without the assumption that the functions $u_1$ and $u_2$ are globaly $C^2$, so we do not give the details of the argument here as they will appear in Section \ref{higherreg}.} $u_1-u_2$. We want to show first af all that the set $\{D(u_1-u_2)=0\}$ is contained in a submanifold of dimension $\leq n-1$. By the condition of opposite signs for the mean curvature, we get that for any $x \in T$ there exists $i  \in \{1, \ldots, n\}$ such that $D^2_{ii} (u_1 -u_2)$ is non-zero at $x$ and therefore the implicit function theorem gives that $D_i(u_1 - u_2)=0$ is locally around $x$ a smooth submanifold of dimension $\leq n-1$. In particular we have that $\text{sing}_T V$ is contained, locally around any point, in a smooth submanifold of dimension $\leq n-1$. 

As a side remark, note that at each point in $T$ we necessarily have an index $i$ for which $D^2_{ii} (u_1 -u_2)$ is non-zero, and in case we have more than one index (say $L$ of them) we can apply the implicit function theorem $L$ times and obtain that $T$ is contained in a submanifold of dimension $n-L$.

\end{proof}

In the proof of our Sheeting Theorem \ref{thm:sheeting} we will need to make use of the following adaptation of \cite[Theorem 2]{SS}. It is through the application of this result at various places in the argument that the stability assumption predominantly enters our proof.

\begin{thm}[Schoen-Simon Sheeting Theorem with codimension $7$ singular set]
\label{thm:SS}
If in Theorem~\ref{thm:sheeting} we assume, in place of the hypotheses 1, 2, 3 (of Theorem~\ref{thm:mainregularityrestated}),  that 

(i) $\text{sing} \, V \setminus \text{sing}_T\,  V  = \emptyset$ in case $n < 6$, $\text{sing} V \setminus \text{sing}_T \, V$ is discrete in case $n=7$ or 
${\rm dim}_{\mathcal H} \, (\text{sing} V \setminus \text{sing}_T V) \leq n-7$ in case $n\geq 8$  

and that  

(ii)  $\text{sing}_T\,  V \subset \Greg{V},$

and keep all other assumptions, then the conclusion of Theorem~\ref{thm:sheeting}  holds.
\end{thm}

\begin{proof}[proof of Theorem \ref{thm:SS}]

Let $W = |\Reg{V}|,$   i.e.\ the multiplicity 1 varifold associated with $\Reg{V}.$ Then of course ${\rm spt} \, \|W\| = {\rm spt} \, \|V\|.$ It suffices to show that under the hypotheses of Theorem~\ref{thm:SS} that 

$$W \res \left(B_{1/2}^{n}(0) \times {\R} \right) = \sum_{j=1}^{\tilde{q}} |{\rm graph} \, \tilde{u}_{j}|$$
where $\widetilde{q} \leq q$, $\tilde{u}_{j} \in C^{1, \alpha} \, (B_{1/2}^{n}(0); {\mathbb R})$ and $\tilde{u}_{1} \leq \tilde{u}_{2} \leq \ldots \leq \tilde{u}_{\tilde{q}}$, with
$$\|\tilde{u}_{j}\|^2_{C^{1, \alpha}(B_{1/2}^{n}(0))} \leq C \left(\int_{B_{1}^{n}(0) \times {\mathbb R}} |x^{n+1}|^{2} \, d\|V\|(X) + \frac{1}{\eps} \left(\int_{B_{1}^{n}(0) \times {\R} } |h|^{p} \, d\|V\|(X) \right)^{\frac{1}{p}}\right)$$ 
for some fixed constants $\alpha  = \alpha(n,p, q, H) \in (0, 1/2),$ $C = C(n, p, q, H) \in (0, \infty)$ and each $j=1, 2, \ldots, q$.

Since $\Greg{W}=\Greg{V}$ is a smooth CMC immersion, the first variation formula $\int_W \text{div}_W X d\mathcal{H}^n \res W = -\int_W \vec{H}|_W \cdot X d\mathcal{H}^n \res W$ holds for every $X\in C^1_c(B_2^{n+1}(0)\setminus (\text{sing} V \setminus \text{sing}_T V);\R^{n+1})$, where $\vec{H}=H  \hat{\nu}$ (for a constant $H \in \R$) is the (classical) mean curvature of $\Reg{V}$.
By (i) and (ii) $\text{sing}V \setminus \Greg{V}$ ($=\text{sing} V \setminus \text{sing}_T V$) has codimension $7$ or higher, therefore a standard cutoff argument (which requires only $\mathcal{H}^{n-1}\left(\text{sing} V \setminus \text{sing}_T V\right)=0$ and the Euclidean volume growth for $W$) allows to check that the first variation formula $\int_W \text{div}_W X d\mathcal{H}^n \res W = -\int_W \vec{H}|_W \cdot X d\mathcal{H}^n \res W$ holds for every $X\in C^1_c(B_2^{n+1}(0);\R^{n+1})$, where $\vec{H}$ is the (classical) mean curvature of $\Reg{V}$.

The proof of the above  can then be obtained by following the arguments in \cite{SS} very closely. We cannot obtain Theorem \ref{thm:SS} directly from \cite{SS} because the singular set would include also 
$\text{sing}_T V$, which can however have dimension $(n-1)$ and this is not allowed by the assumptions in \cite{SS}. However, the proof in \cite{SS} carries over to our setting.
The approximate graph decomposition constructed in \cite{SS} will cover $\text{sing}_T V$ as well. In particular the (smooth) ordered graphs of the approximate decomposition will be weakly ordered, i.e.~they can touch tangentially but cannot cross. We will now describe how to adapt the arguments from \cite{SS}.

The first step is to obtain the analogue of \cite[Lemma 1]{SS} (which, in our case, has $\mu_1=0$) for an arbitrary $\varphi \in C^1_c(B^n_1 \times \R)$. It is enough to obtain this for $\varphi\geq 0$, since only this case will be used later. For the proof of the inequality we need, following \cite[proof of Lemma 1]{SS}, to use the strong stability inequality with the test function $\varphi(1-(\nu\cdot\nu_0)^2)^{1/2}$. In order to justify its use, we note that $\nu$ is defined up to sign on $\Greg{W}\cap \text{sing}_T \, V$ and thus $\varphi(1-(\nu\cdot\nu_0)^2)^{1/2}$ is well-defined on $\Greg{W}$, moreover it is locally Lipschitz on $\Greg{W}$ because $|\nabla\,(1-(\nu\cdot\nu_0)^2)^{1/2}|\leq |A|$; therefore it admits an ambient non-negative Lipschitz extension by Kirszbraun theorem; the strong stability inequality extends from $C^1$ functions compactly supported away from $\text{sing} V \setminus \text{sing}_T V$ to $C^1$ functions with arbitrary compact support in view of the assumption that $\text{sing} V \setminus \text{sing}_T V$ has codimension $7$ (standard cutoff argument, as in \cite{SS}); moreover the strong stability inequality extends from $C^1$ to Lipschitz compactly supported functions, by a simple approximation. The use of the strong stability inequality with the test function $\varphi(1-(\nu\cdot\nu_0)^2)^{1/2}$ leads to \cite[(2.1)]{SS} and the rest of \cite[proof of Lemma 1]{SS} can be carried through with routine modifications considering  the abstract immersed hypersurface and pulling-back the functions on it.

The arguments in \cite[Sections 3, 4]{SS} can now be followed to conclude the proof: we can construct a partial graph decomposition (where the partial CMC graphs are allowed to touch) and then show that the ``excess'' decays, so that all of $\spt{W}$ must be covered by the graph decomposition in the end. We stress that the arguments for the excess decay require (compare \cite[proof of Lemma 3]{SS}):

(i) the use of (the analogue of) \cite[Lemma 1]{SS} with ambient non-negative compactly supported test functions (of the form $\zeta [\log\left(2^{1/\eps}\lambda^{1}\overline{\varphi}_0\right)]_+$, with notations as \cite[p. 763]{SS});

(ii) the use of (the analogue of) \cite[Lemma 1]{SS} with a special non-ambient test function, that is compactly supported on a single sheet $G_i$ of a (previously constructed) partial graph decomposition (of the form $\zeta\psi_i$, with the notations of \cite[p. 763]{SS}).

The first case has been covered in the previous discussion. In the second case, we can use Lemma 1 straight from \cite{SS}, since we are dealing with a smooth CMC graph on some connected open set $\Omega$ and \cite[Lemma 1]{SS} only needs to assume the validity of the strong stability inequality on $G_i$, which is true for arbitrary CMC graphs, as explained in Appendix \ref{stabilitygraphs} below; so the second case is also covered.

The arguments described so far lead to the decay result at the origin (for the fine excess $E$) given by \cite[Lemma 4]{SS}. We can ensure that the $L^2$-excess $\hat{E}$ is uniformly small in $B_2(0)$ and in $B_1(X)$ for any choice of $X \in B_{\frac{1}{2}}(0)$. The functional $J$ that we are addressing does not satisfy assumption \cite[(1.6)]{SS}; on the other hand, we observe that, upon pushing-forward the given varifold $V$ by a translation $T_X$ (so that an arbitrary point $X$ becomes $0$) and restricting to the unit ball, we have that the translated varifold $(T_X)_{\sharp}\,V \res B_1(0)$ fulfils the same assumptions for the same functional $J$ (indeed, the notions of being CMC and stable are independent of the choice of coordinates). Therefore we obtain the decay result at an arbitrary $X \in B_{\frac{1}{2}}(0)$ and can complete the proof following \cite[p. 775]{SS}.
\end{proof}

\begin{oss}[\textit{some differences with \cite{WicAnnals}}] In our Sheeting Theorem we obtain the same qualitative $C^{1,\alpha}$ conclusion as in \cite{WicAnnals}, however the geometry of the problem allows the graphs to touch tangentially without coinciding identically (unlike \cite{WicAnnals}, where if two graphs touch at a point they must agree globally): the touching can clearly happen as shown by the example of two spheres with equal radii touching at a point, or two half-cylinders touching along an affine subspace. Given two tangential $C^{1,\alpha}$ graphs such that the varifold associated to the sum of the two graphs taken each with multiplicity $1$ is in $\mathcal{S}_H$, it is not straightforward to prove that the two graphs are $C^2$ and separately CMC. A second aspect to keep in mind is that we are also allowing multiplicity, so we cannot expect that each graph obtained in Theorem \ref{thm:sheeting} is separately CMC, (unlike \cite{WicAnnals} where each sheet solves the minimal surface equation separately), see Remark \ref{oss:jumpsattouchingsing}: in that example we have three graphs, the bottom and top ones are smooth half-circles but the middle one is only $C^{1,1}$ and agrees partly with the bottom graph and partly with the top one. It is for these reasons that we need to implement, in the induction procedure, the Higher Regularity Theorem, where we look at the support of the varifold and show that it admits a graph decomposition in which each sheet is separately a smooth CMC hypersurface and sheets can touch only in pairs (in the example of Remark \ref{oss:jumpsattouchingsing} we would be neglecting the middle graph from the previous decomposition, as it is redundant for the description of the support). We wish to remark that the assumption that the touching set has zero measure (assumption 3 in Theorem \ref{thm:mainregularity}) plays a crucial role in the proof of Theorem \ref{thm:higher-reg} but is not needed directly in the proof of Theorem \ref{thm:sheeting}, where it is only used inductively by assuming (H3). Compare Remark \ref{oss:withoutzeromeasureC11only} at this stage. We also point out that in the proof of Theorem \ref{thm:sheeting} the inductive assumption (H3) and Proposition \ref{Prop:elementaryconsequence} allow the use of the stability assumption around points of multiplicity $\leq Q-1$, which permits the use of a version of \cite{SS}, i.e. Theorem \ref{thm:SS}.

A further difference with \cite{WicAnnals} lies in the fact that the Sheeting Theorem \ref{thm:sheeting} is obtained with $\alpha <\frac{1}{2}$ (this cannot be changed by a technical improvement, it is intrinsic in the strategy followed). In \cite{WicAnnals} the actual value of $\alpha$ is irrelevant (graphs do not touch and solve separately the minimal surface equation, therefore a $C^{1,\alpha}$ graph is automatically smooth by elliptic regularity). In the present work, at the stage of the graph decomposition for the support of the varifold, i.e.~in the Higher Regularity Theorem \ref{thm:higher-reg}, the fact two $C^{1,\alpha}$ graphs might touch along a $\mathcal{H}^n$-null set does not permit to use the CMC equation separately for each graph and therefore improving the regularity from $C^{1,\alpha}$ with $\alpha<\frac{1}{2}$ to $C^\infty$ requires a lot of work: the main intermediate steps consist in improving the H\"older exponent first to $\alpha\in [\frac{1}{2},1)$ and then obtaining $W^{2,2}$ estimates. These steps are detailed in Section \ref{higherreg}.

In view of the fact that the main thread in Sections \ref{coarseandlinear}, \ref{sketchsheeting} and \ref{min-dist-sec} follows the arguments in \cite{WicAnnals}, we will not give the full details of the proofs but will only describe the main steps and point out the necessary modifications that need to be implemented in our situation.  
\end{oss}

\begin{oss}
\label{oss:C11estimate} 
We point out that, assuming inductively the validity of (H1), (H2), (H3), we are able to improve the regularity of the functions $u_j$ appearing in (H1): they are in fact $C^{1,1}$ (and in general not better), with norm controlled by the $L^2$-excess $\hat{E}^2_{V,\eps,1}$, i.e.~the quantity appearing on the right-hand side of the $C^{1,\alpha}$-estimate in Theorem \ref{thm:sheeting}. On a first reading the reader may prefer to skip the proof, presented in the following lines. To see how to get this improved regularity, we point out that (H1) and (H3) deliver two sets of functions, respectively $u^1\leq ... \leq u^q$ that are $C^{1,\alpha}$ with small gradients and  $\tilde{u}^1\leq ... \leq \tilde{u}^{\tilde{q}}$ that are smooth and solve the CMC equation and $\spt{V \res (B_1\times \R)}$ is given by the union of the graphs of the functions $\{u^j\}_{j=1}^q$ and also by the union of the graphs of the functions $\{\tilde{u}^\ell\}_{\ell=1}^{\tilde{q}}$. Moreover the functions $\{\tilde{u}^\ell\}$ can touch at most in pairs by (H3) and the set $\{\tilde{u}^\ell=\tilde{u}^{\ell+1}\}$ (for any $\ell$) is contained locally in a smooth $(n-1)$-dimensional submanifold. By higher regularity for the CMC PDE we obtain sup bounds $|D \tilde{u}^\ell|+|D^2 \tilde{u}^\ell| \leq C(n,q,H_0)$ for every $\ell$. In particular we can deduce that, denoting by $T=\pi\{x\in B: \tilde{u}^\ell(x)=\tilde{u}^{\ell+1}(x) \text{ for some } \ell\}$, the functions $u^j$ are smooth on $B \setminus T$ and agree locally with exactly one of the $\tilde{u}^\ell$, in particular we have $|D u^j|+|D^2 u^j| \leq C(n,q,H_0)$ on $B \setminus T$ for every $j$. From this we will deduce that actually each $u^j$ is $W^{2,\infty}$ with norm bounded by $C(n,q,H_0)$. The regularity must be proven, in a neighbourhood of every point $x\in T$, for every function $u^j$ such that for some $\ell$ the graphs of $\tilde{u}^\ell$ and $\tilde{u}^{\ell+1}$ touch at $p=(x,u^j(x))$. Note that there is a neighbouhood of $x$ over which the graphs $\tilde{u}^1, ..., \tilde{u}^{\ell-1}$ and the graphs $\tilde{u}^{\ell+2}, ..., \tilde{u}^{\tilde{q}}$ are completely separated from the graphs of $\tilde{u}^{\ell}$ and $\tilde{u}^{\ell+1}$. On such a neighbourhood $U$ of $x$ we consider the function $f=Du^j - D\tilde{u}^\ell$ and note that it is $0$ on the whole of $\{\tilde{u}^\ell=\tilde{u}^{\ell+1}\}\cap U$; indeed, for $y\in \Uc$ either $u^j(y)=\tilde{u}^\ell(y)$ or $u^j(y)=\tilde{u}^{\ell+1}(y)$ by continuity of $u^j$ and moreover $D\tilde{u}^\ell = D\tilde{u}^{\ell+1}$ whenever $\tilde{u}^\ell = \tilde{u}^{\ell+1}$. We know that $f\in C^{0,\alpha}$ and we will show that $f\in W^{1,\infty}(U)$, which clearly suffices. Let $\chi_{\eps}$ be a standard one-parameter family of cut-off functions that are $0$ in an $\eps$-tubular neighbourhood of $\{\tilde{u}^\ell=\tilde{u}^{\ell+1}\}\cap U$ and $1$ away from a $2\eps$-tubular neighbourhood of $\{\tilde{u}^\ell=\tilde{u}^{\ell+1}\}\cap U$. It is a standard fact that this family can be chosen in such a way that $\int |D \chi_{\eps}|$ is equibounded in $\eps$ by $C \mathcal{H}^{n-1}(T)$ For $\zeta \in C^1_c(U)$ we have

$$-\int_U f D\zeta = -\lim_{\eps\to 0}\int  \chi_{\eps} f D\zeta =\lim_{\eps\to 0}\int  (D\chi_{\eps}) f \zeta + \lim_{\eps\to 0}\int   \chi_{\eps} Df \zeta=$$
$$=\lim_{\eps\to 0}\int  (D\chi_{\eps}) f \zeta + \int_{U\setminus T}Df \zeta $$
and using the facts that $f=0$ on $T$ and $f\in C^{0,\alpha}$ we get that the first term is controlled by $C \eps^{\alpha} \|\zeta\|_{L^\infty}$ so in conclusion
$$-\int_U f D\zeta =  \int_{U\setminus T} Df \zeta .$$
Recalling the definition of $f$ we conclude that the distributional derivative $D^2 u^j$ is an $L^\infty$ function.
\end{oss}

\section{Coarse blow-ups and the linear theory} 
\label{coarseandlinear}

\begin{Prop} \cite[Corollary 3.11]{Alm}[\textbf{Almgren's Lipschitz approximation via multiple-valued graph}]
\label{Prop:AlmgrenDec}
Let $V$ be a $n$-varifold in $B^{n+1}_2(0)$ with generalized mean curvature in $L^p(\|V\|)$ for $p>n$. Let $q\in \N$ and $\sigma>0$ be fixed. Assume that  

\begin{equation}
 \label{eq:circaqsheet}
\frac{1}{\om_n 2^n} \|V\|\left( B^{n+1}_2(0) \right) < q+\frac{1}{2};\,\,\,\,\, q-\frac{1}{2}<\frac{1}{\om_n} \|V\|\left( B^{n}_1(0) \times \R \right) < q+\frac{1}{2}.
\end{equation}

Moreover we need to assume smallness of the following ``modified $L^2$-height excess'', namely assume that for $\eps=\eps(n,q, \sigma)$ small enough we have

$$\int\limits_{B^{n}_1(0) \times \R} |x^{n+1}|^2 d\|V\|+  \left(\int\limits_{B^{n}_1(0) \times \R} |h|^p d\|V\|\right)^{\frac{2}{p}}<\eps.$$

Then there exist Lipschitz functions $u^j:B_\sigma \to \R$ for $j\in\{1, 2, ... q\}$ (the \textit{multiple-valued graph}) such that $u^{j} \, : \, B_{\s} 
\to {\mathbb R},$ $j=1, 2, \ldots, q$, with $u^{1} \leq u^{2} \leq \ldots \leq u^{q}$ and 
\begin{equation*}\label{blow-up-2}
{\rm Lip} \, u^{j} \leq 1/2 \;\;\; \mbox{for each} \;\;\;  j \in \{1, 2, \ldots, q\}
\end{equation*}
and a measurable subset $\Sigma$ of $B_{\s}$ (the \textit{bad set}) such that 
\begin{equation}\label{blow-up-4}
{\rm spt} \, \|V\| \cap ((B_{\s} \setminus \Sigma)\times \R) = 
\cup_{j=1}^{q} {\rm graph} \, u^{j} \cap ((B_{\s} \setminus \Sigma)\times \R)
\end{equation}
and 
\begin{equation}\label{blow-up-3}
\|V\|(\Sigma\times \R) + {\mathcal H}^{n} \, (\Sigma) \leq C\left( \int\limits_{B^{n}_1(0) \times \R} |\nabla^{V}x^{n+1}|^2 d\|V\|+ \left(\int\limits_{B^{n}_1(0) \times \R} |h|^p d\|V\|\right)^{\frac{2}{p}} \right)
\end{equation}
where $C  = C(n, q, \s) \in (0, \infty).$ 
\end{Prop}

\noindent For a nice presentation of this result (extended to other contexts as well) see \cite[Theorem 4]{WhiteStratification}.

\medskip

We\footnote{Note the following scaling properties: the invariance is to be understood with respect to the operation (geometrically natural homothethy on the graph of $f:\R^n \to \R$) $\tilde{f}(x):= \frac{f(rx)}{r}$. Under this change of variable we find (i) $\int_{B_1} |\tilde{f}|^2 = r^{-n-2}\int_{B_r}|f|^2$ and (ii) $\left(\int_{B_1} |\tilde{H}|^p\right)^{\frac{1}{p}}=r^{(1-\frac{n}{p})}\left(\int_{B_r}|H|^p\right)^{\frac{1}{p}}$, where $H$ and $\tilde{H}$ denote respectively the mean curvatures of the graphs of $f$ and $\tilde{f}$ in $\R^{n+1}$.} will now recall the concept of \textit{coarse blow-up} for a sequence $\{V_k\}$ of varifolds where each $V_k$ in the class $\mathcal{S}_{H_k}$ for $H_k \leq H_0$. Assume that (\ref{eq:circaqsheet}) with $V_k$ in place of $V$ holds for every $k$. Moreover we assume that a certain ``modified $L^2$-height excess'' is going to zero. The notion of excess that we use is the following (for $\eps\leq 1$ and $\r \leq 1$):
$$\hat{E}^2_{V,\eps,\r}:= \r^{-n-2}\!\!\!\!\int\limits_{B^{n}_{\r}(0) \times \R} |x^{n+1}|^2 d\|V\|\,\,+\,\, \r^{1-\frac{n}{p}}\frac{\left(\int\limits_{B^{n}_{\r}(0) \times \R} |h|^p d\|V\|\right)^{\frac{1}{p}}}{\eps} .$$

\begin{oss}
 \label{oss:comparetwoexcesses}
Since $\eps$ will always be small, the smallness of the excess $\hat{E}_{V, \eps. 1}$ will always guarantee the smallness of the quantity appearing in the assumption in Proposition \ref{Prop:AlmgrenDec}. The reason tfor the choice of the power $1/p$ in the second term of $\hat{E}_{V, \eps. 1}$ (as opposed to the exponent $2/p$ which is customary in \cite{Allard}, \cite{Alm}, \cite{WicAnnals}) will become clear in the forthcoming arguments. This necessary choice will reflect into the fact that, in a first moment, we will get a Sheeting Theorem with $C^{1,\alpha}$-regularity for $\alpha<\frac{1}{2}$.
\end{oss}

We assume that there exists a sequence $\eps_k \to 0$ such that

$$\hat{E}_k:=\hat{E}_{V_k, \eps_k,1} <{\eps}_k \to 0.$$
Let $\s \in (0, 1)$ and apply Almgren's approximation Proposition \ref{Prop:AlmgrenDec}: we obtain for each sufficiently large $k$ Lipschitz functions $u_{k}^{j} \, : \, B_{\s} \to {\mathbb R}$ ($j=1, 2, \ldots, q$) and a measurable subset $\Sigma_{k}$ of $B_{\s}$ with the properties explained above.
Set 
\begin{equation}\label{blow-up-3-1}
v_{k}^{j}(x) = {\hat E}_{k}^{-1}u_{k}^{j}(x)
\end{equation} 
for $x \in B_{\s},$ and write $v_{k} = (v_{k}^{1}, v_{k}^{2}, \ldots, v_{k}^{q}).$ 
Then  $v_{k}$ is Lipschitz on $B_{\s};$ and by (\ref{blow-up-4}) and (\ref{blow-up-3}),
\begin{equation}\label{blow-up-5}
\int_{B_{\s}} |v_{k}|^{2} \leq C, \;\;\; C = C(n, q, \s) \in (0, \infty).
\end{equation} 
Furthermore, by taking $X = x^{n+1}\zeta^{2}e^{n+1}$ for a suitable choice of $\zeta \in C^{1}_{c}(B_{1} \times {\mathbb R})$ in the first variation formula \cite[16.4]{SimonNotes} for an arbitrary $V$ with generalized mean curvature $H \in L^p(\|V\|)$ and a straightforward computation, we also have (compare \cite[22.2]{SimonNotes})

\begin{equation}
\label{eq:tiltcontrolledbyhight}
{\r}^{-n}\!\!\!\int\limits_{B_{\r/2} \times \R} |\nabla^V x^{n+1}|^ 2 d\|V\| \leq C(n) \left[ {\r}^{-n-2}\int_{B_{\r}}|x^{n+1}|^2 d\|V\| +{\r}^{2(1-\frac{n}{p})}\left(\int_{B_{\r}}|h|^p\right)^{\frac{2}{p}}\right].
\end{equation}

\begin{oss}
\label{oss:lipschitzgraphtiltexcess}
Note that, if $V$ is the graph of a Lipschitz function $f$, then $|\nabla^V x^{n+1}|^2=\frac{|Df|^2}{1+|Df|^2}$.
\end{oss}

As a consequence of Remark \ref{oss:lipschitzgraphtiltexcess}, and using ${\rm Lip} \, u^{j} \leq 1/2$ and (\ref{blow-up-3}), we find that the sequence $V_k$ satisfies
\begin{equation}\label{blow-up-6}
\int_{B_{\s}}|Dv_{k}|^{2} \leq C, \;\;\; C = C(n, q,\s) \in (0, \infty).
\end{equation}
In view of the arbitrariness of $\s \in (0, 1),$ by (\ref{blow-up-5}), (\ref{blow-up-6}), Rellich's theorem and a diagonal sequence argument, we obtain a function $v \in W^{1, 2}_{\rm loc} \, (B_{1}; {\mathbb R}^{q}) \cap L^{2} \, (B_{1}; {\mathbb R}^{q})$ and a subsequence $\{k_{j}\}$ of $\{k\}$ such that 
$v_{k_{j}} \to v$ as $j \to \infty$ in $L^{2} \, (B_{\s}; {\mathbb R}^{q})$ and weakly in $W^{1, 2} \, (B_{\s}; {\mathbb R}^{q})$ for every $\s \in (0, 1).$ 

\medskip

\noindent
{\bf Definitions: (1) Coarse blow-ups}: Let $v \in W^{1,2}_{\rm loc} \, (B_{1}; {\mathbb R}^{q}) \cap L^{2} \, (B_{1}; {\mathbb R}^{q})$ correspond, in the manner described above, to (a subsequence of) a sequence $\{V_{k}\}$ of integral $n$-varifolds ($V_k \in \mathcal{S}_{H_k}$) on $B_{2}^{n+1}(0)$ satisfying (\ref{eq:circaqsheet}) and with ${\hat E}_{k} \to 0$. We shall call $v$ a {\em coarse blow-up} of the sequence $\{V_{k}\}$.

\noindent
{\bf (2) The Class ${\mathcal B}_{q}$}: Denote by ${\mathcal B}_{q}$ the collection of all coarse blow-ups of sequences 
of varifolds $\{V_{k}\} \subset {\mathcal S}_{H_k}$ satisfying (\ref{eq:circaqsheet}) with $V_k$ in place of $V$ and for which 
${\hat E}_{k} \to 0.$

We have the following: 

\begin{thm}[\textbf{Linearization theorem}]\label{coarse} If the induction hypotheses (H1), (H2), (H3) hold, then 
${\mathcal B}_{q}$ is a \textit{proper blow-up class}, i.e.~it satisfies conditions $\B{1}$-$\B{7}$ below.
\end{thm}

\begin{df} \cite[Section 4]{WicAnnals}[\textit{proper blow-up class}]
 Fix an integer $q \geq 1$. A collection ${\mathcal B}$ of functions $v  = (v^{1}, v^{2}, \ldots, v^{q})\, : \, B_{1} \to {\mathbb R}^{q}$ is said to be a \emph{proper blow-up class} if it satisfies the following properties for some fixed constant $C \in (0, \infty)$:
\begin{itemize}
\item[(${\mathcal B{\emph 1}}$)] ${\mathcal B} \subset  W^{1, 2}_{\rm loc} \, (B_{1}; {\mathbb R}^{q}) \cap L^{2} \, (B_{1}; {\mathbb R}^{q}).$
\item[(${\mathcal B{\emph 2}}$)] If $v \in {\mathcal B}$, then $v^{1} \leq v^{2} \leq \ldots \leq v^{q}$.
\item[(${\mathcal B{\emph 3}}$)] If $v \in {\mathcal B}$, then $\Delta \, v_{a} = 0$ in $B_{1}$ where $v_{a} = q^{-1}\sum_{j=1}^{q} v^{j}.$
\item[(${\mathcal B{\emph 4}}$)] For each $v \in {\mathcal B}$ and each $z \in B_{1}$, either (${\mathcal B}{\emph 4 \, I}$) or (${\mathcal B}{\emph 4 \, II}$) below is true: 
\begin{itemize}
\item[(${\mathcal B{\emph 4 \,I}}$)] The Hardt-Simon inequality $$\sum_{j=1}^{q} \int_{B_{\r/2}(z)} R_{z}^{2-n} \left(\frac{\partial \, \left((v^{j} - v_{a}(z))/R_{z}\right)}{\partial \, R_{z}}\right)^{2} \leq C \, \r^{-n-2}\int_{B_{\r}(z)} |v - \ell_{v,\, z}|^{2}$$ 
holds for each $\r \in (0, \frac{3}{8}(1- |z|)]$, where $R_{z}(x) = |x - z|,$ $\ell_{v, \, z}(x) = v_{a}(z) + Dv_{a}(z) \cdot (x -z)$ 
and $(v - \ell_{v, \, z}) = (v^{1} - \ell_{v,\,z}, v^{2} - \ell_{v, \, z}, \ldots, v^{q} - \ell_{v, \, z}).$
\item[(${\mathcal B{\emph 4 \, II}}$)] There exists $\s  = \s(z) \in (0, 1 - |z|]$ such that $\Delta \, v = 0$ in $B_{\s}(z).$
\end{itemize}
\item[(${\mathcal B{\emph 5}}$)] (Invariances) If $v \in {\mathcal B}$, then
\begin{itemize}
\item[(${\mathcal B{\emph 5 \, I}}$)]${\widetilde v}_{z, \s}(\cdot) \equiv \|v(z + \s(\cdot))\|_{L^{2}(B_{1}(0))}^{-1}v(z + \s(\cdot)) \in {\mathcal B}$ for each $z \in B_{1}$ and $\s \in (0, \frac{3}{8}(1 - |z|)]$ whenever $v \not\equiv 0$ in $B_{\s}(z);$
\item[(${\mathcal B{\emph 5 \, II}}$)] $v \circ \gamma \in {\mathcal B}$ for each orthogonal rotation $\gamma$ of ${\mathbb R}^{n}$ and
\item[(${\mathcal B{\emph 5 \, III}}$)] $\|v - \ell_{v}\|_{L^{2}(B_{1}(0))}^{-1}\left(v- \ell_{v}\right) \in {\mathcal B}$ whenever $v  - \ell_{v}\not\equiv 0$ in $B_{1}$, where  $\ell_{v}(x) = v_{a}(0) + Dv_{a}(0) \cdot x$ for $x \in {\mathbb R}^{n}$ and $v- \ell_{v}$ is abbreviation for $(v^{1} - \ell_{v},v^{2} -  \ell_{v}, \ldots, v^{q} - \ell_{v}).$ 
\end{itemize}
\item[(${\mathcal B{\emph 6}}$)] (Compactness) If $\{v_{k}\}_{k=1}^{\infty} \subset {\mathcal B}$ then there exists a subsequence $\{k^{\prime}\}$ 
of $\{k\}$ and a function $v \in {\mathcal B}$ such that $v_{k^{\prime}} \to v$ locally in $L^{2}(B_{1})$ and locally weakly in $W^{1, 2}(B_{1}).$
\item[(${\mathcal B{\emph 7}}$)] (Minimum Distance property) If $v \in {\mathcal B}$ is such that for each $j=1, 2, \ldots, q$, there exist linear functions 
$L^{j}_{1}, L^{j}_{2} \; : \; {\mathbb R}^{n} \to {\mathbb R}$ with $L^{j}_{1}(y, 0) = L^{j}_{2}(y, 0)$ for 
$y \in  {\mathbb R}^{n-1},$    
$v^{j}(y, x^{n}) = L^{j}_{1}(y, x^{n})$ if $x^{n} < 0$ and $v^{j}(y, x^{n}) = 
L^{j}_{2}(y, x^{n})$ if $x^{n} \geq 0$, then 
$v^{1} = v^{2} = \ldots = v^{q} = L$ for some linear function $L \; : \; {\mathbb R}^{n} \to {\mathbb R}.$ 
\end{itemize}
\end{df}

Before proving Theorem \ref{coarse}, let us take a close look at the significance of its conclusion. The following regularity theorem says that functions in a proper blow-up class are nothing but harmonic functions. For the proof we refer to \cite{WicAnnals}. 

\begin{thm}[\textbf{\cite[Theorem 4.1]{WicAnnals} Sheeting theorem for proper blow-up classes (linear theory)}]\label{harmonic}
If ${\mathcal B}$ is a proper blow-up class for some $C \in (0, \infty)$, then 
each $v \in {\mathcal B}$ is harmonic in $B_{1}$. Furthermore, if $v \in {\mathcal B}$ and there is a point $z \in B_{1}$ such that 
(${\mathcal B{\it{4} \, I}}$) is satisfied, then $v^{1} = v^{2} = \ldots = v^{q}.$
\end{thm}

\begin{oss}
\label{oss:B7mindistance}
Note the significance of property $({\mathcal B{\emph 7}})$ for Theorem~\ref{harmonic}; an obvious counterexample to the theorem in the absence of this property  is the set of all functions $v = (v^{1}, v^{2}) \, : \, B_{1} \to {\mathbb R}^{2}$
 such that $v^{1} = \min \{\ell, \ell^{\prime}\}$ and $v^{2} = \max \{\ell, \ell^{\prime}\}$ for some affine functions $\ell, \ell^{\prime} \, : \, {\mathbb R}^{n} \to {\mathbb R}$ with $\|\ell\|_{L^{2}(B_{1})}, \|\ell^{\prime}\|_{L^{2}(B_{1})} \leq 1.$ 
 
Note that in view of the compactness property (${\mathcal B{\emph 6}}$), the property 
(${\mathcal B{\emph 7}}$) implies the following: Let $H  = (H^{1}, H^{2}, \ldots, H^{q}) \, : \, {\mathbb R}^{n} \to {\mathbb R}^{q}$ be such that $H^{1} \leq H^{2} \leq \ldots \leq H^{q}$, at least two of $H^{1}, H^{2}, \ldots, H^{q}$ are distinct, and for each $j=1, 2, \ldots, q$, there exist linear functions 
$L^{j}_{1}, L^{j}_{2} \; : \; {\mathbb R}^{n} \to {\mathbb R}$ with $L^{j}_{1}(y, 0) = L^{j}_{2}(y, 0)$ for 
$y \in  {\mathbb R}^{n-1}$ such that     
$H^{j}(y, x^{n}) = L^{j}_{1}(y, x^{n})$ if $x^{n} < 0$ and $H^{j}(y, x^{n}) = 
L^{j}_{2}(y, x^{n})$ if $x^{n} \geq 0$. Then
there exists $\eps = \eps(H, {\mathcal B}) \in (0, 1)$ such that $\int_{B_{1/2}} |v - H|^{2} \geq \eps$ for each 
$v \in {\mathcal B}.$  This explains why property (${\mathcal B{\emph 7}}$) is called the Minimum Distance property. 
\end{oss}

\begin{proof}[\textbf{first part of the proof of Theorem \ref{coarse}}: properties $\B{1}$-$\B{2}$-$\B{3}$-$\B{5}$-$\B{6}$]
 
The case $q=1$ of Theorem \ref{coarse} corresponds to Allard's regularity theorem 
see \cite[p. 880]{WicAnnals} and \cite[p. 115]{SimonNotes}: in this case one only needs to show that the function $v$ is harmonic.

For general case (under the inductive assumptions) it is rather straightforward to verify properties $\B{1}$, $\B{2}$, $\B{5}$, $\B{6}$ and for property $\B{3}$ one needs to use the first variation formula very similarly to the case $q=1$ \footnote{All these properties would be true even without stability and structural assumptions on $V$.}.
\end{proof}

It takes substantially more effort to establish that ${\mathcal B}_{q}$  satisfies the remaining properties $\B{4}$ and $\B{7}$---in particular $\B{7}$ (the Minimum Distance Property) will require a rather lengthy argument. For the sake of clarity we prefer to keep the proofs of $\B{4}$ and $\B{7}$ in distinct subsections.

\subsection{Proof of Theorem \ref{coarse}: verification of $\B{4}$ }
\label{HardtSimon}

To show that property $\B{4}$ holds for ${\mathcal B}_{q}$ (with a constant $C$ depending only on $n$, $q$ and $H_0$), one argues as follows: Fix $v \in {\mathcal B}_{q}$ and $z \in B_{1}$. If $$v - \ell_{v, z} \equiv (v^{1} - \ell_{v, z}, v^{2} - \ell_{v, z}, \ldots, v^{q} - \ell_{v, z}) = 0$$ there is nothing further to prove, so assume $v - \ell_{v, z} \neq 0$. Then by $({\mathcal B{\emph5III}})$, 
$\widetilde{v} = \|v - \ell_{v, z}\|^{-1}_{L^{2}(B_{1})}(v - \ell_{v, z}) \in {\mathcal B}_{q}$. Let $\{V_{k}\} \subset {\mathcal S}$ be a sequence of varifolds whose coarse blow-up is $\widetilde{v}$. Consider the following two alternatives, one of which must hold: 
\begin{itemize}
\item[(1)] there exists $\s > 0$ such that for all sufficiently large $k$, $Z \in  B_{\s}(z)  \times \R \implies 
\Theta \, (\|V_{k}\|, Z) < q;$ 
\item[(2)] there exists a subsequence $\{k^{\prime}\}$ of $\{k\}$ and points 
$Z_{k^{\prime}}  = (z_{k^{\prime}}^{1}, z_{k^{\prime}}^{\prime}) \in {\rm spt} \, \|V_{k^{\prime}}\|$ with 
$$\Theta \, (\|V_{k^{\prime}}\|, Z_{k^{\prime}}) \geq q$$ such that $z_{k^{\prime}}^{\prime} \to z.$  
\end{itemize}
These two alternatives correspond respectively to the validity of $({\mathcal B{\emph 4II}})$ and $({\mathcal B{\emph 4I}})$, as we will show in the rest of the subsection.

\medskip

We begin by showing that, if (1) is true, then alternative $({\mathcal B{\emph 4II}})$ must hold. Indeed (1) implies, by means of Proposition \ref{Prop:elementaryconsequence}, that for all sufficiently large $k$, ${\rm sing} \, V_{k} \res (B_{\s}(z) \times \R) =  \emptyset$ if $n \leq 6$, ${\rm sing} \, V_{k} \res (B_{\s}(z) \times \R)$ is discrete if $n=6$ and 
 $${\rm dim}_{\mathcal H} \, \left({\rm sing} \, V_{k} \res (B_{\s}(z) \times \R)\right) \leq n-7$$ if $n \geq 7.$  Thus in this case, we may apply Theorem~\ref{thm:SS} to conclude that, for all sufficiently large $k$, $V_{k} \res (B_{\s/2}(z) \times \R)$ is given by $q$ (weakly) ordered graphs of $C^{1,\alpha}$ functions $u_k^j$ on $B_{\s/2}(z)$ (for $j=1, ... q$), with $\|u_k^j\|_{C^{1,\alpha}(B_{\s/2}(z))}\leq C \hat{E}_k$ by the associated elliptic estimates (for a constant $C$ independent of $k$). As a consequence, upon extracting a subsequence that we do not relabel, we have that, for every $j$, $\hat{E}_k^{-1} u_k^j$ converges in $C^1(B_{\s/2}(z))$ to a $C^{1,\alpha}$-function $\tilde{v}^j$ (the limit is necessarily the coarse blow up).
 
On the other hand, in view of (H3), we have that locally around every point $$\spt{V_{k} \res (B_{\s/2}(z) \times \R)}$$ is given by the union of one or two smooth graphs (with small gradient) satisfying the CMC equation individually (here we are working on the subsequence implicitly extracted above). Consider the set $U$ of $x \in B_{\s}(z)$ such that $\spt{V_{k} \res (B_{\s/2}(z) \times \R)}$ is embedded at $(x,u_j(x))$ for every $j\in\{1, ..., q\}$. Such a set is open by definition and its complement has Hausdorff dimension $1$, by (H3) and by remark \ref{oss:maxprincsmooth}. On each connected component of $U$ we have that $\spt{V_{k} \res (U \times \R)}$ is given by $\tilde{q}\leq q$ (weakly) ordered graphs of smooth functions $\tilde{u}_k^j$ on $B_{\s/2}(z)$ (for $j=1, ... \tilde{q}$), individually solving the CMC equation. At this stage $\tilde{q}$ might depend on the chosen connected component of $U$, however again by the structure given by (H3) we have that if $y \in \spt{V_{k} \res (B_{\s/2}(z) \times \R)}$ and $\pi(y)$ is on the (topological) boundary of two distinct connected components of $U$, then $\tilde{q}$ cannot change when we pass from one connected component to the other. This allows to obtain a unique well-defined extension to $B_{\s/2}(z)$ of the functions $\tilde{u}_k^j$ (for $j=1, ... \tilde{q}$) such that the union of their graphs describes $\spt{V_{k} \res (B_{\s/2}(z) \times \R)}$. We remark that $\tilde{u}_k^j$ individually solve the CMC equation with $h_k$ on the r.h.s. and with $h_k \to 0$. Hence in this case we conclude with the help of higher order elliptic estimates that $\hat{E}_k^{-1} \tilde{u}_k^j$ converges (upon extracting a subsequence of the previous one, that we again do not relabel) in $C^2(B_{\s/2}(z))$ to a $C^{2,\alpha}$-function $\tilde{\tilde{v}}^j$ and by passing the CMC equation to the limit and recalling that $\hat{E}_k^{-1} h_k \to 0$ we obtain that $\Delta \tilde{\tilde{v}}^j=0$ for every $j$. The harmonicity implies the fact that if $\tilde{\tilde{v}}^j(x)=\tilde{\tilde{v}}^{j+1}(x)$ for a point $x \in B_{\s/2}(z)$ and an index $j\in\{1, ..., \tilde{q}-1\}$, then $\tilde{\tilde{v}}^j\equiv \tilde{\tilde{v}}^{j+1}$ on $B_{\s/2}(z)$.

Comparing the two sequences (with the same indexation) $\hat{E}_k^{-1} u_k^j$ and $\hat{E}_k^{-1} \tilde{u}_k^j$ we have, by construction, that, for every $k$ the support of the graphs of $u_k^j$ agrees with the union of the graphs of the functions $\tilde{u}_k^j$. The convergence respectively in $C^1$ and $C^2$ for the two sequences $\hat{E}_k^{-1} u_k^j$ and $\hat{E}_k^{-1} \tilde{u}_k^j$ forces therefore the union of the graphs of the functions $\tilde{v}^j$ to agree with the union of the graphs of the functions $\tilde{\tilde{v}}^j$; the latter is however a disjoint union of $C^2$ graphs, as established before. Then necessarily for each $\tilde{v}^j$ there is a $\tilde{\tilde{v}}^\ell$ with the same graph, in particular all the functions $\tilde{v}^j$ are harmonic and alternative $({\mathcal B{\emph 4II}})$ must hold.

\medskip

If on the other hand (2) holds, we will use an argument due to Hardt-Simon \cite{HS} to conclude that alternative $({\mathcal B{\emph 4I}})$ must hold.

We will need some preliminary estimates. We follow the reasoning leading to \cite[Theorem 7.1(a)]{WicAnnals} with the due changes. The almost monotonicity formula in our case is given by \cite[17.6(1)]{SimonNotes}, in which we set $\Lambda=C(n,q,p)\|H\|_{L^p}$ in view of \cite[17.9(2)]{SimonNotes}; therefore \cite[(7.1)]{WicAnnals} is replaced by an inequality $\leq$ with an additional term $C(n,q,p)\|H\|_{L^p}$ on the right-hand side. We then follow the argument given in the estimate \cite[(7.2)]{WicAnnals} (which uses the condition $\Theta \, (\|V\|, Z) \geq q$) by replacing (in the second line of \cite[(7.2)]{WicAnnals}) the excess $\hat{E}^2_V$ appearing there with $$\int\limits_{B^{n}_1(0) \times \R} |\nabla^V x^{n+1}|^2 d\|V\|+ \left(\int\limits_{B^{n}_1(0) \times \R} |H|^p d\|V\|\right)^{\frac{2}{p}}.$$ (This term originates in (\ref{blow-up-3}) and enters the argument in the computation at the top of \cite[page 877]{WicAnnals}, replacing the term $\hat{E}^2_V$ appearing there.) The replacements for \cite[(7.1)]{WicAnnals} and \cite[(7.2)]{WicAnnals} that we have just pointed out imply that 
\begin{equation}
 \label{eq:replacement(7.3)}
\int\limits_{B^{n+1}_{3/8}(Z)} \frac{|(X-Z)^\bot|^2}{|X-Z|^{n+2}} d\|V\|(X) \leq C\left(  \int\limits_{B^{n}_1(0) \times \R} |x^{n+1}|^2 d\|V\|+  \left(\int\limits_{B^{n}_1(0) \times \R} |H|^p d\|V\|\right)^{\frac{1}{p}} \right), 
\end{equation}
which replaces \cite[(7.3)]{WicAnnals}. The right-hand side is $C\hat{E}^2_{V,1,1}$ in our notation.

The replacement for \cite[(7.4)]{WicAnnals} is quite immediate, no change is needed except that we need to use the almost monotonicity \cite[17.6 (2)]{SimonNotes} and therefore the estimate becomes
\begin{equation}
 \label{eq:replacement(7.4)}
\int\limits_{B^{n+1}_{1/4}(Z)} \frac{|(X-Z)^\bot|^2}{|X-Z|^{n+2}} d\|V\|(X) \geq C|z^{n+1}|^2 - C\hat{E}^2_{V,1,1}, 
\end{equation}
i.e.~$\hat{E}^2_{V,1,1}$ appears instead of the term appearing in \cite[(7.4)]{WicAnnals}, because in our situation we use (\ref{eq:tiltcontrolledbyhight}).

Putting together the two inequalities (\ref{eq:replacement(7.3)}) and (\ref{eq:replacement(7.4)}), we have the following statement, that replaces \cite[Theorem 7.1(a)]{WicAnnals}: there exists a number $\eps_1>0$ such that, if (\ref{eq:circaqsheet}) holds and $\hat{E}^2_{V,\eps,1}\leq \eps$, then for any $Z$ such that $\Theta(\|V\|, Z)\geq q$ we have 
\begin{equation}
 \label{eq:replacementThm7.1a}
|z^{n+1}|^2 \leq C \hat{E}^2_{V,1,1}, 
\end{equation}

\medskip

Let us now go back to the proof of $({\mathcal B{\emph 4I}})$, assuming that alternative (2) holds. For any fixed $\rho \in \left(0, \frac{3}{8}(1 - |z|)\right)$ we use (\ref{eq:replacement(7.3)}) with $\eta_{Z_{k^{\prime}}, \r \, \#} \, V_{k^{\prime}}$ in place of $V$ and $0$ in place of $Z$ and we follow (without changes) the computation at the top of \cite[page 883]{WicAnnals}. Combining these two inequalities we deduce that, for all sufficiently large values of $k^{\prime}$, (here $u_{k}$, $\Sigma_{k}$ are as in Proposition \ref{Prop:AlmgrenDec}) the following inequality holds (this replaces \cite[(8.9)]{WicAnnals})

 \begin{eqnarray*}\label{step2-7}
&&\sum_{j=1}^{q}\int_{B_{\r/2}(z^{\prime}_{k^{\prime}}) \setminus \Sigma_{k^{\prime}}}
\left(\frac{R_{z^{\prime}_{k^{\prime}}}^{2}}{(u_{k^{\prime}}^{j} - z_{k^{\prime}}^{n+1})^{2} + R_{z_{k^{\prime}}^{\prime}}^{2}}\right)^{\frac{n+2}{2}}
R_{z_{k^{\prime}}^{\prime}}^{2-n} \left(\frac{\partial \, 
\left((u_{k^{\prime}}^{j} - z_{k^{\prime}}^{n+1})/R_{z_{k^{\prime}}^{\prime}}\right)}{\partial \, R_{z_{k^{\prime}}^{\prime}}}\right)^{2}\nonumber\\ 
&&\hspace{.1in}\leq C\, \left[\r^{-n-2}\int_{B_{\r}(z^{\prime}_{k^{\prime}})\times \R}|x^{n+1}|^{2}d\|V_{k^{\prime}}\|(X) + \r^{1-\frac{n}{p}}\left(\int_{B_{\r}(z^{\prime}_{k^{\prime}})\times \R}|H_{k^{\prime}}|^p d\|V_{k^{\prime}}\|\right)^{\frac{1}{p}}\right]
\end{eqnarray*}
where $C = C(n, q, \|H\|_{L^p}) \in (0, \infty)$; for any fixed $y \in {\mathbb R}^{n}$, $R_{y}(x) = |x - y|$ and $\frac{\partial}{\partial \, R_{y}}$ denotes the radial derivative $\frac{x - y}{R_{y}} \cdot D $. One then notes that the first factor (in the left-hand side integrand) goes to $1$ as $k' \to \infty$ (in the almost everywhere sense) thanks to (\ref{eq:replacementThm7.1a}). The final step is then to divide both sides of the above inequality by ${\hat E}^{2}_{k^{\prime}}$ and pass to the limit, which yields the Hardt--Simon inequality in $({\mathcal B\emph{4I}})$. For the right-hand side observe that, under our smallness assumption $\hat{E}^2_{V,\eps,1}\leq \eps$, we have $\left(\int\limits_{B^{n}_1(0) \times \R} |H|^p d\|V\|\right)^{\frac{1}{p}}\leq \eps \hat{E}_{V, \eps, 1}^2$, so after dividing by ${\hat E}^{2}_{k^{\prime}}$, the term involving the mean curvature on the right-hand-side goes to $0$ in the limit and only the term involving the $L^2$-height remains. Follow \cite[Section 8]{WicAnnals} for details.

\subsection{Proof of Theorem \ref{coarse}: verification of $\B{7}$ }
\textit{\textbf{Non-concentration of tilt-excess}}.
The following lemma establishes a certain a priori estimate for integral $n$-varifolds $V$ on $B_{2}^{n+1}(0)$ with $h \in L^p(\|V\|)$ and with small ``modified height-excess'' relative to a plane $P.$ This estimate is inspired by Simon's estimates \cite{Simon2} and says that if the density ratio of $V$ at unit scale is between $q-1/2$ and $q+1/2$ and if there are points of ${\rm spt} \, \|V\|$ with density $\geq q$ ``evenly distributed'' in a certain precise sense near an $(n-1)$-dimensional subspace $L$, then the tilt-excess of $V$ relative to $P$ in a small neighborhood of $L$ is at most a small fraction of the total excess of $V.$ This result, used  in the proof that ${\mathcal B}_{q}$ satisfies property $({\mathcal B\emph{7}})$, may also be of independent interest. The precise claim, assuming without loss of generality that $P = {\mathbb R}^{n} \times \{0\}$, is as follows:

\begin{lem}\label{non-concentration}
Let $q$ be a positive integer, $\tau \in (0, 1/16)$ and $\mu \in (\frac{n}{p}, 1).$There exists a number $\eps_{1} = \eps_{1}(n, q, \tau, \mu)~\in~(0,1/2)$ such that if $V$ is an integral $n$-varifold on $B_{2}^{n+1}(0)$ with $h \in L^p(\|V\|)$ and with 
$$(\omega_{n}2^{n})^{-1}\|V\|(B_{2}^{n+1}(0)) < q + 1/2, \;\;\;  q-1/2 \leq \omega_{n}^{-1}\|V\|(B_{1} \times \R) < q + 1/2 \;\; {\rm and}$$  
$$\hat{E}^2_{V,\eps_1,1}:=\int_{B^{n}_1(0) \times \R} |x^{n+1}|^{2}d\|V\|(X) + \frac{\left(\int\limits_{B^{n}_1(0) \times \R} |H|^p d\|V\|\right)^{\frac{1}{p}}}{\eps_1} \leq {\eps}_{1},$$
and if $L$ is an $(n-1)$-dimensional subspace of $\{0\} \times {\mathbb R}^{n}$ such that 
$$L \cap B_{1/2} \subset \left(\{Z  \in {\rm spt} \, \|V\|\, : \, \Theta \, (\|V\|, Z) \geq q\}\right)_{\tau}, \;\; {\rm then}$$

$\displaystyle \int_{(L)_{\tau} \cap (B_{1/2} \times {\mathbb R} )} |\nabla^{V} \, x^{n+1}|^{2} d\|V\|(X)$

$\displaystyle \hspace{1in}  \leq C\tau^{1-\mu}\left[\int_{B_{1} \times {\mathbb R} } |x^{n+1}|^{2}d\|V\|(X) + \left(\int_{B_{1} \times {\mathbb R} }|H|^p d\|V\|(X)\right)^{1/p} \right].$

\noindent
Here  $C = C(n, q, \mu) \in (0, \infty),$ so
in particular $C$ is independent of $\tau,$ and for a subset $A$ of 
${\mathbb R}^{n+1}$, we use the notation $(A)_{\tau} = \{X \in {\mathbb R}^{n+1} \, : \, {\rm dist} \, (X, A) \leq \tau\}.$
\end{lem}

\begin{proof}
 
The argument follows \cite[Section 7]{WicAnnals}. The number $\mu\in(0,1)$ needs to satisfy $(n-\mu) q <n$, i.e. $\mu>\frac{n}{p}$, in order to ensure the $L^q$-summability of the test function $\psi$ in \cite[page 878]{WicAnnals}, where $q$ is such that $\frac{1}{p} + \frac{1}{q}=1$ (the $L^q$-integrability follows by a decomposition of the ball in diadic annuli). Observe that, in following the argument in the formula at the bottom of \cite[p. 878]{WicAnnals} we find (from the first variation formula) an extra term $\left|\int h \cdot \psi\right|$ on the right-hand side. This term can be bounded using H\"older's inequality and the facts that $h \in L^p$ for $p>n$ and $\psi$ is in $L^q$. The term $\|H\|_{L^p}$ is controlled by $\eps_1 \hat{E}_{V, \eps_1, 1}^2$. This should be combined with (\ref{eq:replacement(7.3)}) and (\ref{eq:replacementThm7.1a}) to obtain, for every $Z=(z^\prime, z^{n+1}) \in \spt{V}$ such that $\Theta(\|V\|,Z)\geq q$ the following:

$\displaystyle \int\limits_{B^{n+1}_{4\tau}(Z)}\!\!\!|x^{n+1} - z^{n+1}|^2 d\|V\|(X)$

$\displaystyle \hspace{1in} \leq C \tau^{n+2-\mu} \left( \int\limits_{B_1 \times \R}|x^{n+1}|^2 d\|V\|(X) + \left(\int\limits_{B_1 \times \R} |H|^p d\|V\|(X)\right)^{\frac{1}{p}} \right).$

In view of the hypothesis, for each $Y \in L \cap B_{1/2}$ we can find $z^{n+1}\in\R$ so that the previous estimate implies

$\displaystyle \int\limits_{B^{n+1}_{3\tau}(Y)}\!\!\!|x^{n+1} - z^{n+1}|^2 d\|V\|(X)$

$\displaystyle \hspace{1in} \leq C \tau^{n+2-\mu} \left( \int\limits_{B_1 \times \R}|x^{n+1}|^2 d\|V\|(X) + \left(\int\limits_{B_1 \times \R} |H|^p d\|V\|(X)\right)^{\frac{1}{p}} \right)$

\noindent
and using (\ref{eq:tiltcontrolledbyhight}) after a suitable translation we find that for any $Y \in L \cap B_{1/2}$

$\displaystyle \int\limits_{B^{n+1}_{\frac{3\tau}{2}}(Y)}\!\!\!|\nabla^V x^{n+1}|^2 d\|V\|(X)$

$\displaystyle \hspace{.5in} \leq C \tau^{n-\mu} \left( \int\limits_{B_1 \times \R}|x^{n+1}|^2 d\|V\|(X) + \left(\int\limits_{B_1 \times \R} |H|^p d\|V\|(X)\right)^{\frac{1}{p}} \right)$ 

$\displaystyle \hspace{2.5in}+ C \tau^{n+2(1-\frac{n}{p})}\left(\int\limits_{B_1 \times \R} |H|^p d\|V\|(X)\right)^{\frac{2}{p}}$
 
$\displaystyle \hspace{1.2in} \leq  C \tau^{n-\mu} \left( \int\limits_{B_1 \times \R}|x^{n+1}|^2 d\|V\|(X) + \left(\int\limits_{B_1 \times \R} |H|^p d\|V\|(X)\right)^{\frac{1}{p}} \right).$

\noindent
We may cover the set $(L)_{\tau} \cap \left( B_{1/2} \times \R \right)$ with balls $B_{\frac{3\tau}{2}}^{n+1}(Y_j)$ for $j=1, ... ,N$ with $N\leq C(n) \tau^{1-n}$ and $Y_j \in L \cap B_{1/2}$ and with a fixed maximal number of balls overlapping (also a constant $C(n)$). Then we can conclude the proof by adding up the previous inequality (written for each $B_{\frac{3\tau}{2}}^{n+1}(Y_j)$) for $j=1, ... N$.
\end{proof}

\textit{\textbf{Easy case of the Minimum Distance Property}}. We next establish the validity of $\B{7}$ in the ``easy case'', namely the case where the coarse blow up $v$ is given by $q$ linear functions on each of the two sides of a hyperplane and on one of the two sides these $q$ linear functions agree, i.e.\ the following:

\begin{lem}
\label{step1} 
Let $v = (v^{1}, v^{2}, \ldots, v^{q})$ be such that 
\begin{equation}\label{linear-blow-up}
v^{j}(y, x^{n}) = L_{1}^{j}(y, x^{n})  \quad {\rm if} \quad  x^{n} <0 \quad {\rm and} \quad  v^{j}(y, x^{n})= L_{2}^{j}(y, x^{n}) \quad {\rm if} \quad x^{n} \geq 0.
\end{equation}
If either $L_{1}^{1} = L_{1}^{2} = \ldots =L_{1}^{q}$ or $L_{2}^{1} = L_{2}^{2} = \ldots =L_{2}^{q}$, then there exists a linear function $L$ such that $v^{j} = L$ for each $j=1, 2, \ldots, q.$
\end{lem}

\begin{proof} Assume without loss of generality that $L_{1}^{1} = L_{1}^{2} = \ldots = L_{1}^{q} = 0.$ Note that in view of property $({\mathcal B\emph {3}})$, it suffices to show that $L_{2}^{1} = L_{2}^{2} = \ldots = L_{2}^{q}$, so assume, for a contradiction, that 
$$L_{2}^{j} \neq L_{2}^{j+1} \quad \mbox{for some} \quad j \in \{1, 2, \ldots, q-1\}.
\leqno{\quad\quad(\dag)}$$ 

Let $\{V_{k}\} \subset {\mathcal S}_H$ be a sequence of varifolds whose coarse blow-up is $v$  and let $\tau \in (0, 1/8)$ be arbitrary. Assumption $(\dag)$ implies, by the argument establishing property $({\mathcal B\emph{4}})$ for ${\mathcal B}_{q}$, 
that for all sufficiently large $k$, $Z \in (B_{1} \times {\mathbb R}) \cap \{x^{n} > \tau\} \implies \Theta \, (\|V_{k}\|, Z) < q,$  so by Proposition \ref{Prop:elementaryconsequence} and Theorem~\ref{thm:SS},  it follows 
that $V_{k} \res ((B_{9/16} \times {\mathbb R}) \cap \{x^{n} > \tau/4 \})$ is given by ordered graphs of $C^{1,\alpha}$ functions $u_k^1\leq \ldots \leq u_k^q$ with small gradients. Thus, with $u_{k}^{j}$ giving rise to a coarse blow-up as in (\ref{blow-up-3-1}), we have 
\begin{equation}\label{no-overlap-3}
 V_{k} \res(((B_{9/16} \times {\mathbb R}) \cap \{x^{n} > \tau/4\}) =   \sum_{j=1}^{q} |{\rm graph} \, u_{k}^{j}| \res 
((B_{9/16} \times {\mathbb R}) \cap \{x^{n} > \tau/4\})
\end{equation}
where $u_{k}^{j}$ are $C^{1,\alpha}$
on $B_{9/16} \cap \{x^{n} > \tau/4\}$ 
and satisfy
\begin{equation}\label{no-overlap-4}
\|u_{k}^{j}\|_{C^{1,\alpha}\left({B_{1/2}\cap \{x^{n} >\tau/4\}}\right)}  \leq C_{\tau}{\hat E}_{k}^{2},
\end{equation}
where $C_{\tau}$ is a constant depending only on $n$ and $\tau$.

To derive the necessary contradiction, take now $\psi(X) = \widetilde{\z}(X)e^{n}$ in the first variation formula \cite[16.4]{SimonNotes} to deduce that
\begin{equation}\label{no-overlap-main}
\int \nabla^{V_{k}} \, x^{n} \cdot \nabla^{V_{k}} \,{\widetilde \z}(X) d\|V_{k}\|(X)= \int h_k\, \widetilde{\z}\, \hat{\nu} \cdot e^n \,d\|V_{k}\|(X)
\end{equation}
for each $k=1, 2, \ldots$ and each ${\widetilde \z} \in C^{1}_{c}(B_{1} \times {\mathbb R}).$ Choosing ${\widetilde \z}$ to be independent of
$x^{n+1}$, i.e. $\widetilde \z(x^{\prime}, x^{n+1}) = \z(x^{\prime})$ in a neighborhood of ${\rm spt} \,\|V_{k}\| \cap (B_{1/4}\times {\mathbb R})$, where 
$\z  \in C^{1}_{c}(B_{1/4})$ is arbitrary, this leads to the identity 
\begin{equation}\label{no-overlap-4-1}
\sum_{j=1}^{q}\int_{B_{1/4}}\sqrt{1 + |D{u}_{k}^{j}|^{2}}\left(D_{n}\z - \frac{D_{n}{u}_{k}^{j} (D\z \cdot D{u}_{k}^{j})}{1 + |D{u}_{k}^{j}|^{2}}\right) = F_{k}, \;\;\; \mbox{where}
\end{equation}
\begin{eqnarray*}
F_{k} = -\int_{(B_{1/4} \cap \Sigma_{k})\times {\mathbb R} } \nabla^{V_{k}} \, x^{n} \cdot \nabla^{V_{k}} \,{\widetilde \z}(X) d\|V_{k}\|(X) &&\\
&&\hspace{-3in}+ 
\sum_{j=1}^{q}\int_{B_{1/4} \cap \Sigma_{k}}\sqrt{1 + |D{u}_{k}^{j}|^{2}}\left(D_{n}\z - \frac{D_{n}{u}_{k}^{j} (D\z \cdot D{u}_{k}^{j})}{1 + |D{u}_{k}^{j}|^{2}}\right) 
\end{eqnarray*}
$$+  \int\limits_{B_{1/4} \times {\mathbb R} } h_k\, \widetilde{\z}\, \hat{\nu} \cdot e^n \,d\|V_{k}\|(X)$$
with $\Sigma_{k}$ as in Proposition \ref{Prop:AlmgrenDec}. Since $\int_{B_{1/4}} D_{n} \z = 0$, it follows from (\ref{no-overlap-4-1}), subtracting $\int_{B_{1/4}} D_{n} \z$ on the left-hand side, that 
\begin{equation}\label{no-overlap-4-2}
\sum_{j=1}^{q}\int_{B_{1/4}}\frac{|D{u}_{k}^{j}|^{2}}{1 + \sqrt{1 + |D{u}_{k}^{j}|^{2}}} D_{n}\z - \frac{D_{n}{u}_{k}^{j} (D\z \cdot D{u}_{k}^{j})}{\sqrt{1 + |D{u}_{k}^{j}|^{2}}} = F_{k}.
\end{equation}
It is not difficult to check using (\ref{no-overlap-3}), (\ref{no-overlap-4}) and the definition of $\Sigma_{k}$ that 
\begin{equation}\label{no-overlap-4-2-1}
B_{1/4} \cap \Sigma_{k} \subset B_{1/4} \cap \{x^{n} < \t/2\},
\end{equation}
and also that for all sufficiently large $k$,  (this uses a Besicovitch covering and the construction of $\Sigma$ in the proof of Proposition \ref{Prop:AlmgrenDec}, see \cite[page 892]{WicAnnals} for details)
\begin{eqnarray}\label{no-overlap-4-3}
&&\|V_{k}\|((B_{1/4} \cap \Sigma_{k})\times {\mathbb R} ) + {\mathcal H}^{n}(B_{1/4} \cap  \Sigma_{k})\nonumber\\ 
&&\hspace{.1in}\leq C \!\!\!\!\!\!\!\!\!\!\!\!\int\limits_{(B_{1/2}\times {\mathbb R}) \cap \{x^{n} < \t\} } \!\!\!\!\!\!\!\!\!|\nabla^{V_{k}} \, x^{n+1}|^{2} d\|V_{k}\|(X) \,+ C \, \left(\!\!\!\!\!\!\int\limits_{(B_{1/2}\times {\mathbb R}) \cap \{x^{n} < \t\} }\!\!\! \!\!\!|h_k|^p d\|V_{k}\|(X)\right)^{2/p}
\end{eqnarray}
where $C \in (0, \infty)$ is a fixed constant depending only on $n$ and $q$. Moreover, it follows from the fact that ${\hat E}_{k}^{-1}{u}_{k} \to 0$ in $L^{2}$ on $B_{9/16} \cap \{x^{n} \leq -\t/2\}$ that 
\begin{equation*}
{\hat E}_{k}^{-2}\int_{(B_{9/16} \times {\mathbb R}) \cap \{x^{n} \leq -\t/2\}} |x^{n+1}|^{2}d\|V_{k}\|(X) \to 0
\end{equation*}
and consequently, by (\ref{eq:tiltcontrolledbyhight}) and Remark \ref{oss:comparetwoexcesses}, that 

\begin{equation}\label{no-overlap-2}
{\hat E_{k}}^{-2}\int_{(B_{1/2}\times {\mathbb R}) \cap \{x^{n} \leq -\t\}} |\nabla^{V_{k}} \, x^{n+1}|^{2}d\|V_{k}\|(X) \to 0.
\end{equation}
Note also that by arguing by contradiction with the help of Proposition \ref{Prop:elementaryconsequence} and Theorem~\ref{thm:SS}, it follows from the assumption $(\dag)$ that for all sufficiently large $k$,
\begin{equation}\label{no-overlap-4-1-0}
\{x^n=0\} \cap B_{1/2}  \subset \left(\{Z  \in {\rm spt} \, \|V_{k}\|\, : \, 
\Theta \, (\|V_{k}\|, Z) \geq q\}\right)_{\t},
\end{equation}
and hence Lemma~\ref{non-concentration} is applicable to the $V_{k}$'s. Inequality (\ref{no-overlap-4-3}), Lemma~\ref{non-concentration} (with a choice of $\mu  > \frac{n}{p}$) 
and (\ref{no-overlap-2}) (using that the integrands in the first two terms in the definition of $F_k$ are bounded independently of $k$ by $C(n,q)|D\zeta|$) then imply that
\begin{equation}
\label{no-overlap-5}
{\hat E}_{k}^{-2}|F_{k}| \leq C (\sup \, |D\z| \t^{1-\mu} + {\eps}_k)
\end{equation}
for all sufficiently large $k$, where $C = C(n, q) \in (0, \infty)$.

Using the abbreviation $w_{k} = \sum_{j=1}^{q}\frac{|D{u}_{k}^{j}|^{2}}{1 + \sqrt{1 + |D{u}_{k}^{j}|^{2}}} D_{n}\z - \frac{D_{n}{u}_{k}^{j} (D\z \cdot D{u}_{k}^{j})}{\sqrt{1 + |D{u}_{k}^{j}|^{2}}}$ for the left-hand side of (\ref{no-overlap-4-2}), we have the straightforward estimate (recall Remark \ref{oss:lipschitzgraphtiltexcess})
$$\int_{B_{1/4} \setminus \Sigma_{k} \cap \{x^{n} \leq \t\}} |w_{k}| \leq C \sup \, |D\z| \int_{(B_{1/2}\times {\mathbb R} ) \cap \{x^{2} \leq \t\}} |\nabla^{V_{k}} \, x^{n+1}|^{2} d\|V_{k}\|(X),$$
and by (\ref{no-overlap-4-3})
$$ \int_{B_{1/4} \cap \Sigma_{k}} |w_{k}|  \leq C \sup \, |D\z| \int_{(B_{1/2}\times {\mathbb R}) \cap \{x^{n} \leq \t\}} |\nabla^{V_{k}} \, x^{n+1}|^{2} d\|V_{k}\|(X),$$
where $C = C(n)$, so that again by Lemma~\ref{non-concentration} (with $\mu$ as before) and (\ref{no-overlap-2}), 
\begin{equation}\label{no-overlap-6}
{\hat E_{k}}^{-2} \left(\int_{B_{1/4} \setminus \Sigma_{k} \cap \{x^{n} \leq \t\}} |w_{k}| + \int_{B_{1/4} \cap  \Sigma_{k}} |w_{k}|\right) \leq C\sup \, |D\z|\t^{1-\mu}
\end{equation}
for all sufficiently large $k$, where $C = C(n)$. Finally, by (\ref{no-overlap-4}), which guarantees the $C^1$- convergence to the coarse blow-up, 
\begin{equation}\label{no-overlap-7}
\lim_{k \to \infty} \, {\hat E_{k}}^{-2} \int_{B_{1/4} \cap \{x^{n} \geq \t\}} w_{k} = -\frac{1}{2}\sum_{j=1}^{q}\int_{B_{1/4} \cap \{x^{n} \geq \t\}} |D_{n}v^{j}|^{2}D_{n} \z
\end{equation}
which also uses the fact that $D_{i} v^{j} \equiv 0$ for $i=1, \ldots, (n-1)$ and 
$j= 1, 2, \ldots, q.$ Dividing (\ref{no-overlap-4-2}) by ${\hat E}_{k}^{2}$ and first letting $k \to \infty$ and 
then letting $\t \to 0$ imply, in view of  (\ref{no-overlap-5}), (\ref{no-overlap-6}) and 
(\ref{no-overlap-7}), that 
$$\sum_{j=1}^{q}\int_{B_{1/4} \cap \{x^{n} \geq 0\}} |D_{n}v^{j}|^{2} D_{n}\z = 0$$
for any $\z \in C^{1}_{c}(B_{1/4}).$ Since $v^{j} = L_{2}^{j}$ on $\{x^{n} \geq 0\}$, this contradicts (for any choice of $\z \in C^{1}_{c}(B_{1/4})$ with 
$\int_{B_{1/4} \cap \{x^{n} \geq 0\}} D_{n}\z \neq 0$) our assumption that  $L_{2}^{j} \neq L_{2}^{j+1}$ 
for some $j \in \{1, 2, \ldots, q-1\}$.
\end{proof}

\textit{\textbf{Hard case of the Minimum Distance Property}}. Let $v$ be as in (\ref{linear-blow-up}). The second case of the proof that ${\mathcal B}_{q}$ satisfies property $({\mathcal B\emph{7}})$ is to rule out the possibility that there are at least two distinct linear functions defining $v$ on either side, i.e.~to establish the following\footnote{This lemma is false already in the minimal case $H=0$ if we drop either the stability hypothesis or the no-classical-singularities hypothesis on the varifolds $V_{k}$ giving rise to $v,$ as can be seen by considering in the former case an appropriate sequence of minimal surfaces in the Scherk's family of surfaces and in the latter an appropriate sequence of pairs of transverse planes.}: 
 
\begin{lem}
\label{step2}
Let $v = (v^{1}, v^{2}, \ldots, v^{q})$ be as in (\ref{linear-blow-up}). There exist no indices $i, j \in \{1, 2, \ldots, q-1\}$ such that $L_{1}^{i} \neq L_{1}^{i+1}$ and $L_{2}^{j} \neq L_{2}^{j+1};$ consequently (in view of Lemma~\ref{step1}), $v^{j} = L$ for some linear function $L$ and each $j=1, 2, \ldots, q,$ and ${\mathcal B}_{q}$ satisfies $({\mathcal B\emph{7}})$.
\end{lem}

Recall that this is a lemma for Theorem \ref{coarse}, so we are working under those assumptions. To describe the basic idea of the proof of Lemma~\ref{step2}, let $v \in {\mathcal B}_{q}$ and let the linear functions $L_{1}^{j}, L_{2}^{j}\, : \, {\mathbb R}^{n} \to {\mathbb R}$, $1\leq j \leq q,$ be as defined in (\ref{linear-blow-up}), and assume, contrary to  the assertion of Lemma~\ref{step2}, that we have for some $i, j \in \{1, 2, \ldots, q\}$, 
\begin{equation}\label{distinct}
L_{1}^{i} \neq L_{1}^{i+1} \quad \mbox{and}  \quad L_{2}^{j} \neq L_{2}^{j+1}. 
\end{equation}
In view of properties $({\mathcal B\emph{3}})$ and $({\mathcal B\emph{5}})$, we may assume without loss of generality that 
\begin{equation}\label{average}
\sum_{j=1}^{q} L_{k}^{j} = 0 \quad \mbox{for $k=1, 2$,} \quad \mbox{and that}
\end{equation}
\begin{equation}\label{normalize}
\|v\|^{2}_{L^{2}(B_{1})} \left(= \sum_{j=1}^{q} \|L_{1}^{j}\|^{2}_{B_{1} \cap \{x^{n} < 0\}} + \|L_{2}^{j}\|^{2}_{B_{1} \cap \{x^{n} \geq 0\}}\right) = 1.
\end{equation}
Let  $\{V_{k}\} \subset {\mathcal S}_{H_0}$ be a sequence of varifolds whose coarse blow-up is $v$. The proof of Lemma~\ref{step2} is accomplished by showing that 
under the hypotheses (\ref{distinct}), (\ref{average}), (\ref{normalize}), one can find $k$ sufficiently large such that there exists a point $Z_{k} \in {\rm sing} \, V_{k}$ with the property that a suitable sequence of re-scalings followed by rotations of $V_{k}$ about $Z_{k}$ converges to a cone supported on at least four distinct half-hyperplanes meeting along a common boundary, contradicting directly the induction hypothesis $(H2)$. 

\medskip

We will use the following notation for the rest of this subsection, where we shall explain in more detail how this contradiction is obtained. 

\noindent
 {\bf (N1)} Let ${\mathcal C}_{q}$ denote the set of hypercones ${\mathbf C}$ of ${\mathbb R}^{n+1}$ of the form 
${\mathbf C} = \sum_{j=1}^{q} |H_{j}| + |G_{j}|,$ where for each $j \in \{1, 2, \ldots, q\}$, $H_{j}$ is the half-hyperplane defined by $H_{j} = \{x^{n} < 0 \;\; \mbox{and} \;\; x^{n+1} = \lambda_{j}x^{n}\},$ 
$G_{j}$ the half-hyperplane defined by $G_{j} = \{x^{n} > 0 \;\; \mbox{and} \;\; x^{n+1} = \mu_{j}x^{n}\},$ with $\lambda_{j}, \mu_{j}$ constants,  
$\lambda_{1} \geq \lambda_{2} \geq \ldots \geq \lambda_{q}$ and $\mu_{1} \leq \mu_{2} \leq \ldots \leq \mu_{q}.$ We do \emph{not} assume that the cones in ${\mathcal C}_{q}$ are stationary in ${\mathbb R}^{n+1}.$

\noindent
{\bf (N2)} For $V \in {\mathcal{S}_H}$ and ${\mathbf C} \in {\mathcal C}_{q}$ define a height excess (``\textbf{fine excess}'') $Q_{V}({\mathbf C})$ of $V$ relative to ${\mathbf C}$ by
\begin{eqnarray*}
Q_{V}({\mathbf C})  = \left(\int_{(B_{1/2} \setminus \{|x^{n}| < 1/16\}) \times {\mathbb R}} {\rm dist}^{2}(X, {\rm spt} \, \|V\|) \,d\|{\mathbf C}\|(X)\right.&&\nonumber\\ 
&&\hspace{-3in}+ \; \left.\int_{B_{1}\times {\mathbb R}} {\rm dist}^{2} \, (X, {\rm spt}\, \|{\mathbf C}\|) \, d\|V\|(X)\right)^{1/2}. 
\end{eqnarray*}

Let ${\mathbf C}_{k} \in {\mathcal C}_{q}$ be the cone corresponding to ${\hat E}_{V_{k}} v$, i.e. 
$${\mathbf C}_{k} = \sum_{j=1}^{q} |{\rm graph} \, {\hat E}_{V_{k}}L_{1}^{j}| \res \{x^{n} < 0\} + |{\rm graph} \, {\hat E}_{V_{k}}L_{2}^{j}|\res\{x^{n} \geq 0\}.$$ 

It is possible to verify with the help of Proposition \ref{Prop:elementaryconsequence} and Theorem~\ref{thm:SS} that for any given $\eps, \gamma \in (0, 1)$, after passing to appropriate  subsequences of $\{V_{k}\}$ and $\{{\mathbf C}_{k}\}$ without changing notation, there exist points $Z_{k} \in {\rm spt} \, \|V_{k}\|$ with $\Theta \, (\|V_{k}\|, Z_{k}) \geq q$ and $Z_{k} \to 0$ such that for each $k$, Hypotheses~\ref{hyp} below  are satisfied with $\widetilde{V}_{k} \equiv \eta_{Z_{k}, 2(1 - |Z_{k}|) \, \#} \, V_{k}$ in place of $V$ and ${\mathbf C}_{k}$ in place of ${\mathbf C}$ (see the proof of \cite[Corollary 14.2]{WicAnnals} for details). Note that ${\hat E}_{{\widetilde V}_{k,\eps_k,1}} \to 0,$ and that the coarse blow-up of the sequence $\{\widetilde V_{k}\}$ (relative to $\{{\mathbf C}_{k}\}$) is still $v$. 

\begin{hyp}[Contradiction Hypotheses]\label{hyp}
\begin{itemize}
\item[]
\item[(1)] $V \in {\mathcal S}_H$, \; $\Theta \, (\|V\|, 0) \geq q$, \;$(\omega_{n}2^{n})^{-1}\|V\|(B_{2}^{n+1}(0)) < q + 1/2$, \; $\omega_{n}^{-1}\|V\|( B_{1}{\times \mathbb R}) < q + 1/2.$ 
\item[(2)] ${\mathbf C}  \in {\mathcal C}_{q}.$ 
\item[(3)] ${\hat E}_{V,\eps,1}^{2} \equiv \int_{B_{1} \times {\mathbb R}} |x^{n+1}|^{2} d\|V\|(X)+ \frac{1}{\eps}\left(\int_{B_{1} \times {\mathbb R}} |h|^{p} d\|V\|(X) \right)^{\frac{1}{p}} < \eps.$
\item[(4)] $\{Z \, : \, \Theta \, (\|V\|, Z) \geq q\} \cap \left((B_{1/2} \setminus \{|x^{n}| < 1/16\})\times {\mathbb R}  \right) = \emptyset.$
\item[(5)] $Q_{V}^{2}({\mathbf C}) < \gamma {\hat E}_{V,\eps,1}^{2}$.
\end{itemize} 
\end{hyp}

One can also check, in view of (\ref{average}) and (\ref{normalize}), that passing to a subsequence of $\{{\widetilde V}_{k}\}$ without changing notation, the following hypothesis is satisfied with $M=1$ and with ${\widetilde V}_{k}$ in place of $V$:

  \medskip
  
\noindent
{\bf Hypothesis ($\star$)}:
\begin{equation*}
{\hat E}_{V,\eps,1}^{2} < \frac{3}{2}M\inf_{\{P = \{x^{n+1} = \lambda x^{n}\} \in G_{n} \, : \, \lambda \in {\mathbb R}\}} \, \int_{B_{1}\times \R} {\rm dist}^{2} \, (X, P) d\|V\|(X).
\end{equation*}
The contradiction will be found by iteratively applying Lemma \ref{multi-scale} below, starting with $V = \widetilde{V}_{k}$ and ${\mathbf C} = {\mathbf C}_{k}$ for suitably large fixed $k$, to produce a sequence of numbers $\s_{j} \to 0$, a sequence of rotations $\G_{j}$ of ${\mathbb R}^{n+1}$ and a cone ${\mathbf W}$ supported on four or more \emph{distinct} half-hyperplanes meeting along a common axis and with $\Theta \, (\|{\mathbf W}\|, 0)  = q$ such that  
\begin{equation}\label{convergence}
\eta_{\s_{j} \, \#} \G_{j \, \#}\widetilde{V_{k}} \to {\mathbf W} 
\end{equation}
as varifolds as $j \to \infty$, directly contradicting the induction hypothesis $(H2)$ as desired.  In this lemma (upon which we comment in Remark \ref{oss:multiscale}) and subsequently, $M_{0}$ denotes a certain explicit constant $>1$ that depends only on $n$ and $q$. 

\begin{lem}\label{multi-scale} 
Let $q \geq 2$ be an integer. For $j=1, 2, \ldots, 2q-3,$ let $\th_{j} \in (0, 1/4)$ be such that $\th_{1} > 8\th_{2} > 64\th_{3} > \ldots >8^{2q- 4}\th_{2q-3}$. 
There exist numbers $$\eps = \eps(n, q,  \th_{1}, \th_{2}, \ldots, \th_{2q-3}) \in (0, 1/2),$$ 
$$\gamma= \gamma(n, q, \th_{1}, \th_{2} \, \ldots, \th_{2q-3}) \in (0, 1/2)$$ such that if the varifold $V$ and the cone ${\mathbf C}$ satisfy Hypotheses~\ref{hyp} and Hypothesis~($\star$) with  $M = M_{0}$, and if the induction hypotheses $(H1)$, $(H2)$, $(H3)$ hold,  
then there exist an orthogonal rotation $\G$ of ${\mathbb R}^{n+1}$ and a cone ${\mathbf C}^{\prime} \in {\mathcal C}_{q}$ satisfying
\begin{enumerate}
 \item[(a)]  $\displaystyle |e_{n+1} - \Gamma(e_{n+1})|^2 \leq \kappa \left(Q_{V}^2({\mathbf C})+\frac{1}{\eps} \left(\int\limits_{B_1 \times \R}|h|^p d\|V\|\right)^{\frac{1}{p}}\right), $ 
 
 $\displaystyle |e_{j} - \Gamma(e_{j})|^2 \leq \kappa {\hat E}_{V,\eps,1}^{-2}\left(Q_{V}^2({\mathbf C})+\frac{1}{\eps} \left(\int\limits_{B_1 \times \R}|h|^p d\|V\|\right)^{\frac{1}{p}}\right)$ 
 
 for each  $j=1, 2, \ldots, n;$
 
 \item[(b)] $\displaystyle {\rm dist}_{\mathcal H}^{2} \, ({\rm spt} \, \|{\mathbf C}^{\prime}\| \cap (B_{1} \times {\mathbb R}), {\rm spt} \, \|{\mathbf C}\| \cap (B_{1} \times {\mathbb R}))$ 
 
 $\displaystyle \hspace{2in} \leq C_{0}\left(Q_{V}^2({\mathbf C})+\frac{1}{\eps} \left(\int\limits_{B_1 \times \R}|h|^p d\|V\|\right)^{\frac{1}{p}}\right)$
\end{enumerate}

and the following for some $j \in \{1, 2, \ldots, 2q-3\}:$
  
\begin{enumerate}
 \item[(c)] $\displaystyle \th_{j}^{-n-2}\int_{\G\left(\left(B_{\th_{j}/2} \setminus \{|x^{n}| \leq \th_{j}/16\}\right) \times {\mathbb R}\right)} {\rm dist}^{2} \, (X, {\rm spt} \, \|V\|) \, d\|\G_{\#} \, {\mathbf C}^{\prime}\|(X)$

$\displaystyle  \hspace{.5in} + \; \th_{j}^{-n-2}\int_{\G(B_{\th_{j}} \times {\mathbb R})} {\rm dist}^{2} \, (X, {\rm spt} \, \|\G_{\#} \, {\mathbf C}^{\prime}\|) \, d\|V\|(X)$ 

 $\hspace{1in}+\theta_j^{1-\frac{n}{p}} \frac{1}{\eps} \left(\int_{\G(B_{\th_{j}} \times {\mathbb R})}|h|^p \, d\|V\|(X)\right)^{1/p}$ 
 
 $\displaystyle \hspace{1.8in} \leq \nu_{j}\th_{j}^{1-\frac{n}{p}}\left(Q_{V}^2({\mathbf C})+\frac{1}{\eps} \left(\int\limits_{B_1 \times \R}|h|^p d\|V\|\right)^{\frac{1}{p}}\right);$

 \item[(d)] $\displaystyle \left(\th_{j}^{-n-2}\int_{B_{\th_{j}} \times {\mathbb R}}  {\rm dist}^{2}\,(X, P) \, d\|\G^{-1}_{\#} \,V\|(X)\right)^{1/2}$
 
 $\displaystyle \hspace{.5in}\geq C_{1}\,{\rm dist}_{\mathcal H}\, ({\rm spt} \, \|{\mathbf C}\| \cap (B_{1} \times {\mathbb R}), P \cap (B_{1} \times {\mathbb R}))$ 

$\displaystyle \hspace{2in} - C_{2}\left(Q_{V}^2({\mathbf C})+\frac{1}{\eps} \left(\int\limits_{B_1 \times \R}|h|^p d\|V\|\right)^{\frac{1}{p}}\right)^{\frac{1}{2}};$
 
 for any $P \in G_{n}$ of the form $P = \{x^{n+1} = \lambda x^{n}\}$ for some $\lambda \in (-1, 1);$
 
 \item[(e)] $\displaystyle \{Z \, : \, \Theta \, (\|\G_{\#}^{-1} \, V\|, Z) \geq q\} \setminus \left((B_{\th_{j}/2} \cap \{|x^{n}| < \th_{j}/16\}) \times {\mathbb R}\right) = \emptyset;$
 
 \item[(f)] $\displaystyle \left(\omega_{n}\th_{j}^{n}\right)^{-1}\|\G_{\#}^{-1} \, V\|(B_{\th_{j}} \times {\mathbb R}) < q + 1/2.$
\end{enumerate}  
     
Here $C_{1} = C_{1}(n)$; the constants $\kappa, C_{0}, C_{2} \in (0, \infty)$ depend only on $n$ in case $q=2$ and only on $n$, $q$ and $\th_{1}, \th_{2}, \ldots, \th_{2q-4}$ in case $q \geq 3$; $\nu_{1} = \nu_{1}(n, q);$ and, 
in case $q \geq 3,$  $\nu_{j} = \nu_{j}(n, q,  \th_{1}, \ldots, \th_{j-1})$ for each $j=2, 3, \ldots, 2q-3.$ In particular, $\nu_{j}$ is independent of $\th_{j}, \th_{j+1}, \ldots, \th_{2q-3}$ for each $j = 1, 2, \ldots, 2q-3$. 
\end{lem}

\begin{oss}
\label{oss:multiscale}
Allowing multiple scales is a natural idea used in \cite{WicAnnals} to overcome certain issues arising from the presence of higher multiplicity.
In view of the fact that for each $j\in \{1, 2, \ldots, 2q-3\}$ the constant $\nu_{j}$ in the lemma is independent of the scales $\th_{j}, \th_{j+1}, \ldots, \th_{2q-3},$ it is possible to choose the scales 
$\th_{1}, \th_{2}, \ldots, \th_{2q-3}$ depending only on $n$ and $q$ such that $\nu_{j}\th_{j}^{2} < 4^{-1}$ for each $j=1, 2, \ldots, 2q-3.$ 
With this choice of scales and under the hypotheses of the lemma, part (c) says that the excess (fine height excess plus the mean curvature term) of suitably rotated and rescaled $V$ relative to the new cone $\in {\mathcal C}_{q}$ improves by a factor $4^{-1}$ at each iteration of the lemma.  Parts (b), (d), (e), (f) of the conclusion imply that the original hypotheses (i.e.\ Hypothesis~\ref{hyp} and Hypothesis~($\star$)) are satisfied at the new scale, so the lemma can be iterated indefinitely to produce controlled scales $\s_{j} \to 0$, rotations $\G_{j}$ and 
a non-planar cone ${\mathbf W} \in {\mathcal C}_{q}$ such that the desired conclusion (\ref{convergence}) holds.
\end{oss}

The first step in proving Lemma~\ref{multi-scale} is to reduce it to a result (Lemma~\ref{intermediate} below) in the same spirit but with stronger hypotheses and stronger conclusions; specifically, to show that the conclusions of Lemma~\ref{multi-scale} hold at a fixed given \emph{single} scale $\th \in (0,  1/4)$ (rather than at one of a finite set of scales
$\{\th_{1}, \th_{2}, \ldots, \th_{2q-3}\}$ allowed in Lemma~\ref{multi-scale}) provided we make one additional hypothesis. To describe this hypothesis, it is convenient to use the following notation: 

\medskip

\noindent
{\bf (N3)} For $p \in \{2, 3, \ldots, 2q\}$, let ${\mathcal C}_{q}(p)$ denote the set of  hypercones ${\mathbf C}  = \sum_{j=1}^{q} |H_{j}| + |G_{j}| \in {\mathcal C}_{q}$ as defined in {\bf (N1)} above such that the number of {\em distinct} half-hyperplanes in the set 
$$\{H_{1}, \ldots, H_{q}, G_{1}, \ldots, G_{q}\}$$ 
is $p$. Then, of course, ${\mathcal C}_{q} = \cup_{p=2}^{2q} \, {\mathcal C}_{q}(p).$

\noindent
{\bf (N4)} For $q \geq 2$ and $p \in \{4, \ldots, 2q\}$, let 
$$Q_{V}^{\star}(p) = \inf_{{\mathbf C} \in \cup_{k=4}^{p} {\mathcal C}_{q}(k)} \, Q_{V}({\mathbf C})$$
where $Q_{V}({\mathbf C})$ is as defined in {\bf (N2)} above.

\medskip

With this notation, consider  for small $\beta \in (0, 1)$ to be chosen depending only on $n$ and $q$ and for $V \in {\mathcal S}_H$ the following:

\medskip

\noindent
{\bf Hypothesis ($\star\star$)}: \emph{Either
\begin{itemize}
\item[(i)] ${\mathbf C} \in {\mathcal C}_{q}(4)$ or
\item[(ii)] $q \geq 3$, ${\mathbf C} \in {\mathcal C}_{q}(p)$ for some $p \in \{5, \ldots, 2q\}$, $Q_{V}^{2}({\mathbf C}) < \beta \left(Q_{V}^{\star}(p-1)\right)^{2}.$
\end{itemize}}

The idea is that under Hypotheses~\ref{hyp} \emph{and} Hypothesis~($\star\star$) for suitably small $\eps$, $\gamma$ and $\beta$, the sheets of $V$ away from a small tubular neighborhood $T$ of the singular axis of ${\mathbf C}$ organize themselves as sets of ordered graphs over the distinct half-hyperplanes making up ${\rm spt} \, \|{\mathbf C}\|,$ with the absolute value of the height function defining the graphs equal to the distance to ${\rm spt} \, \|\mathbf C\|$. (See conclusion (a) of Lemma~\ref{L2-est-1} below.) This prevents in a strong way the possibility of concentration of the fine excess $Q_{V}({\mathbf C})$ in any small region outside $T$.

With Hypothesis ($\star\star$) in place in addition to the other hypotheses as in Lemma~\ref{multi-scale}, the claim now is the following:

\begin{lem}
\label{intermediate}
For any given $\th \in (0, 1/4)$, there exists 
$\eps = \eps(n, q, \th) \in (0, 1/2)$, $\gamma = \gamma(n, q, \th) \in (0, 1/2)$, $\beta = \beta(n, q, \th) \in (0, 1/2)$ such that if $V \in {\mathcal S}_H$, ${\mathbf C} \in {\mathcal C}_{q}$ satisfy Hypotheses~\ref{hyp}, Hypothesis~($\star$) with $M = M_{0}^{2}$ and Hypothesis~($\star\star$), and if the induction hypotheses $(H1)$, $(H2)$, $(H3)$ hold, then there exist an orthogonal rotation $\Gamma$ of ${\mathbb R}^{n+1}$ and a cone ${\mathbf C}^{\prime} \in {\mathcal C}_{q}$ such that 

conclusions (a) and (b) of Lemma~\ref{multi-scale} hold with the same formulation and with $\kappa$ depending only on $n$, $q$ and $\th$, 

conclusions (d), (e) and (f) of Lemma~\ref{multi-scale} hold with $\theta$ in place of $\theta_j$ and with $C_{1} = C_{1}(n)$ and $C_{0}$, $C_{2}$ depending only on $n$, $q$ and $\th$, 

and moreover

 $\displaystyle \th^{-n-2}\int_{\G\left(\left(B_{\th /2} \setminus \{|x^{n}| \leq \th /16\}\right) \times {\mathbb R}\right)} {\rm dist}^{2} \, (X, {\rm spt} \, \|V\|) \, d\|\G_{\#} \, {\mathbf C}^{\prime}\|(X)$

$\displaystyle \hspace{1in} + \; \th^{-n-2}\int_{\G(B_{\th} \times {\mathbb R})} {\rm dist}^{2} \, (X, {\rm spt} \, \|\G_{\#} \, {\mathbf C}^{\prime}\|) \, d\|V\|(X)$
 
$\displaystyle \hspace{2in} \leq \nu\th^{2}\left(Q_{V}^2({\mathbf C})+\frac{1}{\eps} \left(\int\limits_{B_1 \times \R}|h|^p d\|V\|\right)^{\frac{1}{p}}\right)$

where $\nu = \nu(n, q)$.
\end{lem} 

Once Lemma~\ref{intermediate} is in place it is not difficult to deduce Lemma~\ref{multi-scale}; that is to say to remove Hypothesis~($\star\star$) from Lemma~\ref{intermediate},  provided we do not insist that the conclusions hold at a single smaller scale (unless $q=2$) and instead allow ourselves the freedom of a fixed set of finitely many  (in fact $2q-3$) scales $\th_{1}, \th_{2}, \ldots, \th_{2q-3}$ at one of which the conclusions are required to hold. This is done by inducting on the unique number $p \in \{4, 5, \ldots, 2q\}$ for which ${\mathbf C} \in {\mathcal C}_{q}(p)$. We refer to \cite[pp 943-946]{WicAnnals} for details, describing here only the key idea to prove Lemma \ref{multi-scale} assuming Lemma \ref{intermediate}. 

\begin{proof}[Sketch of the proof of Lemma \ref{multi-scale} (assuming Lemma \ref{intermediate}).]
For the sake of simplicity, we make two technical simplifications, i.e.~we ignore the effect of the rotation $\Gamma$ appearing in Lemma \ref{intermediate}  and we assume that ${\mathbf C}^{\prime}={\mathbf C}$ (in other words we pretend that the excess improvement given by Lemma \ref{intermediate} is verified with repect to exactly the same cone ${\mathbf C}$ in Lemma \ref{intermediate}), and we focus on the procedure needed to establish assertion (c), in order to make evident how the multiple scales in Lemma \ref{multi-scale} are a natural way to deal with high multiplicity. We use the notation
$$Q^2_V(C,\rho)=\rho^{-n-2}\int_{(B_{\rho/2} \setminus \{|x^n|\leq \frac{\rho}{16}\})\times \R} \text{dist}^2(X,\spt{V})d\|C\|(X)+$$
$$+\rho^{-n-2}\int_{B_{\rho/2}\times \R} \text{dist}^2(X,\spt{C})d\|V\|(X).$$
With these technical simplifications and notation, the content of Lemma \ref{intermediate} (c) is as follows: given $\theta$ there are $\eps(\theta)$ and $\beta(\theta)$ such that, assuming the validity of Hypothesis~($\star\star$) with $\beta=\beta(\theta)$, we have
 $$Q^2_V(C,\theta) \leq \nu \theta^2 \left(Q^2_V(C,1) + \frac{1}{\eps}\left(\int_{B_1}|h|^p d\|V\|\right)^{1/p}\right),$$
 with $\nu>0$ independent of $\theta$. Moreover it is straightforward that
 $$\frac{\theta^{1-\frac{n}{p}}}{\eps}\left(\int_{B_\theta}|h|^p d\|V\|\right)^{1/p}\leq \frac{\theta^{1-\frac{n}{p}}}{\eps}\left(\int_{B_1}|h|^p d\|V\|\right)^{1/p}.$$
Putting together, Lemma \ref{intermediate} says that the following excess improvement holds for some $\tilde{\nu}$ independent of $\theta$ and with $\eps=\eps(\theta)$, assuming the validity of Hypothesis~($\star\star$) with $\beta(\theta)$: 
\begin{equation} 
\label{eq:excessimprovementintermediatesimplified}
Q^2_V(C,\theta) + \frac{\theta^{1-\frac{n}{p}}}{\eps}\left(\int_{B_\theta}|h|^p d\|V\|\right)^{1/p} \leq \tilde{\nu} \theta^{1-\frac{n}{p}}  \left(Q^2_V(C,1) + \frac{1}{\eps}\left(\int_{B_1}|h|^p d\|V\|\right)^{1/p}   \right).
\end{equation}

Now, in order to prove Lemma \ref{multi-scale}, let $p\leq 2q$ be such that ${\mathbf C} \in C(p)$; we start with $\theta_1$ and from Lemma \ref{intermediate} we obtain $\eps(\theta)$ and $\beta(\theta)$ such that, if  Hypothesis~($\star\star$) is fulfilled (at scale $1$) with $\beta(\theta_1)$, then the excess improvement (\ref{eq:excessimprovementintermediatesimplified}) holds with $\theta=\theta_1$ and $\eps=\eps(\theta_1)$, so in this case Lemma \ref{multi-scale} is proved with $\nu_1=\tilde{\nu}$. On the other hand, if Hypothesis~($\star\star$) with $\beta(\theta_1)$ does not hold, then we can choose a cone $C^{(1)}\in C(p-1)$ (at scale $1$) such that 
\begin{equation}
\label{negationstarstar}
Q_V^2(C^{(1)}) \leq \frac{Q_V^2(C)}{\beta(\theta_1)};
\end{equation}
we pass now to the scale $\theta_2$ and ask whether Hypothesis~($\star\star$) holds or not with $\beta=\beta(\theta_2)$ and with $C^{(1)}$ on the left hand side and $Q^*_V(p-2)$ on the right hand side. If yes, then we have that the excess improvement (\ref{eq:excessimprovementintermediatesimplified}) holds with $\theta=\theta_2$, $\eps=\eps(\theta_2)$ and with $C^{(1)}$ instead of $C$ on the left- and right-hand side of (\ref{eq:excessimprovementintermediatesimplified}). Using the fact that we are working under the validity of (\ref{negationstarstar}), we then conclude (using $\beta<1$)
\begin{equation}
Q^2_V(C^{(1)},\theta_2) + \frac{\theta_2^{1-\frac{n}{p}}}{\eps(\theta_2)}\left(\int_{B_{\theta_2}}|h|^p\right)^{1/p} \leq \frac{\tilde{\nu}}{\beta(\theta_1)}\theta_2^{1-\frac{n}{p}}  \left(Q^2_V(C,1) + \frac{1}{\eps(\theta_2)}\left(\int_{B_1}|h|^p\right)^{1/p}   \right), 
\end{equation}
i.e.~the desired excess improvement holds at scale $\theta_2$, with $\eps=\eps(\theta_2)$ and $\nu_2 = \nu_2(\theta_1)=\frac{\tilde{\nu}}{\beta(\theta_1)}$. This is the desired conclusion for Lemma \ref{multi-scale} at scale $\theta_2$, and it was obtained under the assumptions that Hypothesis~($\star\star$) does not hold with $C$ and $\beta(\theta_1)$ and it holds with $C^{(1)}$ and $\beta(\theta_2)$. Continuing with the analysis of possibilities, we address the next alternative, i.e.~that Hypothesis~($\star\star$) fails both with $C$, $p$, $\beta=\beta(\theta_1)$ and with $C^{(1)}$ in place of $C$, $p-2$ in place of $p$ and $\beta=\beta(\theta_2)$; in view of the latter we can choose $C^{(2)}\in C(p-2)$ such that 
\begin{equation}
\label{negationstarstarII}
Q_V^2(C^{(2)}) \leq \frac{Q_V^2(C^{(1)})}{\beta(\theta_2)}.
\end{equation}
Then we consider $\theta_3$ and ask whether Hypothesis~($\star\star$) holds or not with $\beta=\beta(\theta_3)$, $C^{(2)}$ on the left hand side and $Q_V^*(p-3)$ on the right hand side. If it does, we obtain, as before, the desired excess improvement at scale $\theta_3$ with $\nu_3=\nu_3(\theta_1, \theta_2)=\frac{\tilde{\nu}}{\beta(\theta_1)\beta(\theta_2)}$, namely

$\displaystyle Q^2_V(C^{(2)},\theta_3) + \frac{\theta_3^{1-\frac{n}{p}}}{\eps(\theta_3)}\left(\int_{B_{\theta_3}}|h|^p d\|V\|\right)^{1/p}$ 

$\displaystyle \hspace{2in} \leq \nu_3 \theta_3^{1-\frac{n}{p}}  \left(Q^2_V(C,1) + \frac{1}{\eps(\theta_3)}\left(\int_{B_1}|h|^p d\|V\|\right)^{1/p}   \right);$

\noindent
otherwise we continue the procedure, which can go on for at most $2q-3$ steps (i.e.~we get at most to the scale $\theta_{2q-3}$), which happens in the case that Hypothesis~($\star\star$) fails, for every $j=1, ..., 2q-4$, with $\beta=\beta(\theta_j)$, with the cone $C^{(j-1)}$ on the left-hand side and with $Q_V^*(p-j)$ on the right-hand side. In this case, however, from the last negation of Hypothesis~($\star\star$), we choose a cone $C^{(j)}\in C(p-2q+4)$, which necessarily fulfils Hypothesis~($\star\star$), namely alternative (i), because $p\leq 2q$. Remark that by construction the constant $\nu_j$ (appearing in the excess improvement at scale $\theta_j$) depends only on the scales $\theta_1, ... \theta_{j-1}$. In conlusion the desired excess decay in assertion (c) holds at one of the chosen scales.
 
Keeping track of the choices of $C^{(j)}$ we can also see how assertion (b) can be verified (we are still working under the simplified assumption that the excess improvement given by Lemma \ref{intermediate} is verified with repect to exactly the same cone ${\mathbf C}$, the actual proof would require to absorbe the additional errors). Then by the triangle inequality we have 
$$\text{dist}^2_{\mathcal{H}}(C^{(1)},C) \leq 2 \left( \text{dist}^2_{\mathcal{H}}(C^{(1)},V) +\text{dist}^2_{\mathcal{H}}(V,C)\right);$$
using (\ref{negationstarstar}) and the fact that $Q^2_V(\tilde{C})\geq \text{dist}^2_{\mathcal{H}}(V,\tilde{C})$, we get

$$\text{dist}^2_{\mathcal{H}}(C^{(1)},C) \leq 2 \left(1+\frac{1}{\beta(\theta_1)} \right)Q^2_V(C).$$
Similarly, using (\ref{negationstarstarII}) we get
$$\text{dist}^2_{\mathcal{H}}(C^{(2)},C) \leq 2 \left(1+\frac{1}{\beta(\theta_1)\beta(\theta_2)} \right)Q^2_V(C),$$
and so on for any $C^{(j)}$.
\end{proof}

The method used for the proof of Lemma~\ref{intermediate} ---the so called blow-up method---is inspired by the pioneering work of Simon \cite{Simon2} where this method was first used to prove asymptotic decay results near non-isolated singularities for minimal submanifolds in certain multiplicity 1 classes. 

The main a priori estimates needed for the blow up argument in the present context are collected in the lemma below.

\begin{lem}
\label{L2-est-1}
Let $q$ be an integer $\geq 2,$ $\tau \in (0, 1/8),$ $\delta \in (0, 1/8)$ and $\mu \in (\frac{n}{p},1)$. For each $\r \in (0, 1/4]$, there exist numbers $\eps_{0} = \eps_{0}(n , q, \delta, \tau, \r) \in (0,1),$ 
$\gamma_{0} = \gamma_{0}(n, q, \delta, \tau, \r) \in (0, 1)$ and $\beta_{0} = \beta_{0}(n, q, \tau, \r) \in (0, 1)$ such that the following is true:
Let $V \in {\mathcal S}_H$, ${\mathbf C} \in {\mathcal C}_{q}$ satisfy
Hypotheses~\ref{hyp}, Hypothesis~($\star$) and Hypothesis~($\star\star$) with $M = M_{0}^{3}$ and $\eps_{0}$,$\gamma_{0}$, $\beta_{0}$ in place of $\eps$, $\gamma$, $\beta$ respectively. Suppose also that the induction hypotheses $(H1)$, $(H2)$, $(H3)$ hold. 
Write ${\mathbf C} = \sum_{j=1}^{q}|H_{j}| + |G_{j}|$ where for each $j \in \{1, 2, \ldots, q\}$, $H_{j}$ is the half-space defined by
$H_{j} = \{x^{n} < 0 \;\; \mbox{and} \;\; x^{n+1} = \lambda_{j}x^{n}\},$ 
$G_{j}$ the half-space defined by $G_{j} = \{x^{n} > 0 \;\; \mbox{and} \;\; x^{n+1} = \mu_{j}x^{n}\},$ with $\lambda_{j}, \mu_{j}$ constants,  
$\lambda_{1} \geq \lambda_{2} \geq \ldots \geq \lambda_{q}$ and $\mu_{1} \leq \mu_{2} \leq \ldots \leq \mu_{q}$; for 
$(y, x^{n}) \in {\mathbb R}^{n}$ and $j = 1, 2, \ldots, q$, define $h_{j}(y, x^{n}) = \lambda_{j}x^{n}$ and 
$g_{j}(y, x^{n}) = \mu_{j}x^{n}.$ Then, 
after possibly replacing ${\mathbf C}$ with another cone ${\mathbf C}^{\prime} \in {\mathcal C}_{q}$ with 
${\rm spt} \, \|{\mathbf C}^{\prime}\| = {\rm spt} \, \|{\mathbf C}\|$ and relabelling ${\mathbf C}^{\prime}$ as ${\mathbf C}$, the following must hold for each $Z = (\eta, \z^{n}, \z^{n+1}) \in {\rm spt} \, \|V\| \cap ( B_{3/8} \times {\mathbb R} )$ with $\Theta \, (\|V\|, Z) \geq q$: 
\begin{description}
 \item[(${a}$)] $\;\;\;V \res ({\mathbb R} \times (B_{3/4} \setminus \{|x^{n}| < \t\}))= \sum_{j=1}^{q} |{\rm graph} \, (h_{j} + u_{j})| + |{\rm graph} \, (g_{j} + w_{j})| $,
 
 where, for each $j=1, 2, \ldots, q,$ 
 
$u_{j} \in C^{1,\alpha} \, (B_{3/4} \,\cap \,\{x^{n} < -\t\})$;

$w_{j} \in C^{1,\alpha} \, (B_{3/4} \, \cap \, \{x^{n} > \t\});$ 

$h_{1} + u_{1} \leq h_{2} + u_{2} \leq \ldots \leq h_{q} + u_{q}$ ; 

$ g_{1} + w_{1} \leq g_{2} + w_{2} \leq \ldots \leq g_{q} + w_{q}$; 

${\rm dist} \, ((h_{j}(y, x^{n}) + u_{j}(x^{n},y), x^{n}, y), {\rm spt} \,\|{\mathbf C}\|) =  (1 + \lambda_{j}^{2})^{-1/2}|u_{j}(y, x^{n})|$

for $(y, x^{n}) \in B_{3/4} \cap \{x^{n} < -\tau\};$

${\rm dist} \, ((g_{j}(y, x^{n}) + w_{j}(x^{n},y), x^{n}, y), {\rm spt} \,\|{\mathbf C}\|) = (1 + \mu_{j}^{2})^{-1/2}|w_{j}(y, x^{n})| $ 

for $(y, x^{n}) \in B_{3/4} \cap \{x^{n} > \tau\};$

if $h_{j}(x) + u_{j}(x) = h_{j+1}(x) + u_{j+1}(x)$ for some $x$ then $h_j \equiv h_{j+1};$

if $g_{j}(x) + w_{j}(x) = g_{j+1}(x) + w_{j+1}(x)$ for some $x$ then $g_j \equiv g_{j+1};$

\item[($a^{\prime}$)] $\;\;\; \cup_{j=1}^{q} {\rm graph} \, (h_{j} + u_{j})= \cup_{j\in \tilde{J}_1} {\rm graph} \, (h_{j} + \tilde{u}_{j})$,

$\;\;\; \cup_{j=1}^{q} {\rm graph} \, (g_{j} + w_{j})= \cup_{j\in \tilde{J}_2} {\rm graph} \, (g_{j} + \tilde{w}_{j})$
 
 where $\tilde{J}_1,\tilde{J}_2  \subset \{1, ..., q\}$, $\#\tilde{J}_1 = \tilde{q}_1 \leq q$, $\#\tilde{J}_2 = \tilde{q}_2 \leq q$,
 
 and, for each $j\in \tilde{J}_1$ or $\tilde{J}_2$,
 
 $\tilde{u}_{j} \in C^{2} \, (B_{3/4} \,\cap \,\{x^{n} < -\t\})$;

$\tilde{w}_{j} \in C^{2} \, (B_{3/4} \, \cap \, \{x^{n} > \t\});$ 

 $h_{j} + \tilde{u}_{j}$ and $g_{j} + \tilde{w}_{j}$ solve the CMC equation on their domains;

\item[($b$)]\;\;\; $|\zeta^{n+1}|^{2} + {\hat E}_{V, \eps,1}^{2}|\zeta^{n}|^{2} \leq C \int_{B_{1} \times {\mathbb R}} {\rm dist}^{2} \, (X, {\rm spt} \, \|{\mathbf C}\|) \, d\|V\|(X)$ 

$\hspace{3in} + \frac{C}{\eps}\left(\int\limits_{B_{1} \times {\mathbb R}} |h|^p d\|V\|\right)^{\frac{1}{p}};$

\item[($c$)]\;\;\;\;$\int_{B_{5\r/8}^{n+1}(Z)} \frac{{\rm dist}^{2} \, (X, {\rm spt} \, \|T_{Z \, \#} \, {\mathbf C}\|)}{|X - Z|^{n+2 - \mu}} \, d\|V\|(X)$ 

$\hspace{1in}\leq \widetilde{C}\r^{-n-2+\mu}\int_{B_{\r}(\eta, \z^{n} )\times {\mathbb R}} {\rm dist}^{2} \, (X, {\rm spt} \,\|T_{Z\, \#} \, {\mathbf C}\|) \, d\|V\|(X)$

$\hspace{3in}+ \tilde{C}\r^{1-\frac{n}{p}}\frac{1}{\eps}\left(\int\limits_{B_{\rho} \times {\mathbb R}} |h|^p d\|V\|\right)^{\frac{1}{p}}.$

\item[($d$)]\;\;\;\;$B^{n+1}_{\delta}(y, 0) \cap \{Z \, : \, \Theta \, (\|V\|, Z) \geq q\} \neq \emptyset$ 

for each point $(y, 0) \in ({\mathbb R}^{n-1} \cap B_{1/2})\times \{0\}$.  

\item[($e$)]\;\;\;\; $\int_{B_{1/2}^{n+1}(0) \cap \{|(x^{n}, x^{n+1})| < \sigma\}} {\rm dist}^{2} \, (X, {\rm spt} \, \|{\mathbf C}\|) \, d\|V\|(X)$ 

$\leq C_{1}\sigma^{1-\mu}\left(\int_{B_{1} \times {\mathbb R}} {\rm dist}^{2} \, (X, {\rm spt} \, \|{\mathbf C}\|) \, d\|V\|(X) + \frac{1}{\eps}\left(\int\limits_{B_{1} \times {\mathbb R}} |h|^p d\|V\|\right)^{\frac{1}{p}}\right)$

for each $\sigma \in [\delta, 1/4)$. 
\end{description}

Here $T_{Z} \, : \, {\mathbb R}^{n+1} \to {\mathbb R}^{n+1}$ is the translation $X \mapsto X + Z$;
$C = C(n, q) \in (0, \infty)$, $\widetilde{C} = \widetilde{C}(n, q,  \mu) \in (0, \infty)$ and $C_{1} = C_{1}(n, q, \mu) \in (0, \infty)$ (In particular, $C$, $\widetilde{C},$ $C_{1}$ do not depend on $\r$ or $\tau$.) 
\end{lem}

For conclusions ($a$) and ($a^{\prime}$) of Lemma ~\ref{L2-est-1} (which are regularity statements for $V$ in a region a fixed distance away from the singular axis of the cone ${\mathbf C}$), it turns out that the argument is fairly directly based on the induction hypotheses $(H1)$, $(H2)$, $(H3)$. Notice that the role of ($a'$) when compared to ($a$) is similar to the role of (H3) when compared to (H1), namely we remove some of the graphs obtained in ($a$) that do not affect the resulting support of $V_k \res \{|x^n|>\tau\}$ and gain, in doing so, the $C^2$ regularity and the fact that $h_j+\tilde{u}_j$ and $g_j+\tilde{w}_j$ solve the CMC equation. The subset $\tilde{J}_1$ is such that $\cup_{j \in \tilde{J}}H_j=\cup_{j =1}^q H_j$ (similarly for $\tilde{J}_2$ and $G_j$). 
For the other conclusions however (which require information on the part of $V$ very near the axis of ${\mathbf C}$),  a considerable amount of delicate additional technical arguments are needed to implement Simon's \cite{Simon2} basic strategy. See \cite[Section~10]{WicAnnals} for the proofs.

\medskip

With the help of Lemma~\ref{L2-est-1}, the blow up argument needed to establish Lemma~\ref{intermediate} proceeds as follows (the proof continues until the end of this subsection). 

\textbf{\textit{Proof of Lemma~\ref{intermediate}}}. Let $\{\e_{k}\}$,$\{\gamma_{k}\}$ and $\{\beta_{k}\}$ be sequences of positive numbers such that 
$$\e_{k}, \gamma_{k}, \b_{k} \to 0.$$ 
Consider  sequences of varifolds $V_{k} \in {\mathcal S}_H$ and cones ${\mathbf C}_{k} \in {\mathcal C}_{q}$ such that, for each $k=1, 2, \ldots,$ with $V_{k}$, ${\mathbf C}_{k}$  in place of $V$, ${\mathbf C}$ respectively, Hypotheses~\ref{hyp} hold with $\e_{k}$, $\g_{k}$ in place of $\e$, $\g$; Hypothesis ($\star$) holds with $M = M_{0}^{3}$ and Hypothesis~($\star\star$) holds with $\b_{k}$ in place of $\b$. 

Let $\{\d_{k}\}, \{\t_{k}\}$ be sequences of decreasing positive numbers converging to $0$, and let

$$E_{k}= \left(\int_{B_{1} \times {\mathbb R}} {\rm dist}^{2} \, (X, {\rm spt} \, \|{\mathbf C_{k}}\|) \, d\|V_{k}\|(X)+\frac{1}{\eps_k}\left(\int_{B_{1} \times {\mathbb R}} |h|^p \, d\|V_{k}\|(X)\right)^{\frac{1}{p}}\right)^{1/2}.$$ 
By passing to appropriate 
subsequences of $\{V_{k}\}$, $\{{\mathbf C}_{k}\}$, and possibly replacing ${\mathbf C}_{k}$ with a cone 
${\mathbf C}_{k}^{\prime} \in {\mathcal C}_{q}$ with ${\rm spt} \, \|{\mathbf C}_{k}^{\prime}\| = {\rm spt} \, \|{\mathbf C}_{k}\|$ without changing notation, we obtain functions 
$u^{(k)}_{j} \in C^{2}\left(B_{3/4} \cap \{x^{n}  < -\t_{k}\}\right)$ and $w^{(k)}_{j} \in C^{2}\left(B_{3/4} \cap \{x^{n} > \t_{k}\}\right)$, $1 \leq j \leq q$,  satisfying the conclusions of Lemma~\ref{L2-est-1} with $V_{k}$ in place of $V$, ${\mathbf C}_{k}$ in place of ${\mathbf C}$,  $u^{(k)}_{j}$, $w^{(k)}_{j}$ in place of $u_{j}$, $w_{j}$ and with $\d_{k}$, $\t_{k}$ in place of $\d$, $\t$ respectively. By Lemma~\ref{L2-est-1}(a), (a') and elliptic estimates, there exist harmonic functions $\varphi_{j} \in C^{2}\left(B_{3/4} \cap \{x^{n} < 0\}\right)$ and $\psi_{j} \in C^{2}\left(B_{3/4} \cap \{x^{n} > 0\}\right)$, $1 
\leq j \leq q$ such that 
$$E_{k}^{-1} u_{j}^{(k)} \to \varphi_{j} \quad {\rm and} \quad E_{k}^{-1} w_{j}^{(k)} \to \psi_{j}$$  
where the convergence is locally in $C^{1}$ in the respective domains; assertion (a) and $C^{1,\alpha}$-estimates allow to find a $C^1$-limit, while assertion (a') and $C^{2,\alpha}$ estimates on the $\tilde{u}_j$ and $\tilde{w}_j$ allow to prove the harmonicity of $\varphi_j$ and $\psi_j$ with an argument similar to the one given in Section \ref{HardtSimon}. It also follows from Lemma~\ref{L2-est-1}(e) that $\varphi_{j} \in L^{2}(B_{1/2} \cap \{x^{n} < 0\})$ and $\psi_{j} \in L^{2} \, (B_{1/2} \cap \{x^{n} > 0\})$, and that the convergence above is also in  $L^{2}(B_{1/2} \cap \{x^{n} < 0\})$ and $L^{2}\, (B_{1/2} \cap \{x^{n} > 0\})$ respectively. Write $\varphi = (\varphi_{1}, \varphi_{2}, \ldots, \varphi_{q})$ and 
$\psi = (\psi_{1}, \psi_{2}, \ldots, \psi_{q})$.

\begin{df}
 Let  ${\mathcal B}_{q}^{F}$ be the (smaller) class of pairs of functions $(\varphi, \psi)$ (the \emph{\textbf{fine blow-ups}}) arising as above corresponding to sequences $V_{k}$, ${\mathbf C}_{k}$ satisfying all of the hypotheses 
required as above but with $M = M_{0}^{2}$ in Hypothesis ($\star$). 
\end{df}

What is required next in order to prove Lemma~\ref{intermediate} is sufficiently strong uniform regularity estimates for the functions $\varphi$, $\psi$ with $(\varphi, \psi) \in {\mathcal B}_{q}^{F}$ in fixed smaller regions in their domains, e.g.\  in  
$B_{1/4}  \cap \{x^{n}< 0\}$ and $B_{1/4} \cap \{x^{n} > 0\}$ respectively, necessarily up to the boundary $B_{1/4} \cap \{x^{n} = 0\}.$  For instance a  uniform $C^{1, \alpha}$ estimate up to the boundary $B_{1/4} \cap \{x^{n} = 0\}$ for some fixed $\alpha \in (0, 1)$ suffices. 

\medskip

\textit{First step: H\"older continuity up to the boundary.} The first step in obtaining such regularity estimates is a uniform continuity estimate up to the boundary for $\varphi$ and $\psi$ (compare \cite[proof of Lemma 12.1]{WicAnnals}). This is obtained as follows, essentially as a consequence of Lemma~\ref{L2-est-1}(b), (c), (d). For any $Y\in B_{5/16} \cap \{x^n=0\}$ we can find, thanks to (d) (upon extracting a subsequence from the sequence $V_k$ giving rise to the fine blow up $(\varphi, \psi)$, that we do not relabel) $Z_k=(\eta_k, \zeta_k^n, \zeta_k^{n+1})$ such that $Z_k \to Y$ and we set (taking, in the following, subsequential limits that we do not relabel)

$${\k}_1(Y):=\lim_{k \to \infty}  \frac{\zeta_k^{n+1}}{E_k},\,\,\,  {\k}_2(Y):=\lim_{k \to \infty} \hat{E}_k\frac{\zeta_k^n}{E_k};$$
these limits are finite thanks to (b) in Lemma~\ref{L2-est-1}, moreover we can see that they depend only on $Y$ and not on the choice of $\{Z_k\}$, so the notation is appropriate. Indeed, using Lemma~\ref{L2-est-1}(c) for the chosen subsequence $V_k$, $Z_k=(\eta_k, \zeta_k^n, \zeta_k^{n+1})$ (in place of $V$ and $Z=(\eta, \zeta^n, \zeta^{n+1})$) and with $\mu \in (\frac{n}{p},1)$, dividing by $E_k^2$ and passing to subsequential limit, we obtain the validity of inequality (\ref{continuity-1}) below; the finiteness of the integrals forces the continuity of $\varphi$ and $\psi$ up to the boundary with boundary values respectively $(\k_{1}(Y) - \ell_{j}\k_{2}(Y))$ and $(\k_{1}(Y) - m_{j}\k_{2}(Y))$. The choice of a different subsequence $Z_k=(\eta_k, \zeta_k^n, \zeta_k^{n+1}) \to Y$ will not alter the fine blow up $(\varphi, \psi)$, which has been fixed, so necessarily ${\k}_1$ and ${\k}_2$ depend only on $Y$. Summarizing, we have established:

\noindent
\emph{There exist 
two functions $\k_{1}, \k_{2}\,  : \, B_{5/16} \cap \{x^{n} = 0\} \to {\mathbb R}$  with
\begin{equation*}\label{continuity-1-1}
|{\k}_{1}(Y)|, |{\k}_{2}(Y)| \leq C \quad for \quad Y \in B_{5/16} \cap \{x^{n} = 0\}
\end{equation*}
and numbers $\ell_{j}, m_{j}$ $(1 \leq j \leq q)$ with $\ell_{1} \geq \ldots \geq \ell_{q}$, 
$m_{1} \leq \ldots \leq m_{q},$  
\begin{eqnarray}\label{excess-6}
c \leq {\rm max} \, \{|\ell_{1}|, |\ell_{q}|\} \leq c_{1}, \;\;\; c \leq {\rm max} \,\{|m_{1}|, |m_{q}|\} 
\leq c_{1} \;\;\; \mbox{and}&&\nonumber\\
&&\hspace{-3in}{\rm min} \, \{|\ell_{1} - \ell_{q}|, |m_{1} - m_{q}|\} \geq 2c 
\end{eqnarray}
where $C = C(n, q)$, $ c = c(n, q)$, $c_{1} = c_{1}(n, q),$ such that} 

\begin{equation}
\label{continuity-1}
\sum_{j=1}^{q} \int_{B_{\r/2}(Y) \cap \{x^{n} <0\}} \frac{|\varphi_{j}(x) - (\k_{1}(Y) - \ell_{j}\k_{2}(Y))|^{2}}
{|x- Y|^{n+2-\mu}} \, dx 
\end{equation}
$$+ \;\;\sum_{j=1}^{q}\int_{B_{\r/2}(Y) \cap \{x^{n}> 0\}}\frac{|\psi_{j}(x) - (\k_{1}(Y)- 
m_{j}\k_{2}(Y))|^{2}}{|x - Y|^{n+2-\mu}} \, dx\leq $$
$$\leq C_{1} \r^{-n-2+\mu}\sum_{j=1}^{q}\int_{B_{\r}(Y) \cap \{x^{n} <0\}} |\varphi_{j} - ({\k}_{1}(Y) - \ell_{j}{\k}_{2}(Y))|^{2}+$$
$$+\, C_{1}\r^{-n-2+\mu}\sum_{j=1}^{q} \int_{B_{\r}(Y) \cap 
\{x^{n} > 0\}} |\psi_{j} - ({\k}_{1}(Y) - m_{j}{\k}_{2}(Y))|^{2} +\, C_{1}\r^{1-\frac{n}{p}},$$

\emph{where $C_{1} = C_{1}(n, q) \in (0, \infty).$} 

\medskip

This estimate implies (see \cite[Lemma~12.1]{WicAnnals}) that $\varphi$, $\psi$ are respectively uniformly 
H\"older continuous in $B_{5/16} \cap \{x^{n} \leq 0\}$, $B_{5/16} \cap \{x^{n} \geq 0\}$ with 
a fixed H\"older exponent $\b = \b(n, p, q) \in (0, 1/2)$ and the functions 
$\varphi(Y) \equiv (\k_{1}(Y) - \ell_{1}\k_{2}(Y), \ldots, \k_{1}(Y) - \ell_{q}\k_{2}(Y)),$   
$\psi(Y) \equiv  (\k_{1}(Y) - m_{1}\k_{2}(Y), \ldots, \k_{1}(Y) - m_{q}\k_{2}(Y))$ are respectively the continuous boundary values of $\varphi$, $\psi$ on $B_{5/16} \cap \{x^{n} = 0\}.$ We point out that the choice of $\mu$ in \cite[Lemma~12.1]{WicAnnals} is $\mu=\frac{1}{2}$, which is not generally allowed here, where we choose $\mu \in(\frac{n}{p},1)$; when replicating the proof of \cite[Lemma~12.1]{WicAnnals}, the different value of $\mu$ will affect the value of $\gamma$ appearing there. Moreover, due to the last term on the right-hand side of (\ref{continuity-1}), inequality \cite[(12.5)]{WicAnnals} must be written with an extra term $\sigma^{1-\frac{n}{p}}$ on the left-hand side and an additional term $C_1 \left(\frac{\sigma}{\rho}\right)^{1-\frac{n}{p}} \rho^{1-\frac{n}{p}}$ : this will also alter the value of $\beta$, which will be $0<\beta<\min\left\{\frac{1-\frac{n}{p}}{2}, -\frac{1}{2 \log_4 \gamma}\right\}$;  apart from this specific value, the statement of \cite[Lemma~12.1]{WicAnnals} holds in the present setting with exactly the same formulation.

\medskip

\textit{Second step: higher regularity up to the boundary.} In order to improve regularity of $(\varphi, \psi)$, it then suffices to show that the boundary values 
$\varphi(Y)$, $\psi(Y)$ as functions of $Y \in B_{1/4} \cap \{x^{n} = 0\}$ are more regular. This is achieved by employing two separate first variation arguments based on the fact that the varifolds $V_{k}$  giving rise to the blow-up $(\varphi, \psi)$ are CMC. The first one is the first variation induced by a vector field $\tilde{\zeta} e^{n+1}$ (``vertical vector field''), while the second one by a vector field $\tilde{\zeta} e^{n}$ (``horizontal and orthogonal to the axis of $C$''), following the notations in \cite[Theorem~12.2]{WicAnnals}. We point out that when writing the analogue of \cite[(12.14)]{WicAnnals} we have $\int h_k \tilde{\zeta}e^{n+1} \cdot  \hat{\nu} d\|V_k\|$ instead of $0$ on the right-hand side and, keeping track of this term, we have to replace the right-hand side of inequality \cite[(12.16)]{WicAnnals} with $C(\sup|D\zeta| \tau^{1-\mu}+\eps_k E_k) E_k$. Dividing by $E_k$ and sending $k\to \infty$ at fixed $\tau$, and then letting $\tau \to 0$, we obtain \cite[(12.24)]{WicAnnals}. Analogue changes must be made for the arument involving the second type of first variation, thanks to which we obtain the analogue of \cite[(12.33)]{WicAnnals} with the right-hand side replaced by $C(\sup|D\zeta| \tau^{\frac{1-\mu}{2}}\hat{E}_k+\eps_k E_k) E_k$. Dividing by $\hat{E}_k E_k$ and sending $k\to \infty$ at fixed $\tau$, and then letting $\tau \to 0$, we obtain \cite[(12.39)]{WicAnnals}. From \cite[(12.24)]{WicAnnals} and \cite[(12.39)]{WicAnnals} it can be shown, following \cite{WicAnnals}, that the two linear combinations of $\k_{1}$, $\k_{2}$ defined by 

$$\Phi(Y) = 2q\kappa_{1}(Y) - \left(\sum_{j=1}^{q}  (\ell_{j} + m_{j})\right)\kappa_{2}(Y) \quad {\rm and}$$

$$\Psi(Y) = \left(\sum_{j=1}^{q}(\ell_{j} + m_{j})\right)\kappa_{1}(Y)  - \left(\sum_{j=1}^{q} (\ell_{j}^{2} + m_{j}^{2})\right)\kappa_{2}(Y)$$
for $Y \in B_{9/32} \cap \{x^{n} = 0\}$ are smooth functions satisfying

$$\|\Phi\|_{C^{3}(B_{9/32} \cap \{x^{n} = 0\})} +  \|\Psi\|_{C^{3}(B_{9/32} \cap \{x^{n} = 0\})}\leq $$
$$\leq C\left(\int_{B_{1/2} \cap \{x^{n} \leq 0\}} |\varphi|^{2} + \int_{B_{1/2} \cap \{x^{n} \geq 0\}} |\psi|^{2} \right)^{1/2},$$
where $C = C(n, q) \in (0, \infty).$ Since the Jacobian $J$ associated with the above linear system for $\k_{1}$, $\k_{2}$ is given by 
\begin{eqnarray*}
&&J = 2q\sum_{j=1}^{q}(\ell_{j}^{2} + m_{j}^{2}) - \left(\sum_{j=1}^{q}(\ell_{j} + m_{j})\right)^{2}\nonumber\\ 
&&\hspace{1in}= \frac{1}{2}\sum_{i=1}^{q}\sum_{j=1}^{q} \left((m_{i} - m_{j})^{2} + (\ell_{i} - \ell_{j})^{2} + 2(\ell_{i} - m_{j})^{2}\right),
\end{eqnarray*}
$J$ satisfies, by (\ref{excess-6}), $\widetilde{C} \geq J \geq C >0$ for constants $\widetilde{C} =\widetilde{C}(n, q)$, $C = C(n, q).$  
It follows that $\k_{j}$, $j=1, 2$, and hence $\left.\varphi\right|_{B_{9/32} \cap \{x^{n} =0 \}}$ and 
$\left.\psi\right|_{B_{9/32} \cap \{x^{n} = 0\}},$ are $C^{3}$ on $B_{9/32} \cap \{x^{n} = 0\},$ 
and moreover  
\begin{eqnarray*}
&&\|\varphi\|_{C^{3}(B_{9/32} \cap \{x^{n} = 0\})} +  \|\psi\|_{C^{3}(B_{9/32} \cap \{x^{n} = 0\})}\nonumber\\ 
&&\hspace{1in}\leq C\left(\int_{B_{1/2} \cap \{x^{n} \leq 0\}} |\varphi|^{2} + \int_{B_{1/2} \cap \{x^{n} \geq 0\}} |\psi|^{2}\right)^{1/2}
\end{eqnarray*}
for a constant $C = C(n, q)$. By standard boundary regularity theory for harmonic functions, this leads to the desired improved regularity for $(\varphi, \psi)$, specifically to the conclusion that 
$\varphi \in C^{2} \, (B_{1/4} \cap \{x^{n} \leq 0\}), \;\;\; \psi \in C^{2} \, (B_{1/4} \cap 
\{x^{n} \geq 0\})$ with 
\begin{eqnarray*}
&&\|\varphi\|_{C^{2}(B_{1/4} \cap \{x^{n} \leq 0\})} + \|\psi\|_{C^{2}(B_{1/4}\cap \{x^{n} \geq 0\})}\nonumber\\ 
&&\hspace{1in}\leq C\left(\int_{B_{1/2} \cap \{x^{n} < 0\}}|\varphi|^{2} + \int_{B_{1/2} \cap \{x^{n} > 0\}} |\psi|^{2}\right)^{1/2}
\end{eqnarray*}
where $C = C(n, q).$  Lemma~\ref{intermediate} follows from this estimate in  a fairly standard way, exploiting the corresponding decay property up to the boundary for the harmonic functions $\varphi, \psi$.

\section{Proof of the Sheeting Theorem  subject to the induction hypotheses} 
\label{sketchsheeting}

Once it is established that ${\mathcal B}_{q}$ is a proper blow up class, it follows from Theorem~\ref{harmonic} that each $v \in {\mathcal B}_{q}$ is given by a set of $q$ (classical) harmonic functions on $B_{1}.$ Moreover, given $v \in {\mathcal B}_{q},$ a sequence of varifolds $\{V_{k}\} \subset {\mathcal S}_H$ giving rise to $v$ and any $\s \in (0, 1)$, we have the dichotomy that either $\Theta \, (\|V_{k}\|, Z) < q$ for all $k$ sufficiently large and $Z \in B_{\s}(0) \times {\mathbb R}$, or (by Lemma~\ref{harmonic} and by the argument described in the proof of Theorem \ref{coarse} in establishing property $({\mathcal B}\emph{4})$ for ${\mathcal B}_{q}$) that $v^{1} = v^{2} = \ldots  = v^{q}$ on $B_{1}.$ 
Given this fact, it is fairly straightforward to establish, using elliptic estimates for harmonic functions and with the help of Proposition \ref{Prop:elementaryconsequence} and Theorem~\ref{thm:SS}, the following lemma:
\begin{lem}\label{excess-s}
Let $q$ be an integer $\geq 2$, $\th \in (0, 1/4)$ and suppose that the induction hypotheses $(H1)$, $(H2)$, $(H3)$ hold. There exists a number $\b_{0} = \b_{0}(n, q, \th) \in (0, 1/2)$ such that if $V \in {\mathcal S}_H$, $(\omega_{n}2^{n})^{-1}\|V\|(B_{2}^{n+1}(0)) < q + 1/2$, $q-1/2 \leq (\omega_{n})^{-1}\|V\|(B_{1} \times {\mathbb R}) < q + 1/2$, and $$\int_{B_{1} \times {\mathbb R}} {\rm dist}^{2} \, (X, P) \, d\|V\|(X) + \frac{\left( \int_{B_{1} \times {\mathbb R}}|h|^p d\|V\|(X)\right)^{\frac{1}{p}}}{\b_{0}}< \b_{0}$$
for some affine hyperplane $P$ of ${\mathbb R}^{n+1}$ with ${\rm dist}^{2}_{\mathcal H} \,(P \cap (B_{1} \times {\mathbb R}), B_{1} \times \{0\}) < \b_{0},$ then the following hold:
\begin{itemize}
\item[(a)] \underline{Either}  $V \res(B_{1/2} \times {\mathbb R}) = \sum_{j=1}^{q} |{\rm graph} \, u_{j}|$
where $u_{j} \in C^{2} \,(B_{1/2}; {\mathbb R})$ for $j=1, 2, \ldots, q$ with  
$u_{1} \leq u_{2} \leq \ldots \leq u_{q}$ on $B_{1/2},$ $u_{j_{0}} < u_{j_{0}+1}$ on $B_{1/2}$ for some $j_{0} \in \{1,2, \ldots, q-1\},$ and for each $j \in \{1,2, \ldots, q\}$, 
\begin{eqnarray*}
&&{\rm sup}_{B_{1/2}} \, |u_{j} - p|^{2} + |D \, u_{j} - D \, p|^{2} + |D^{2} \, u_{j}|^{2}\nonumber\\ 
&&\hspace{1.0in}\leq C\int_{B_{1} \times {\mathbb R}} {\rm dist}^{2} \, (X, P) \, d\|V\|(X) + C\frac{\left( \int_{B_{1} \times {\mathbb R}}|h|^p d\|V\|(X)\right)^{\frac{1}{p}}}{\b_0}
\end{eqnarray*}
where $C = C(n, q) \in (0, \infty)$ and $p \, : \, {\mathbb R}^{n} \to {\mathbb R}$ is the affine function such that ${\rm graph} \,  p = P;$ 
\underline{or}, there exists an  affine hyperplane $P^{\prime}$ with 
\begin{eqnarray*}
&&{\rm dist}_{\mathcal H}^{2} \, \left( P^{\prime} \cap (B_{1} \times {\mathbb R}), P \cap (B_{1} \times {\mathbb R})\right)\nonumber\\ 
&&\hspace{0.5in}\leq C_{1}\int_{B_{1} \times {\mathbb R}} {\rm dist}^{2} \,(X, P) \,  d\|V\|(X) + \frac{\left( \int_{B_{1} \times {\mathbb R}}|h|^p d\|V\|(X)\right)^{\frac{1}{p}}}{\b_{0}} \;\;\; and
\end{eqnarray*}
\begin{eqnarray*}
&&\th^{-n-2}\int_{B_{\th} \times {\mathbb R}} {\rm dist}^{2} \, (X, P^{\prime}) \, d\|V\|(X)\nonumber\\ 
&&\hspace{0.5in}\leq C_{2}\th^{2}\left(\int_{B_{1} \times {\mathbb R}} {\rm dist}^{2} \,(X, P) \,  d\|V\|(X) + \frac{\left( \int_{B_{1} \times {\mathbb R}}|h|^p d\|V\|(X)\right)^{\frac{1}{p}}}{\b_{0}}\right)\,\,\,\,\,
\end{eqnarray*}
where $C_{1} = C_{1}(n, q) \in (0, \infty)$ and $C_{2} = C_{2}(n, q) \in (0, \infty).$

\item[(b)] $\left(\omega_{n}(2\th)^{n}\right)^{-1}\|V\|(B_{2\th}^{n+1}(0)) < q + 1/2$ and 
$$q-1/2 \leq (\omega_{n}\th^{n})^{-1} \|V\|(B_{\th} \times {\mathbb R}) < q + 1/2.$$
\end{itemize}
\end{lem} 

For any given $z \in B_{3/4}$, we may apply this lemma iteratively starting with $\eta_{(0, z), 1/4 \, \#} \, V$ in place of $V$, stopping the iteration when we reach a scale at which the first alternative in (a) occurs.  Notice that when this happens, the lemma shows that $V \res ( B_{\s}(z) \times {\mathbb R})$ for some some $\s \in (0, 1/8)$ is the union of embedded graphs of $q$ smooth functions over $B_{\s}(z) \subset {\mathbb R}^{n}$ each with small gradient. If the iteration proceeds indefinitely, then ${\rm spt} \, \|V\|$ contains precisely one point above $z$, and $V$ at that point has a unique multiplicity $q$ tangent hyperplane that is almost parallel to ${\mathbb R}^{n} \times \{0\}$, to which $V$ decays upon rescaling (by the decay of $\hat{E}^2_{V,\beta,\th}$ that holds in case the second alternative of (a) is verified). With a little extra work, it is not difficult to see that these  facts lead to Theorem~\ref{thm:sheeting} with $\alpha \in\left(0, \frac{1}{2}\left(1-\frac{n}{p}\right)\right)$.

\section{Proof of the Minimum Distance Theorem subject to the induction hypotheses}
\label{min-dist-sec}

The next part of our simultaneous proof by induction is the completion of induction for the Minimum Distance Theorem, i.e. steps (ii) and (iiii) listed in Section \ref{mainsteps}. With regard to this, the idea is to complete step (ii) first, using among other things the result of step (i), and then use both step (i) and step (ii)  to complete step (iii). 

In both cases, the argument is by contradiction and uses a version of the blow-up method described in the preceding section, the argument this time being much closer to its original form used by Simon in \cite{Simon2}. More specifically, the goal is to show if there is a sequence of varifolds $V_{j} \in {\mathcal S}_H$ converging to a cone ${\mathbf C}$ as in the Minimum Distance Theorem (with the appropriate density at the origin, namely, $\Theta \, (\|{\bf C}\|, 0) = q + 1/2$ in step (ii) and $\Theta \, (\|{\mathbf C}\|, 0) = q + 1$ in  step (iii)), then for sufficiently large $j$, $V_{j}$ must have a classical singularity, in direct contradiction to the definition of ${\mathcal S}_H$. It is in completing step (ii) and step (iii) that the no-classical-singularities hypothesis is used non-inductively for the first and the only time in the entire proof. 

To describe steps (ii) and (iii) in more detail, let ${\mathbf C}$ be as in the Minimum Distance Theorem \ref{thm:min-dist}. Then ${\mathbf C} = \sum_{j=1}^{m} q_{j}|P_{1}|$  for some $m$ ($\geq 3$) distinct half-hyperplanes $P_{1}, P_{2}, \ldots, P_{m}$ meeting along a common boundary $L$ and  for some positive integers $q_{1}, q_{2}, \ldots, q_{m}$. If $\eta_{j}$ is the unit vector along $P_{j}$ orthogonal to $L$, it is an easy exercise to show that stationarity of ${\mathbf C}$ in ${\mathbb R}^{n+1}$ is equivalent to the requirement that 
$$\sum_{j=1}^{m} q_{j} \eta_{j} = 0,$$ 
and to show, using this fact, that
$$\Theta \, (\|{\mathbf C}\|, 0) \in \{q+1/2, q+1\} \implies  q_{j} \leq q \quad \mbox{for each} \quad j=1, 2, \ldots, m.$$  
Thus, in view of step (i), we have, whenever the induction hypotheses (H1), (H2), (H3) hold and a cone ${\mathbf C}$ and a varifold $V$ are as in Theorem~\ref{thm:min-dist} with $\Theta \, (\|{\mathbf C}\|, 0) \in \{q+1/2, q+1\}$,   then $V$ is regular away from any $\e$-neighborhood of the axis of ${\mathbf C}$  provided $V$ is sufficiently close to ${\mathbf C}$ (depending on ${\mathbf C}$ and $\e$). 

Consider now step (ii), namely, the case when $\Theta \, (\|{\mathbf C}\|, 0) = q + 1/2.$ In this case, the first task in preparation for the blow up argument is to establish the analogues of Simon's estimates (as in the conclusion of Lemma~\ref{L2-est-1} above) for $V \in {\mathcal S}_H$ satisfying the appropriate set of contradiction hypotheses (analogous to 
the hypotheses of Lemma~\ref{L2-est-1}) corresponding to step (ii), including the hypothesis that $V$ is close as a varifold to ${\mathbf C}$ in the unit ball $B_{1}^{n+1}(0).$   

Unlike in the case of Lemma~\ref{L2-est-1} above, the estimates this time are established following Simon's argument in \cite{Simon2} very closely; the reason why it is possible to do so even though, strictly speaking, the ``multiplicity 1 class'' hypothesis needed in \cite{Simon2} does not hold in the present situation is because by hypothesis (of the Minimum Distance Theorem with $\Theta \, (\|{\mathbf C}\|, 0)  = q + 1/2$), the density ratio of $V$ is not much more than $q+1/2$, and that implies,  by $(H1)$ and step (i) (which provides the Sheeting Theorem for multiplicity up to and including $q$), complete regularity of $V$ in any small toroidal region around and close to the axis of ${\mathbf C}$ whenever $V$ is close in distance to ${\rm spt} \, \|{\mathbf C}\|$ in a slightly larger torus. This fact serves the purpose that Allard's regularity theorem served in the argument of \cite{Simon2}. See \cite[Theorem~16.2, Corollary~16.3 and Corollary~16.4]{WicAnnals}.

\medskip

Note also that the crucial \textit{``no $\delta$ gaps'' condition}, i.e.\ the analogue of  Lemma~\ref{L2-est-1}(d),  holds in the present setting. For if not, then there is a ball $B_{\delta}  = B_{\delta}^{n+1}(Y)$ of fixed size $\delta$ about some point $Y$ on the axis of ${\mathbf C}$  
in which all singularities of $V$ have density $< q + 1/2$, implying, in view of Proposition \ref{Prop:elementaryconsequence}, which uses the Almgren--Federer Generalized stratification theorem, the induction hypotheses $(H1)$, $(H2)$, $(H3)$ and Step (i), that $V \res B_{\delta}$ is ``generalized regular'' up to a singular set of codimension 7. The following argument leads to a contradiction whenever $V$ in distance is sufficiently close to ${\mathbf C}$ (depending on $\delta$). 

Note first of all that in view of the fact that $V \res B_{\delta}$ is ``generalized regular'' up to a singular set of codimension 7 allows us to set the multiplicity equal to $1$ almost everywhere without losing the fact that the new varifold $M$ belongs to $\mathcal{S}_H$, as explained in the first paragraph of the proof of Theorem \ref{thm:SS}. A standard capacity argument shows that ${\rm sing} \, V \setminus \Greg{V}$ is ``removable for the stability inequality'' in the sense that the inequality $\int_{{\rm \Greg} \, V} |A|^{2}\z^{2} \leq \int_{{\rm \Greg} \, V} |\nabla \, \z|^{2}$ holds for every 
$\z \in C^{1}_{c}(U)$ (and not just for $\z \in C^{1}_{c}(\Greg{V})$ (as initially required). In particular, we may take a standard cut-off function $\z$ in this improved stability inequality  to deduce that
$$\int_{M \cap B_{3/4}^{n+1}(0)} |A|^{2} \leq C$$ for a fixed constant $C= C(n, \Lambda)$ independent of $M,$ where $\Lambda$ is any upper bound for ${\mathcal H}^{n} \, (M \cap B_{1}^{n+1}(0))$. This, the Cauchy--Schwarz inequality, and the volume growth bounds for $M$ (implied by the monotonicity formula) imply  that if $L$ is an $(n-1)$-dimensional subspace of ${\mathbb R}^{n+1}$, then for each fixed $\Lambda >0$, 
\begin{equation}\label{secondff}
\sup_{M \in {\mathcal M}, {\mathcal H}^{n}(M \cap B_{1}^{n+1}(0)) \leq \Lambda} \, \int_{M \cap B_{3/4}^{n+1}(0) \cap (L)_{\t}} |A| \to 0
\end{equation}
as $\t \to 0$. 

On the other hand, if ${\mathbf C}$ is as in Theorem~\ref{thm:min-dist}, $L$ is the singular axis of ${\mathbf C}$, $p \, : \, {\mathbb R}^{n+1} \to L$ is the orthogonal projection, and if, contrary to Theorem~\ref{thm:min-dist}, there is a sequence $M_{j} \in {\mathcal M}$ with $\omega_{n}^{-1}{\mathcal H}^{n}(M_{j} \cap B_{1}^{n+1}(0)) \leq \Theta(\|{\mathbf C}\|, 0) + 1/2$ and ${\rm dist}_{\mathcal H} \, (M_{j} \cap B_{1/2}^{n+1}(0), {\rm spt} \, \|{\mathbf C}\|\cap B_{1/2}^{n+1}(0)) \to 0$, then for sufficiently large $j,$ $p^{-1}(y) \cap M_{j} \cap (L)_{\t}$ is a non-empty finite collection of disjoint, 
smooth embedded curves $\gamma_{j}^{k}$,  $1 \leq k \leq \ell_{j},$ for ${\mathcal H}^{n-1}$-a.e. $y \in L \cap B_{1/4}^{n+1}(0)$, and by Theorem~\ref{thm:sheeting} (inductive assumption), 
if $\nu_{j}$ denotes a choice of unit normal to $M_{j}$, then 
\begin{equation}\label{lower-bound}
\theta \leq \inf_{1 \leq k \leq \ell_{j}} \left|\nu_{j}(P_{j}^{k, 1}) - \nu_{j}(P_{j}^{k, 2})\right| \leq 
\int_{p^{-1}(y) \cap M_{j} \cap (L)_{\t}} |A| \, d{\mathcal H}^{1}
\end{equation}
where $P_{j}^{k, 1}, P_{j}^{k, 2}$ are the end points of $\gamma_{j}^{k}$ and $\theta$ is a fixed positive constant determined by ${\rm spt} \, \|{\mathbf C}\|$.  Integrating over $y \in L \cap B_{1/4}^{n+1}(0)$ and using the co-area formula, we see that this leads to a contradiction with  (\ref{secondff}) for sufficiently large $j$. \footnote{Notice that this simple argument clearly breaks down  if we do not have the (inductive) assumptions $(H1)$ for $q$ and $(H3)$ for $q-1$. Indeed, even  under the hypothesis ${\mathcal H}^{n-1} \, ({\rm sing} \, V) = 0$ which is stronger than the no-classical-singularities hypothesis and which together with Theorem~\ref{thm:sheeting}  makes it possible to justify (\ref{lower-bound}) for ${\mathcal H}^{n-1}$-a.e. $y \in L \cap B_{1/4}^{n+1}(0)$, the above argument breaks down because in this case, a priori  we do not know even whether $\int_{{\rm reg} \, V \cap B_{3/4}^{n+1}(0) \cap (L)_{\t}}|A| < \infty$, let alone the statement (\ref{secondff}). }

\medskip

Once Simon's estimates are established, a blow up argument analogous to the one described above in the context of Lemma~\ref{multi-scale} can be used to show that when $V \in {\mathcal S}_H$ is sufficiently close (as a varifold) to ${\mathbf C}$ in $B_{1}^{n+1}(0)$, $V$ must contain a classical singularity somewhere near  the origin, contradicting directly the definition of ${\mathcal S}_H$ and thereby completing step (ii). See \cite[Section~16]{WicAnnals} for details. 

\medskip

Armed with step~(ii), we can repeat the blow up argument, in an identical fashion to how it is used in  step~(ii) but with the assumption that 
$\Theta \, (\|{\mathbf C}\|, 0) = q + 1,$ to complete step (iii).

\section{Proof of the Higher Regularity Theorem subject to the induction hypotheses}
\label{higherreg}

The last part of our simultaneous proof by induction is the completion of induction for the Higher Regularity Theorem, i.e.~step (iv) listed in Section \ref{mainsteps}. We inductively assume the validity of (H1), (H2), (H3) and step (i), step (ii), step (iii). 

In view of this, by assumption we have that $V$ decomposes as $q$ sheets:
$$V \res \left(B_{1/2}^{n}(0) \times {\R} \right) = \sum_{j=1}^{q} |{\rm graph} \, u_{j}|,$$
with $u_{j} \in C^{1, \alpha} \, (B_{1/2}^{n}(0); {\mathbb R})$ and $u_{1} \leq u_{2} \leq \ldots \leq u_{q}$, $\alpha \in (0,\frac{1}{2})$. Whenever $\spt{V}$ is embedded at $(x,u_j(x))$ then the CMC equation is valid (in weak form) for $u_{j}$ in a neighbourhood of $x$, implying $C^2$ regularity for $u_i$ in that neighbourhood (and even $C^\infty$ by bootstrapping). Moreover keep in mind that for $V$ as in assumptions of Theorem \ref{thm:higher-reg} the density $\theta$ is everywhere integer-valued and we can inductively assume that $\spt{V} \cap \{\theta \leq q-1\} \subset \Greg{V}$. Note that we cannot hope that the $u_j$'s are separately smooth and CMC, since the assumptions allow the possibility presented in example \ref{oss:jumpsattouchingsing}: in such a situation, we would have $u_1 \leq u_2 \leq u_3$ with $u_1$ and $u_3$ smooth and CMC but $u_2$ is only $C^{1,1}$ and it coincides partly with $u_1$ and partly wth $u_3$. This is why in order to express the higher regularity result, we need to find a representation of $\spt{V}$ by neglecting multiplicities.

\subsection{Absence of $\ell$-fold touching singularities for $\ell \geq 3$}
\label{Hopf}

Our first claim is that, for $H\neq 0$ and $X \notin \text{reg} \, V$, whenever there exist (under the assumptions of Theorem \ref{thm:higher-reg}) sheets touching at $X$ then \textit{exactly two sheets} touch at $X$ when we discard multiplicities  and moreover \textit{the mean curvature vector points upwards on the top sheet and downwards on the bottom sheet} on the embedded parts. In particular three-fold, four-fold etc.~touching singularities are ruled out.

\begin{lem}
 \label{lem:Hopf}
Under the assumptions of Theorem \ref{thm:higher-reg} and assuming the validity of (H1), (H2), (H3), let $X=(x, X^{n+1}) \in \spt{V}$ be a point of density $q$ where $\spt{V}$ is not embedded. Assume that $H\neq 0$. Then
\begin{enumerate}
 \item $\spt{V}={\rm graph} \, u_1 \cup {\rm graph} \, u_q$ and $X=(x, X^{n+1}) \in \text{sing}_T \, V;$ 
 \item in a neighbourhood of $X$ and away from $\text{sing} \, V$, we have that $\vec{H} \cdot \hat{e}^{n+1} >0$ on ${\rm graph} \, u_{q}$ and $\vec{H} \cdot \hat{e}^{n+1} <0$ on ${\rm graph} \, u_{1}$.
\end{enumerate}
\end{lem}

\begin{oss}
The same argument shows that if $H=0$ then under the assumptions of Theorem \ref{thm:higher-reg} any two graphs $u_i, u_{i+1}$ either coincide identically (i.e.\ $u_i \equiv u_{i+1}$) or have empty intersection (i.e.\ $u_i < u_{i+1}$). Then the minimal surface PDE very directly implies that each $u_j$ is smooth.
\end{oss}

\begin{proof}
Recall that any two graphs touching must touch tangentially (otherwise we would create a classical singularity). Denoting with $\pi:B_{1/2}^n(0)\times \R \to B_{1/2}^n(0)$ the standard projection, let $C=\pi\{y\in \spt{V}: \theta(y)=q\}$ be the closed set in $B_{1/2}^n(0)$ above which the density of the varifold is $q$, i.e. the set of $x\in B_{1/2}^n(0)$ such that $u_1(x)=...=u_q(x)$. If $C$ is the whole ball $B_{1/2}^n(0)$ then $V$ is embedded in the cylinder $B_{1/2}^n(0)\times \R$ (and we can use standard bootstrapping to infer that $u_1\equiv...\equiv u_q(x)$ is a smooth function), against the assumption. 

Therefore $C \subsetneq B_{1/2}^n(0)$. Denote by $B$ the ball $B_{1/2}^n(0)$ and consider the open set $(B_{1/2}^n(0) \setminus C)\times \R$: here we have the validity of the inductive assumptions, as the density is $\leq q-1$. Consider $x \in B$ such that ${\rm spt} \, \|V\|$ is embedded at $(x,u_j(x))$ for every $j$. The set $U$ of such points $x$ is open and dense (by Allard's theorem). Then we can ``neglect multiplicities'' and define $\tilde{q}(x)$ as the number of points without multiplicity above $x$. By embededness $\tilde{q}$ is locally constant on $U$, i.e. constant on each connected component of $U$. This allows to define $\tilde{u}_1 < ... < \tilde{u}_{\tilde{q}}$ on $U$, with $\tilde{q}$ constant on each connected component of $U$ and $\spt{V \res (U \times \R)}$ described by the union of the graphs of $\tilde{u}_j$. We will however show that $\tilde{q}$ actually extends to a locally constant function on $B \setminus C$. Indeed, clearly $U \subset B \setminus C$ and on $(B \setminus C) \times \R$ we can use the inductive assumptions, thanks to which we know that  $\spt{V \res (B \setminus C)\times \R} \subset \Greg{V}$: by definition of $\Greg$ this means that at any point $y$ in $(B \setminus C)\times \R$ where $V$ is not embedded the structure of $\spt{V}$ is, in a neighbourhood, exactly the union of two $C^2$ graphs touching tangentially on a set of $0$-measure. Therefore if $\pi(y)$ is on the (topological) boundary of any two connected components of $U$, the number $\tilde{q}$ cannot change when we pass from one connected component to the other. This allows to extend uniquely the definition of the functions $\tilde{u}_1 \leq ... \leq \tilde{u}_{\tilde{q}}$ to $B \setminus C$ so that the union of their graphs is equal to $\spt{V \res (B \setminus C)\times \R}$ and moreover each $\tilde{u}_j$ is $C^2$ on $B\setminus C$ and satisfies the CMC PDE. In particular $\tilde{q}$ is locally constant on $B \setminus C$. The great advantage in passing to the functions $\tilde{u}_1 \leq ... \leq \tilde{u}_{\tilde{q}}$ is that these are smooth and CMC on $B\setminus C$. Note that necessarily we have $u_1=\tilde{u}_1$ and $u_q=\tilde{u}_{\tilde{q}}$.

Take any connected component $A$ of $B \setminus C$ on which  $\tilde{q} \geq 2$ (such a component exists otherwise $V$ is embedded everywhere). We will show that there exists $m\in \{1, ... \tilde{q}-1\}$ such that $\tilde{u}_1|_A= ... \tilde{u}_m|_A$ and $\tilde{u}_{m+1}|_A = \tilde{u}_{\tilde{q}}|_A$ ($m$ might depend on which connected component has been chosen). For $y \in A$ consider open balls centred at $x$ and contained in $A$ and pick the supremum $R$ of the radii of such balls: then $B_R(y) \subset A$ and there exists $p\in \p B_R(y) \cap C$. On the ball $B_R$, by the inductive assumptions, each $\tilde{u}_j$ is smooth and solves the CMC equation 
$$\text{div} \left( \frac{D \tilde{u}_j}{\sqrt{1+|D\tilde{u}_j|^2}} \right) =|h| \text { or } -|h|$$
classically, with the sign on the right-hand side depending on whether the mean curvature vector $\vec{H}$ points upwards or downwards, i.e. whether respectively $\vec{H} \cdot \hat{e}^{n+1} >0$ or $\vec{H} \cdot \hat{e}^{n+1} <0$ on $|\text{graph}(\tilde{u}_j|_{B_R(y)})|$. For any couple $\tilde{u}_i, \tilde{u}_j$ (with $i\neq j$) we then consider the possibility that the mean curvatures on these two graphs point in the same direction. Subtracting the PDE for $\tilde{u}_{j+1}$ from the one for $\tilde{u}_j$ we find that $\tilde{u}_j - \tilde{u}_{j+1}$ is $0$ on $C$ and solves an elliptic PDE on $B_R(y)$, namely (\ref{eq:PDEv}) with the right-hand side replaced by $0$. Unless $\tilde{u}_j \equiv \tilde{u}_{j+1}$ on $B_R(y)$ we can apply Hopf boundary point lemma, yielding that the derivative of $\tilde{u}_j - \tilde{u}_{j+1}$ along the outward normal to the boundary of the ball at $\xi$ must be strictly positive: this contradicts the vanishing of the gradient of $\tilde{u}_j - \tilde{u}_{j+1}$ on $C$ (tangentiality of the graphs, equivalently absence of classical singularities).

Summing up, we have by now that $\spt{V}$ in the cylinder $B_{1/2}^n(0)\times \R$ is represented by the union of the graphs of $u_1=\tilde{u}_1$ and $u_q=\tilde{u}_{\tilde{q}}$ (since every other $\tilde{u}_j$ coincides with one of these two on $B\setminus C$ and on $C$ clearly all graphs coincide). Moreover there exists a non-empty open set $A$ on which $\tilde{u}_1 < \tilde{u}_q$. We can conclude that actually $u_1 = u_q$ only on the set $C$; on $B_{1/2}^n(0) \setminus C$ we have $u_1 < u_q$. Indeed Lemma \ref{lem:constancysecond} forces $V=q_1(x)|\text{graph}(u_1)|+(q-q_1(x))|\text{graph}(u_q)|$, with $q_1$ integer-valued and locally constant on the embedded part. The set $\{u_1=u_q\}$ is therefore exactly the set $C$. \footnote{Actually we can even prove that $C$ has $0$-measure. All points above $C$ in the support of the varifold are either embedded or two-fold touching singularities. Embeddedness is an open condition, so above $\p C$ we must have touching singularities, and $\p C \neq \emptyset$. So let $p\in\p C$ and consider $(p,u_1(p))\in\text{sing}_T V$. In a neighbouhood of $p$ the set $C$ must have then $0$-measure: if that were not the case, then we would contradict assumption 3 of the class $\mathcal{S}_H$, in that we would produce on $C$ the presence of a twofold toucing singularity $(p,u_1(p))$ for which the zero-measure condition fails. Now consider the following: at every point of $C$ we either have (if the point is in the interior) an open neighbourhood that is all contained in $C$ or (if the point is on the boundary) an open neighbourhood in which $C$ has $0$-measure. This gives rise to an open cover of $C$ and we can extract a finite cover. Denote by $N$ the open set that is the union of all open neighbouhoods that lie in the interior, clearly $N \subsetneq C$. Then for every point $p\in \p N$ we have that $p$ must be contained in another set $L$ from the cover that was centred at the boundary of $C$, so $C$ should have $0$-measure in $L$. On the other hand $L \cap N$ is open and contained in the interior of $C$, contradiction.}

The previous argument settles part 1 of the lemma and moreover shows that, for any two sheets touching, the mean curvature vectors must point in opposite directions. However, if the top sheet $\tilde{u}_{j+1}$ has mean curvature vector pointing downwards and the bottom sheet $\tilde{u}_j$ has mean curvature pointing upwards, then the CMC equation is solved by $\tilde{u}_j$ with right hand side $|h|$ and by $\tilde{u}_{j+1}$ with right hand side $-|h|$. So the same argument yields that $\tilde{u}_j - \tilde{u}_{j+1}$ is a subsolution of the same elliptic PDE of which it was solution before, more precisely $\tilde{u}_j - \tilde{u}_{j+1}$ solves the PDE (\ref{eq:PDEv}) with the right-hand side replaced by $-2|h|\int \zeta$. Hopf lemma applies again, so part 2 is proved as well.
\end{proof}

\subsection{Average and semi-difference and their PDEs}

In order to complete the proof of Theorem \ref{thm:higher-reg} we need to prove $C^2$ regularity of each sheet across the touching points, i.e. we need to show that these touching points belong to $\Greg{V}$. Note that we only know that the set of points where sheets touch (tangentially) is a set of ${\Hc}^n$-measure zero, but such an estimate is too weak to allow any capacity argument for the extension of the CMC equation to each sheet. In other words, with our knowledge so far there could still be (for a single sheet) a singular part of the generalized mean curvature that is concentrated on the touching set (with the mean curvature of the other sheet that annihilates the first). In the following subsections we will rule out this possibility. By Lemma \ref{lem:Hopf} we only need to consider the case of two distinct sheets, so for the sequel we will work in the following situation.

\begin{hyp}
We consider two sheets $u_1 \leq u_2$, each of them is $C^{1,\alpha}(B^n_{1}(0))$ for $\alpha<\frac{1}{2}$, they can touch tangentially (the touching singularities $\text{sing}_T$ are contained in this touching set) and they cannot cross. Away from the touching set $\{(x,u_1(x)): u_1(x)=u_2(x)\}$ we have $C^{2}$-regularity and constant mean curvature in the classical sense. 
\end{hyp}

\begin{oss}
\label{oss:twosheetstouching}
By Lemma \ref{lem:Hopf} on any open set $B^n_\rho \times \R$ where there are no singular points, the two sheets have mean curvature vectors pointing in opposite directions, precisely the top sheet $u_2$ must have mean curvature pointing upwards and the bottom sheet $u_1$ must have mean curvature pointing downwards. The varifold $V$ is given by integration on the union of the two sheets endowed with an integer multiplicity. The multiplicity is $q \in \N$ on the touching set, it is denoted by $q_1(x)$ on $u_1$ (on the set where $u_1 < u_2$) and, by Lemma \ref{lem:constancysecond}, it is $q-q_1(x)$ on $u_2$ (on the set where $u_1 < u_2$). Moreover $q_1(x)$ is locally constant on the $C^2$ embedded portions.
\end{oss}

The two auxiliary functions that will play a key role in this section are the average $u_a:=\frac{u_1+u_2}{2}$ and the semi-difference $v:=\frac{u_2-u_1}{2}$ of the two sheets.

\medskip

\textbf{PDE for the semi-difference}
Consider $v=\frac{u_2-u_1}{2}$, well-defined on $B_1^n(0)$ and assume without loss of generality that $(0,0) \in B_1^n(0) \times \R$ is a touching singularity of the varifold. Then $v(0)=0$ and $Dv(0)=0$. The set 
\begin{equation}
T=\{x \in B_1^n : v(x)=0, Dv(x)=0\}
\end{equation}
is the projection on $B_1^n$ of $\{(x,u_1(x)): u_1(x)=u_2(x)\}$ (i.e. where the two sheets agree - we know that they must be tangential there).

We are assuming that on the embedded $C^2$ parts the two sheets are CMC with mean curvatures having the same absolute value $H>0$. Moreover for $u_1$ we have $\vec{H}=-H  \hat{\nu}$  and for $u_2$ we have $\vec{H}=H  \hat{\nu}$, with $ \hat{\nu} \cdot \vec{e}^{n+1} >0$ (i.e. normal vectors $ \hat{\nu}$ to the sheets pointing upwards).

Then consider the (possibly several) connected components of $B^n_1 \setminus T$, these are open sets above which $u_1$ and $u_2$ are $C^2$ and classically CMC. Then we can neglect multiplicities on the two sheets (since on such a connected component each sheet is counted with constant multiplicity) and deduce the first variation formulae

\begin{eqnarray}
\label{eq:fistvarformulae}
&&\int_{B_1^n} \frac{D_i u_1}{\sqrt{1 + |D u_1|^2}} D_i \zeta =H \int_{B_1^n} \zeta \nonumber\\
&&\int_{B_1^n} \frac{D_i u_2}{\sqrt{1 + |D u_2|^2}} D_i \zeta =-H \int_{B_1^n} \zeta 
\end{eqnarray}
for all test functions $\zeta$ compactly supported in $B^n_1 \setminus T$. Then we write

$$F^i(p)=\frac{p_i}{\sqrt{1+|p|^2}}\,\,\, \text{ for } p \in \R^n$$

$$F^i(Du_2) - F^i(D u_1) = \int_0^1 \frac{d}{dt}[F^i(t Du_2 + (1-t) Du_1)] dt=$$
$$=  \int_0^1 \frac{\p F^i}{\p p_j}(t Du_2 + (1-t) Du_1) (D_ju_2 - D_j u_1) dt =  \int_0^1 \frac{\p F^i}{\p p_j}(t Dv +  Du_1) (D_j u_2 - D_j u_1) dt = $$
$$= 2 D_j v \int_0^1 \frac{\p F^i}{\p p_j}(Du_a -(1-t) Dv) dt . $$
Set $a_{ij}(p,q) =  \int_0^1 \frac{\p F^i}{\p p_j}(p -(1-t) q) dt$ and, taking the difference of the two equations in (\ref{eq:fistvarformulae}), we find

$$\int_{B^n_1} a_{ij}(Du_a, Dv) D_j v D_i \zeta =  -H \int_{B^n_1} \zeta$$
for all test functions $\zeta$ compactly supported in $B^n_1 \setminus T$. 

Since $\frac{\p F^i}{\p p_j} = \frac{\delta_{ij}}{\sqrt{1+|p|^2}} - \frac{p_i p_j}{(1+|p|^2)^{\frac{3}{2}}}$, $a_{ij}(p,q) = \delta_{ij} + b_{ij}(p,q)$ with $|b_{ij}(p,q)| \leq C(|p|^2+|q|^2)$, which will allow us to view the PDE obtained for $v$, namely, 

\begin{equation}
 \label{eq:PDEv}
 \int_{B^n_1} (\delta_{ij}+b_{ij}(Du_a, Dv)) D_j v D_i \zeta =  -H \int_{B^n_1} \zeta
\end{equation}
as a perturbation of (the weak form of) $\Delta v = H$. Recall that we only use test functions $\zeta$ compactly supported in $B^n_1 \setminus T$. By assumption $T$ has zero $\mathcal{H}^n$-measure and $v=0$, $Dv=0$ on $T$. 

\medskip

\textbf{PDE for the average}. Consider $u_1$ with multiplicity $q_1(x)$ and $u_2$ with multiplicity $q_2(x)=q-q_1(x)$, with $q_j(x)$ constant on any open set disjoint from $T$. The first variations give

\begin{equation}
 \label{eq:firstvarformulaglobal}
\int \left( \frac{q_1(x) D_i u_1}{\sqrt{1+|Du_1|^2}} + \frac{q_2(x) D_i u_2}{\sqrt{1+|Du_2|^2}}\right) D_i \zeta = H \int (q_1(x)-q_2(x))\zeta.
\end{equation}
Writing $F^i(p) = \frac{p_i}{\sqrt{1+|p|^2}}$ for $p \in \R^n$ we have

$$q_2(x) F^i(D u_2) - q_1(x) F^i(-D u_1) = \int_0^1 \frac{d}{dt}[(tq_2(x) + (1-t)q_1(x))F^i(t Du_2 - (1-t)Du_1)]dt = $$

$$\int_0^1  (q_2(x)-q_1(x))F^i((2t-1)Du_a +Dv) dt + 2 \int_0^1 (q_1(x) + t(q_2(x) - q_1(x)) \frac{\p F^i}{\p p_j}((2t-1)Du_a +Dv) D_j u_a dt$$
Set $a_{ij}(p,q)=2 \int_0^1  (q_1(x) + t(q_2(x) - q_1(x)) \frac{\p F^i}{\p p_j}((2t-1)p +q)  dt$. Then the initial equation (\ref{eq:firstvarformulaglobal}) becomes

$$\int_\Omega \left(a_{ij}(Du_a, Dv)D_ju_a + (q_2(x)-q_1(x))\int_0^1 \frac{\delta_{ij}(2t-1)D_ju_a+D_iv}{\sqrt{1+|(2t-1)Du_a + Dv|^2}}dt\right)D_i\zeta = H\int_\Omega (q_2(x)-q_1(x))\zeta.$$
Setting $b_{ij}(p,q)=a_{ij}(p,q)+(q_2(x)-q_1(x))\delta_{ij}\int_0^1 \frac{(2t-1)}{\sqrt{1+|(2t-1)Du_a + Dv|^2}}dt$, we can rewrite the equation as follows:

$$\int_\Omega b_{ij}(Du_a, Dv)D_ju_a  D_i \zeta=$$ $$= H\int_\Omega (q_2(x)-q_1(x))\zeta -  \int_\Omega(q_2(x)-q_1(x))\left(\int_0^1 \frac{1}{\sqrt{1+|(2t-1)Du_a + Dv|^2}}dt\right) D_iv D_i\zeta$$

\subsection{Sheeting theorem - higher H\"older regularity for the gradients}
\label{higherholderexponent}

Next we need to improve exponent $\alpha$ of the $C^{1,\alpha}$ regularity, that was obtained in the Sheeting Theorem \ref{thm:sheeting} with $\alpha <\frac{1}{2}$:

\begin{Prop}
\label{Prop:alphaimproved}
Under the assumption that $V \in \mathcal{S}_H$ and 
$$V \res \left(B_{1/2}^{n}(0) \times {\R} \right) = \sum_{j=1}^{q} |{\rm graph} \, u_{j}|,$$
with $u_{j} \in C^{1, \alpha} \, (B_{1/2}^{n}(0); {\mathbb R})$ and $u_{1} \leq u_{2} \leq \ldots \leq u_{q}$ for some $\alpha (0, \frac{1}{2})$, we have that $u_{j} \in C^{1, \alpha} \, (B_{1/2}^{n}(0); {\mathbb R})$ for any $\alpha\in(0,1)$.
\end{Prop}

In particular we will need $\alpha\geq \frac{1}{2}$ at a later stage in this section. Note that it suffices, in view of Lemma \ref{lem:Hopf}, to focus on two sheets $u_1$ and $u_2$, for which we may assume the knowledge of $C^{1,\alpha}$ regularity for some $\alpha <\frac{1}{2}$: the aim is to conclude that then actually the H\"older exponent can be improved. Roughly speaking the strategy for the proof of Proposition \ref{Prop:alphaimproved} is to show a De Giorgi type decay for each separate sheet by obtaining such a decay for the semi-difference $v$ and for the average $u_a$ at all touching points. 

\medskip

Consider the semi-difference $v$ on $B_1=B^n_1(0)$ (without loss of generality we are considering graphs on $B_1(0)$, which can always be obtained by homothetic rescaling) and let $K:=\{v=0, Dv=0\}$ and assume without loss of generality that $0 \in K$. We saw in (\ref{eq:PDEv}) that $v$ satisfies a PDE of the form

\begin{equation}
\label{eq:PDEsemidifferenceforschauder}
 \int_{B_1} \left( \delta_{ij}+b_{ij}(x) \right) D_j v D_i \zeta = -H \int_{B_1} \zeta \,\,\,\,\,\, \text{ for } \zeta \in C^1_c(B_1 \setminus K)\,\,\,\,\,\,\, \text{ with } H>0.
\end{equation}
In the case of (\ref{eq:PDEv}), the coefficients $b_{ij}(x)$ are, more precisely, the $C^{0,\alpha}$ functions $b_{ij}(Du_a(x), Dv(x))$.
In order to produce a De Giorgi type decay for $v$ we will use a contradiction argument/blow up method for which the following a priori estimate is needed.

\begin{Prop}[Schauder estimates for the semi-difference]
 \label{Prop:Schauderforsemidifference}
There exist $\beta, H_0 >0$ such that if $v \in C^{1, \alpha}(B_1)$ satisfies the PDE (\ref{eq:PDEsemidifferenceforschauder}) on $B_1 \setminus K$ with $\sup_{B_1}|b_{ij}|+[b_{ij}]_{\alpha,B_1}\leq \beta$ and $H \leq H_0$ then

$$[Dv]_{\alpha, B_{1/2}} \leq C(\|v\|_{L^2(B_1)} + H),$$
where $C$ depends only on $n, H_0, \beta$.

\end{Prop}

The key difficulty here is that the PDE is satisfied only away from the closed set $K$, of which no regularity properties are known. To prove Proposition \ref{Prop:Schauderforsemidifference} we adapt the scaling argument due to L.~Simon (\cite{Simon3}). 

\begin{lem}
\label{lem:SimonslemmaforSchauder} 
  $\forall \delta>0 \, \exists\,  C>0, C=C(\delta,H_0,\beta,\delta)$ such that 

$$[Dv]_{\alpha,B_{1/2}} \leq \delta [Dv]_{\alpha, B_1} + C(|v|_{0, B_1} + |Dv|_{0, B_1} +H). $$

\end{lem}
 
\begin{proof}[proof of Lemma \ref{lem:SimonslemmaforSchauder}]
 
The proof proceeds by contradiction: assume that we can find $\delta>0$ such that for every $k \in \N$ there is a $v_k$ satisfying the same assumptions and such that the reverse inequality holds, i.e.~for every $k$ we have (with $0\leq H_k\leq H$)

\be
\label{eq:SimonReverseIneq}
[Dv_k]_{\alpha,B_{1/2}} \geq \delta [Dv_k]_{\alpha, B_1} + k(|v_k|_0 + |Dv_k|_0 +H_k). 
\ee
Choose $x_k, y_k \in B_{1/2}$ such that 

\be
\label{eq:choicexkyk}
\frac{|Dv_k(x_k)-Dv_k(y_k)|}{|x_k - y_k|^\alpha} > \frac{1}{2}[Dv_k]_{\alpha, B_{1/2}}. 
\ee
Denote $\rho_k = |x_k - y_k|$. Observe that in view of (\ref{eq:choicexkyk}) and (\ref{eq:SimonReverseIneq}) we have $\frac{1}{2}[Dv_k]_{\alpha, B_{1/2}} \leq \frac{2|Dv_k|_0}{\rho_k^\alpha} \leq \frac{2[Dv_k]_{\alpha, B_{1/2}}}{k \rho_k^\alpha}$. Therefore we must have $\rho_k \to 0$ as $k \to \infty$.  

\medskip

We distinguish two cases: (I) $B_{\rho_k}(x_k) \subset B_1 \setminus K_k$ for all $k$, (II) there exists a subsequence (not relabelled) such that there is a point $z_k$ in the set $K_k \cap B_{\rho_k}(x_k)$.

If we are in case (I) consider the rescaling of $v_k$ as follows: 
$$w_k(x):= \frac{v_k(x_k + \rho_k x)-v_k(x_k)-\rho_k \sum_{i=1}^n D_iv_k(x_k) x_i}{\rho_k^{1+\alpha}[Dv_k]_{\alpha, B_1}} .$$

If we are in case (II) consider the rescaling of $v_k$ as follows: 
$$w_k(x):= \frac{v_k(x_k + \rho_k x)}{\rho_k^{1+\alpha}[Dv_k]_{\alpha, B_1}} $$
and denote $p_k = \frac{z_k - x_k}{\rho_k}$. Then we have $w_k = 0$ and $Dw_k = 0$ at the point $p_k \in B_1(0)$ (the point $p_k$ plays the role of an ``anchor point'' so there is no need to subtract the first jet, which we must do in case (I) in order to make the origin the ``anchor point'').

\medskip

In either case we consider the sequence $w_k$ thus obtained and we will now prove that it converges (up to extraction of a subsequence) in $C^1$ on any compact set to a $C^{1,\alpha}$ function $w$ defined on $\R^n$ and such that $w$ is harmonic on the open set $\{w \neq 0\}$. Then we will show that such a $w$ is harmonic in the whole of $\R^n$.

\medskip

Case (I): by definition we have $w_k(0)=0$, $Dw_k(0)=0$ and let $\zeta_k:=\frac{y_k-x_k}{\rho_k}$ so that $|\zeta_k|=1$. By (\ref{eq:choicexkyk}) and (\ref{eq:SimonReverseIneq})

$$|Dw_k(\zeta_k) - Dw_k(0)| > \frac{\delta}{2},$$
and moreover we know that $[Dw_k]_{\alpha, B_{\frac{1}{2\rho_k}}(0)}\leq 1$ by rescaling properties.  Since $Dw_k(0)=0$ then for $0<\sigma<\frac{1}{2\rho_k}$ and $x \in B_\sigma(0)$ we have $|Dw_k(x)| = |Dw_k(x)-Dw_k(0)| \leq [Dw_k]_{\alpha, B_{\frac{1}{2\rho_k}}(0)} |x-0|^\alpha \leq \sigma^\alpha$. Similarly $|w_k(x)| \leq \sigma^{1+\alpha}$.

We now want to send $k \to \infty$. We have just seen that on any compact set a tail of the sequence satisfies the requirements of Ascoli-Arzel\`{a}'s theorem: this yields a $C^{1,\alpha}$ function $w$ on $\R^n$ to which a subsequence of $w_k$ converges in $C^1$ on any compact set (a diagonal argument is needed here, by using an exhausting sequence of compact sets $\overline{B}_{\frac{1}{2\rho_k}}$). We may assume, up to extracting a further subsequence that we do not relabel, that $\zeta_k \to \zeta$ as $k\to \infty$ with $|\zeta|=1$. The function $w$ satisfies (by the $C^1$ convergence)

\be
\label{eq:inequalitiesforlimitingblowup}
[Dw]_{\alpha, \R^n} \leq 1 \,\,\, \text{ and } \,\, |Dw(\zeta) - Dw(0)| > \frac{\delta}{2}.
\ee

\medskip

Case (II): this time, with the notation $p_k = \frac{z_k - x_k}{\rho_k}$, we have $w_k(p_k)=0$, $Dw_k(p_k)=0$. We have still $\zeta_k:=\frac{y_k-x_k}{\rho_k}$ so that $|\zeta_k|=1$. By (\ref{eq:choicexkyk}) and (\ref{eq:SimonReverseIneq})

$$|Dw_k(\zeta_k) - Dw_k(0)| > \frac{\delta}{2},$$
and moreover we know that $[Dw_k]_{\alpha, B_{\frac{1}{2\rho_k}}(0)}\leq 1$ by rescaling properties. For $0<\sigma<\frac{1}{2\rho_k}$ and $x \in B_\sigma(0)$ we have $|Dw_k(x)| = |Dw_k(x)-Dw_k(p_k)| \leq [Dw_k]_{\alpha, B_{\frac{1}{2\rho_k}}(0)} |x-p_k|^\alpha \leq (\sigma+1)^\alpha$. Similarly $|w_k(x)| \leq (\sigma+1)^{1+\alpha}$. As before we can extract a converging subsequence using Ascoli-Arzel\`{a}'s theorem and get as above $w$ of class $C^{1,\alpha}$ on $\R^n$ with (\ref{eq:inequalitiesforlimitingblowup}) valid.

\medskip

Since $v_k \geq 0$ for all $k$, in case (II) we will have $w \geq 0$ and the set $\{w=0\}$ is the set of $z$ such that there exists a sequence $z_k \to z$ with $x_k+\rho_k z_k \in K_k$. We will consider the open set $\{w > 0\}$ and will show that $w$ is harmonic there. 

In case (I), on the other hand, let $Z_k$ be the closed set such that $x_k+\rho_k Z_k = K_k$; then $w_k$ takes on the set $Z_k$ the same value as the affine function $\frac{-v_k(x_k)-\rho_k \sum_{i=1}^n D_iv_k(x_k) \cdot_i}{\rho_k^{1+\alpha}[Dv_k]_{\alpha, B_1}}$; here $\cdot$ denotes the variable and $\cdot_i$ its $i$-th coordinate. Therefore the limsup $Z$ (as $k \to \infty$) of the sets $Z_k$ will be such that the value of $w$ on $Z$ coincides with the value taken by a certain affine function, which is the $C^1$-limit of those exhibited above. We will consider the open set $\R^n \setminus Z$ and will conclude that $w$ is harmonic there. 

\medskip

Observe that both in case (I) and case (II) we can obtain a PDE for $w$ on the open set $\R^n \setminus Z$ by suitable rescaling the (weak) PDEs for $v_k$ and passing to the limit by the $C^1$ convergence: indeed the open set on which we are focusing comes from dilations of the open sets on which the PDEs for $v_k$ are valid. The limiting process follows the lines of the blow up argument in L. Simon's \cite{Simon3} and will be omitted.

\medskip

We now have $w$ that is harmonic away from a closed set where it coincides with an affine function. By subtracting this affine function we obtain both in cases (I) and (II) that we have a function $w \in C^{1,\alpha}$ that is harmonic wherever it is non-zero. Then $w$ is harmonic in the whole of $\R^n$. To see this we show first of all that $|D^2 w|$ (which is well-defined on $\R^n \setminus Z$ since $w$ is smooth there) is locally $L^2$-summable on $\R^n \setminus Z$.

Consider a convex smooth and monotone increasing function $\gamma_{\delta}:[0, \infty[ \to \R$, for $\delta>0$, such that 
$$\gamma_{\delta}(t) = \left\{\begin{array}{ccc}
                                                   0 & \text{ for } 0 \leq t \leq  \delta  \\
                                                   t-\delta & \text{ for } t \geq 2 \delta
                                                    \end{array} ,\right.$$
and $\gamma_\delta$ is positive increasing. Noting that $w$ is smooth wherever $\gamma_{\delta}(|\nabla w|^2)  \neq 0$ we can compute

$$\Delta \left(\gamma_{\delta} (|\nabla w|^2)  \right) = \text{div}\left( \gamma_{\delta}^\prime  (|\nabla w|^2) \nabla(|\nabla w|^2) \right) =$$
$$=  \gamma_{\delta}^{\prime \prime} (|\nabla w|^2)  \left|\nabla(|\nabla w|^2)\right|^2 + \gamma_{\delta}^\prime  (|\nabla w|^2) \Delta(|\nabla w|^2))$$
(keeping in mind that the higher derivatives of $w$ have only appeared where they are classically well-defined). Using the convexity assumption $\gamma_{\delta}^{\prime \prime} \geq 0$ and Bochner's formula for harmonic functions we conclude that

$$\Delta \left(\gamma_{\delta} (|\nabla w|^2)  \right) \geq 2 \gamma_{\delta}^\prime  (|\nabla w|^2) |D^2_{ij} w|^2.$$
Let $\psi$ be any non-negative bump function that is compactly supported in $B_R \subset \R^n$ and identically $1$ in $Z \cap B_{R-1}$ ($R$ large). Then the previous inequality yields 

$$2 \int_{|Dw|\geq \sqrt{2\delta}} |D^2 w|^2 \psi \leq  2 \int_{\R^n} \gamma_{\delta}^\prime  (|\nabla w|^2)  |D^2 w|^2 \psi \leq \int_{\R^n} \Delta\left(\gamma_{\delta} (|\nabla w|^2)\right)\psi \leq $$ $$\leq \int_{\R^n} \gamma_{\delta} (|\nabla w|^2) \Delta \psi \leq C_R \|Dw\|_{C^0}\int \Delta \psi ,$$
and sending $\delta \to 0$ (since the r.h.s. does not depend on $\delta$) we have the desired local summability on the open set $\{Dw \neq 0\}$. At points in the set $\{Dw = 0\} \setminus Z$ the function $w$ is $C^2$ (and harmonic), in particular $D^2 w=0$ almost everywhere on this set. In view of this we conclude the $L^2$-summability of $D^2 w$ on any domain $B_R \cap (\R^n \setminus Z)$.  We note that actually we would need only $L^1$ summability for the next step.

\medskip

Now with $f=D_i w$ for $i\in \{1, ...n\}$ we can use Lemma \ref{lem:distributionalder} to conclude that the function that takes the value $D_l f$ on $\R^n\setminus \{f=0\}$ and the value $0$ on $\{f=0\}$ is in $L^2_{\text{loc}}(\R^n)$ and is the distributional $l$-derivative of $f=D_i v$ in $\R^n$. In other words we have obtained that $w \in W^{2,2}_{\text{loc}}$ and so $\Delta w =0$ on the whole of $\R^n$: so $w$ is everywhere harmonic (recall that adding back the affine function in case (I) does not change the harmonicity of the blow up function $w$).

\medskip

To conclude the proof of Lemma \ref{lem:SimonslemmaforSchauder} we only need to recall that we have obtained a harmonic function $w$ on $\R^n$ satisfying (\ref{eq:inequalitiesforlimitingblowup}). Then $D_l w$ is harmonic (for any $l$) and with sublinear growth (first inequality in (\ref{eq:inequalitiesforlimitingblowup})), which implies by Liouville theorem that $D_l w$ is constant. But then we contradict the second inequality in (\ref{eq:inequalitiesforlimitingblowup}).
\end{proof}

\begin{proof}[proof of Proposition \ref{Prop:Schauderforsemidifference}]
 The estimate follows as in \cite{Simon3} from Lemma \ref{lem:SimonslemmaforSchauder} by means of the ``simple abstract lemma'' \cite[p. 398]{Simon3}.
\end{proof}

Now we are in a position to prove the De Giorgi-type decay for the ``modified $L^2$-excess'' of $v$.

\begin{Prop}[De Giorgi decay for the semi-difference]
\label{Prop:DeGiorgitypedecayforthesemidifference}
Let $v \in C^{1,\alpha}(B_1)$ (with $0<\alpha<\frac{1}{2}$) solve the PDE (\ref{eq:PDEsemidifferenceforschauder}) on $B_1 \setminus K$. Assume that $\sup_{B_1}|b_{ij}|+[b_{ij}]_{\alpha,B_1}\leq \beta$ and $H \leq H_0$, with $\beta, H_0$ as in Proposition \ref{Prop:Schauderforsemidifference}. Let $\mu \in (0,1)$. For every $\theta \in (0, \frac{1}{4})$ there exists $C,\eps>0$ (depending on $\theta$) such that if $\int_{B_1}|v|^2 + \frac{H^2}{\eps}\leq \eps$ then

\be
\label{eq:DeGiorgiDecayforsemidifference}
\frac{1}{\theta^{n+2}}\int_{B_\theta} |v|^2  + \frac{\theta^{2}}{\eps} H^2 \leq C \theta^{2\mu} \left(\int_{B_1} |v|^2 + \frac{H^2}{\eps}\right).
\ee
\end{Prop}

\begin{proof}[proof of Proposition \ref{Prop:DeGiorgitypedecayforthesemidifference}]
We consider a sequence of $v_k$ satisfying the PDE with $H_k$ on the r.h.s. and such that $\int_{B_1}|v_k|^2 + \frac{H_k^2}{\eps_k} \leq \eps_k \to 0$. Consider the blow-up obtained by dividing by the $L^2$-height excess:

$$\tilde{v}_k := \frac{v_k}{\left(\int_{B_1}|v_k|^2 + \frac{H_k^2}{\eps_k}\right)^{\frac{1}{2}}}$$
and remark that $\|\tilde{v}_k\|_{C^{1,\alpha}(B_{1/2})}$ is bounded by $C$ independently of $k$ thanks to the a priori Schauder type estimate in Proposition \ref{Prop:Schauderforsemidifference}. This allows to find $\tilde{v} \in C^{1,\alpha}(B_{1/2})$ to which a subsequence of $\tilde{v_k}$ converges in $C^1(B_{1/2})$. 
Moreover $\|\tilde{v}_k\|_{L^2(B_{1/2})}$ are equibounded $\leq 1$ and for any $R<1/2$ the norms $\|\tilde{v}_k\|_{W^{1,2}(B_{R})}$ are equibounded, implying the strong $L^2$-convergence on every compact set in $B_{1/2}$. Then $\|\tilde{v}\|_{L^2(K)} \leq 1$ independently of $K$ on any compact set $K$ contained in $B_{1/2}$. So $\tilde{v} \in L^2(B_{1/2})$. Moreover $\tilde{v}$ is harmonic on $B_{1/2}$ (by the rescaling properties of the PDE the function $\tilde{v}_k$ solves the PDE (\ref{eq:PDEv}) with $ H_k \leq \eps_k$ on the r.h.s. and we have $C^1$ convergence so we can pass the PDE to the limit) and by $C^1$ convergence we have that $D\tilde{v} (0)=0$ (since $Dv_k(0)=0$ for every $k$ as $0 \in T$). Harmonic function theory and the vanishing of the gradient at $0$ guarantee the validity of the inequality

$$\frac{1}{\theta^{n+2}}\int_{B_\theta} |\tilde{v}|^2  \leq C \theta^{2\mu} \int_{B_{1/2}} |\tilde{v}|^2 , \,\,\, \text{ with } C=C_n.$$
By strong $L^2$-convergence on $B_{\theta}$ we have, for $k$ large enough (we need to ensure the validity of $\int_{B_\theta} |\tilde{v} - \tilde{v}_k|^2 <  \theta^{n+2+2\mu}$) the inequality

$$\frac{1}{\theta^{n+2}}\int_{B_\theta} |\tilde{v_k}|^2  \leq  \theta^{2\mu} \left(C_n \int_{B_{1/2}} |\tilde{v}|^2+1\right),$$
in other words (recall that $\int_{B_{1/2}}|\tilde{v}|^2 \leq 1$), by definition of $\tilde{v}_k$,
$$\frac{1}{\theta^{n+2}}\int_{B_\theta} |v_k|^2  \leq  (C_n+1) \theta^{2\mu} \left(\int_{B_1} |v_k|^2 + \frac{H_k^2}{\eps_k} \right).$$
The desired inequality (\ref{eq:DeGiorgiDecayforsemidifference}) follows now immediately, since the second term on the l.h.s. of (\ref{eq:DeGiorgiDecayforsemidifference}) is easily bounded.
\end{proof}

The decay result obtained for (\ref{eq:PDEsemidifferenceforschauder}) immediately applies to the PDE (\ref{eq:PDEv}) for the semi-difference, upon ensuring the bounds for $b_{ij}$ and $H$ by performing suitable homothetic dilations.

We will now focus on the average of the two sheets and prove a similar decay property. Rather than a PDE for the average, however, we will seek a PDE for the weighted average $q_1 u_1 + q_2 u_2$. Comparing with the case of the semi-difference, we wish to point out that this time we will have no control over the values of the weighted average at the points corresponding to the ``touching set'', nor on the gradient; on the plus side, however, we will have a PDE that is valid across the touching set. 

The PDE is obtained from the first variation formula (\ref{eq:firstvarformulaglobal}): in a first step we have

$$\int \left( \underbrace{\frac{q_1(x) D_iu_1}{\sqrt{1+|Du_1|^2}}+\frac{q_2(x) D_iu_2}{\sqrt{1+|Du_2|^2}}}\right) D_i \zeta = H \int (q_1(x)-q_2(x)) \zeta ,$$
and we can rewrite the braced expression as
$$q_1(x) D_i u_1 \left(1+ \left(\frac{1}{\sqrt{1+|Du_1|^2}}-1\right)\right) + q_2(x) D_i u_2 \left(1+ \left(\frac{1}{\sqrt{1+|Du_2|^2}}-1\right)\right) $$
so that the equation becomes

$$\int \left( q_1(x) D_iu_1 + q_2(x) D_iu_2 \right) D_i \zeta = H \int (q_1(x)-q_2(x)) \zeta \,+ \, \int O(|Du_1|^3)D_i\zeta \, + \,\int O(|Du_2|^3)D_i\zeta .$$
Rewriting $ q_1(x) D_i u_1 = D_i( q_1(x)  u_1) - u_1 D_i q_1 $ (and similarly for $u_2$) we get
$$\int  D_i(q_1(x) u_1 + q_2(x)  u_2 ) D_i \zeta = H \int (q_1(x)-q_2(x)) \zeta \,+ \, \int O(|Du_1|^3)D_i\zeta \, + \,\int O(|Du_2|^3)D_i\zeta \, + $$
$$+ \, \int (u_1 D_i q_1 + u_2 D_i q_2) D_i \zeta. $$
The last term is however zero, since on the support of the distributions $D_i q_1 , D_i q_2$ (the touching set of the two graphs, see remark \ref{oss:twosheetstouching}) we have $u_1=u_2$ and $D_i q_1 = -D_i q_2$. This gives the PDE for the weighted average $q_1 u_1 + q_2 u_2$:

\begin{equation}
 \label{eq:PDEweightedaverage}
\int  D_i(q_1 u_1 + q_2 u_2 ) D_i \zeta = H \int (q_1 -q_2) \zeta \,+ \, \int O(|Du_1|^3)D_i\zeta \, + \,\int O(|Du_2|^3)D_i\zeta . 
\end{equation}

\medskip

\textit{Choice of excess for the weighted average}. We let, for a reference affine function $L$ and $\rho \in (0,1]$ (here $\eps>0$ is a constant to be chosen in Proposition \ref{Prop:DeGiorgitypedecayfortheweightedaverage} below):
$$E^2_L(\rho):= \rho^{-n-2}\int_{B_{\rho}} |q_1 u_1 + q_2 u_2 -L|^2 + \rho^2 \frac{H^2}{\eps} + \frac{\left(\rho^{3\alpha}\right)^2}{\eps} \left([D u_1]_{\alpha, B_\rho}^3\right)^2+ \frac{\left(\rho^{3\alpha}\right)^2}{\eps} \left([D u_2]_{\alpha, B_\rho}^3\right)^2.$$

\noindent We will prove the following decay at points where $u_1 = u_2$, assuming that $0$ is such a point and that \textit{we have rotated coordinates so that $Du_1 (0) = D u_2 (0) =0$} and $q$ is fixed.

\begin{Prop}[De Giorgi decay for the weighted average]
 \label{Prop:DeGiorgitypedecayfortheweightedaverage}
Let $\mu \in (\frac{1}{2}, 1)$ and $\theta \in (0, \frac{1}{4})$. Assume that $Du_1(0)=Du_2(0)=0$. There exists $\eps={\eps}(n,\mu,\theta)$ and $C=C(n,\mu)$ such that, if $E^2_L(1)\leq \eps$ for a certain $L$, then there exists $L^\prime$ for which 
\be
\label{eq:DeGiorgitypedecayfortheweightedaverage}
E^2_{L^\prime}(\theta) \leq C \theta^{2\mu} E_L^2(1),
\ee
\be
\label{eq:DeGiorgitypedecayfortheweightedaveragetilting}
|L - L'|_{C^1(B_1)}^2\leq C E^2_L(1).
\ee
\end{Prop}

\noindent \textit{Notational remark}: the excess $E^2_L(\rho)$ computed for the function $u_k$ will be denoted in the following proof by $E^2_L(\rho,k)$.

\begin{proof}
Note that the last two terms in the definition of the excess $E$ are defined so that they are naturally rescaling under the geometric transformation $\frac{u(rx)}{r}$ when evaluating the H\"older seminorm at $0$. 

Elliptic $L^2$-theory for a PDE as (\ref{eq:PDEweightedaverage}) of the form $\text{div}(\nabla w) = f + \text{div}g$ with $f \in L^\infty(B_1)$ and $g \in L^2(B_1(0))$ give $\|w\|_{W^{1,2}(B_{1/2})} \leq C \left( \|w\|_{L^2(B_1)}+ \|f\|_{L^\infty(B_1)} + \|g\|_{L^2(B_1)}\right)$, see e.g. \cite[Chapter 8]{GT}. Note that we can subtract a linear function from $q_1 u_1 + q_2 u_2$ and find (since a linear function is harmonic)

$$\int  D_i(q_1 u_1 + q_2 u_2 -L) D_i \zeta = H \int (q_1 -q_2) \zeta \,+ \, \int O(|Du_1|^3)D_i\zeta \, + \,\int O(|Du_2|^3)D_i\zeta . $$
For this PDE the elliptic estimate reads

$$\|q_1 u_1 + q_2 u_2-L\|_{W^{1,2}(B_{1/2})} \leq C \left(\|q_1 u_1 + q_2 u_2-L\|_{L^2(B_1)}+ \right.$$ $$\left. +Hq + \left(\int_{B_1}O\left(|Du_1|^3\right)^2\right)^{1/2}  + \left(\int_{B_1}O\left(|Du_2|^3\right)^2\right)^{1/2}\right)$$
Using $Du_1(0)=0$ and $|Du_1|(x)=|Du_1(x)-Du_1(0)| \leq [Du_1]_{\alpha, B_1}|x|^\alpha$ we can continue 
$$\leq C\left(\|q_1 u_1 + q_2 u_2-L\|_{L^2(B_1)}+ Hq + \left(\int_{B_1}\left([Du_1]_{\alpha,B_1}^3\right)^2\right)^{1/2}+ \left(\int_{B_1}\left([Du_2]_{\alpha,B_1}^3\right)^2\right)^{1/2}\right) \leq $$
$$\leq C\left(\|q_1 u_1 + q_2 u_2-L\|_{L^2(B_1)}+ H +  [Du_1]_{\alpha,B_1}^3 + [Du_2]_{\alpha,B_1}^3\right) \leq K E_L(1). $$

We can thus perform a blow up argument, assuming to have a sequence $\eps_k \to 0$ and $u_1^k, u_2^k$ with $E_{L}^2(1,k) \leq \eps_k$. Define

$$\tilde{u}_k = \frac{q_1 u^k_1 + q_2 u^k_2 -L}{E_{L}(1,k)}.$$
The elliptic estimate just discussed gives that $\|\tilde{u}_k\|_{W^{1,2}(B_{1/2})} \leq K$ and therefore $\tilde{u}_k \to \tilde{u} \in W^{1,2}(B_{1/2})$ strongly in $L^2(B_{1/2})$ and weakly in $W^{1,2}(B_{1/2})$. The blow up function $\tilde{u}$ is harmonic on $B_{1/2}$. Indeed the PDE (\ref{eq:PDEweightedaverage}) for $\tilde{u}_k$ is of the form

$$\int  D_i\tilde{u}_k D_i \zeta = \int f_k  \zeta \,+ \, \int g_k D_i\zeta $$
with $|f_k| \leq \eps_k$ and $|g_k| \leq \eps_k$, for any $\zeta \in C^{\infty}_c(B_{1/2})$. Passing to the limit thanks to the $W^{1,2}$-weak convergence on the left-hand-side we get 

$$\int_{B_{1/2}}  D_i\tilde{u} D_i \zeta = 0, $$
as desired. By the harmonicity of $\tilde{u}$ we get an affine function $\tilde{L}'$ (tangent to the graph of $\tilde{u}$ above the origin) and a constant $\tilde{C}=C(n,\mu)$ such that

$$ \theta^{-n-2}\int_{B_{\theta}} |\tilde{u} -\tilde{L}^\prime|^2  \leq \tilde{C} \theta^{2\mu} \int_{B_{1/2}} |\tilde{u}|^2.$$
Recalling the definition of $\tilde{u}_k$ and by the $L^2$-strong convergence to $\tilde{u}$ we can see that for $k$ large enough it holds

$$\theta^{-n-2}\int_{B_{\theta}} |q_1 u^k_1 + q_2 u^k_2 -L - E_{L}(1,k)\tilde{L}^\prime|^2 \leq \left(\tilde{C} \int_{B_{1/2}} |\tilde{u}|^2 + 1\right) \theta^{2\mu} E^2_{L}(1,k),$$
and we can set $L'$ to be the affine function $L + E_{L}(1,k)\tilde{L}^\prime$ and recall that $\int_{B_{1/2}} |\tilde{u}|^2 \leq K$ (the constant in the elliptic estimate): this bounds the first term in the excess that appears on the left-hand-side of (\ref{eq:DeGiorgitypedecayfortheweightedaverage}). The remaining three terms of the excess on the left-hand-side of (\ref{eq:DeGiorgitypedecayfortheweightedaverage}) are easily bounded by using $\theta \leq 1$ and $[Du_j]_{\alpha, B_\theta} \leq [Du_j]_{\alpha, B_1}$ (for $j=1,2$).

Note that $\nabla \tilde{u}$ is harmonic and so by the mean vaue theorem $|\nabla\tilde{L}^\prime |=|\nabla{\tilde{u}}(0)|$ and $|\tilde{u}(0)|=|\tilde{L}'(0)| \leq \|\tilde{u}\|_{L^2(B_{1/2})}\leq K$, so $|\tilde{L}^\prime |_{C^1(B_1)} \leq \|\tilde{u}\|_{W^{1,2}(B_{1/2})}\leq K$ (the constant in the elliptic estimate). We therefore have also obtained the desired control (\ref{eq:DeGiorgitypedecayfortheweightedaveragetilting}) over the tilting of the plane.
\end{proof}

\begin{oss}[Iteration step]
The purpose of the decay lemma (Proposition \ref{Prop:DeGiorgitypedecayfortheweightedaverage}) is to perform an iterative procedure. Start with the given $u_1$ and $u_2$ with coordinates set such that $u_1=u_2=Du_1=Du_2=0$ at the origin. If the varifold was dilated enough we can also ensure smallness of $H$ and of $[Du_j]_{\alpha, B_1}$ in order to have $E^2_L(1)\leq \eps$ (at this initial step the choice of $L=0$ will do).
Now consider the homotetic dilation of a factor $\theta$ (which dilates the ball $B^{n+1}_\theta(0)$ to the unit ball $B^{n+1}_1(0)$). The corresponding tranformations for the two graphs are $\tilde{u}_1(x)=\frac{u_1(\theta x)}{\theta}$ and $\tilde{u}_2(x)=\frac{u_2(\theta x)}{\theta}$. 

The term $\theta^{-n-2}\int_{B_{\theta}} f^2$ is designed to be scale invariant, i.e. setting $\tilde{f}(x)=\frac{f(\theta x)}{\theta}$ we have by a change of variable $\theta x = y$ the equality $ \int_{B_1}\tilde{f}^2 =\int_{B_1} \frac{f(\theta x)^2}{\theta^2} = \frac{1}{\theta^{n+2}}\int_{B_\theta} f^2$. Therefore for $q_1\tilde{u}_1+q_2\tilde{u}_2$ we find
 
$$\int_{B_1} |q_1\tilde{u}_1+q_2\tilde{u}_2 - \tilde{L}^\prime|^2 = \theta^{-n-2}\int_{B_{\theta}} |q_1 u^k_1 + q_2 u^k_2 -L^\prime|^2 ,$$
where $\tilde{L}^\prime$ is the affine function obtained by dilating $L^\prime$.

Consider the PDE (\ref{eq:PDEweightedaverage}) written for $q_1\tilde{u}_1+q_2\tilde{u}_2$, with $\tilde{u}_1$, $\tilde{u}_2$, $\tilde{H}$ and note that $\tilde{H}=\theta H$ by rescaling properties. Therefore the second term $\frac{1}{\eps}\theta^2 H^2$ that appears (for $q_1 u_1 + q_2 u_2$) in $E^2(\theta)$ is equal to $\frac{1}{\eps} \tilde{H}^2$, i.e. the second term that appears in $E^2(1)$ for $q_1\tilde{u}_1+q_2\tilde{u}_2$.

The dilation preserves the fact that $\tilde{u}_1=\tilde{u}_2=D\tilde{u}_1=D\tilde{u}_2=0$ at the origin. Moreover, since $D\tilde{u}_1(x) = Du_1(\theta x)$ we have

$$[D \tilde{u}_1]_{\alpha,B_1} = \sup_{x,y \in B_1} \frac{|D\tilde{u}_1(x) - D\tilde{u}_1(y)|}{|x-y|^\alpha} =   \theta^\alpha \sup_{x,y \in B_1} \frac{|Du_1(\theta x) - D{u}_1(\theta y)|}{|\theta x- \theta y|^\alpha} = \theta^\alpha [D u_1]_{\alpha,B_\theta}$$
and therefore (for $j\in\{1,2\}$) the term $\theta^{3\alpha}[Du_j]^3_{\alpha, B_\theta}$ that appears (for $q_1 {u}_1+q_2 {u}_2$) in $E^2(\theta)$ is equal to the term $[D\tilde{u}_j]^3_{\alpha, B_1}$ that appears, for $q_1\tilde{u}_1+q_2\tilde{u}_2$, in $E^2(1)$.

In other words we have seen that the excess (for $q_1\tilde{u}_1+q_2\tilde{u}_2$) $E^2_{\tilde{L}^\prime}(1)$ is exactly equal to the excess $E^2_{L^\prime}(\theta)$ of $q_1 u_1 + q_2 u_2$ at scale $\theta$ with respect to $L^\prime$ and so the decay (\ref{eq:DeGiorgitypedecayfortheweightedaverage}) ensures that the smallness condition of the excess at scale $1$ is still satisfied by $q_1\tilde{u}_1+q_2\tilde{u}_2$, upon choosing $\theta$ small enough so that $C \theta^{2\mu} <1$, which can be done since $C$ does not depend on $\theta$. We are thus able to apply Proposition \ref{Prop:DeGiorgitypedecayfortheweightedaverage} to $\tilde{u}_1, \tilde{u}_2$ and iterate. Rescaling back to the original picture this iteration gives, for $d \in \N$, an affine function $L^{(d)}$ such that

$$E^2_{L^{(d)}}(\theta^d) \leq C^d \theta^{2 \mu d} E^2_L(1).$$
Moreover the second line in (\ref{eq:DeGiorgitypedecayfortheweightedaveragetilting}) gives the convergence of $L^{(d)}$ to a certain affine function $L_\infty$ (both the gradients and the translations, in the original picture, converge by the geometric control provided by the inequality). Note that $C$ is independent of $\theta$.

Choosing $0<\delta<\mu$ ($\delta$ much smaller than $\mu$) and $\theta$ small enough\footnote{For any chosen $\delta$ such that $0<\delta<\mu$ we can get $d_0 \in \N$ such that $2\delta d_0 >n+2$: then there exists $\theta$ small enough such that $\theta^{2\delta d_0-n-2}\leq \frac{1}{4C^{d_0}}$. On the other hand $\theta$ is also small enough so that $C \theta^{2\mu}<1$, therefore $\theta^{2\delta d-n-2}\leq \frac{1}{4C^{d}}$ holds for all $d \geq d_0$. For $\rho<\theta^{d_0}=R_\delta$ we have that there exists $d\geq d_0$ such that $\theta^{d+1}\leq \rho\leq \theta^d$. By the decay inequality $$\frac{1}{\rho^{n+2}}\int_{B_\rho}|f-L^{(d)}|^2 \leq \frac{1}{\theta^{(d+1)(n+2)}}\int_{B_{\theta^d}}|f-L^{(d)}|^2 = \frac{1}{\theta^{(n+2)}}\frac{1}{\theta^{d(n+2)}}\int_{B_{\theta^d}}|f-L^{(d)}|^2\leq \frac{C^d}{\theta^{n+2-2\mu d}}E^2_L(1)\leq \frac{1}{4}E^2_L(1).$$ Together with the triangle inequality and (\ref{eq:DeGiorgitypedecayfortheweightedaveragetilting}), and recalling the geometric rescaling of the other terms appearing in the definition of the excess $E^2$, we obtain $E^2_{L_\infty}(\rho) \leq  \rho^{2 (\mu-\delta)} E^2_{L_\infty}(1)$ for $\rho\leq R_\delta$.}, a standard argument delivers that for any $\rho \in (0,R_\delta)$ the following inequality holds:

$$E^2_{L_\infty}(\rho) \leq  \rho^{2 (\mu-\delta)} E^2_{L_\infty}(1).$$

Moreover we know that $Du_1(0)=Du_2(0)=0$: this means that $L_\infty$ is actually the zero-function and the previous inequality shows how the excess decays to it.
\end{oss}

\begin{proof}[proof of Proposition \ref{Prop:alphaimproved}]
The decay obtained in Propositions \ref{Prop:DeGiorgitypedecayforthesemidifference} and \ref{Prop:DeGiorgitypedecayfortheweightedaverage} is valid (for a certain fixed $\mu-\delta \in (\frac{1}{2},1)$) at all points of the touching set. Moreover away from these points we have complete smoothness and elliptic estimates for $u_1$ and $u_2$: combining these two facts (see \cite[Lemma 4.3]{WicAnnals}) with standard theory of Campanato spaces gives that the gradients of $q_1 u_1 + q_2 u_2$ decay to $L_\infty\equiv 0$ in a $C^{0,\mu-\delta}$ fashion. 

With an analogous iteration step based on the decay established in Proposition \ref{Prop:DeGiorgitypedecayforthesemidifference} we get that the gradients of $u_1-u_2$ decay to $L_\infty\equiv 0$ in a $C^{0,\mu-\delta}$ fashion. 

Here $\mu$ can be chosen as close as we want to $1$ and $\delta$ is as small as we wish. Therefore (since we have two independent linear combinations of $u_1$ and $u_2$) the proof of Proposition \ref{Prop:alphaimproved} is complete.  
\end{proof}

\subsection{The semi-difference is in $W_{\rm loc}^{2,2}$}
\label{semidifferenceisW22}

The aim of this section is to conclude that $v \in W_{\rm loc}^{2,2}(B)$. 
The function $v$ solves (\ref{eq:PDEv}) for test functions $\zeta$ compactly supported in $B^n_1 \setminus T$. By assumption $v=0$, $Dv=0$ on $T$. Moreover $v > 0$ on $B \setminus T$ and $Dv$ is H\"older continuous on $B$ and actually $C^1$ on $B\setminus T$ (the fact that $T$ has zero $\mathcal{H}^n$-measure is not needed for this argument). With Proposition \ref{Prop:L2boundonsecondder} below we will ensure a $L^2$-bound for $D^2 v|_{B\setminus T}$: once this is achieved, Lemma \ref{lem:distributionalder} will immediately imply that $v \in W_{\rm loc}^{2,2}(B)$.

\begin{Prop}
\label{Prop:L2boundonsecondder}
We have the summability statement $\int_{B\setminus T} |D^2v|^2 <\infty$.
\end{Prop}

The proof occupies the rest of this section. For the purpose of obtaining the $L^2$-bound for $D^2 v|_{B\setminus T}$ we will use the PDE (\ref{eq:PDEv}) with the test function $D_k(\gamma_\delta(D_k v) \zeta)$ and $\zeta \in C^\infty_c(B)$ (for a fixed $k$), where $\gamma_\delta$ is as in Lemma \ref{lem:distributionalder}. Plugging in, switching the order of differentation and integrating by parts we obtain

$$\int_{B} (\delta_{ij}+b_{ij}(Du_a, Dv))\, D_k D_j v\, D_i (\gamma_\delta(D_k v) \zeta)+\int_{B} D_k \left((\delta_{ij}+b_{ij}(Du_a, Dv))\right) D_j v\, D_i (\gamma_\delta(D_k v) \zeta) =  0.$$
Standard computations, with the notation $b_{ij}(p,q)$ and $p=(p_1, ..., p_n)$, $q=(q_1, ..., q_n)$, yield

$$\int\limits_{|D_kv|>\delta} (\delta_{ij}+b_{ij}(Du_a, Dv)) (D_j D_k v) (D_i D_k v) \zeta + \int\limits_{|D_kv|>\delta} (D_{q_l}b_{ij})(Du_a, Dv)\, D_k D_l v \, D_j v\, D_i D_k v\, \zeta =$$
$$=-\int\limits_{|D_kv|>\delta} (\delta_{ij}+b_{ij}(Du_a, Dv)) (D_j D_k v) \gamma_\delta(D_k v) D_i\zeta - \int\limits_{|D_kv|>\delta} (D_{q_l}b_{ij})(Du_a, Dv)\, D_k D_l v \, D_j v\, \gamma_\delta(D_k v) D_i\zeta -$$
\begin{equation}
\label{eq:fourterms}
 -\int\limits_{|D_kv|>\delta} (D_{p_l}b_{ij})(Du_a, Dv)\, D_k D_l u_a \, D_j v\, D_i D_k v\, \zeta -   \int\limits_{|D_kv|>\delta} (D_{p_l}b_{ij})(Du_a, Dv)\, D_k D_l u_a \, D_j v\, \gamma_\delta(D_k v) D_i\zeta
\end{equation}
and relabeling indexes in the second term on the left-hand-side we can rewrite the left hand side of the previous equality (\ref{eq:fourterms}) in the form

\begin{equation}
\label{eq:lhsfourterms}\int\limits_{|D_kv|>\delta} (\delta_{ij}+b_{ij}(Du_a, Dv)) (D_j D_k v) (D_i D_k v) \zeta + \int\limits_{|D_kv|>\delta} (D_{q_j}b_{il})(Du_a, Dv)\, D_k D_j v \, D_l v\, D_i D_k v\, \zeta .
\end{equation}
Here we can notice the elliptic matrix whose coefficients are 
$$\delta_{ij}+b_{ij}(Du_a, Dv))+ \sum_l (D_l v)(D_{q_j}b_{il})(Du_a, Dv);$$ the ellipticity is guaranteed by the smallness of $b_{ij}$ in $C^1$-norm and the H\"older continuity of $Du_a$ and $Dv$ with $Du_a(0)=Dv(0)=0$ (either start with a good dilation so that these quantities are small enough, or work at this stage in a small enough ball around $0$). So there exists $c>0$ (independent of $\delta$) such that 

\begin{equation}
 \label{eq:lhs}
c \int\limits_{|D_kv|>\delta} |\nabla(D_k v)|^2\zeta \leq (\ref{eq:lhsfourterms})\,\,\, \text{ for all } \zeta \geq 0.
\end{equation}
Recall that (\ref{eq:lhsfourterms}) is the left-hand side of (\ref{eq:fourterms}) and so we can replace (\ref{eq:lhsfourterms}) by the right-hand side of (\ref{eq:fourterms}). Therefore we will now analyse the four terms on the right-hand side of (\ref{eq:fourterms}) and find suitable bounds. We begin with the third term: by Young's inequality, choosing $\eps>0$ small compared to $c$ in (\ref{eq:lhs}), we find

\begin{equation}
 \label{eq:rhsIII}
\left|\int\limits_{|D_kv|>\delta} (D_{p_l}b_{ij})(Du_a, Dv)\, D_k D_l u_a \, D_j v\, D_i D_k v\, \zeta \right|\leq 
\end{equation} $$ \leq\frac{1}{\eps}\int\limits_{|D_kv|>\delta} |(D_{p_l}b_{ij})(Du_a, Dv)|^2 |D_k D_l u_a|^2 |D_j v|^2 \zeta + \eps\int\limits_{|D_kv|>\delta} |D_i D_k v|^2 \zeta.$$
The second term on the r.h.s. of (\ref{eq:rhsIII}) will be absorbed on the left-hand side of (\ref{eq:lhs}): it is not a priori clear, however, that the first term on the r.h.s. of (\ref{eq:rhsIII}) is summable with a uniform bound, independently of $\delta$. This is obtained in the following lemma (which will also be needed later on in Section \ref{extensionPDEv}).
 
\begin{lem}
 \label{lem:boundsonsecondder}
We have the following summability statements:
$$ \int\limits_{B\setminus T}  |D^2 u_a|^2 |D_j v|^2 \zeta  <\infty \,\,\, \text{ and } \,\,\, 
 \int\limits_{B\setminus T}  |D^2 v|^2 |D_j v|^2  \zeta  <\infty.$$
  
\end{lem}

\begin{proof}
 The summability of $\int\limits_{B\setminus T}  |D^2 u_a|^2 |D_j v|^2 $ comes from the following observations. Let $x \in B\setminus T$ and take the largest open ball $B_R(x)$ that is disjoint from $T$, so that $R=\text{dist}(x,T)$. Then $\exists y \in T \cap \p B_R(x)$. The function $D_l u_1$ satisfies (by differentiating the CMC equation)

\begin{equation}
 \label{eq:differentiateCMC}
\int_{B_R(x)} \frac{1}{\sqrt{1+|D u_1|^2}}\left(\delta_{ij}-\frac{D_i u_1 D_j u_1}{1 + |D u_1|^2}\right) D_j(D_l u_1)D_i \zeta =0
\end{equation}
with uniform ellipticity on $B \setminus T$. In other words $f=D_l u_1$ satisfies $\text{div}(A \nabla f)=0$ with $A=(a_{ij})_{i,j=1}^n$, where $a_{ij}=\frac{\p a^i}{\p p_j} (Du_1)=\frac{1}{\sqrt{1+|D u_1|^2}}\left(\delta_{ij}-\frac{D_i u_1 D_j u_1}{1 + |D u_1|^2}\right)$ and $a^i(p)=\frac{p_i}{\sqrt{1+|p|^2}}$, so $a_{ij}$ are smooth uniformly elliptic entries (we can assume that $\nabla u_1$ is small) and hence by Schauder estimates, we have that for any constant $\lambda_{1} \in {\mathbb R}$,  
\begin{equation}
 \label{eq:SchauderDGN1}
|\nabla(D_l u_1)|(x)\leq   \frac{C}{R} \sup_{B_{R}(x)} |D_{l}u_{1} - \lambda_{1}|.
\end{equation}

Similarly, for any constant $\lambda_{2} \in {\mathbb R}$,

\begin{equation}
 \label{eq:SchauderDGN2}
|\nabla(D_l u_2)|(x) \leq   \frac{C}{R} \sup_{B_{R}(x)} |D_{l}u_{2} - \lambda_{2}|.
\end{equation}

Taking the average $u_a=\frac{u_1 +u_2}{2}$ we get
$$\nabla(D_l u_a) = \frac{1}{2} (\nabla(D_l u_1)+\nabla(D_l u_2))$$
so from (\ref{eq:SchauderDGN1}) and (\ref{eq:SchauderDGN2}) we infer,
after choosing $\lambda_1=D_l u_1(y)$ and $\lambda_2=D_l u_2(y)$ and using the H\"older continuity of $Du_1$ and $Du_2$, that 
$|D^2u_a|(x)\leq C \text{dist}(x,T)^{\alpha-1}$. Therefore 

\begin{equation}
\label{eq:boundsonsecondder}
 \int\limits_{B\setminus T}  |D^2 u_a|^2 |D_j v|^2 \zeta \leq \int\limits_{B \setminus T}  \text{dist}(x,T)^{2(2\alpha-1)} \zeta <\infty
\end{equation}
since we can choose $\alpha\geq \frac{1}{2}$, as proved in Proposition \ref{Prop:alphaimproved}. Note that the same summability behaviour can be obtained for $v=\frac{u_1-u_2}{2}$, i.e. $|D^2v|(x)\leq C \text{dist}(x,T)^{\alpha-1}$ and 

\begin{equation}
\label{eq:boundsonseconddersecond}
 \int\limits_{B\setminus T}  |D^2 v|^2 |D_j v|^2 \zeta \leq \int\limits_{B \setminus T}  \text{dist}(x,T)^{2(2\alpha-1)} \zeta <\infty.
\end{equation}

\end{proof}

Lemma \ref{lem:boundsonsecondder} establishes the desired bound on the first term on the r.h.s. of (\ref{eq:rhsIII}). In order to conclude the $L^2$-bound for $D^2 v|_{B\setminus T}$ in Proposition \ref{Prop:L2boundonsecondder}, we now go back to the remaining three terms on the right hand side of (\ref{eq:fourterms}). Let us analyse the second:

$$\left|\int\limits_{|D_kv|>\delta} (D_{q_l}b_{ij})(Du_a, Dv)\, D_k D_l v \, D_j v\, \gamma_\delta(D_k v) D_i\zeta\right|\leq$$ 
\begin{equation}
 \label{eq:secondoffour}
\leq \int\limits_{|D_kv|>\delta} |(D_{q_l}b_{ij})(Du_a, Dv)|\, |D_k D_l v| \, |D_j v|\, (|D_k v|+\delta) |D_i\zeta| 
\end{equation}
and we can use the summability of $|D^2v||Dv|$ on $B\setminus T$, which is a consequence of Lemma \ref{lem:boundsonsecondder}, so the right-hand-side of (\ref{eq:secondoffour}) is finite and bounded independently of $\delta$. We can argue similarly for the first and fourth terms on the right hand side of (\ref{eq:fourterms}). With this we conclude (recall that there is an absorption on the left hand side to be performed, for each fixed $\delta$, when splitting the third term on the right hand side of (\ref{eq:fourterms})), sending $\delta \to 0$, that

$$  \int\limits_{B\setminus T} |\nabla(D_k v)|^2\zeta<\infty.$$
We thus have $D^2 v|_{B \setminus T}$ is in $L_{\rm loc}^2$. As mentioned at the beginning of this section, this implies $v \in W_{\rm loc}^{2,2}(B)$ since we can now apply Lemma \ref{lem:distributionalder} ($D^2 v$ develops no further distribution supported on $T$).

\subsection{Extension of the PDE for $v$ across $T$ and higher regularity conclusions}
\label{extensionPDEv}

Consider the PDE for $v$ (\ref{eq:PDEv}): the aim of this subsection is to extend the PDE for $v$ across $T$ and, as a straightforward consequence, complete the proof of (iv), i.e. the inductive step of Theorem \ref{thm:higher-reg}. While it is indeed true that the PDE for $v$ is valid for arbitrary test functions $\zeta \in C^\infty_c(B)$ (this can be proved exactly with the same argument that we are going to use in this subsection) what we will need is the extension of the PDE for $v$ that takes into accound multiplicities as well.

Consider the function $f(x)=q_1(x) (\delta_{ij}+b_{ij}(Du_a, Dv)) D_j v$ on $B$. Recall that $q_1$ is constant on each connected component of $B\setminus T$. The function $f$ is $C^1$ on $B \setminus T$ and its derivative is summable on $B\setminus T$, in view of the bounds established in Lemma \ref{lem:boundsonsecondder}. Moreover it vanishes everywhere on the set $T$ (since $Dv$ vanishes there\footnote{$q_1$ is not really defined on $T$ but we just set the function to be $0$ since $Dv$ vanishes.}) and it is H\"older continuous on $B$. Indeed for $y \in T$ and $x \in B\setminus T$ we have (recall that $0< q_1(x)<q$)

$$|f(x)-f(y)|=|f(x)|=|q_1(x)| |\delta_{ij}+b_{ij}(Du_a, Dv)| |D_j v| \leq K_{b_{ij},q} |Dv(x)| =$$ $$=K_{b_{ij},q} |Dv(x)-Dv(y)|\leq K_{b_{ij},q}|x-y|^\alpha. $$
For $x,y \in B\setminus T$ and $x$ and $z$ in distinct connected components of $B\setminus T$ we consider the segment joining $x$ and $z$, which must intersect $T$ at a point $y$. Then $|f(x)-f(z)|\leq |f(x)-f(y)|+|f(z)-f(y)|\leq K_{b_{ij},q}|x-y|^\alpha+K_{b_{ij},q}|z-y|^\alpha \leq 2K_{b_{ij},q}|x-z|^\alpha$. For $x,y$ in the same connected component of $B\setminus T$ then $q_1$ is constant, so the H\"older continuity follows. 

We then have by Lemma \ref{lem:distributionalder} that the distributional derivative of $f$ on $B$ is an $L^1$ function, namely is is just the classical derivative on $B \setminus T$ extended by setting it to be $0$ on $T$. So we have $f\in W^{1,1}(B)$. We then have, for an arbitrary $\zeta \in C^\infty_c(B)$ :

$$\int_B q_1(x) (\delta_{ij}+b_{ij}(Du_a, Dv)) D_j v D_i\zeta = -\int_B D_i\left(q_1(x) (\delta_{ij}+b_{ij}(Du_a, Dv)) D_j v \right)\zeta =   $$
$$= -\int_{B\setminus T} D_i\left(q_1(x) (\delta_{ij}+b_{ij}(Du_a, Dv)) D_j v \right)\zeta = -\int_{B\setminus T} q_1(x) D_i\left((\delta_{ij}+b_{ij}(Du_a, Dv)) D_j v \right)\zeta =$$
$$ =-\int_{B\setminus T} q_1(x) H\zeta,$$
where in the second equality we use the fact that $T$ has measure $0$, in the third equality the fact that $q_1$ is locally constant on $B \setminus T$ and in the last equality the fact that $D_i\left((\delta_{ij}+b_{ij}(Du_a, Dv)) D_j v\right)=H$ on $B\setminus T$ (by the PDE (\ref{eq:PDEv}), which is valid in its strong form on $B \setminus T$ since $v$ is $C^2$ there). Since $T$ has zero measure we can write the last term of the previous chain of equalities as $-\int_{B} q_1(x) H\zeta$: in conclusion we find a version of the PDE for $v$ that holds on the whole of $B$ also when we take multiplicity into account, namely:

\begin{equation}
\label{eq:PDevwithmultiplicity}
\int_B q_1(x) (\delta_{ij}+b_{ij}(Du_a, Dv)) D_j v D_i\zeta =   -\int_B q_1(x) H\zeta. 
\end{equation}

\medskip

Now consider the first variation identity (\ref{eq:firstvarformulaglobal}), valid for any arbitrary $\zeta \in C^\infty_c(B)$:

$$\int_B \left( \frac{q_1(x) D_i u_1}{\sqrt{1+|Du_1|^2}} + \frac{q_2(x) D_i u_2}{\sqrt{1+|Du_2|^2}}\right) D_i \zeta =  \int_B (q_1(x)-q_2(x))H \zeta.$$
We rewrite it, using $q_2(x)=q-q_1(x)$, as

$$\int_B  q\frac{D_i u_2}{\sqrt{1+|Du_2|^2}}D_i\zeta + \int_B q_1(x)\left(\frac{ D_i u_1}{\sqrt{1+|Du_1|^2}} - \frac{ D_i u_2}{\sqrt{1+|Du_2|^2}}\right) D_i \zeta =  \int_B (2q_1(x)-q)H\zeta.$$
The term $\left(\frac{ D_i u_1}{\sqrt{1+|Du_1|^2}} - \frac{ D_i u_2}{\sqrt{1+|Du_2|^2}}\right)$ is exactly $2(\delta_{ij}+b_{ij}(Du_a, Dv)) D_j v$, as computed pointwise when we obtained the PDE (\ref{eq:PDEv}) for $v$ (it is $0$ on $T$). Using (\ref{eq:PDevwithmultiplicity}) we then find

$$\int_B  q\frac{D_i u_2}{\sqrt{1+|Du_2|^2}}D_i\zeta  = - \int qH\zeta.$$
Simplifying $q$ we have the standard CMC equation in weak form for $u_2$, where the test function $\zeta$ is arbitrary and $u_2 \in C^{1,\alpha}(B)$. The difference quotients method yields $u_2\in W^{2,2}(B)$ and differentiating the PDE proves $Du_2$ is $C^{1,\alpha}$. Bootstrapping yields $u_2 \in C^\infty$.

With the conlcusions obtained so far we have achieved the proof of the Higher Regularity Theorem \ref{thm:higher-reg} subject to the inductive assumptions, i.e. step (iv) of the general inductive program from Section \ref{mainsteps}. At this stage the entire induction process is complete, and with it the proof of Theorem \ref{thm:mainregularityrestated} (and consequently Theorem \ref{thm:mainregularity} is proved as well).
 
\section{Proof of the Compactness Theorem}

For the sake of clarity we will split the proof in three main parts. We begin by proving the following:

\begin{thm}[\textbf{first part of Theorem \ref{thm:compactness}}]
\label{thm:compactnesssfirstpart}
Let $n \geq 2$ consider an open set $\mathcal{U} \subset \R^{n+1}$. Let $V_j$ be a sequence of integral $n$-varifolds and assume that each $V_j$ satisfies the assumptions of Theorem \ref{thm:mainregularity} and that the $V_j$'s have uniformly bounded masses $M(V)\leq K_0$ and uniformly bounded mean curvatures $|H|\leq H_0$ for $K_0, H_0 \in \R$ (here $H$ is the mean curvature appearing in assumption 3). Then whenever $V_j \to V$ as varifolds, $V$ must be an integral $n$-varifold that satisfies properties 1 and 2 of Theorem \ref{thm:mainregularity}.
\end{thm}

\begin{proof}[proof of Theorem \ref{thm:compactnesssfirstpart}]
Thanks to the Regularity Theorem \ref{thm:mainregularity} we have that, for any varifold $V$ in $\mathcal{S}_H$ with $H \leq H_0 \in \R^+$, the mean curvature is actually in $L^\infty$ and therefore the first variation of $V$ satisfies $\|\delta V\|(X)\leq H_0 \|X\|_{L^1(\spt{V})}$ for any $X \in C^1_c(\spt{V})$. This gives that, under the assumptions of Theorem \ref{thm:compactness}, we also have a uniform bound for the first variations. In view of this we can apply Allard's compactness theorem for integral varifolds \cite{Allard} and obtain that, given $V_j \in \mathcal{S}_{H_j}$ for $H_j\leq H_0 \in \R^+$, we can extract a subsequence converging to an integral varifold $V_\infty$ with bounded generalized mean curvature and $\|\delta V_\infty\|\leq \liminf_{j \to \infty} \|\delta V_j\|$. Therefore the generalized mean curvature of $V_\infty$ is in $L^\infty$ as well, bounded by $H_0$ in $L^\infty$, so condition 1 is checked. Since we are working up to the extraction of a subsequence, we may also assume without loss of generality that $H=\lim_{j \to \infty} H_{V_j}$ exists. 

For condition 2, i.e. the absence of classical singularities, we assume by contradiction that $V_\infty$ possesses a classical singularity, with $V_j \rightharpoonup V_\infty$ weakly as varifolds. By definition we have a tangent cone $C$ at $y \in \text{sing}_C V_\infty$ such that $\spt{C}$ is a sum of three or more half planes, with at least two of them transversal. Let $\eps>0$ be as in the minimum distance Theorem \ref{thm:min-dist}. Choose $\rho>0$ such that $\text{dist}(\eta_{y,\rho}V_\infty, C) \leq \eps/2$, where $\eta_{y,\rho}$ denotes the dilation of factor $\rho$ around $y$. Since we also have that $\eta_{y,\rho}V_j \rightharpoonup \eta_{y,\rho}V_\infty$ we will find $j$ large enough such that $\text{dist}(\eta_{y,\rho}V_\infty, \eta_{y,\rho}V_j) \leq \eps/2$ and therefore $\text{dist}(C, \eta_{y,\rho}V_j) \leq \eps$: since $\eta_{y,\rho}V_j$ satisfies the assumption of the minimum distance Theorem \ref{thm:min-dist} we have a contradiction.
\end{proof}
Next we have

\begin{thm}[\textbf{second part of Theorem \ref{thm:compactness}}]
\label{thm:compactnessstronstability}
Let $n \geq 2$ consider an open set $\mathcal{U} \subset \R^{n+1}$. Let $V_j$ be a sequence of integral $n$-varifolds and assume that each $V_j$ satisfies the assumptions of Theorem \ref{thm:mainregularity} and that the $V_j$'s have uniformly bounded masses $M(V)\leq K_0$ and uniformly bounded mean curvatures $|H|\leq H_0$ for $K_0, H_0 \in \R$ (here $H$ is the mean curvature appearing in assumption 3). Then whenever $V_j \to V$ as varifolds, $V$ must be an integral $n$-varifold that satisfies properties 3 and 4 of Theorem \ref{thm:mainregularity}.
\end{thm}

\begin{oss}[improved sheeting theorem]
\label{oss:improvedsheeting}
We will make use of the following ``improved sheeting theorem''. Note that the Sheeting Theorem and Higher Regularity Theorem previously established during the inductive proof are at this stage statements in their own right, as the induction is over. Note that $u_1=\tilde{u}_1$ and $u_q=\tilde{u}_{\tilde{q}}$ when we compare Theorems \ref{thm:sheeting} and \ref{thm:higher-reg}. Combining these two theorems we can state the following. Under the assumptions of Theorem \ref{thm:sheeting} (note in particular the assumption of strong stability of $\Greg V$ with repect to the functional $J$) we can conclude that $V$ is the sum (as varifolds, just as in Theorem \ref{thm:sheeting}) of the graphs of $q$ $C^{1,\alpha}$ functions $u_1 \leq ... \leq u_q$ on $B_{1/2}^n(0)$ (actually they are $C^{1,1}$ but this is not needed) with $u_1$ and $u_q$ smooth and satisfying separately the CMC PDE with either $|h|$ or $-|h|$ on the right-hand side, and moreover $u_1$ and $u_q$ fulfil the elliptic estimate
$$\|u_j\|_{C^{2,\alpha}(B_{1/2})} \leq C\left(\int_{B_1 \times \R} |x^{n+1}|^2 d\|V\|(X)+ |h| \right),$$
for $C=C(n,\alpha,q)$ and for $j=1$ and $j=q$. In other words, the improvement with respect to the statement of Theorem \ref{thm:sheeting} concerns the top and bottom graphs: they enjoy higher regularity and higher order elliptic estimates. This is not necessarily true for the remaining graphs, since our assumptions allow the presence of jumps in multiplicities at the touching singularity (as in example \ref{oss:jumpsattouchingsing}) that can indeed cause $u_j$ to be $C^{1,1}$ but not $C^2$ for $2\leq j \leq q-1$.
\end{oss}

\begin{proof}[proof of Theorem \ref{thm:compactnessstronstability}]
Note that conditions 3 and 4 (and also 5, that we will deal with later) are to be checked respectively at touching singularities and on the embedded $C^{1,\alpha}$ part of $V$ (and condition 5 on $\Greg{V})$): for this reason, in the following step we will address the case of a point $x \in \spt{V_\infty}$ where the tangent exists and is a plane counted $q$ times. In the following step we aim to ensure the applicability of the sheeting theorem  along a subsequence in a fixed ball around an arbitrary such point $x$: with this at hand, in step 2 we will check conditions 2 and 4.

\medskip

\textit{Step 1: ensure the applicability of the sheeting theorem  along a subsequence (not relabeled) in a fixed ball around an arbitrary point}. As a preliminary step, we will need to ensure that we can work in a fixed ball where the stability (for $J$) holds for $\Greg{V_j}$ at least along a subsequence. We assume that $x \in \spt{V_\infty}$ is a point at which the tangent cone is a plane counted $q$ times for some $q \in \N$. Given $0<R<1$ look at the disjoint open sets $B^{n+1}_R(x)$ and $B^{n+1}_1(x) \setminus B^{n+1}_R(x)$: by Remark \ref{oss:weakimpliesstrong} we have that either for all $j$ large enough $\Greg{V_j}$ is stable in $B^{n+1}_R(x)$ for the functional $J$, or this stability is valid for a subsequence suitably extracted from $V_{j}$ in the annulus $B^{n+1}_1(x) \setminus B^{n+1}_R(x)$. In view of this, either we find $R>0$ so that the first alternative holds for all $j$ large enough or we have a subsequence ${j_k}$ such that each $\Greg{V_{j_k}}$ is stable for $J$ in any annulus of the form $B^{n+1}_1(x) \setminus B^{n+1}_{{1}/{2^N}}(x)$ (for $N\in \N$: this uses a diagonal argument). In the second case, however, the arbitraryness of $N$ and the fact that $n\geq 2$ allow to extend, by a standard capacity argument, the stability condition to the whole ball $B_1^{n+1}(x)$. Either way we have identified a subsequence (that by abuse of notation we will not relabel from now on) $V_j$ and a ball of fixed radius around $x$ such that  for every $j$ the immersion $\Greg{V_{j}}$ is stable for $J$ in the selected ball.

We now need to ensure the remaining conditions for the applicability of the sheeting theorem around $x$ to $V_j$: the mass bounds and the smallness of the excess (all unifomly in $j$). Without loss of generality we can assume $x=0$ and $T_x V_\infty = \R^n \times \{0\}$. 
This is simply a matter of restricting to a small enough ball around $x$ and dilating to the size $B_2$. Indeed, we can choose $\sigma>0$ small enough, up to replacing $V_j$ and $V_\infty$ with a suitable dilation (with factor independent of $j$) such that for all $j$ large enough:

\noindent (i) $\spt{V_j} \subset \{|x^{n+1}| < \sigma\}$ for every $X=(x,x^{n+1}) \in \spt{V_j} \cap \left(B^n_2(0) \times (-1, 1) \right)$ (by the weak convergence and the monotonicity formula we have that the supports of $V_j$ converge to the support of $V_\infty$ in Hausdorff distance, so we just need to ensure the condition $\spt{V_\infty} \subset \{|x^{n+1}| < \sigma\}$ on $V_\infty$).

\noindent (ii) $|h_{V_j}| \leq \sigma H_0$ for all $V_j$ and $\int_{B^n_2(0) \times (-1, 1)}|h|^p d\|V_j\| \leq \sigma H_0 \|V_j\|\left( B^n_2(0) \times (-1, 1)\right)^{1/p} \leq \sigma H_0 \Lambda^{1/p}$, where $\Lambda$ is a bound for the masses of $V_j$

\noindent (iii) Denoting with $q$ the density of $V_\infty$ at $0$, $q-\frac{1}{2}<\|V_j\|\left( B^n_2(0) \times (-1, 1)\right) < q+\frac{1}{2}$ (by the weak convergence, the masses converge to the mass and we can ensure, by the dilation step, that $\|V_\infty\|\left( B^n_2(0) \times (-1, 1)\right)$ is as close to $q$ as we wish.

\medskip

\textit{Step 2: conditions 3 and 4.} Let us check condition 3 (on the touching singularities). Assume that $x \in \spt{V_\infty}$ is a (two-fold) touching singularity, $x \in \text{sing}_T(V_\infty)$: so the support of $V_\infty$ in a small enough ball around $x$ is described by two $C^{1,\alpha}$ graphs touching only tangentially (and surely touching at $x$). In particular we have a well-defined tangent plane to $V_\infty$ at $x$. We assume without loss of generality $x=0$ and $T_x V_\infty = \R^n \times \{0\}$. The previous step yields $r>0$ such that $V_j \res \left( B_r^{n}(x)\right) \times (-r,r)$ is the varifold sum of the graphs of $q$ $C^{1,\alpha}$ functions $u_j^1 \leq ...\leq u_j^q$. Upon choosing $r$ small enough we have that $V_\infty \res \left( B_r^{n}(x)\right) \times (-r,r)$ is supported on the union of exactly two $C^{1,\alpha}$ functions. The elliptic estimates in Theorem \ref{thm:sheeting} imply uniform $C^{1,\alpha}$ bounds on this sequence of $q$-valued functions and therefore by Ascoli-Arzela's compactness theorem we have $C^1$-convergence (up to a subsequence) to a $q$-valued graph $u_\infty^1 \leq ...\leq u_\infty^q$ (naturally ordered in an increasing fashion since so are those in the sequence) that describes necessarily the support of $V_\infty$. Remark \ref{oss:improvedsheeting}, however, gives smoothness and higher elliptic estimates for $u_j^1$ and $u_j^q$, uniformly in $j$. Therefore, again using Ascoli-Arzela's theorem, we get $C^2$ convergence of $u_j^q \to u_\infty^q$ and $u_j^1 \to u_\infty^1$, which forces in particular the $C^2$-regularity of $u_\infty^1$ and $u_\infty^q$. While it is not necesarily true that the remaining graphs $u_\infty^\ell$ are $C^2$ (jumps in multiplicities as in example \ref{oss:jumpsattouchingsing} might prevent that) we obtain that the \textit{support} of $V_\infty \res \left( B_r^{n}(x)\right) \times (-r,r)$ is exactly given by the union of the graphs of $u_\infty^1$ and $u_\infty^q$, that are $C^2$. Moreover, again by Remark \ref{oss:improvedsheeting}, $u_j^1$ and $u_j^q$ satisfy the CMC PDE with either $|h_j|$ of $-|h_j|$ on the right-hand side and therefore (up to a subsequence that we do not relabel) also  $u_\infty^1$ and $u_\infty^q$ satisfy the CMC PDE with either $|h_\infty|=\lim_j|h_j|$ of $-|h_\infty|$. So the two sheets $u_\infty^1$ and $u_\infty^q$ are separately CMC and Hopf boundary point Lemma \ref{lem:Hopf} implies that either $u_\infty^1\equiv u_\infty^q$ (in which case $x$ is in $\text{reg}{V_\infty}$ which contradicts $x \in \text{sing}_T(V_\infty)$, so this cannot happen) or (only possible case) the graphs of $u_\infty^1$ and $u_\infty^q$ have mean curvatures pointing in opposite directions, which (by continuity of the second derivatives, since there are $C^2$ graphs) forces the set where $u_1=u_q$ to have $n$-dimensional measure $0$ (Remark \ref{oss:maxprincsmooth}).

Let us check condition 4. Whenever we have a $C^{1,\alpha}$ embedded part of $V_\infty$, we can use locally around each point the same argument just discussed. This time, when we get to Hopf boundary point lemma, the fact that $x\in \text{reg}{V_\infty}$ will force $u_\infty^1\equiv u_\infty^q$ and $u_\infty^1$ smooth and CMC, by passing the PDE for $u_j^1$ to the limit.

\end{proof}

\begin{oss}
As a byproduct of the proof we can see that whenever $x \in \spt{V_\infty}$ is such that the tangent at $x$ is unique and is a plane with multiplicity, then $x \in \Greg{V}$. On the other hand it is always true that at $x \in \Greg{V}$ the tangent at $x$ is unique and is a plane with multiplicity. It would be possible to continue at this stage along this direction in order to establish the fact that $V_\infty$ enjoys the same regularity properties as $V_j$ - however, since this will follow from the compactness theorem that we are in the process of proving, we will not do that.
\end{oss}

With Theorem \ref{thm:compactnessstronstability} in place, we will now prove the compactness theorem by showing that the limit $V$ satisfies the weak stability condition on $\Greg V$. 
We will need the following (see lemma \ref{lem:neckregion} below): given a sequence of weakly stable CMC varifolds either we have a suitable notion of global sheeting or we can identify a ``neck point'' characterized by the fact that, away from it, we can obtain a suitable ``global sheeting'' and moreover the strong stability holds for a suitable subsequence\footnote{Intuitively this can be thought of as an analogue of the neck region (localized around the origin) for a sequence of blow-downs of the standard catenoid (this sequence converges to a double copy of the plane).}.Remark that the stability condition needs to be checked only on $\Greg{V}$, in other words we need to check it for volume-preserving variations that are compactly supported in the open set $\Greg{V}$. Note that $\text{sing}V \setminus \Greg V$ is a closed set: indeed $\text{sing}V$ is closed and if $x_n \to x$ with $x_n \in \text{sing}V \setminus \Greg V$ and $x \in \text{sing}V \cap \Greg V$ then we would have, by definition of $\Greg{V}$, that $T_x V$ is well-defined and in a small ball around $x$ we can write $\spt{V}$ as the union of two smooth graphs (on the hyperplane $T_x V$) touching only tangentially and touching at $x$. Any point in this small ball is then either a smooth point (if the graph is embedded there) or is in $\Greg{V} \cap \text{sing}V$. A variation compactly supported in $\Greg{V}$ is therefore supported in an open set that is a positive distance away from $\text{sing}V \setminus \Greg V$. on this open set we also have a $L^\infty$-bound on the second fundamental form of the immersion. We will therefore restrict for the rest of the proof to such an open set, on which $\Greg{V}$ is represented by an immersion $i$ of an abstract $n$-dimensional surface $\Sigma$ and the second fundamental form of $i(\Sigma)$ is bounded by $A>0$. This allows to consider a tubular neighbourhood of fixed width $\frac{1}{C_n A}$ (note that we are considering the tubular neighbouhood of the immersion, so it overlaps with itself at the touching singularities of $\Greg{V}$).

\begin{lem}[\textbf{identification of at most one ``neck region'' and sheeting selection away from it}]
\label{lem:neckregion}
Let $V_j \to V$, where the $V_j$'s are as in Theorem \ref{thm:compactnessstronstability}. Denote by $i:\Sigma \to \mathcal{U}$ the immersion describing $\Greg V$ restricted to the open set described above, where the variation is supported and where we have a global bound $A$ on the second fundamental form of the immersion. Then either 

(i) for every $k$ large enough we can find a smooth immersion $i_k:\Sigma \to \mathcal{U}$ such that $i_k \stackrel{C^2}{\to} i$ and $i_k\left(\Sigma\right) \subset \Greg{V_k}$.

(ii) there exists a point $y \in i(\Sigma)$ such that, for any $R>0$, the following holds: there exists a subsequence $V_{j'}$ (possibly depending on $R$) such that $\Greg V_{j'}$ is strongly stable (in the sense of the functional $J$) in $\mathcal{U} \setminus B_R(y)$ and moreover we we can find a smooth immersion $i_{j'}:\Sigma \setminus i^{-1}(B_R(y)) \to \mathcal{U}$ such that $i_{j'} \stackrel{C^2}{\to} i$ (where $i$ is restricted to $\Sigma  \setminus i^{-1}(B_R(y))$) and $i_{j'}\left(\Sigma  \setminus i^{-1}(B_R(y))\right) \subset \Greg{V_{j'}}$.
\end{lem}

\begin{proof}[\textbf{proof of Lemma \ref{lem:neckregion}}]
\textit{Step 1: identification of at most one neck region}. In the following $C>0$ is a constant depending only on the dimension $n$, on $\eps$ in Theorem \ref{thm:sheeting} and on the bound $A$ on the second fundamental form of the immersion: it will be specified below how to choose this $C$.

For $x \in \overline{i(\Sigma)}$ denote by $q_x$ the density at $x$ (this is an integer since we are in $\Greg{V}$ - note also that $T_x V$ is well-defined for all $x \in  \overline{i(\Sigma)}$). For any fixed $r>0$ we can find, for every $x \in  \overline{i(\Sigma)}$, a radius $r_x<r$ such that the mass ratio of $V$ in $B_{4C r_x}^{n+1}(x)$ is smaller than  $q_x+\frac{1}{2}$, the support of $V \res B_{4C r_x}^{n+1}(x)$ is contained in the cone of slope $\frac{1}{4}$ around $T_x V$, the mass ratio in the cylinder $(B^{n+1}_{2C r_x}(x) \cap T_x V) \times (T_x V^\bot \cap B_{2C r_x}(x))$ is a value between $q_x-\frac{1}{2}$ and $q_x+\frac{1}{2}$ and the $L^2$-excess (from the Sheeting theorem) with respect to $T_xV$ in the cylinder $(B^{n+1}_{2C r_x}(x) \cap T_x V) \times (T_x V^\bot \cap B_{2C r_x}(x))$ is controlled by $\eps$. Moreover we can ensure that the ball $B_{C r_x}(x)$ does not intersect $\text{sing}V \setminus \Greg{V}$. Note that this choice (we give more details below) is made in order to ensure the applicability of the Sheeting Theorem \ref{thm:sheeting} at scale $C r_x$ around $x$ (both for $V$ and for $V_k$ with $k$ large enough), which gives sheeting on the hyperplane $T_x V$. Later on we will need to restrict the sheeting result to a smaller scale, namely $r_x$, around $x$: this will be done when we will replace the graphical decomposition on $T_x V$ with a graphical decomposition with respect to the normal bundle to $i(\Sigma)$: at that stage the meaning of the constant $C=C(n,\eps,A)$ will become evident. 

The collection $\{B_{r_x}(x)\}_{x \in  \overline{i(\Sigma)}}$ is an open cover of $ \overline{i(\Sigma)}$ so we can extract a finite cover $\{B_{r_x}(x)\}_{x \in L}$. We will use this as a finite cover of $i(\Sigma)$. Then either (a) all the balls $\{B_{4 C r_x}(x)\}_{x \in L}$ are such that we have strong stability for all $V_j$'s for the functional $J$ or (b) there exists a ball $B_{r_y}(y)$ for which the strong stability fails for some $V_{j_y}$ in the ball $B_{4C r_y}(y)$. 
If (a) holds true for some (initially chosen) $r>0$, then we will show that alternative (i) stated in the lemma holds; if not (i.e. if (b) holds for every initially chosen $r>0$) we will have alternative (ii) in the lemma.

Assume that (a) holds, for some $r>0$, for every $j$ large enough (i.e. for a tail of the sequence $V_j$). Then in each ball $B_r$ we can apply the sheeting theorem to $V_j$ for every $j$ large enough; note indeed that there are finitely many balls and that the varifold convergence implies mass convergence, therefore the mass condition for the sheeting theorem is satisfied for $j$ large enough in all balls, and moreover, still by varifold convergence, we obtain the $L^2$-excess condition for $V_j$. The idea will be, by unique continuation, to suitably paste together the sheeting decomposition and obtain a sort of ``global sheeting'' (for $j$ suitably large). 

If (b) holds true for every $r>0$, then we have a sequence of radii $r_n \to 0$ and the corresponding sequence of points $y_n \in V_{j_{y_n}}$ given by (b); upon extracting a subsequence (from $r_n$) that we do not relabel, we can assume that there exists a point $y$ such that $y_n \to y$. Notice that in this case we also have $j_{y_n} \to \infty$ (for, otherwise, if $j_{y_n}$ stays bounded, we would have the strong stability in all balls for a tail of the sequence $V_j$, i.e. (a) would be verified). Now for an arbitrary $R>0$ we look at $B_R(y)$: surely there is $n_0$ large enough ($n_0$ depending on $R$) such that for $n \geq n_0$ the balls $B_{r_{y_n}}(y_n)$ are contained in $B_{R/2}(y)$ and $r_n <\frac{R}{2}$, in particular for all $n \geq n_0$ there is $j_n$ such that the varifold $V_{j_{y_n}}$ is strongly stable in $\mathcal{U} \setminus B_{R/2}(y)$ (indeed $V_{j_{y_n}}$ fails to be strongly stable in $B_{r_{y_n}}(y_n)$ by construction and the open sets $B_{r_n}(y_n)$ and $\mathcal{U} \setminus B_{R/2}(y)$ are disjoint, so we can use Remark \ref{oss:weakimpliesstrong}). We have thus produced a subsequence $V_{j'}=V_{j_{y_n}}$ such that all the $V_{j'}$'s are strongly stable (in the sense of the functional $J$ - away from their codimension $7$ singular set) in $\mathcal{U} \setminus B_{R/2}(y)$, and a fortiori in $\mathcal{U} \setminus B_{R}(y)$. Here $R>0$ can be taken arbitrarily small (but as we change $R$ we get a possibly different subsequence since we may have to take a further tail of the previously extracted subsequence). For every $j'=j_{y_n}$ (so for every $n$ fixed) in all the balls $\{B_{r_x}(x)\}_{x \in L}$ that are contained in $\mathcal{U} \setminus B_{R/2}(y)$ the sheeting theorem applies to $V_{j'}$ (since the strong stability holds in these balls). Moreover the union of these balls contains $i(\Sigma) \cap \left(\mathcal{U}\setminus B_R(y)\right)$ (since, among all the balls from the original cover $\{B_{r_x}(x)\}_{x \in L}$,  we are taking those that do not intersect $B_{R/2}(y)$ and moreover $r_x<r_n < \frac{R}{2}$). The idea again will be to suitably paste together the sheeting decompositions in each of these good balls and obtain a ``global sheeting'' for the subsequence $V_{j'}$ away from $B_R(y)$. 

\textit{Step 2: sheeting in each ball.} In the first alternative identified in the previous step, we have a collection of balls $\{B_{r_x}(x)\}_{x \in L}$ (covering all of $i(\Sigma)$) such that for every $j$ and for every $x\in L$ the varifold $V_j \res B_{4C r_x}(x)$ satisfies the assumptions of the Sheeting theorem. In the second alternative identified in the previous step, for every $R>0$ we have a subcollection of balls $\{B_{r_x}(x)\}_{x \in L_R}$ (covering all of $i(\Sigma) \setminus B_R(y)$) and a subsequence $V_{j'}$ such that for every $x \in L_R$ the varifold $V_{j'} \res B_{4C r_x}(x)$ satisfies the assumptions of the Sheeting theorem. In either of the two cases, in each $B_{4C r_x}(x)$ we can obtain that in the cylinder $(B^{n+1}_{C r_x}(x) \cap T_x V) \times (T_x V^\bot \cap B_{2C r_x}(x))$ the support of $V_j$ (or $V_{j'}$) is described by a collection of $q$ graphs $u_j^1\leq ... u_j^q$ (or $u_{j'}^1\leq ... u_{j'}^q$ in the second alternative) of $C^{1,1}$ functions on $T_x V \cap B_{C r_x}^{n+1}(x) \to T_x V^\bot$ of which $u_j^1$ and $u_j^q$ are separately smooth and CMC. Moreover the estimate in Remark \ref{oss:improvedsheeting} holds for $u_j^1$ and $u_j^q$, which gives a uniform $C^{2,\alpha}$ bound, independently of $j$, on these two graphs.  The uniform $C^{1,\alpha}$ estimate from Theorem \ref{thm:sheeting} gives a uniform $C^{1,\alpha}$ bound on the graphs $u_j^1\leq \ldots \leq u_j^q$, independently of $j$ (independently of $j'$). By Ascoli-Arzel\`{a}'s compactness theorem we can then obtain $C^1$-convergence for these $q$-valued graphs and by necessity the limiting $q$-valued graph must describe the support of $V \res (B^{n+1}_{C r_x}(x) \cap T_x V) \times (T_x V^\bot \cap B_{2C r_x}(x))$ (since varifold convergence implies convergence in Hausdorff distance). Ordering this $q$-valued graph increasingly we obtain $q$ graphs $u^1\leq \ldots \leq u^q$ describing $\spt{V}$ in this cylinder.  On the other hand it is known by assumption that $\Greg{V}$ decomposes in this cylinder as the union of two $C^2$ graphs, which forces $u_1$ and $u_q$ to be smooth and CMC. Recalling the $C^{2,\alpha}$ bound on $u_j^1$ and $u_j^q$ we conclude that actually $u_j^1$ and $u_j^q$ converge in $C^2$ respectively to $u^1$ and $u^q$. By PDE limit $u^1$ and $u^q$ are separately CMC.

In the case that $x$ is a point where $\spt{V}$ is embedded, i.e. if $x \in \Greg{V} \setminus \text{sing}V$, then actually all the graphs describing $\spt{V}$ around $x$ must coincide, $u^1=  \ldots =u^q$; if, on the other hand, $x \in \Greg{V}\cap \text{sing}V$ (i.e. we are at a twofold touching singularity where two $C^2$ sheets touch tangentially) then we must have $u^1\neq u^q$. If $V$ is CMC but not minimal (which is the harder case since otherwise touching of sheets is ruled out) then, if $x \in \Greg{V} \setminus \text{sing}V$ then the mean curvature provides a natural orientation for $V$ and the ordering of the $q$ graphs can be chosen coherently with this orientation (this is intrinsic). If $x \in \Greg{V}\cap \text{sing}V$ then the mean curvature vectors on $u_1$ and $u_{q}$ have opposite directions. By $C^2$ convergence we also conclude that (for $j$ large enough) the CMC graphs $u_j^1$ have mean curvatures pointing in one direction (the same as $u_m$) and the CMC graphs $u_j^q$ have mean curvatures pointing in the opposite direction (the same as $u^{q}$). In view of this, at $x \in \Greg{V}\cap \text{sing}V$ we can choose the sheet $u_j^{1}$ (as top sheet, for the intrisic orientation, above $u_1$) and select the sheet $u_j^{q}$ as the top sheet on $u^{q}$. On the other hand, when $x \in \Greg{V} \setminus \text{sing}V$ then we select the top sheet for the natural orientation, i.e.\ $u_j^{q}$.

\textit{Step 3: selection of the global sheet (away from the neck region).} In this step we will paste together the sheets selected in the previous step in order to produce a global sheet, which will be an immersion $i_j$ of $\Sigma$ close to $i$ in $C^2$-norm (more precisely $i_j \to i$ as smooth immersions on $\Sigma$). We will deal only with alternative (i), since in the case of alternative (ii) the only difference lies in the fact that we only need to perform the same procedure on a subcollection of balls (note that the convergence is a local matter and only needs to be checked in each ball).

We assume that $V$ is CMC but not minimal (this is the harder case since otherwise touching of sheets is ruled out). If $x \in \Greg{V} \setminus \text{sing}V$ we have chosen in step 2 the top sheet describing $\spt{V_j}$ (in the cylinder of size $r_x$ at $x$) with respect to the orientation induced by the mean curvature vector; if $x \in \Greg{V}\cap \text{sing}V$ we have chosen $u_j^1$ and $u_j^{q}$. We will see that this selection can be pasted together to produce an immersion $i_j$. 

The sheets are given, in step 2, as graphs on $T_x V$. On the other hand they are also graphs on $V$ with respect to the graphing direction given by the unit normal to $V$ (we take the orientation induced by the mean curvature), as the support of $V_j$ is contained in the tubular neighbourhood of $\Greg{V}$ chosen in the beginning and we only need to restrict to $\Greg{V} \cap B_{r_x}(x)$ (this is where we lose a factor $C$ in the size of the region where we have a graphical representation). The graphs of $u^j_1:T_xV \to  (T_xV)^\bot$ and $u^j_q:T_xV \to  (T_xV)^\bot$ will be represented, with respect to the normal bundle coordinate system, as smooth functions $n^j_1:i(\Sigma)\cap B_{r_x}(x) \to  \R  \hat{\nu}$ and $n^j_q:i(\Sigma)\cap B_{r_x}(x) \to  \R  \hat{\nu}$. 

Remark that when $x \in \text{reg}V$ then the counter image $i^{-1}\left(B_{r_x}(x) \cap i(\Sigma)\right)$ is made of a (simply) connected open set, when $B_r$ is centred on $\text{sing}_T V$ then the counter image $i^{-1}\left(B_{r_x}(x) \cap i(\Sigma)\right)$ is made of two (simply) connected components (as $r_x$ were chosen so that the is sheeting for $V$ at scale $r_x$ with the number of sheets equal to the density $q_x$). Consider the set $\tilde{L}=\{i^{-1}(x):x\in L\}$. The set $\tilde{L}$ has cardinality $\sharp \{x \in L: x\in  \text{reg}V\} + 2 \sharp \{x \in L: x\in  \text{sing}_T V\}$: this is the set over which we can index the cover $\{A_p\}_{p \in \tilde{L}}$ of $\Sigma$ obtained by taking the counter-images of $B_x$ via $i$: each element $A_p$ in the cover is diffeomorphic to a ball (in other words, when $x \in \text{sing}_T V$ and so the counter-image is made of two connected components we make them two distinct elements in the cover). We will define the parametrization $i_j$ on each element $A_p$ by setting $i_j=i+n_q^j$ if $i(A_p)=B_{r_x}(x)\cap i(\Sigma)$ is embedded; when $i(p)=x$ is a touching singularity of $V$, on the other hand, then (with $q=q_x$) we set $i_j=i+n_q^j$ if $i(A_p)=\text{graph}(u_q)$ and $i=i+n_1^j$ if $i(A_p)=\text{graph}(u_1)$. In other words we are choosing everywhere the ``top sheet'', taking into account the orientation induced by $\vec{H}$ everywhere on the immersed hypersurface. Whenever $A_p \cap A_b \neq \emptyset$, then $i(A_p)\cap i(A_b) = i(A_p\cap  A_b)$ by construction and on $i(A_p)\cap i(A_b)$ the mean curvature vector identifies uniquely the direction according to which we choose the top sheet. As we have used the system of coordinates induced by the normal bundle, the map $i_j$ is uniquely defined on the overlap.

The $C^2$-convergence of $i_j$ to $i$ only needs to be checked locally, i.e. in each $A_p$, where it follows immediately from the fact that $u^j_1:T_xV \to  (T_xV)^\bot$ and $u^j_q:T_xV \to  (T_xV)^\bot$ both converge as $j\to\infty$ in $C^2$ repspectively to $u^\infty_1:T_xV \to  (T_xV)^\bot$ and $u^\infty_q:T_xV \to  (T_xV)^\bot$ and from the fact that the system of coordinates induced by the normal bundle is smooth and independent of $j$.

\end{proof}

\begin{proof}[\textbf{proof of Theorem \ref{thm:compactness}, third part: condition 5 (stability)}]

Recall that we are working in an open set in which the variation $\zeta$ (for which we want to check the stability) is supported, so $i(\Sigma) \subset \subset \Greg{V}$. It is convenient to write the stability inequality for the immersion $i:\Sigma \to \mathcal{U}$ using a test function $\zeta$ defined on $\Sigma$ rather than in $\mathcal{U}$. Let $A$ denote the second fundamental form of the immersion ($A$ is defined on $\Sigma$ here). The stability inequality then reads

\begin{equation}
 \label{eq:stabilityimmersedversion}
\int_{\Sigma} |A|^2 \zeta^2 dvol_g \leq \int_{\Sigma} |g^{-1} d\zeta|_g^2 dvol_g
\end{equation}
where $g$ is the metric on $\Sigma$ induced by the immersion, namely $g(X,Y) = \langle di(X), di(Y) \rangle$ (the scalar product is the one induced by the ambient metric in $\mathcal{U}$ - in this case the euclidean metric) . The constraint for $\zeta$ is that it must induce a volume-preserving deformation of the immersion and so it reads $\int_{\Sigma} \zeta dg =0$.  On the right-hand side $g^{-1} d\zeta$ is the contravariant version of the exterior differential on $(\Sigma, g)$.

If alternative (i) of Lemma \ref{lem:neckregion} holds then we have a subsequence (not relabeled) $V_k$ such that for every $k$ there exists an immersion $i_k:\Sigma \to \mathcal{U}$ of class $C^2(\Sigma)$ such that $i_k(\Sigma) \subset \Greg{V_k}$ and $i_k \to i$ in $C^2$ as $k \to \infty$. In this case we aim to prove the weak stability inequality for a volume-preserving variation on $\Greg{V}$ with initial speed $\phi \nu$, for any given $\phi\in C^1_c(\mathcal{U}\setminus (\text{sing}\,V \setminus \text{sing}_T\,V))$ such that $\int_{\Greg{V}} \phi=0$, by passing to the limit the stability inequalities for the $V_k$'s thanks to the ``global sheeting''. We must pay attention to the fact that we cannot write the stability inequality for $V_k$ for the given $\phi$ because $c_k:=\int_{\Greg{V_k}} \phi$ is not necessarily $0$; however we have $c_k \to 0$ as $k\to \infty$. We therefore pick a point $p\in \Reg{V}$ that lies in the complement of $\text{spt}\,\phi$ and consider an open ball $B_R(p)$ all contained in the complement  of $\text{spt}\,\phi$. In $B_R$ we have convergence with sheeting of the $V_k$'s and we pick an ambient function $\zeta\in C^1_c(B_R(p))$ such that $\zeta =1$ on $B_{R/2}(p)$ and $|\nabla \zeta| \leq \frac{2}{R}$. A suitable constant $a_k \in \R$ ensures that $\int_{\Greg{V_k} \cap B_R(p)} a_k \zeta = -c_k$ (note that $c_k \to 0$ forces $a_k \to 0$) and therefore the ambient test function $\phi+a_k \zeta$ can be plugged in the weak stability inequality for $V_k$:

$$\int_{\Greg{V_k}} |\phi+a_k \zeta|^2 |A_k|^2 \leq \int_{\Greg{V_k}} |\nabla \phi+ a_k \nabla \zeta|^2 .$$
The disjointness of $\text{spt}\,\phi$ and $\text{spt}\,\zeta$ imply

$$\int_{\Greg{V_k}} |\phi|^2 |A_k|^2 +a_k^2\int_{\Greg{V_k}} |\zeta|^2 |A_k|^2 \leq \int_{\Greg{V_k}} |\nabla \phi|^2 + a_k^2 \int_{\Greg{V_k}}| \nabla \zeta|^2 $$
and $a_k \to 0$, $|\nabla \zeta|\leq \frac{2}{R}$ and the $C^2$ convergence give that the second term on each side of the inequality goes to $0$ in the limit, so
$$\int_{\Greg{V_k}} |\phi|^2 |A_k|^2  \leq \int_{\Greg{V_k}} |\nabla \phi|^2 .$$
We thus conclude that (\ref{eq:stabilityimmersedversion}) holds if alternative (i) of Lemma \ref{lem:neckregion} is satisfied.

If alternative (ii) of Lemma \ref{lem:neckregion} holds then for the ``neck point'' $y$ consider $i^{-1}(\{y\})$, which is given by either a single point or two points (respectively when $y \notin Sing_T V$ and when $y \in Sing_T V$). Let $\chi$ be a cutoff function that vanishes in a neighbourhood $\mathcal{O}$ of $i^{-1}(\{y\})$ and is identically $1$ away from a neighbourhood of $i^{-1}(\{y\})$. Choose $R$ small enough so that $i^{-1}(B_R(x)) \subset \mathcal{O}$ and consider the corresponding sequence $V_{j'}$. Then we have immersions $i_{j'}:\Sigma \setminus i^{-1}(B_R(y)) \to \mathcal{U} \setminus B_R(y)$ of class $C^2$ such that $i_{j'}(\Sigma \setminus i^{-1}(B_R(y))) \subset \Greg{V_{j'}}$ and $i_{j'} \to i$ in $C^2(\Sigma \setminus i^{-1}(B_R(y)))$ as $j' \to \infty$. The stability inequality holds for the immersions $i_{j'}$ even without the zero-average constraint on the test functions, since we proved that the immersions $i_{j'}$ are stable in the strong sense, i.e. for the functional $J$. Let $\tilde{\zeta}=\zeta \chi$, this is smooth and compactly supported in $\Sigma$ and such that $i^{-1}(\{y\})$ does not lie in the support, so we can plug $\tilde{\zeta}$ in the (strong) stability inequality for $i_{j'}$. 
Passing to the limit as $j' \to \infty$ thanks to the $C^2$ convergence we obtain that (\ref{eq:stabilityimmersedversion}) holds for the immersion $i$ with such $\tilde{\zeta}$. Upon letting $\chi$ converge to $1$, a standard capacity argument leads to the validity of (\ref{eq:stabilityimmersedversion}) with the test function $\zeta$ (one or two points have zero capacity for the inequality). Note that in the case of alternative (ii) we actually obtain even more than we wanted, as $\zeta$ does not need to have zero-average, i.e. we have proved (\ref{eq:stabilityimmersedversion}) for variations that are not necessarily volume-preserving.

The proof of \ref{thm:compactness} is thus complete, as inequality (\ref{eq:stabilityimmersedversion}) for zero-average test functions is equivalent to the  variational stability under volume-preserving deformations of $\Greg{V}$.

\end{proof}

\section{Proof of Corollary \ref{cor:Caccioppolianalytic}}
\label{Caccioppolicorollariesproofs}

The key argument for the proof consists in showing how the stationarity assumption taken with respect to the $L^1_{\text{loc}}$-topology on Caccioppoli sets allows to rule out 

(i) the existence of classical singularities in $|\p^* E|$,

(ii) the existence of touching singularities $p \in \Greg{|\p^* E|}$ such that, locally around $p$, $\text{Sing}_T\,|\p^* E| \cap \Greg{|\p^* E|}$ is an $(n-1)$-dimensional submanifold (that locally disconnect $\Greg{|\p^* E|}$).

\medskip

\textit{(i) Ruling out classical singularities}. Assume that $|\p^* E|$ has a classical singularity $p$, namely there exists three or more $C^{1,\alpha}$ hypersurfaces-with-boundary that intersect only along their common $C^{1,\alpha}$-boundary $L$ and such that $|\p^* E|$ is described, in a neighbourhood $B$ of $p$, by the union of these hypersurfaces-with-boundary (clearly $p\in L$). Set coordinates so that the tangent to $L$ at $p$ is the subspace $\{x_1=0, x_2=0\}$ and so that the two of the hypersurfaces-with-boundary intersect transversely at $p$ with tangents $x_2 = a x_1$ and $x_2 = -a x_1$, where $a>0$, and no other hyper-surface with boundary intersects the open set $\{-a x_1 <x_2 <a x_1\}$ in the chosen neighbourhood $B$. Let $A$ denote the connected component of $E \cap B$ whose boundary is given by the restriction to $B$ of the two selected hypersurfaces-with-boundary. 

Consider the vector field $\partial_{x_1}=(1,0,...0)$ and fix a compactly supported smooth (ambient) function $\Psi\in C^1_c(B)$ such that $\Psi\geq 0$ and $\Psi=1$ in a smaller neighbouhood of $p$. The choice of coordinates made ensures, in the upcoming argument, that the vector field $\partial_{x_1}=(1,0,...0)$ pushes $A$ to its interior near $p$, for small positive $t$. More precisely, we consider $E_t=(E\setminus A) \cup (Id+t \Psi \partial_{x_1})(A)$ as one parameter family of deformation of $E$ in $B$ for $t\in[0,\eps)$. The fact that $t>0$ guarantees that (the interiors of) $(E\setminus A)$ and $(Id+t \Psi\partial_{x_1})(A)$ stay disjoint, so that we have a well-defined Cacioppoli set for $t\in[0,\eps)$ and this is a continuous curve in the $L^1_{\text{loc}}$-topology. Set $\varphi_{t}=Id+t \Psi \partial_{x_1}$.
 
The stationarity condition implies the validity of the following inequality: 
 $$\left.\frac{d}{dt}^+\right|_{t=0} J(E_t) \geq 0, \text{ where } t\in [0,\eps)\to \chi_{E_t}, \text{ for $E_t$ constructed above}. \,\,\, (\dagger)$$ 
At the level of the stationarity assumption, such a ``one-sided'' deformation is actually the only requirement that we must allow in the assumptions of Corollary \ref{cor:Caccioppolianalytic} when it comes to non-ambient deformations\footnote{This is clearly weaker than requiring $\left.\frac{d}{dt}\right|_{t=0}=0$ for all deformations parametrized on $t\in(-\eps, \eps)$. Moreover we do not need to worry about how we continue the deformation when we go to negative $t$, where overlapping of $E \setminus A$ and $(Id+t \Psi\partial_{x_1})(A)$ would happen and we would have to redefine $E_t$ for $t<0$.}.

Assume $(\dagger)$ for the deformation constructed, i.e.~$t\to E_t=(E\setminus A) \cup (Id+t \Psi \partial_{x_1})(A)$, where $\chi_{E_t}=\chi_{E\setminus A} + \chi_{(Id+t \Psi \partial_{x_1})(A)}$ and $t\in [0,\eps)$. Since $|D\chi_{E_{t}}| = |D\chi_{\varphi_{t}(A)}| + |D\chi_{E \setminus A}|$ in this case and the volume enclosed by $E_t$ is the sum of the volumes enclosed by $E\setminus A$ and by $(Id+t \Psi \partial_{x_1})(A)$, noting that $E\setminus A$ does not change in the deformation we conclude that $(\dagger)$ holds also for $A$ under the deformation $t \to (Id+t \Psi \partial_{x_1})(A)$ for $t\in [0,\eps)$.

Set $X=\Psi \partial_{x_1}$; we will contradict the inequality $(\dagger)$ for $A$ with the choice of one-parameter family of deformations of $A$ in $B$ induced by $X$ for $t>0$ (as decribed above). Take a tangent cone at $p \in L$ (with rescaling factors $r_k \to 0$): we get a sequence $A_k \to C$ where $C$ is made of two half-hyperplanes meeting transversely at their common boundary. The first variations  $\delta_{A_k}$ converge to the first variation of $C$ as functionals on $C^1_c$ vector fields $X$, $\delta_{A_k}(X) \to \delta_{C}(X)$. Let $\Psi_r (\cdot)= \Psi(\frac{\cdot}{r})$, where $\Psi \in C^1_c(B_1)$. By homothetic rescalings 

$$\delta_{A_k}(\Psi \p_{x_1})= \frac{1}{{r_k}^{n-1}} \delta_A (\Psi_{r_k} \p_{x_1}).$$
By convergence of the first variations 
$$\delta_{A_k}(\Psi \p_{x_1}) \to \delta_{C}(\Psi \p_{x_1})=\int -c(x) \Psi(y) \vec{w}\cdot \p_{x_1} d\mathcal{H}^{n-1}(y)\,\, \text{ with } c(x)> 0,$$
where the latter comes from the analysis of the $1$-dimensional situation arising for $C=V\times \R^{n-1}$; here $V$ is $1$-dimensional in the plane spanned by the coordinates $x_1, x_2$, it is made of two segments with a common extremal point, joined at an angle $2\alpha=2\tan(a)$, the first variation of $V$ is a Dirac-$\delta$ at the corner with weight equal to $c(x)=2\cos(a)$ and lenght decreases if we push to the inside of $V$.

From $(\dagger)$ for $A$, we get the inequality (writing the first variation of the enclosed volume)

$$\frac{1}{r^{n-1}} \delta_A (\Psi_r \p_{x_1}) \geq -\frac{1}{r^{n-1}} \lambda \int \Psi_r \p_{x_1}\cdot \vec{\nu} d\mathcal{H}^n \res \p^*A ;$$
writing this inequality for $r=r_k$ and sending $k\to \infty$ we get (the right-hand-side goes to $0$ because the integral is of order $r_k^n$) we get

$$-c(x)\geq 0, \,\,\, \text{ contradiction}.$$

\medskip
 
\textit{(ii) Ruling out a touching set of dimension $(n-1)$.} This can be done as for classical singularities: note that the topological structure now is exactly the same as in the case of a classical singularity, therefore in a ball $B$ around a touching singularity $p$ we can consider $(B\setminus E)$, which gets disconnected when removing $\text{Sing}_T \,V$, so a deformation of one of its connected components of the type exhibited above will produce the desired contradiction (note that in this case $c(p)=2$). 

\medskip

In view of the fact that a $C^2$ CMC hypersurface is actually analytic, two CMC hypersurfaces that intersect tangetially must necessarily intersect along a finite union of (analytic) submanifolds with integer dimensions between $0$ and $(n-1)$; therefore, once we have ruled out the possilbility that said intersection is locally a submanifold of dimension $(n-1)$, we immediately have that $\text{Sing}_T\,|\p^* E| \cap \Greg{|\p^* E|}$ is a finite union of submanifolds of dimensions between $0$ and $(n-2)$. In particular 
\begin{equation}
 \label{eq:analyticdimensional}
\mathcal{H}^{n-2}\left(\text{Sing}_T\,|\p^* E| \cap \Greg{|\p^* E|}\right)<\infty \text{ locally.}
\end{equation}
Recall that, for any $x\in \text{spt}\,|\p^* E|$ there exists $B^{n+1}_r(x)$ such that the strong stability inequality holds for any $\phi \in C^1_c(B^{n+1}_r(x) \setminus \text{Sing} \,|\p^* E|)$ (Remark \ref{oss:weakimpliesstrongforJ}). Thanks to (\ref{eq:analyticdimensional}) and to a slightly refined capacity argument (see \cite[7.4.2]{EvansGar}) we can then show that the strong stability inequality holds for any $\phi \in C^1_c(B^{n+1}_r(x) \setminus (\text{Sing} \,|\p^* E| \setminus \text{Sing}_T \,|\p^* E|))$; equivalently, for any $x\in \text{spt}\,|\p^* E|$ there exists $B^{n+1}_r(x)$ such that the strong stability inequality holds for $\Greg{|\p^* E|}$ in $B^{n+1}_r(x)$, and not just for $\Reg{|\p^* E|}$. Corollary \ref{cor:Caccioppolianalytic} now follows immediately from Corollary \ref{cor:CaccioppoliwithoutLp} or from Theorem \ref{thm:mainregularity}.

\appendix
\section{Two constancy lemmas for integral varifolds}
\label{appendixconstancy}

We collect in this appendix two general facts regarding the density of integral $n$-varifolds with generalized mean curvature in $L^p_{\text{loc}}(\spt{V})$ for some $p>n$. These statements should be implicitly born in mind when reading the statement of the main Theorem \ref{thm:mainregularity}.

\begin{lem}[first constancy lemma]
\label{lem:constancyfirst}
Let $V$ be an integral $n$-varifold in an open set $\mathcal{U} \subset \R^N$ and assume that $\spt{V}=M$, where $M$ is a $C^1$ connected submanifold of dimension $n$ in $\mathcal{U}$ and that the generalized mean curvature $\vec{H}$ of $V$ is in $L^p_{\text{loc}}(\spt{V})$ for some $p>n$. Then $V$ is the varifold of integration on $M$ with a \textit{constant} integer multiplicity, namely $V=(M, \theta_0)$ for a certain $\theta_0 \in \N$.
\end{lem}

\begin{oss}
We need to apply this result to the embedded $C^{1, \alpha}$ part of our varifold (that is non-empty by Allard's result): we could then deduce the constancy of the density from \cite{Duggan}, where however the varifold is rectifiable (and not necessarily integral) and the conclusions are that the support is actually $W^{2,p}$ and $\theta$ is $W^{1,p}$ (thus a fortiori continuous). The proof in \cite{Duggan} is however rather lengthy and non-trivial: the extra assumption of integrality on the varifold allows us to give the following very direct proof. 
\end{oss}

\begin{proof}[proof of Lemma \ref{lem:constancyfirst}]
Thanks to the assumption that the generalized mean curvature is in $L^p_{\text{loc}}(\spt{V})$ for $p>n$ we have the monotonicity formula \cite[\S 17]{SimonNotes} and in particular its consequence \cite[Corollary 17.8]{SimonNotes} that the density $\theta$ is \textit{everywhere} well-defined and it is an \textit{upper semi-continuous} function. The assumption that the support is a $C^1$-submanifold $M$ forces the support of any tangent to the varifold at $x \in M$ to be the unique tangent plane to $M$ at $x$. Moreover each tangent is a stationary varifold by the assumption on $\vec{H}$ - this forces each tangent to be the tangent plane to $M$ endowed with constant multiplicity\footnote{The fact that a plane with multiplicity is stationary if and only if the multiplicity is constant can be directly checked very easily, so this does not really require the stronger statement of \cite[\S 41]{SimonNotes}.}. Since by the compactness theorem for integral varifolds \cite[Theorem 42.7-42.8]{SimonNotes} we have that every tangent is itself an integral varifold, then this multiplicity is an integer and therefore $\theta$ is \textit{everywhere} integer valued\footnote{The proof of this lemma can go through with minor modifications even without noticing that the density is everywhere integer-valued, but this observation makes it a bit more slender.}.

Consider an arbitrary point $p$ and let $Q \in \N$ be its density: by upper semicontinuity there exists $B^N_\rho(p) \subset \subset  \mathcal{U}$ in which the density is $\leq Q$ for all points. We wish to conclude that the density of $V$ is exactly $Q$ in a neighbourhood of $p$. If this is not the case, then the \textit{closed} set 

$$\Sigma_Q=\{x\in B^N_\rho(p) : \theta(x)=Q \} ,$$
is not the whole of $M \cap B^N_{\rho/3}(p)$. Its complement $B^N_\rho(p) \setminus \Sigma_Q$ is then an open and non empty set in $M \cap B^N_\rho(p)$, so we can choose $y \in (M \cap B^N_{\rho/3}(p)) \setminus \Sigma_Q$. Consider the open balls centered at $y$ and contained in  $(M \cap B^N_\rho(p)) \setminus \Sigma_Q$: take the sup $d$ of the radii of such balls. Then $\p B_d(y)$ intersects $\Sigma_Q$ (actually interesects its boundary) at least at one point $q$ and  $B_d(y) \cap \Sigma_Q = \emptyset$. 

We now consider $q$, where we have density $Q$, and perform a blow up: by the monotonicity formula we have $Q = \lim_{\sigma\to 0} \frac{\|V\|(B^N_\sigma(q))}{\om_n \sigma^n}$, where $\om_n$ denotes the volume of the unit $n$-dimensional ball. However by the $C^1$ assumption we can assume without loss of generality that $M \cap B^N_\rho(p)$ is given by the graph of a $C^1$ function $g:\R^n \to \R^{N-n}$ with $Dg(\pi(q))=0$ ($\pi$ is the projection on the first factor $\R^n$) we get

$$\|V\|(B_\sigma^N(q))=\|V\|(A_\sigma \times \R^{N-n}) = \int_{A_\sigma} (\theta \circ g) |\text{Jac}\,(Id,g)|,$$
where $A_\sigma = \pi(B_\sigma^N(q) \cap M)$. By the $C^1$ regularity we have $||\text{Jac}\,(Id,g)| - 1|\leq C_n|Dg|^2 = 1+o(\sigma)$ and $\frac{|A_\sigma|}{\om_n \sigma^n} \to 1$ as $\sigma \to 0$. The existence of the ball $B^N_d(y)$ where $\theta \leq Q-1$ and such that $q \in \p B_d(y)$ implies that 
$$\limsup_{\sigma \to 0}\frac{\|V\|(B^N_\sigma(q))}{\om_n \sigma^n} \leq \frac{Q-1}{2} + \frac{Q}{2}=Q - \frac{1}{2},$$
contradiction.
\end{proof}

\begin{lem}[second constancy lemma]
\label{lem:constancysecond}
Let $V$ be an integral $n$-varifold in an open set $B^n_R(0) \times \R$ and assume that $\spt{V}=M_1 \cup M_2$, where each $M_j$ is the graph of a $C^{1,\alpha}$ function $u_j:B^n_R(0) \to \R$ and that the generalized mean curvature $\vec{H}$ of $V$ is in $L^p_{\text{loc}}(\|V\|)$ for some $p>n$. Let $T:=\{(x, t) \in B^n_R(0) \times \R: u_1(x) = u_2(x) =t\}$. Denote by $\pi:B^n_R(0) \times \R \to B^n_R(0)$ the standard projection. For $x \in B^n_R(0) \setminus \pi(T)$ let $s(x)=\Theta(\|V\|, (x, u_1(x))) + \Theta(\|V\|, (x, u_2(x)))$ and for $x \in \pi(T)$ let $s(x)=\Theta(\|V\|, (x, u_1(x)))$. Then there exists $N\in\N$ such that $s=N$ everywhere on $B^n_R(0)$.
\end{lem}

\begin{proof}[proof of Lemma \ref{lem:constancysecond}]
Recall that the density $\Theta(\|V\|, \cdot)$ is everywhere well-defined thanks to the condition that $\vec{H}$ is in $L^p(\spt{V})$, in view of the monotonicity formula \cite[\S 17]{SimonNotes}. Note that if $C$ is a stationary integral $n$-varifold supported on the union of two transversal planes intersecting along an $(n-1)$-subspace then, setting coordinates so that $\spt{C}=\{x^{n+1}=L_1(x)=a x^1\} \cup \{x^{n+1}=L_2(x)=b x^1\}:=P_1 \cup P_2$ for some $a\neq b \in \R$, the density $\Theta(\|C\|, x)$ is constant on each of the four half-planes $\{x^{n+1}=L_1(x), x^1>0\} , \{x^{n+1}=L_1(x), x^1<0\} , \{x^{n+1}=L_2(x), x^1>0\}, \{x^{n+1}=L_2(x), x^1<0\}$ by Lemma \ref{lem:constancyfirst}, respectively with values $\theta_1^+$, $\theta_1^-$, $\theta_2^+$, $\theta_2^-$. The stationarity of $C$ implies the vanishing of the first variation along $\{x^1=0, x^{n+1}=0\}$ and therefore $\theta_1^+-\theta_1^-=0$ and $\theta_2^+-\theta_2^-=0$. Then $C=\theta_1 |P_1| + \theta_2 |P_2|$ for some $\theta_1, \theta_2 \in \N$. This implies that $\Theta(\|C\|, (x, 0))=\theta_1+\theta_2 \in \N$ for every $x$ such that $x^1=0$. By the assumptions on $\spt{V}$, any tangent to $V$ is supported on the union of two planes, possibly coinciding, so it is either a cone of the type just described or it is a plane with constant multiplicity (again by stationarity, similarly to Lemma \ref{lem:constancyfirst}). This implies that $\Theta$ is everywhere integer-valued on $\spt{V}$ and as a consequence $s$ is integer-valued.

On the complement of $T$ the varifold is supported on a $C^1$ submanifold, so $\Theta(\|V\|, \cdot)$ is locally constant by Lemma \ref{lem:constancyfirst}, i.e. $\Theta(\|V\|, \cdot)$ is constant on each connected component of $(M_1 \cup M_2)\setminus T$. Moreover $\Theta(\|V\|, \cdot)$ is upper semi-continuous, and in particular the restriction of $s$ to $\pi(T)$ is upper semi-continuous. Choose an arbitrary $x \in \pi(T)$ and fix a neighbourhood $B^n_\sigma(x)$ so that $s|_{\pi(T) \cap B^n_\sigma(x)} \leq Q$. We wish to prove, as a first step, that $s$ is upper semi-continuous in $B^n_{\sigma/3}(x)$. Arguing by contradiction, we assume that this fails and choose $y \in B^n_{\sigma/3}(x)$ with $s(y)\geq Q+1$ $y \notin T$. Since $B^n_{\sigma}(x) \setminus \pi(T)$ is open, we can choose balls centred at $y$ and disjoint from $\pi(T)$ and take the sup $d$ of the radii of such balls. Then $B_d^n(y) \subset B^n_{\sigma}(x) \setminus \pi(T)$ and $\p B_d^n(y)$ intersects $\pi(T)$ at least at a point $q$, which lies in $B^n_{\sigma}(x)$ by the choice of $y$. The fact that $B_d^n(y)$ lies in a single connected component of $B^n_{\sigma}(x) \setminus \pi(T)$ implies that $s$ is constant on it, i.e. $s \geq Q+1$ on $B_d^n(y)$. However $q \in \pi(T)\cap B^n_\sigma(x)$ and so $s(q)\leq Q$: take the unique point on the varifold above $q$, namely $(q, u_1(q))$, and blow up at this point: $s(q)\leq Q$ this must give as tangent varifold either a plane with constant multiplicity $\leq Q$ or the sum of two transversal planes with constant multiplicities that add up to a value $\leq Q$. However the presence of $B_d^n(y) \times \R$ in which the multiplicities of the two sheets add up to an integer $\geq Q+1$ provides a contradiction, since in the limit we would find either a half-plane counted with multiplicity $\geq Q+1$ or the sum of two transversal planes with constant multiplicities that add up to a value $\geq Q+1$.

With the upper semi-continuity of $s$ we can conlude the proof as follows. Choose $x \in \pi(T)$ with density $Q \in \N$ and fix a neihbourhood $B^n_\sigma(x)$ such that $s \leq Q$ in it. The set 
$$\Sigma_Q=\{z \in B^n_\sigma(x): s(z)=Q\}$$
is closed and we wish to conclude that it contains a neighbouhood of $x$, namely $B^n_{\sigma/3}(x)$. If that is not the case then its open complement $B^n_\sigma(x) \setminus \Sigma_Q$ contains a point $y \in B^n_{\sigma/3}(x) \setminus \Sigma_Q$. Choose the sup $d$ of the radii of balls centred at $y$ and contained in $B^n_\sigma(x) \setminus \Sigma_Q$: then $B_d^n(y) \subset B^n_\sigma(x) \setminus \Sigma_Q$ and $\p B_d^n(y)$ intersects $\Sigma_Q$ at least at a point $q$, which lies in $B^n_{\sigma}(y)$ by the choice of $y$. So the value of $s$ at $q$ is $Q$ and on $B_d^n(y)$ we have $s \leq Q-1$. If $q \notin \pi(T)$ then we would have a contradiction, in that $s$ is jumping in the interior of $B^n_\sigma(x) \setminus T$, where we know that it is locally constant. If $q \in \pi(T)$ then we perform a blow-up at $q$ which must give as tangent varifold either a plane with constant multiplicity $Q$ or the sum of two transversal planes with constant multiplicities that add up to a value $\geq Q+1$. However the presence of $B_d^n(y) \times \R$ in which the multiplicities of the two sheets add up to an integer $\leq Q-1$ provides a contradiction, since in the limit we would find either a half-plane counted with multiplicity $\leq Q-1$ or the sum of two transversal planes with constant multiplicities that add up to a value $\leq Q-1$. So we have concluded that $s$ is locally constant at every point on $\pi(T)$, which finishes the proof.
\end{proof}

\section{Strong stability for CMC graphs}
\label{stabilitygraphs}

The following result is well-known for minimal graphs and is likely to be known to experts in the case of CMC graphs. Due to its importance within the proof of Theorem \ref{thm:SS} (and to non-immediate availability of its proof in the literature) we present in this appendix the statement and two proofs.

\begin{Prop}
Let $f:\Om \to \R$ be a $C^2$ function (on $\Om\subset \R^n$ open) such that the graph $M$ of $f$ is a CMC hypersurface in $\Om \times \R \subset \R^{n+1}$. Let $H$ denote the mean curvature.
Then the strong stability inequality

$$\int |A|^2 \phi^2 d\mathcal{H}^n \leq \int |\nabla_M \phi|^2 d\mathcal{H}^n$$
holds for every $\phi\in C^\infty_c(M)$, where $\nabla_M$ is the gradient on $M$.
\end{Prop}

\begin{proof}[first proof - semi-calibration argument]
This argument actually shows something more, namely that $M$ is a minimizer of $J$ for its boundary. Let $N$ be a hypersurface with the same boundary and let $L$ be a $(n+1)$-dimensional submanifold with boundary $M-N$. Denote by $\nu$ the unit normal vector to $M$ such that $\nu \cdot e^{n+1} >0$ and extend $\nu$ to a vector field on $\Om \times \R$ by setting $\nu(x,y) = \nu(x,f(x))$, for $x\in \Om$ and $y \in \R$ (vertically invariant extension). Define the $n$-form $\om=\iota_{\nu} d\text{vol}^{n+1}$, where $d\text{vol}^{n+1} = dx^1 \wedge \ldots \wedge dx^{n+1}$ is the standard volume form on $\R^{n+1}$ and $\iota$ denotes the inner product (contraction). Then $d\om = \text{div} \nu \, d\text{vol}^{n+1}= H\, d\text{vol}^{n+1} $, where $H$ is the constant mean curvature of $M$. Note also that $\int_M \om = \mathcal{H}^n(M)$ and that $\om$ has comass $1$. Now using Stokes' theorem we compute

$$\mathcal{H}^n(M)=\int_M \om = \int_N \om + \int_L d\om = \int_N \om + H \int_L d\text{vol}^{n+1} \leq \mathcal{H}^n(N)+ H \int_L d\text{vol}^{n+1}.$$
Since $J(N)=\mathcal{H}^n(N) - H \text{vol}(N) = \mathcal{H}^n(N) - H ( \text{vol}(M) - \int_L d\text{vol}^{n+1})$ we obtain 
$$J(N) \geq \mathcal{H}^n(M)- H  \text{vol}(M) =J(M),$$
so $M$ is a minimizer of $J$ for its boundary $\p M$ and, in particular, $J''(\phi)\geq 0$ for any $\phi\in C^\infty_c(M)$, as desired. 

Note that, if $N$ encloses the same volume as $M$, the inequality above says that $\mathcal{H}^n(M)\leq \mathcal{H}^n(N)$, i.e.~$M$ minimizes area among hypersurfaces with the same boundary and with the same enclosed volume.
\end{proof}

\begin{proof}[second proof - Jacobi fields argument]
Let $u:=\nu \cdot e^{n+1}$, where $\nu$ is the unit normal to $M$ chosen so that $u$ is a positive function on $M$. We will prove that $u$ satisfies $\Delta_M u = -|A|^2 u$, where $\Delta_M$ is the Laplacian on $M$. We will compute at an arbitrary point $x$, in a system of normal coordinates around $x$. Let $v_i$ ($i=1, ..., n$) be a orthonormal basis for $M$.

$$\Delta_M u = \nabla_{v_i}  \nabla_{v_i} u =  \nabla_{v_i}  \nabla_{v_i} \langle \nu, e^{n+1} \rangle =  \nabla_{v_i} \langle  \nabla_{v_i}\nu, e^{n+1} \rangle=$$
$$= -\nabla_{v_i} \langle  A_{ij}v_j, e^{n+1} \rangle = - \langle  (\nabla_{v_i}A_{ij}) v_j, e^{n+1} \rangle - \langle  A_{ij} (\nabla^{\R^3}_{v_i}v_j), e^{n+1} \rangle=$$
and using Codazzi equations for the first term and the fact that the tangential component of $\nabla_{v_i}v_j$ is $0$ (we are in normal coordinates) for the second term we continue the chain of equalities as follows:
$$= - \langle  (\nabla_{v_j}A_{ii}) v_j, e^{n+1} \rangle - \langle  A_{ij} A_{ij} \nu, e^{n+1} \rangle=- \langle  (\nabla_{v_j}H) v_j, e^{n+1} \rangle - |A|^2 u =$$
$$ = - |A|^2 u,$$
where, in the last step, we used the constancy of $H$. The function $w=\log u$ (well-defined since $u>0$) satisfies, by an algebraic manipulation,

$$\Delta_M w = \frac{\Delta_M u}{u} - |\nabla_M w|^2;$$
in view of $\Delta_M u = -|A|^2 u$, we have immdiately

$$|A|^2 + |\nabla_M w|^2=-\Delta_M w .$$
Multiplying this identity by $\phi^2$ and integrating by parts we obtain

$$\int |A|^2 \phi^2+ \int |\nabla_M w|^2 \phi^2=-\int \Delta_M w \,\phi^2 \leq 2 \int |\phi| |\nabla_M w| |\nabla_M \phi| \leq \int |\phi|^2 |\nabla_M w|^2 +\int  |\nabla_M \phi|^2$$
and simplyfing we conclude, as desired, that
$$\int |A|^2 \phi^2  \leq  \int  |\nabla_M \phi|^2.$$
\end{proof}

\section{On the distributional derivative}
The following fact is used twice in the inductive step for the proof of the higher regularity Theorem \ref{thm:higher-reg}.

\begin{lem}
\label{lem:distributionalder} 
Let $B \subset \R^n$ be an open ball and $T \subset B$ be a closed subset. Assume that $f:B \to \R$ is a continuous function with the following properties: $f=0$ on $T$, $f|_{B \setminus T}$ is $C^1$ and $Df$ is summable on $B \setminus T$ (i.e. $\int_{B \setminus T} |D f| <\infty$). Then the distributional derivative of $f$ on $B$ is given by the $L^1$-function on $B$ that is equal to $0$ on $T$ and to $Df$ on $B \setminus T$. 
\end{lem}

\begin{proof}
For $\delta >0$, let $\gamma_{\delta}:\R \to \R$ be a smooth non-decreasing function such that  $\gamma_{\delta}(t) = 0$ for $|t| < \delta/2$ , $\gamma_{\delta}(t) = t - \delta$ for $t > \delta$, $\gamma_{\delta}(t) = t + \delta$ for $t < -\delta$ and $\gamma_{\delta}^{\prime}(t) \leq 1$ for all $t \in {\mathbb R}$. 
Note that $\gamma_\delta(f)$ is $C^1$ on the open set \{$|f|>\delta/4\}$ and is identically $0$ on the open set $\{|f|<\delta/2\}$, so $\gamma_\delta(f)$ is $C^{1}$ on $B$ and so we can integrate by parts and use the chain rule to get, for any $\zeta \in C^{1}_{c}(B)$ and $l \in \{1, 2, \ldots, n\}$,  $$\int_B \gamma_\delta(f) D_l \zeta = -\int_B \gamma_{\delta}^{\prime}(f)D_{l}f\zeta = -\int_{B \setminus T} \gamma_{\delta}^{\prime}(f) D_{l}f \zeta.$$ 
Since $f$ is $C^{1}$ on $B \setminus T$, we have that ${\mathcal H}^{n}(\{x \in B \setminus T \, : \, f(x) = 0, \; Df(x) \neq 0\}) = 0,$ so in view of the fact that $\gamma_{\delta}^{\prime}(f(x)) \to 1$ as $\delta \to 0^{+}$ for every $x \in B \setminus T$ with $f(x) \neq 0$, and $D_{l}f \in L^{1}(B \setminus T)$, it follows from the dominated convergence theorem that $$\int_{B \setminus T} \gamma_{\delta}^{\prime}(f) D_{l}f \zeta\to \int_{B \setminus T} D_{l}f \zeta$$ 
as $\delta \to 0^{+}$. Since  $\gamma_{\delta}(f(x)) \to f(x)$ for every $x \in B$,  again by the dominated convergence theorem  $\int_{B} \gamma_{\delta}(f) D_{\l} \zeta \to \int_{B}f D_{l}\zeta$ as $\delta \to 0^{+}$, so letting $\delta \to 0^{+}$ in the above we conclude that 
$$\int_B f D_l\zeta=- \int_{B \setminus T} D_l f \, \zeta$$
for every $\zeta \in C^{1}_{c}(B).$ From this the conclusion follows. 
\end{proof}

\section{Notation used in the paper}
 
$IV_n(\mathcal{U})$: integral $n$-dimensional varifolds on the open set $\Uc$.
 
 \medskip
 
$\spt{V}$: support of the varifold $V$.

 \medskip
 
$\Reg V$: set of regular (smooth) points of the varifold $V$.

 \medskip
 
$\text{reg}_1 V$: set of $C^1$-embedded points of the varifold $V$.

 \medskip
 
$\text{sing} V$: set of singular points of the varifold $V$ (complement of $\Reg V$ in $\spt{V}$).

 \medskip
 
$\text{sing}_C V$: set of classical singularities of the varifold $V$.

 \medskip
 
$\text{sing}_T V = \text{sing}_T^2 V$: set of (two-fold) touching singularities of the varifold $V$.

 \medskip
 
$\text{sing}_T^\ell V$: set of $\ell$-fold touching singularities of the varifold $V$.

 \medskip
 
$\Greg{V}$: set of generalized regular points of the varifold $V$.

 \medskip
 
$\mathcal{S}_H$: class of integral $n$-varifolds satisfying assumptions 1, 2, 3, 4, 5 of Theorem \ref{thm:mainregularityrestated} with ${\mathcal U} = B_{2}^{n+1}(0)$ and $|h| \leq H$.

 \medskip
 
$\text{vol}_{\mathcal{O}}(W)$: volume enclosed by $V \res \Oc$, for $\Oc$ open such that $V \res \Oc \subset \text{reg}_1 V$ and $V \res \Oc$ is orientable.

 \medskip
 
$\text{A}_{\mathcal{O}}(W)$: total weight of $V$ in $\Oc$, i.e.~$\|V\|(\Oc)$ (hypersurface area in $\Oc$), for $\Oc$ open such that $V \res \Oc \subset \text{reg}_1 V$.

 \medskip
 
$J_{\mathcal{O}}(W)$: for $\lambda \in \R$, this denotes the functional $\text{A}_{\mathcal{O}}(W)+\lambda \text{vol}_{\mathcal{O}}(W)$.

 \medskip
 
$\mathscr{vol}(\mathscr{i})$: for an oriented immersion $\mathscr{i}$, this denotes the enclosed volume.

 \medskip
 
$\mathscr{a}(\mathscr{i})$: for an oriented immersion $\mathscr{i}$, this denotes the hypersurface area.

 \medskip
 
$J(\mathscr{i})$: for $\lambda \in \R$, this denotes the functional $\mathscr{a}(\mathscr{i})+\lambda \mathscr{vol}(\mathscr{i})$.

 \medskip
 
${\rm dist}_{\mathcal H}$: Hausdorff distance.

\bigskip
\hskip-.2in\vbox{\hsize3in\obeylines\parskip -1pt 
  \small 
Costante Bellettini
Department of Mathematics
University College London
London WC1E 6BT, United Kingdom
\vspace{4pt}
{\tt C.Bellettini@ucl.ac.uk}} 
\vbox{\hsize3in
\obeylines 
\parskip-1pt 
\small 
Neshan Wickramasekera
DPMMS 
University of Cambridge 
Cambridge CB3 0WB, United Kingdom
\vspace{4pt}
{\tt N.Wickramasekera@dpmms.cam.ac.uk}}

\end{document}